%% file: iterated_traces.tex
\documentclass[11pt,oneside]{amsart}
\usepackage{fouriernc} 

\usepackage{STY-Definitions}    
\usepackage{STY-PageSetup}      
\usepackage{STY-Environments}   

\usepackage{afterpage}
\usepackage{adjustbox}

\usepackage{STY-text}
\usepackage{STY-circuit_diagrams}
\usepackage{subcaption}
\usepackage{mathtools}

\usepackage{tikz-3dplot}
\usepackage{multirow}
\usepackage{resizegather}

\usetikzlibrary{decorations.markings}
\tikzset{mid vert/.style={/utils/exec=\tikzset{every node/.append style={outer sep=0.8ex}},
postaction=decorate,decoration={markings,
mark=at position 0.5 with {\draw[-] (0,#1) -- (0,-#1);}}},
  mid vert/.default=0.75ex}

\tikzset{shorten <>/.style={shorten >=#1,shorten <=#1}} 
\newcommand{\Sp}{\operatorname{Sp}}

\ifdraft{

\usepackage{STY-DraftSetup}

} 
{}

\hypersetup{
 pdfkeywords={latex starter,template},
 pdfauthor={Niles Johnson},
}

\usepackage{tikz-cd} 
\usetikzlibrary{decorations.pathmorphing}

\subjclass[2010]{18D05,18D20,19K14,37C25}

\title{Iterated traces in 2-categories and Lefschetz theorems}
\date{\today}
\author{Jonathan A. Campbell}
\email{jonathan@math.duke.edu, jonalfcam@gmail.com}
\author{Kate Ponto} 
\email{kate.ponto@uky.edu}

\begin{document}
\maketitle
\begin{abstract}
While not obvious from its initial motivation in linear algebra, there are many context where iterated traces can be defined.  
In this paper we  prove a very general theorem about iterated 2-categorical traces. We show that many Lefschetz-type theorems in the literature are consequences of this result and the new perspective we  provide allows for immediate spectral generalizations. 

\end{abstract}

\setcounter{tocdepth}{1}
\tableofcontents

  \section{Introduction}

In recent years, there has been a proliferation of Lefschetz-type theorems in noncommutative geometry \cite{ben_zvi_nadler_1,ben_zvi_nadler_2,polishchuk,shklyarov,lunts,cisinski_tabuada, hoyois}.
These results are comparisons of invariants and in their simplest form they compare the dimension of the Hochschild homology of a bimodule with the trace of the map induced by tensoring with that bimodule.  For example Lunts \cite{lunts} showed that for a sufficiently nice dg-algebra $A$ and perfect $(A,A)$-bimodule $M$, 
  \[
  \sum_i (-1)^i\dim \HH_i (A;M) = \sum_i \tr\left(\HH_i\left(\Mod^\mathrm{perf}_A\right)\xto{\HH_i (-\otimes M)}\HH_i\left(\Mod^\mathrm{perf}_A\right)\right)
  \]  

In this paper we describe the underlying formal structure for these theorems and demonstrate that they are consequences of the following generalization of the main result of \cite{ben_zvi_nadler_1}.   
\begin{thm}[\cref{thm:main_umbra,thm:monoidal_to_umbra}]\label{thm:main_bzn}
  Let $X, Y$ be endomorphism 1-cells in a symmetric monoidal bicategory $\mc{B}$ where all 0-cells are 2-dualizable (\cref{defn:2_dualizable}).  If $\phi\colon X \odot Y \to Y \odot X$ is a 2-morphism and the traces of $\phi$ with respect to $X$ and $Y$ are defined then
  \[
	\tr_X\left( \tr_Y(\phi)\right)=\tr_Y\left( \tr_X(\phi)\right)
  \]
\end{thm}

There are many examples of bicategories that satisfy the conditions of \cref{thm:main_bzn}.  
Identifying these allows us to demonstrate that seemingly unrelated theorems are immediate consequences of \cref{thm:main_bzn}.
\begin{itemize}
\item   The bicategory of dg-categories and their bimodules recovers the results of \cite{lunts, cisinski_tabuada,polishchuk}.
\item  In the bicategory of spectral categories and their bimodules, \cref{thm:main_bzn} extends the main result of \cite{lunts,cisinski_tabuada} to the following new result.
\end{itemize}  

\begin{thm}[\cref{thm:lunts_2}]
  Let $\mc{A}$ be a smooth, proper spectral category with $\mc{M}$ an $(\mc{A},\mc{A})$-module. Then 
\[
\chi (\thh(\mc{A}; \mc{M})) = \tr (\thh (- \sma \mc{M})) 
\]
where $- \sma \mc{M}$ is a  functor $_{\mc{A}} \Mod_{\mc{A}} \to _{\mc{A}} \Mod_{\mc{A}}$ given by tensoring with $\mc{M}$. 
\end{thm}

In the two examples above the dualizability conditions have similar and familiar flavors.  There are other examples of the bicategories satisfying \cref{thm:main_bzn} where the dualizability conditions are more difficult to describe.  We record these bicategories here, but we do not address the consequences of \cref{thm:main_bzn} in this paper. 
\begin{itemize}
\item Dualizablity in the  bicategory of parameterized spectra \cite{may_sigurdson} 
can be satisfied in the $K(n)$-local category (here $K(n)$ is Morava $K$-theory).
\item 
The 2-category of varieties discussed in \cite{caldararu_mukai} is also an example. Much of the theory of Fourier-Mukai transforms fits inside the framework of this paper. This is current work.
\item Finally, it would also be interesting to know if this formalism can be applied to matrix factorizations (see e.g. \cite{polishchuk_vaintrob}) or knot and link invariants (see e.g.\cite{quantum_link}).
  \end{itemize}

Ben-Zvi and Nadler were motivated by a very different application of \cref{thm:main_bzn}. 
Motivated by \cite{Hopkins-Kuhn-Ravenel}, Ganter and Kapranov  \cite{ganter_kapranov} generalized representations and characters  to 
2-representations and 2-characters.  
The 2-characters of \cite{ganter_kapranov} are iterated traces in the sense of \cref{thm:main_bzn} and in this context, \cref{thm:main_bzn} implies the definition of the 2-character is independent of choices.  

A fundamental observation in representation theory is that the character of the representation is conjugation invariant.
One of the main results of 
 \cite{ben_zvi_nadler_1} shows 2-characters have a much richer invariance than characters.  
\begin{thm}\cite{ben_zvi_nadler_1} \label{lem:bzn_invariance}
2-class functions have an $SL_2(\bZ)$ action and 2-characters are $SL_2 (\Z)$-invariant.
\end{thm}

In \cref{sec:two_characters} we prove a generalization of this theorem while avoiding the cobordism hypothesis.

\subsection*{Comparison to previous perspectives}
There are several important observations to make about \cref{thm:main_bzn} and the proof given here.  Like \cite{ben_zvi_nadler_1}, our proof of \cref{thm:main_bzn} is entirely formal.  In particular, new examples of symmetric monoidal bicategories with 2-dualizable objects would immediately give rise to new Lefschetz theorems.  

There are also two important points of contrast with \cite{ben_zvi_nadler_1}.  We state this result for symmetric monoidal bicategories rather than symmetric monoidal $(\infty,2)$-categories.  This simplifies exposition, but more importantly, a bicategory is the correct context for the examples of interest, since the underlying homotopy bicategory captures all information about dualizability in symmetric monoidal $(\infty,2)$-categories.  
Second, the proof we give here for \cref{thm:main_bzn} requires only minor generalizations of duality and traces in bicategories \cite{may_sigurdson,p:thesis,ps:bicat}.  It does not rely on the cobordism hypothesis. Indeed, it could not, since the cobordism hypothesis does not apply:  we use 2-dimensional data that does not have a manifold analog.

\subsection*{Organization}
The technical elements that go in to proving this theorem are somewhat formidable. There are two main technical hurdles: proving \cref{thm:main_bzn} and producing useful examples of symmetric monoidal bicategories. 

We begin the first of these with a review of bicategorical duality  in \cref{sec:duality}. This is a very terse review of the machinery that we will need for this paper. More leisurely and thorough treatments of this material can be found in \cite{p:thesis, ps:bicat, campbell_ponto}.

Before going on to proofs of the main theorem, we discuss the applications. We first give some further explanation of \cref{thm:main_bzn} in \cref{sec:explanation_main}.
\cref{sec:two_characters} discusses our first application: the modular invariance of 2-characters. This first appeared in \cite{ben_zvi_nadler_1}. We also discuss a generalization to $n$-characters.
\cref{sec:lefschetz} shows that a variety of Lefschetz-type theorems follow from \cref{thm:main_bzn}. The only additional inputs are results from \cite{campbell_ponto} on Morita equivalence and traces. 

In \cref{sec:iterated} we discuss a formalism called \textit{umbras}, a variant of shadows. We show that in any bicategory equipped with an umbra we can verify the main theorem. By design, this section is purely formal, and the result becomes a diagram chase.

In \cref{sec:monoidal_bicat} we show any symmetric monoidal bicategory with suitably dualizable objects yields an example of an umbra. The verifications in this section are somewhat arduous, but doable, 2-category theory.
 This is the technical core of the paper --- the main categorical computations occur here. 
Some of the difficulty in this section is eased by the use of ``circuit diagrams''.

Finally, we produce interesting examples of symmetric monoidal bicategories in \cref{sec:enriched}. Many of the categories we work with are homotopical, i.e. possess a notion of equivalence much weaker than isomorphism, and enriched, i.e. have hom objects in some category rather than vanilla hom sets. To properly work with these examples, our bicategories must be suitably homotopical, and they must be symmetric monoidal. This work was done by Shulman \cite{shulman, shulman_symmetric} and we summarize his results in \cref{sec:enriched}.

  \subsection*{Acknowledgements} We thank David Ben-Zvi and David Nadler for writing the paper \cite{ben_zvi_nadler_1} which provided significant impetus and inspiration. We also thank John Lind for a very careful reading of an ealier version of this paper. He caught many errors in writing and thinking, and we especially thank him for educating us on how 2-characters should work. JC thanks  
the University of Kentucky for many pleasant visits.  
KP was partially supported by DMS-1810779 and the Royster Research Professorship.

\section{Duality and trace}\label{sec:duality}

In order to keep this paper fairly self-contained, 
we first review our perspective on duality and trace. All of this is developed in \cite{dold_puppe,lms, may_sigurdson,p:thesis,ps:bicat}, and nothing in \cref{sec:duality} 
is new, except, perhaps, for some examples. However, these sections demonstrate the wide applicability and utility of duality theory in category theory and higher category theory. The main point is that dualizability allows for the extraction of interesting invariants of the dualizable object.

\subsection{Symmetric monoidal duality}\label{sec:sym_duals}
We first recall duality theory in a symmetric monoidal category.

\begin{defn}\label{defn:symm_mon_duality}
  Let $(\mc{C}, \otimes,1)$ be a symmetric monoidal category. An object $X$ of $\mc{C}$ is \textbf{dualizable} if there is an object $Y$ of $\mc{C}$ and  maps
  \[
  \eta\colon  1 \to X \otimes Y \qquad \epsilon \colon Y \otimes X \to 1
  \]
  such that
  \begin{align*}
    &X \xrightarrow{\eta \otimes \id} X \otimes Y \otimes  X\xrightarrow{\id \otimes \epsilon} X \\
    &Y \xrightarrow{\id \otimes \eta} Y \otimes X \otimes Y \xrightarrow{\epsilon \otimes \id} Y 
  \end{align*}
are both the identity. 
\end{defn}

Throughout, it will be useful to have in mind particular categories. We work mostly in vector spaces, dg-algebras, dg-categories and spectra. 

\begin{eg}
A vector space $V$ over a field $k$ is dualizable if and only if it is finite dimensional. Its dual is given by $\operatorname{Hom}_k (V, k)$. 
\end{eg}

\begin{eg}
A spectrum $S$ is dualizable if and only if it is compact as an $S$-module. 
\end{eg}

The existence of a dual allows for the extraction of some interesting invariants.

\begin{eg}
  If $V$ is a vector space over $k$ and $V^\ast$ is its dual, then the composition
  \[
  k \xrightarrow{\eta} V^\ast \otimes V \xrightarrow{\cong} V \otimes V^\ast \xrightarrow{\epsilon} k
  \]
is an element of $\hom_k(k, k)$ and is multiplication by $\dim V$. 
\end{eg}

\begin{eg}
  If $X$ is a compact CW complex, $\Sigma^\infty_+ X$ is a dualizable spectrum, with dual $DX$ (this is the Spanier-Whitehead dual) then
  \[
  S \to \Sigma^\infty_+ X \sma DX \simeq DX \sma \Sigma^\infty_+ X \to S
  \]
  is a map in $[S, S] = \pi_0 (S) \cong \Z$ which is multiplication by $\chi(X)$, the Euler characteristic of $X$. 
\end{eg}

The above examples gives us more: inserting maps in various points give \textit{traces}.

\begin{eg}
  Let $f\colon V \to V$ be an endomorphism of a vector space $V$ over a field $k$. Then the composite
  \[
  k \xrightarrow{\eta} V \otimes V^\ast \xrightarrow{f \otimes \id} V \otimes V^\ast \xrightarrow{\cong} V^\ast \otimes V \xrightarrow{\epsilon} k
  \]
  is multiplication by $\operatorname{tr}(f)$. It is important to note there that the trace is a \textit{map} rather than a \textit{number}. 
\end{eg}

\begin{eg}
  Let $f\colon X \to X$ be a map of topological spaces. Then the composite
  \[
  S \to X \sma DX \xrightarrow{f \sma \id} X \sma DX \simeq DX \sma X \to S
  \]
  is the Lefschetz number $L(f)$. The Lefschetz theorem is a formal consequence of this fact \cite{dold_puppe}. 
\end{eg}

\subsection{Bicategorical duality}\label{sec:bicat_duals}

We move on to duality in bicategories. For definitions of bicategories see \cite{leinster}. The most useful bicategory to keep in mind is the Morita bicategory of rings, bimodules and bimodule maps.

\begin{notn}  We denote the bicategorical composition in a bicategory $\mc{B}$ by $\odot$.  If $A$ is an object of $\mc{B}$ we denote the identity 1-cell for $A$ by $U_A$.  
In the category of bimodules, $U_A =\, _A A_A$ and $\odot$ is the tensor product.
\end{notn}

The following definition first appeared in \cite{may_sigurdson}.

\begin{defn}
  Let $M$ be a 1-cell in a bicategory $\mc{B}(C, D)$. We say $M$ is \textbf{right dualizable} if there is a 1-cell $N$ together with 2-cells
  \[
  \eta\colon U_C \to M \odot N \qquad \epsilon\colon N \odot M \to U_D 
  \]
  such that the triangle identities hold. We say $N$ is \textbf{right dual} to $M$. 
  We  say that $(M,N)$ is a \textbf{dual pair}, that $N$ is \textbf{left dualizable}, and that $M$ is its \textbf{left dual}.
\end{defn}

\begin{rmk}
In an unfortunate clash of nomenclature, in the bicategory of categories, 1-cells are functors and right duals correspond to \textit{left} adjoints. 
\end{rmk}

The following lemma is easy, but critical. 

\begin{lem}\label{lem:duals_compose}
If $M_1 \in \mc{B}(A, B)$ and $M_2 \in \mc{B}(B, C)$ are right dualizable, then so is $M_1 \odot M_2$. 
\end{lem}

Again, given duality data one would like to extract invariants. However, it is not the case that $M \odot N \cong N \odot M$. To fix this, the second author introduced the formalism of shadows \cite{p:thesis}. A shadow is a gadget that repairs this defect. We define a shadow in terms of formal properties, and give more examples in \cref{sec:applications,sec:enriched}. 

\begin{defn}\label{defn_shadow}
A \textbf{shadow} for a bicategory $\mc{B}$ consists of functors
 \[\sh{-}\colon \mc{B}(R, R)\to \mathbf{T}\]
for each object $R$ of $\mc{B}$ and some fixed category $\mathbf{T}$, equipped with a natural isomorphism
\[\theta\colon  \sh{ M \odot  N}\to \sh{N \odot M }\]
for $M \in \mc{B}( R, S)$  and $N \in \mc{B}( S,  R)$ such that the following diagrams commute whenever they make sense:
  \[\xymatrix@C=12pt{
    \sh{(M\odot N)\odot P} \ar[r]^\theta \ar[d]_{\sh{a}} &
    \sh{P \odot (M\odot N)} \ar[r]^{\sh{a}} &
    \sh{(P\odot M) \odot N}
    \\
    \sh{M\odot (N\odot P)} \ar[r]^\theta 
    & \sh{(N\odot P)\odot M} \ar[r]^{\sh{a}} 
   & \sh{N\odot (P\odot M)}\ar[u]_\theta 
  }\hfill \xymatrix@C=12pt{
    \sh{M\odot U_C} \ar[r]^\theta \ar[dr]_{\sh{r}} &
    \sh{U_C\odot M} \ar[d]^{\sh{l}} \ar[r]^\theta &
    \sh{M\odot U_C} \ar[dl]^{\sh{r}}
    \\
    &\sh{M}}\]
\end{defn}

For this paper the most important example of a shadow is Hochschild homology which is a shadow on the bicategory of rings, bimodules and bimodule maps and valued in graded abelian groups.

\begin{defn}
   The \textbf{Euler characteristic} of a (right) dualizable 1-cell $M \in \mc{B}(A, B)$  is the map
  \[
  \sh{U_A} \to \sh{M \odot N} \cong \sh{N \odot M} \to \sh{U_B}
  \]
\end{defn}

In what follows, a more general construction is be needed.

\begin{defn}\label{defn:twisted_trace}
  Let $\phi\colon P \odot M \to M \odot Q$ be a 2-cell where $M$ is right dualizable. The {\bf twisted trace} of $\phi$ is  the composite
  \[
  \sh{P} \cong \sh{P \odot U_A} \to \sh{P \odot M \odot N} \to \sh{M \odot Q \odot N} \cong \sh{N \odot M \odot Q} \to \sh{U_B \odot Q} \cong \sh{Q} 
  \]
\end{defn}

One can think of this as a dualizable object, $M$, witnessing a trace between $\sh{P}$ and $\sh{Q}$.  There is a corresponding construction for twisted endomorphisms of $N$.

\begin{eg}
The case of the Euler characteristic for a 1-cell $M \in \mc{B}(A, B)$ corresponds to performing the above procedure for the canonical isomorphism 2-cell $U_A \odot M \xrightarrow{\cong} M \odot U_B$. 
\end{eg}

One can imagine longer strings of such maps. For example, suppose we are given $M_1 \in \mc{B}(A, B)$ and $M_2 \in \mc{B}(B, C)$ and $Q_1, Q_2, Q_3$ which twist endomorphisms of $M_1, M_2$: 
\[
f_1\colon Q_1 \odot M_1 \to M_1 \odot Q_2 \qquad f_2 \colon Q_2 \odot M_2 \to M_2 \odot Q_3. 
\]
These will witness maps $\sh{Q_1} \to \sh{Q_2}$ and $\sh{Q_2} \to \sh{Q_3}$. The following theorem says that we can obtain the composite of these all at once.

\begin{thm}\cite[7.5]{ps:bicat}\label{thm:composite}
  Let $M_1 \in \mc{B}(A, B)$, $M_2 \in \mc{B}(B, C)$ be right dualizable and $Q_1 \in \mc{B}(A,A)$, $Q_2 \in \mc{B}(B, B)$ and $Q_3 \in \mc{B}(C, C)$. Let $f_1, f_2$ be as above. Then the trace of
  \[
  Q_1 \odot M_1 \odot M_2 \xrightarrow{f_1 \odot \id_{M_2}} M_1 \odot Q_2 \odot M_2 \xrightarrow{\id_{M_1} \odot f_2} M_1 \odot M_2 \odot Q_3
  \]
  is
  \[
  \sh{Q_1} \xrightarrow{\tr(f_1)} \sh{Q_2} \xrightarrow{\tr(f_2)} \sh{Q_3} 
  \]
\end{thm}

The utility of this result cannot be overstated, and it will be lurking in the background of many examples below. When applied to the isomorphisms 
  \[
  U_A \odot M_1 \odot M_2 \xrightarrow{\cong} M_1 \odot U_B \odot M_2 \xrightarrow{\cong} M_1 \odot M_2 \odot U_C 
  \] it gives the following statement.  

\begin{cor}\label{thm:composite_euler}
  If $M_1 \in \mc{B}(A, B)$ and $M_2 \in \mc{B}(B, C)$ are right dualizable then
  \[
  \chi(M_1 \odot M_2) = \chi(M_2) \circ \chi(M_1)
  \]
\end{cor}

\begin{defn}\label{defn:mate}
Let $f\colon P \odot M \longrightarrow M \odot Q$ be a twisted endomorphism where $M$ is right dualizable with right dual $N$. The {\bf mate} of $f$ is the map $f^*\colon N \odot P \longrightarrow  Q \odot N$ defined as follows:
\begin{align*}
f^*:N \odot P \xrightarrow{\id \odot \eta}  N \odot P \odot M \odot N  \xrightarrow{\id \odot f \odot \id} N \odot M \odot Q \odot N \xrightarrow{\epsilon \odot \id} Q \odot N  
\end{align*}
\end{defn}

\begin{prop}\cite[7.6]{ps:bicat} \label{trace_with_dual}
  Let $f\colon P \odot M \longrightarrow M \odot Q$ be a twisted endomorphism where $M$ is right dualizable.  Then $\operatorname{tr}(f) = \operatorname{tr}(f^*)$. 
\end{prop}

\begin{eg}
  Let $V$ be a representation of a group $G$. This data is equivalent to the bimodule $_{k[G]} V_k$. When $V$ is a finite dimension vector space, $_{k[G]} V_k$ is right dualizable, and the Euler characteristic is a map
  \[
  \sh{k[G]} \to \sh{k} 
  \]
  When the shadow is given by $\HH_0$ we then obtain a map $\HH_0(k[G],k[G]) \to \HH_0 (k)$. That is, a map of (formal sums of) class functions to $k$. This is the character $\chi_V$. 

  This can be used to easily recover the induction formula for group representations. Given a representation of $H$ on $V$, i.e. a bimodule $_{k[H]} V_k$ the induced representation is given by a composition of bimodules $_{k[G]} k[G]_{k[H]} \odot _{k[H]} V_k$. If $[G:H]<\infty$ then $_{k[G]} k[G]_{k[H]}$ is right dualizable. By \cref{thm:composite}  $\chi_{\operatorname{Ind}^G_H(V)}$ is the composite
  \[
  \sh{k[G]} \xrightarrow{\chi^G_H} \sh{k[H]} \xrightarrow{\chi_V} \sh{k}. 
  \]

  If one wishes, this says that induction formulae follow from computing a universal example.
\end{eg}

\begin{eg}
The following simple example appears to not be well-known, but is extremely useful in computations in $\thh$. Let $X$ be a compact $CW$-complex. If we consider $S$ as a $(S, \Sigma^\infty_+ \Omega X)$-module then it is right dualizable (by compactness) and we obtain an Euler characteristic $S \to \Sigma^\infty_+ \mc{L} X$. Similarly, if we consider $S$ as a $(\Sigma^\infty_+ \Omega X, S)$-module we obtain an Euler characteristic $\Sigma^\infty_+ \mc{L} X \to S$ (this map always exists and is induced by the ``collapse to a point'' map). Then, \cref{thm:composite} says that the composite of these two maps is the Euler characteristic of $S \sma^L_{\Sigma^\infty_+ \Omega X} S$. This is $\Sigma^\infty_+ X$ and the Euler characteristic is the usual symmetric monoidal one. Thus, the composite $S \to \Sigma^\infty_+ \mc{L}  X \to S$ is the element $\chi(X) \in \pi_0 (S)$ (this is a special case of the main theorem of \cite{ps:mult}). 
\end{eg}

\section{Explanation of the Main Theorem}\label{sec:explanation_main}

Having laid out the necessary background, we explain the main theorem. Understanding the motivation and statement will allow one to read \cref{sec:two_characters,sec:lefschetz}  without wading into the significant categorical computations and homotopical technicalities of \cref{sec:iterated,sec:monoidal_bicat,sec:enriched}. 

In a symmetric monoidal bicategory $\mc{B}$ where all 0-cells are 2-dualizable (\cref{defn:2_dualizable}) and $I$ is the monoidal unit there is a shadow and it takes values in $\mc{B}(I,I)$ (\cref{prop:shadow}).  If $M$ is a left dualizable endomorphism 1-cell and $N$ is a right dualizable endomorphism 1-cell, for example $M: A \to A$ and $N: A \to A$, and 
\[\phi: M \odot N \to N \odot M\] 
is a 2-morphism we have a choice of bicategorical traces.  If we  take the trace of $\phi$ with respect to $M$ as in \cref{defn:twisted_trace} this gives us a map 
\[\tr_M (\phi): \sh{N} \to \sh{N}.\]
 If $\sh{N}$ is dualizable in $\mc{B}(I,I)$, we are entitled to take the trace of this map and obtain a map 
\[\tr_\sh{N} (\tr_M (\phi)): I \to I.\] 
Alternatively, we can  first take the trace with respect to $N$. In this case, we get a map 
\[\tr_\sh{M} (\tr_N (\phi)): I \to I.\]
 Our main theorem \cref{thm:main_bzn} states that these coincide.

This statement deserves significant amplification. We reach to the language of topological field theories for intuition, since our diagrammatic language below is reminiscent of them (though independent).

Symmetric monoidal traces have a well-known interpretation in terms of field theories of one dimensional framed manifolds (see, e.g. \cite{lurie_cobordism}). A {\bf framed zero-manifold} is a point labeled either  $+$ or  $-$ for a positive or negative orientation.  We adopt the convention of representing these as in \cref{fig:famed_zero_mfld}.
Ignoring worries about homotopy, gluing, etc, these oriented zero manifolds form the objects of a category whose morphisms are framed 1-manifolds. We represent a framing by a fattened edge (\cite{douglas_schommer_pries_snyder, schommer-pries,ps:indexed}).  \cref{fig:famed_one_mfld} are instances of 1-morphisms.

A {\bf 0 dimensional topological field theory}, as defined in \cite{atiyah}, is a symmetric monoidal functor $F: \mbf{Bord}^{\text{fr}}_1 \to \mc{C}$, where $\mc{C}$ is a  symmetric monoidal category. Such a functor will associate an object of $\mc{C}$ to each of the framed 0-manifolds and morphisms of $\mc{C}$ to each of the framed 1-manifolds in \cref{fig:famed_one_mfld}.  Note that the structure of $\mbf{Bord}^{\text{fr}}_1 $ 
and  the diagrams in \cref{fig:motivation_triangle} show that $F(+)$ and $F(-)$ must be dual.
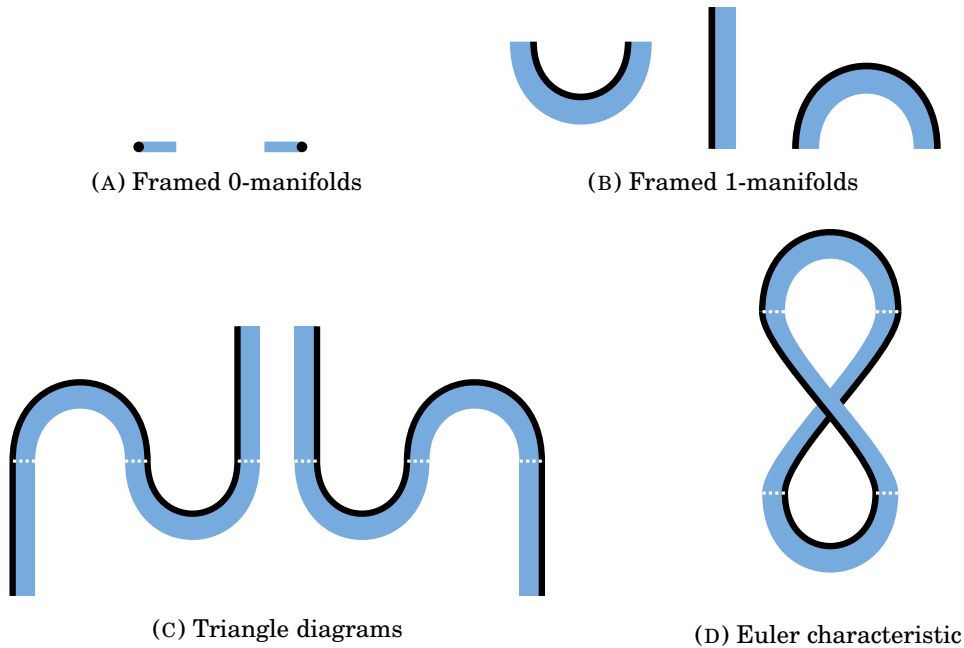
\begin{figure}
\input{II-motivation_dual_pair}
\caption{Framed manifolds}
\end{figure}

Now turning to traces, \cref{fig:motivation_euler_char} is the Euler characteristic of $F(+)$.
Note that the framing requires that we think of the trace as a figure-8 rather than a circle.  At this stage this may feel unnecessarily pedantic but this level of specificity will be essential later.  It is also  useful for understanding diagrams when we replace monoidal categories with bicategories.

Typically, topological field theories are thought of as sources of invariants of manifolds.  Here we will use manifolds to represent invariants of maps in a category. The distinction is important: we will mark manifolds in ways that are not geometrically motivated, but are categorically motivated. For example, given a morphism $f: A \to A$ of a dualizable object, we can depict $\tr(f)$ as in \cref{fig:motivation_trace_1}
where we have marked a region of the framed $S^1$ by a morphism. That section of the $S^1$ should be regarded as having a different nature from the rest of the manifold.

Moving up dimensions, and further away from field theory motivation, we view objects in 2-categories as vertices, morphisms as (framed) lines, and 2-morphisms as sheets between one morphisms. We run into issues with clarity of presentation here, so we do not illustrate this with a diagram. The 2-morphisms become difficult to picture and not particularly helpful in many cases.

The trace of a 2-cell is the cylindrical picture in \cref{fig:motivation_bicat_trace}. The disk marked $f$ should be interpreted as a 2-dimensional region remembering how to swap the red and blue regions. As one moves down the cylinder, a blue dual pair appears, one of the duals is swapped with the red strand, and after a rotation the blue duals are canceled. This is a pictorial representation of the maps in  \cref{defn:twisted_trace}.

The cylinder in \cref{fig:motivation_bicat_trace} is a morphism in the category where shadows take values, which we assume to be symmetric monoidal. Taking a trace in this category would amount to gluing that morphism into a figure-8, as in \cref{fig:motivation_trace_1}. Thus,  combining \cref{fig:motivation_trace_1,fig:motivation_bicat_trace}, the iterated trace is the colored torus in  \cref{fig:motivation_iterated_traces}.

The main theorem, \cref{thm:main_bzn}, is a statement about the equivalence of iterated trace diagrams, and is depicted in \cref{fig:motivation_main_thm}. Intuitively, these diagrams should be equivalent. However, the difficulty is turning the geometric intuition, i.e. all of the geometric moves, into 2-category theory.

\begin{figure}
\begin{tabular}{cc}
\input{II-framed_small_mflds}
&    \multirow{2}{*}[7cm]{ \begin{subfigure}[t]{.42\textwidth}
\input{II-motivation_iterated_traces}\vspace{-1cm}
\caption{Iterated trace}\label{fig:motivation_iterated_traces}
\end{subfigure}}
\hfill\\
  \begin{subfigure}[t]{.30\textwidth}
\begin{centering}{
\input{II-motivation-bicat_trace}}
\end{centering}
\caption{Bicategorical trace}\label{fig:motivation_bicat_trace}
\end{subfigure}
\end{tabular}
\caption{Traces}
\end{figure}

\begin{figure}
\input{II-motivation_iterated_traces_main_thm}
\caption{The traces in \cref{thm:main_bzn}.  The statement of that theorem is that these two pictures depict the same composite. }\label{fig:motivation_main_thm}
\end{figure}
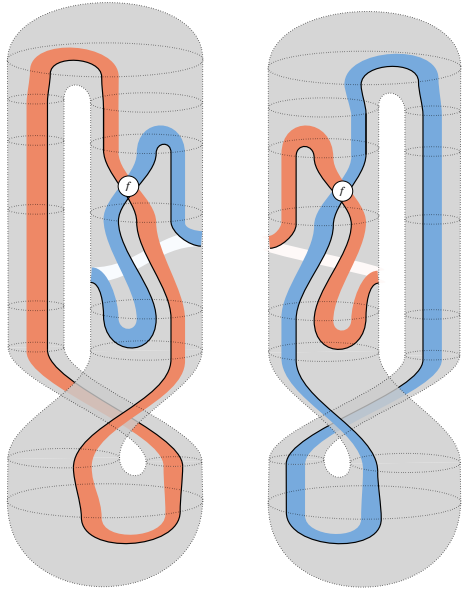

\begin{rmk} We are not interpreting these diagrams as formal proofs both because the relevant coherence theorem would take us too far afield and because it is difficult to depict the proper framings.  For example, the framings require that the two legs of the torus in \cref{fig:motivation_iterated_traces} switch sides near the bottom.  \cref{sec:iterated,sec:monoidal_bicat} are concerned with a rigorous proof motivated by \cref{fig:motivation_iterated_traces}.
\end{rmk}

\section{Applications: generalized Lefschetz theorems} \label{sec:lefschetz}
\label{sec:applications}

We now consider consequences of combining \cref{thm:main_bzn} with one of the main results in \cite{campbell_ponto} and some basic facts about the trace.  Special cases of these results have already appeared in the literature \cite{lunts,polishchuk,cisinski_tabuada} and we will describe the comparisons below.  The proofs we give here offer significant generalizations.  

\subsection{Lefschetz theorems}
In this section we will restrict our attention to the following  bicategories.

\begin{example}\label{ex:enriched_bicat}\cref{main_categorical_construction} implies the following bicategories are monoidal. 
\begin{itemize}
\item The bicategory of dg-algebras, and the derived category of bimodules and homomorphisms.   The bicategorical product is the tensor product.
\item The bicategory of dg-categories, where 1 cells are dg-categories and 2-cells are the derived category of bimodules.  We will abuse notation and also denote this category $\Mod(\Cat_{\dg})$.
\item The bicategory of spectral categories where 1 and 2 cells are the homotopy category of bimodules. We will  denote this category $\Mod(\Cat_{\Sp})$.
\end{itemize}
The bicategorical product in the second and third examples is a generalization of the usual tensor product of modules.  The shadow can be interpreted as the bicategorical product with the diagonal module associated to a category.   

See \cref{sec:enriched_bimod} for careful definitions of these bicategories.  
\end{example}

The first of these examples is a subbicategory of the second where we only consider the categories with a single object.  Correspondingly, there is a spectral analog of the first example where the objects are replaced by ring spectra.

\begin{example}We will be most interested in very specific objects and 1-cells in these bicategories.
\begin{itemize}
\item If $A$ is a dg-algebra then the categories of $A$-modules and perfect $A$-modules  are dg-categories.  So $\Mod_A$ and $\Mod^\perf_A$ are 0-cells in the bicategory $\Mod(\Cat_{\dg})$. The category $\Mod^\perf_A$ is Morita equivalent to $A$, it can be treated as functionally equivalent to $A$. 
\item If $\mc{A}$ is a dg-category then the categories of $\mc{A}$-modules  and dualizable $\mc{A}$-modules are dg-categories and so  are 0-cells in  the bicategory $\Mod(\Cat_{\dg})$.  
\item Similarly, if $\mc{A}$ is a spectral category then the categories of $\mc{A}$-modules  and dualizable $\mc{A}$-modules are spectral categories and so  are 0-cells in  the bicategory $\Mod(\Cat_{\mathrm{Sp}})$. 
\end{itemize}
\end{example}

Between a pair of dg-categories or spectral categories we have enriched functors.  (These become algebra homomorphism when we restrict to a category with one object.)  
\begin{defn}\label{defn:base_change}
If $F\colon \mc{C}\to \mc{D}$ is an enriched functor there is a $(\mc{D},\mc{C})$-bimodule  $\mc{D}_F$  defined by 
\[\mc{D}_F(d,c)\coloneqq \mc{D}(d,F(c)).\] 
 Morphisms of $\mc{C}$ act through $F$.   This is a  {\bf base change} 1-cell associated to $F$.
There are similar base change 1-cells $_F\mc{D}$ and $_F\mc{D}_G$.
\end{defn}

\begin{example} While there are many interesting base change 1-cells we are most interested in functors defined by tensoring.
\begin{itemize}
\item If $A$ is a dg-algebra and $Q$ is an $(A,A)$-module there is a $(\Mod_A^\perf,\Mod_A^\perf)$-bimodule, and so 1-cell in $\Mod(\Cat_{\dg})$, whose value on compact modules $M$ and $N$ is the dg-category of module homomorphisms 
\[\Mod_A(M,N\odot Q)\]
$\Mod_A$ acts on the left by composition and on the right by tensoring with the identity map of $Q$ and composing.  

This is the base change object for the functor $-\odot Q\colon \Mod_A^\perf\to \Mod_A$.  There are corresponding bimodules for tensoring in the domain and tensoring in both domain and target. 

\item If $\mc{A}$ is a dg-category (respectively spectral category) and $\mc{Q}$ is an $(\mc{A},\mc{A})$-module there is a $(\Mod_\mc{A}^\perf,\Mod_\mc{A}^\perf)$-bimodule, and so 1-cell in $\Mod(\Cat_{\dg})$ (respectively, $\Mod(\Cat_{\Sp})$), whose value on compact modules $\mc{M}$ and $\mc{N}$ is the dg-category (respectively spectral category)  of module homomorphisms 
\[\Mod_{\mc{A}}(\mc{M},\mc{N}\odot \mc{Q})\]
The actions are as above.

\end{itemize}
\end{example}

 If $\mc{Q}$ and $\mc{M}$ are endomorphism 1-cells and  $\phi\colon \mc{Q}\odot \mc{M}\to \mc{M}\odot \mc{Q}$ is a 2-cell in any of the bicategories above, then $\phi$ induces a map 
	\[(\Mod_\mc{A})_{-\odot \mc{Q}}\to \,_{-\odot\mc{M}}(\Mod_\mc{A})_{-\odot \mc{M}\odot \mc{Q}}\] 
as the composite 
\begin{align*}
\Mod_\mc{A}(\mc{X},\mc{Y}\odot \mc{Q})&\xto{-\odot \id_\mc{M}} 
	\Mod_\mc{A}(\mc{X}\odot \mc{M},\mc{Y}\odot \mc{Q}\odot \mc{M})
\\
	&\xto{(\id_\mc{Y}\odot \phi)_*} 
	\Mod_\mc{A}(\mc{X}\odot \mc{M},\mc{Y}\odot \mc{M}\odot \mc{Q})
\end{align*}
for $\mc{X}, \mc{Y}\in \Mod_{\mc{A}}^{\perf}$.
There is also a forgetful map 
\[_{-\odot \mc{M}}(\Mod_\mc{A})_{-\odot \mc{M}\odot \mc{Q}}\to
  (\Mod_\mc{A})_{-\odot \mc{Q}}.\]
Together these define an endomorphism $\Phi$ of $  (\Mod_\mc{A})_{-\odot \mc{Q}}$.

\begin{prop}\label{thm:cp}\cite[5.21]{campbell_ponto}  
If $\mc{M}$ is a right dualizable $(\mc{A},\mc{A})$-bimodule,  $\mc{Q}$ is an $(\mc{A},\mc{A})$-bimodule, and $\phi\colon \mc{Q}\odot \mc{M}\to \mc{M}\odot \mc{Q}$ is a 2-cell then the 
 following diagram commutes. 
\[\xymatrix{
  \sh{\mc{Q}}	\ar[d]^\chi	\ar[r]^-{\tr_\mc{M}(\phi)}
  &\sh{\mc{Q}}\ar[d]^\chi
  \\
  \sh{(\Mod_\mc{A})_{-\odot \mc{Q}}}\ar[r]^{\sh{\Phi}}
  &\sh{(\Mod_\mc{A})_{-\odot \mc{Q}}}
}\]
The vertical maps are isomorphisms and the horizontal maps are as above. 
\end{prop}
  
To apply \cref{thm:main_bzn} we need to restrict to 2-dualizable 0-cells.  We will formally define this condition in \cref{defn:2_dualizable} but for this section it is enough to describe the relevant 0-cells. 

In the first of the bicategories in \cref{ex:enriched_bicat} this condition is well studied in the literature.

\begin{defn}\cite[2.3]{toen}
A dg-$k$-algebra $B$ is {\bf proper} if $B$ is perfect as an object of $D(k)$ (i.e. $B$ is perfect when regarded as a chain complex of $k$ modules).
It is {\bf smooth} if $B$ is perfect as an object of $D(B \otimes^{\mathbf{L}}_k B^\op)$
\end{defn}

This definition admits an immediate generalization to the other two bicategories.  If $\mc{A}$ is a dg-category (respectively spectral category), the unit 1-cell $U_\mc{A}$ can be regarded as a $(k,\mc{A}\otimes \mc{A})$-bimodule (respectively $(\bS,\mc{A}\otimes \mc{A})$-bimodule).  Denote this module $\overrightarrow{U}_\mc{A}$. 
\begin{defn}
A 0-cell 
 $\mc{A}$ in a dg-category (respectively spectral category) is {\bf proper} if $\overrightarrow{U}_\mc{A}$ is left dualizable.  It is {\bf smooth} if $\overrightarrow{U}_\mc{A}$ is right dualizable. 
\end{defn}

The following theorem is essentially contained in \cite[Thm.~5.8]{cisinski_tabuada_dg}, and shows that the usage above is consistent with standard usage of these terms.  It also provides a satisfying explanation for the necessity of two different conditions needed for the 2-dualizability of dg or spectral categories: the conditions correspond to taking different duals in a bicategory. 

\begin{thm}
For spectral or dg-categories, $\mc{C}$ is 2-dualizable if and only if it is smooth and proper. 
\end{thm}

\begin{rmk}
If $\mc{A}$ is a smooth or proper dg-category, then so is $\Mod^{\perf}_{\mc{A}}$. Indeed, $\Mod^{\perf}_{\mc{A}}$ is Morita equivalent to $\mc{A}$ and so has all of the same dualizability properties. 
\end{rmk}

This result and \cref{thm:main_bzn} then imply the following statement.

\begin{prop}\label{thm:pol_2}
If $\mc{A}$ is smooth and proper, $\mc{M}$ is a right dualizable $(\mc{A},\mc{A})$-bimodule,  $\mc{Q}$ is a left dualizable $(\mc{A},\mc{A})$-bimodule, and $\phi\colon \mc{Q}\odot \mc{M}\to \mc{M}\odot \mc{Q}$ is a 2-cell then 
\[\tr(\sh{\mc{M}} \xrightarrow{\tr_{\mc{Q}}(\phi)} \sh{\mc{M}})=\tr(  \sh{(\Mod_\mc{A})_{-\odot \mc{Q}}}\xto{\sh{\Phi}}
  \sh{(\Mod_\mc{A})_{-\odot \mc{Q}}})\]
\end{prop}

\begin{proof}
By \cref{thm:main_bzn} $\tr_{\sh{\mc{M}}}(\tr_\mc{Q}(\phi))=\tr_{\sh{\mc{Q}}}(\tr_\mc{M}(\phi))$.
By \cref{thm:cp}
  \begin{align*}
   \sh{\mc{Q}} \xrightarrow{\tr_\mc{M}(\phi)} \sh{\mc{Q}}
	&= \sh{\mc{Q}} \xrightarrow{\chi} \sh{(\Mod_\mc{A})_{-\odot \mc{Q}}}  \xrightarrow{\sh{\Phi}} 
	\sh{(\Mod_\mc{A})_{-\odot \mc{Q}}} \xrightarrow{\chi^{-1}} \sh{\mc{Q}}
\end{align*}
Since the trace is invariant under cyclic permutation 
\[
\tr \left(\sh{\mc{Q}} \xrightarrow{\chi} \sh{(\Mod_\mc{A})_{-\odot \mc{Q}}}  \xrightarrow{\sh{\Phi}} 
	\sh{(\Mod_\mc{A})_{-\odot \mc{Q}}} \xrightarrow{\chi^{-1}} \sh{Q} \right)\] is 
    \[\tr \left( \sh{(\Mod_\mc{A})_{-\odot \mc{Q}}}  \xrightarrow{\sh{\Phi}}
 \sh{(\Mod_\mc{A})_{-\odot \mc{Q}}}\right) 
\]
\end{proof}

There is an interesting simplification of this result when $\phi$ is the identity.  It requires a preliminary lemma.

\begin{lem}\label{thm:lunts_1}  Let $\sB$ be a symmetric monoidal bicategory where all zero cells are 2-dualizable.  If  $M\in \sB(A,A)$ is right dualizable then 
\[\chi(\sh{M})=  \tr_{\sh{U_A}}\left(\sh{U_A} \xrightarrow{\chi(M)} \sh{U_A}\right) .\] 
\end{lem}

\begin{proof}
The unit $U_A$ is left dualizable, so 
applying \cref{thm:main_bzn} to the unit isomorphism $i\colon U_A \odot M \to M \odot U_A$ we get the following equality 
  \[
  \tr(\sh{M} \xrightarrow{\tr_{U_A}(i)} \sh{M}) = \tr(\sh{U_A} \xrightarrow{\tr_M(i)} \sh{U_A}).  
  \]
By \cite[7.4]{ps:bicat} $\sh{M} \xrightarrow{\tr_{U_A}(i)} \sh{M} = \sh{M} \xrightarrow{\sh{\id_M}} \sh{M}$.   By definition the trace of this map is $\chi(\sh{M})$. 

To identify the right hand side,  the following commutative diagram demonstrates that the trace of $i$ with respect to $M$ is $\chi(M)$
\begin{center}\resizebox{\textwidth}{!}{\xymatrix{
	\sh{U_A\odot U_A}\ar[r]^-{\sh{\id\odot \eta}}\ar[d]_\sim^{\sh{l}}
	& \sh{U_A\odot M\odot N}\ar[r]^-{\sh{i\odot \id}}\ar[d]^{\sh{l\odot \id}}
	&\sh{M\odot U_A\odot N}\ar[r]^-\theta\ar[d]^{\sh{r\odot \id}}
	&\sh{N\odot M\odot U_A}\ar[r]^-{ \sh{\epsilon\odot\id}}\ar[d]^{\sh{\id\odot r}}
	&\sh{U_A\odot U_A}\ar[d]_\sim^{\sh{r}}
\\
	\sh{U_A}\ar[r]^-{\sh{\eta}}
	& \sh{M\odot N}\ar[r]^-{\sh{\id}}
	&\sh{M\odot N}\ar[r]^-\theta
	&\sh{N\odot M}\ar[r]^-{\sh{\epsilon}}
	&\sh{U_A}
}}\end{center}
\end{proof}

\begin{rmk}
The proof of \cref{thm:main_bzn} shows it is enough to assume that $A$ is 2-dualizable rather than assuming all zero cells are 2-dualizable.
\end{rmk}

\begin{cor}\label{thm:lunts_2}
If $\mc{M}$ is a right dualizable $(\mc{A},\mc{A})$-bimodule and $\mc{A}$ is smooth and proper 
	then \[\chi(\sh{\mc{M}})= \tr \left( \sh{\Mod_\mc{A}}  \xrightarrow{\sh{- \odot \mc{M}}} \sh{\Mod_\mc{A}}\right) \]
\end{cor}

\begin{proof}
\cref{thm:lunts_1} identifies $\chi(\sh{\mc{M}})$ with $\tr(\chi(\mc{M}))$.  If $\phi$ is the identity map of $\mc{M}$, 
\cref{thm:pol_2} identifies $\tr(\chi(\mc{M}))$ with $\tr(\sh{-\odot \mc{M}})$.
\end{proof}

If $A$ is a smooth and proper dg-algebra over a field $k$ and $M \in \Mod^{\text{perf}}_{(A,A)}$ is a perfect DG bimodule then 
 \[\sum (-1)^i \dim \operatorname{HH}_i (A, M)=\chi(\sh{M})\] and 
\[\sum_j (-1)^j \tr \HH_j (-\odot  M) =
	\tr \left( \sh{\Mod^{\text{perf}}_A}  \xrightarrow{\sh{- \odot  M}} \sh{\Mod^{\text{perf}}_A}\right) 
\]
With these identifications, the following theorem is a consequence of \cref{thm:lunts_2}.

\begin{thm}\cite[Theorem 1.4]{lunts}\label{thm:lunts}
  Let $A$ be a smooth, proper dg-algebra over a field $k$ and $M \in \Mod^{\text{perf}}_{(A,A)}$ be a perfect DG bimodule. Then
  \[
  \sum_i (-1)^i \dim\HH_i (A;M) = \sum_i \tr\HH_i (-\odot  M)
  \]
\end{thm}

Similarly, \cref{thm:lunts_2} implies the following generalization due to Cisinski and Tabuada \cite{cisinski_tabuada} .

\begin{cor}\cite[Theorem 1.7]{cisinski_tabuada}
  Let $\mc{A}$ be a smooth, proper dg-category over a field $k$ and $\mc{M}$ be a perfect $(\mc{A}, \mc{A})$-bimodule.   Then
  \[
  \sum_i (-1)^i \dim\HH_i (\mc{A};\mc{M}) = \sum_i \tr\HH_i (-\otimes {\mc{M}}) 
  \]
\end{cor}

\subsection{Lefschetz reciprocity}
In \cite{polishchuk} Polishchuk also proved a consequence of \cref{thm:main_bzn} but his result relies on mates (\cref{defn:mate}) rather than the main result of \cite{campbell_ponto}.

\begin{thm}\label{thm:mate}
If $(M,N)$ is a dual pair in a monoidal bicategory $\sB$ where all objects are 2-dualizable, 
$Q$ is left dualizable and $f\colon Q\odot M\to M\odot Q$ then
\[\tr_\sh{M}\left(\sh{M}\xto{\tr_Q(f)}\sh{M}\right)=\tr_\sh{Q}\left(\sh{Q}\xto{\tr_N(f^*)}\sh{Q}\right)\]
\end{thm}

Here  $f^*$ is the dual of $f$ as in \cref{trace_with_dual}.

\begin{proof}
By \cref{thm:main_bzn} \[\tr_\sh{M}\left(\sh{M}\xto{\tr_Q(f)}\sh{M}\right)=\tr_\sh{Q}\left(\sh{Q}\xto{\tr_M(f)}\sh{Q}\right).\]  
 \cref{trace_with_dual} implies $\tr_\sh{Q}(\sh{Q}\xto{\tr_M(f)}\sh{Q})
=\tr_\sh{Q}(\sh{Q}\xto{\tr_N(f^*)}\sh{Q})$.
\end{proof}

\begin{rmk}
If $Q$ is also right  dualizable then  we have the sequence of equalities 
\begin{align*}
\tr_\sh{M}(\sh{M}\xto{\tr_Q(f)}\sh{M})&=\tr_\sh{Q}(\sh{Q}\xto{\tr_M(f)}\sh{Q})
=\tr_\sh{Q}(\sh{Q}\xto{\tr_N(f^*)}\sh{Q})
\\
&=\tr_\sh{N}(\sh{N}\xto{\tr_Q(f^*)}\sh{N})
\end{align*}
\end{rmk}

To compare this to  \cite[Eq. 0.4]{polishchuk} requires a careful comparison of notation.
For a dg-category $\sC$ over a field $k$, let  $\Mod^\perf_{\dg}(\sC)$
let  be the derived category of perfect $\sC^{\op}$-modules.

Up to an appropriate equivalence, every dg-functor $F :\Mod^\perf_{\dg}(\sC) \to  \Mod^\perf_{\dg}(\sD)$  is of the form \[F( M) = M \odot ^\mathbb{L}_{\sC} K_F\] 
for 
a {\bf kernel} $K_F \in \Mod^\perf_{\dg}(\sC^{\op} \otimes \sD)$.  Similarly, a natural transformation $F\to G$ is given by a homomorphism $K_F\to K_G$. 
Then \cite{polishchuk} defines \[\ptr_{\sC} (F )\coloneqq \sh{K_F}\] for an endofunctor $F : \Mod^\perf_{\dg}(\sC) \to \Mod^\perf_{\dg}(\sC)$ given by a kernel $K_F \in \Mod^\perf_{\dg}(\sC^{\op}\otimes \sC)$.  

If $F : \Mod^\perf_{\dg}(\sC) \to  \Mod^\perf_{\dg}(\sD)$ is a dg-functor with adjoint $G$ and $f\colon  F\circ \Psi\to \Psi\circ F$ is a natural transformation \cite{polishchuk} defines  $(F,f)^∗ $ be the composite.
\[\ptr_\sC(\Psi)\to \ptr_\sC(G\circ F\circ \Psi) \xto{G\circ f} \ptr_\sC(G\circ \Psi \circ F) \to  \ptr_\sC(F\circ G\circ \Psi)\to \ptr_\sC(\Psi).\] 
Since an adjoint for $F$ is equivalent to a dual for $K_F$,  
the map
$(F,f)^∗ $ is the bicategorical trace of the map 
\[K_\Psi\odot K_F\to K_F\odot K_\Psi\] induced by $f$.

Then \cref{thm:mate} implies 
\[\tr_{\sh{K_{\Psi}}}(\tr_{K_F}(f))=\tr_{\sh{K_{G}}}(\tr_{K_\Psi}(f^*)).\]
(Note the conflicting uses of $^*$.  In this case it refers to the mate of $f$.)
Using the notational comparisons above this equality is that in the following result of \cite{polishchuk}. 

\begin{thm}\cite[Equation 0.4]{polishchuk}\label{thm:pol}
If $F :\Mod^\perf_{\dg}(\sC) \to  \Mod^\perf_{\dg}(\sD)$ is a dg-functor with adjoint $G$ and $f\colon  F\circ \Psi\to \Psi\circ F$ is a natural transformation then 
  \[
  \operatorname{str}((F,f)^\ast, \operatorname{Tr}(\Psi)) = \operatorname{str}((\Psi, \psi)^\ast, \operatorname{Tr}(G))
  \]
where  $\psi\colon \Psi\circ G\to G\circ \Psi$ is the composite:
\[\Psi\circ G\to G\circ F\circ \Psi\circ G\to G\circ \Psi\circ F\circ  G \to G\circ \Psi.\]
\end{thm}

Here we have also identified  $\operatorname{str}$, the supertrace, as the trace in the category of graded vector spaces.

\section{Umbras and iterated traces}\label{sec:iterated}
We now turn to the more technical work of proving \cref{thm:main_bzn}.  It follows from \cref{thm:main_umbra}  in this section and   \cref{thm:monoidal_to_umbra} in the next section.  

\cref{thm:main_umbra} identifies a small set of additional conditions that a bicategory must satisfy in order to define iterated traces.  We call this structure an {\it umbra} since it extends the shadows of \cite{p:thesis}.  It is also related to shadows philosophically since in both cases we seek to retain a minimum of structure.

In all examples of interest to us, an umbra comes from a monoidal bicategory and so we are not interested in this axiomatization because we expect further generalizations but because this structure is easier to work with.

\begin{defn}
A {\bf penumbra} on a bicategory $\sB$ taking values in a monoidal category $(\mathbf{T},\otimes, I)$ is 
functors \[\sh{-},\dsh{-},\csh{-},\esh{-} \colon \sB(A,A)\to \mathbf{T}\]
for all 0-cells $A$ of $\sB$ 
and  natural maps 
\begin{align*}
\rspl\colon \dsh{M}\otimes \csh{N}&\to \dsh{M\odot N}
&
\luspl\colon \dsh{M\odot N}&\to \esh{M}\otimes \dsh{N}
\\
     \uspl\colon \dsh{M}\otimes \sh{N} &\to \esh{M\odot  N}
&\spl\colon \csh{M\odot  N}&\to \sh{M}\otimes \dsh{N} 	  
\\
\lspl\colon \csh{M}\otimes \sh{N}&\to \sh{M\odot N}
&
\ruspl\colon \sh{M\odot N}&\to \sh{M}\otimes \esh{N}
\end{align*}
 \begin{align*} \iunit\colon I&\to \csh{U_A}&\ounit\colon \esh{U_A}&\to I\end{align*}
making the  diagrams in \cref{def:penumbra} commute. 
\end{defn}

  The following is essentially the corresponding statement for monoidal functors \cite{dold_puppe,p:thesis}.

\begin{lem}\label{lem:penumbra_dual}
If $(N,M)$ is a dual pair and $(\sh{-},\dsh{-},\csh{-},\esh{-})$ is a penumbra, then 
$(\sh{N},\dsh{M})$ is a dual pair in $\mathbf{T}$. 
 \end{lem}

 \begin{proof} 
The coevaluation and evaluation for $(\sh{N}, \dsh{M})$ are the composites
\begin{center}\begin{tikzcd}
I
\ar[r,"\iunit"]
&
\csh{U_A}\ar[r,"\csh{\eta_N}"]&\csh{N\odot M}\ar[r,"\spl"]&\sh{N}\otimes \dsh{M}
\end{tikzcd}
\begin{tikzcd}
\dsh{M}\otimes \sh{N}\ar[r,"\uspl"]
&
\esh{M\odot N}\ar[r,"\esh{\epsilon_N}"]&\esh{U_A}\ar[r,"\ounit"]&I
\end{tikzcd}
\end{center}
The triangle diagrams are in \cref{fig:penumbra_dual}.
All small regions commute by definition of a penumbra or a dual pair.
  \end{proof}

\afterpage{\clearpage
\begin{figure}
\[\xymatrix@C=30pt{\csh{M\odot N}\otimes \sh{P}\ar[r]^-{\spl\otimes \id}\ar[d]^-{\lspl}
	&\sh{M}\otimes \dsh{N}\otimes \sh{P}\ar[d]^-{\id\otimes \uspl}
	\\
	\sh{M\odot N\odot P}\ar[r]^-{\ruspl}
	&\sh{M}\otimes \esh{N\odot P}
	}\hspace{1cm}
\xymatrix@C=30pt{\dsh{M}\otimes\csh{N\odot P}\ar[r]^-{\id\otimes \spl}\ar[d]^{\rspl}
	&\dsh{M}\otimes \sh{N}\otimes \dsh{P}\ar[d]^{\uspl\otimes \id}
	\\
	\dsh{M\odot N\odot P}\ar[r]^-{\luspl}
	&\esh{M\odot N}\otimes \dsh{P} 
}\]
\[\xymatrix@C=30pt{
	\csh{M\odot N\odot P} \ar[d]^{\spl}\ar[r]^-{\spl}
	&\sh{M}\otimes \dsh{N\odot P}\ar[d]^{\id\otimes\luspl}
	\\
	\sh{M\odot N}\otimes \dsh{P}\ar[r]^-{\ruspl\otimes \id}
    	&\sh{M}\otimes \esh{N}\otimes \dsh{P}}
	\hspace{1cm}
\xymatrix@C=30pt{
	\dsh{P} \otimes \csh{N}\otimes\sh{M}\ar[d]^{\id\otimes \lspl}\ar[r]^-{\rspl\otimes\id}
	&\dsh{P\odot N}\otimes \sh{M}\ar[d]^{\uspl}
	\\
	\dsh{P}\otimes \sh{N\odot M}\ar[r]^-{\uspl}
    &\esh{P\odot N\odot M}}
\]
\[\xymatrix{\csh{U_A}\otimes \sh{M}\ar[r]^-{\lspl}&\sh{U_A\odot M}\ar[d]
	\\
	I\otimes \sh{M}\ar[u]^{\iunit\otimes \id}\ar[r]&\sh{M}
}\hspace{1cm}
\xymatrix{\dsh{M}\otimes\csh{U_A}\ar[r]^-{\rspl}&\dsh{M\odot U_A}\ar[d]
	\\ 
	\dsh{M}\otimes I\ar[u]^{\id\otimes \iunit}\ar[r]&\dsh{M}
}\]
\[\xymatrix{\sh{M}\otimes \esh{U_A}\ar[r]^-{\id\otimes \ounit}&\sh{M}\odot I\ar[d]
	\\
	\sh{M\odot U_A}\ar[u]^-{\ruspl}\ar[r]&\sh{M}} \hspace{1cm}
\xymatrix{\esh{U_A}\otimes\dsh{M}\ar[r]^-{\ounit\otimes \id}& I\odot\dsh{M}\ar[d]
	\\ \dsh{ U_A \odot M}\ar[u]^-{\luspl}\ar[r]&\dsh{M}
}\]
\caption{The commuting squares for a penumbra.  The unlabeled maps are unit isomorphisms. }\label{def:penumbra}
\end{figure}

\begin{figure}
\[\xymatrix@C=35pt{
 I\otimes \sh{N}\ar[r]^-{\iunit\otimes \id}\ar[d]
    &\csh{U_A}\otimes \sh{N}\ar[r]^-{\csh{\eta_N}\otimes \id}\ar[d]^{\lspl}
    &\csh{N\odot M}\otimes \sh{N}\ar[r]^-{\spl\otimes \id}\ar[d]^{\lspl}
    &\sh{N}\otimes \dsh{ M}\otimes \sh{N}\ar[d]^{\id\otimes \uspl}
    \\
    \sh{N}\ar[ddrr]^-\id
    &\sh{U_A\odot N}\ar[r]^-{\sh{\eta_N\odot \id}}\ar[l]
    &\sh{N\odot M\odot N}\ar[d]^{\sh{\id\odot \epsilon_N}}\ar[r]^-{\ruspl}
    &\sh{N}\otimes \esh{M\odot  N}\ar[d]^{\id\otimes \esh{\epsilon_N}}
    \\
    &&\sh{N\odot U_A}\ar[d]\ar[r]^-{\ruspl}
    &\sh{N}\otimes \esh{U_A}\ar[d]^{\ounit}
    \\
    &&\sh{N}
    &\sh{N}\otimes I\ar[l]
}\]
\[
\xymatrix@C=35pt{
\dsh{M} \otimes I\ar[r]^-{\id\otimes \iunit}\ar[d]
    &\dsh{M}\otimes \csh{U_A}\ar[r]^-{\id\otimes\csh{\eta_N}}\ar[d]^-{\rspl}
    &\dsh{M}\otimes \csh{N\odot M}\ar[r]^-{\id\otimes \spl}\ar[d]^-{\rspl}
    &\dsh{M}\otimes \sh{ N}\otimes \dsh{M}\ar[d]^-{\uspl\otimes \id}
    \\
    \dsh{M}\ar[r]\ar[ddrr]^\id
    &\dsh{M\odot U_A}\ar[r]^-{\dsh{\id\odot \eta_N}}
    &\dsh{M\odot N\odot M}\ar[r]^-{\luspl}
	\ar[d]^-{\dsh{\epsilon_N\odot \id}}
    &\esh{M\odot N}\otimes\dsh{M} \ar[d]^-{\esh{\epsilon_N}\otimes\id}
    \\
    &&\dsh{U_A\odot M}\ar[d]\ar[r]^-{\luspl}
    &\esh{U_A}\otimes \dsh{M}\ar[d]^{\ounit}
    \\
    &&\dsh{M}
    &I \otimes \dsh{M}\ar[l]
}\]\caption{Commuting diagrams for \cref{lem:penumbra_dual}}\label{fig:penumbra_dual}
\end{figure}
\clearpage}

We now turn to symmetry conditions. 
If the functors $\sh{\hspace{1mm}}$ and $\dsh{\hspace{1mm}}$ in a penumbra are shadows and $\mathbf{T}$ is a symmetric monoidal category, we can use the shadow isomorphisms and symmetry isomorphisms to produce two permutations of three 1-cells.  To prove \cref{thm:main_bzn} we need to know these maps are the same.

\begin{defn}\label{defn:umbra}
 A penumbra $(\sh{-}, \dsh{-}, \csh{-}, \esh{-})$ on a bicategory $\sB$ taking values in a symmetric monoidal category $(\mathbf{T},\otimes, I)$  is an {\bf umbra}  if 
\begin{enumerate}
\item $\sh{\hspace{1mm}}$ and $\dsh{\hspace{1mm}}$ are shadows (\cref{defn_shadow}) and 
\item the following diagram relating symmetry and shadow isomorphisms commutes.
\end{enumerate}
\begin{equation}
\xymatrix{
\sh{M\odot N}\otimes \dsh{P}\ar[d]^{\theta\otimes \id}
    &\csh{(M\odot N)\odot P}\cong \csh{M\odot (N\odot P)}\ar[l]_-{\spl}\ar[r]^-{\spl}
    &\sh{M}\otimes \dsh{N\odot P}\ar[d]^{\id \otimes \theta}
\\
\sh{N\odot M}\otimes \dsh{P}\ar[d]^{\gamma}
&&\sh{M}\otimes \dsh{P\odot N}\ar[d]^{\gamma}
\\
\dsh{P}\otimes \sh{N\odot M}\ar[r]^-{\uspl}
    &\esh{P\odot (N\odot M)}\cong \esh{(P\odot N)\odot M}
    &\dsh{P\odot N}\otimes \sh{M}\ar[l]_-{\uspl}
}
\label{eq:big_square}
\end{equation}
\end{defn}

\begin{thm}[Umbra version of \cref{thm:main_bzn}]\label{thm:main_umbra}
If $(\sh{\hspace{1mm}},\dsh{\hspace{1mm}},\csh{\hspace{1mm}},\esh{\hspace{1mm}})$ is an umbra on a bicategory $\sB$, $M$ is left dualizable and $N$ is right dualizable in $\sB(A,A)$
and $\phi\colon M\odot N\to N\odot M$ is a 2-cell in $\sB$
then \[\tr\left(\sh{N}\xto{\tr_M(\phi)} \sh{N}\right)=\tr\left(\dsh{M}\xto{\tr_N(\phi)}\dsh{M}\right).\]
\end{thm}

\begin{proof} 
This follows from  the diagram in  \cref{fig:commuting_traces}.  To clarify notation we denote the dual of $M$ by $\ldual{M}$ and the dual of $N$ by $\rdual{N}$.

 All regions with a dashed arrow commute since the dashed 
arrow is defined to be the composite of the remaining arrows.  The remaining small squares commute by naturality of the penumbra structure maps.  The large central square is \cref{eq:big_square}.
\begin{figure}
\input{II-main_umbra_diagram}
\caption{Diagram for the proof of \cref{thm:main_umbra}}\label{fig:commuting_traces}
\end{figure}
\end{proof}

\section{Monoidal bicategories to umbras}\label{sec:monoidal_bicat}

All of the examples of umbras in this paper arise from a monoidal bicategory and we show in \cref{thm:monoidal_to_umbra} that a monoidal bicategory where all objects are 2-dualizable (\cref{defn:2_dualizable}) has an associated umbra.  This completes the proof of \cref{thm:main_bzn}. The main work in this section is in verifying that the second condition in \cref{defn:umbra} holds.

A rough outline of this section is as follows. We introduce monoidal bicategories and the definitions of 1-dualizability internal to them. The opacity of these definitions lead us to adopt a graphical calculus that we call \emph{circuit diagrams} to better represent operations in monoidal bicategories. We then define a shadow in any symmetric monoidal bicategory with suitably dualizable 1-cells (and verify the shadow axioms). Under the additional constraints of 2-dualizability (\cref{defn:2_dualizable}) we can define all of the structure of an umbra. We emphasize that the verification of all of the umbral properties are the technical core of this paper.

For the definition of a monoidal bicategory we follow \cite{stay} and briefly recall the relevant structure here.  

\begin{defn}\cite{stay} A {\bf monoidal bicategory} consists of:
\begin{itemize}
\item  a bicategory $\sB$.
\item  a  functor  $\otimes \colon  \sB \times \sB \to \sB$ and with an invertible  2-morphism 
\[( M \otimes N) \odot  ( M ' \otimes  N') \Rightarrow ( M \odot M ') \otimes (N \odot N').\] 
\begin{itemize}
\item An adjoint equivalence 
\[a: (A\otimes B)\otimes C \to  A\otimes (B\otimes C)\] that is pseudonatural in $A, B, C$.
\item  An invertible modification $\pi$ relating the two different ways of moving parentheses from being clustered at the left of four objects to being clustered at the right. 
\item An equation of modifications relating the various ways of getting from the parentheses clustered at the left of five objects to clustered at the right.
\end{itemize}
\item  a 0-cell $I$
\begin{itemize}
\item  Adjoint equivalence 1-cells $L_A\colon I\otimes A \to A$ and  $R_A\colon A\otimes I \to  A$ that are pseudonatural in A.
\item  Invertible modifications $\lambda\colon l\otimes B\Rightarrow l\circ a$, $\mu\colon r\otimes B\Rightarrow (A\otimes l)\circ a$, and $\rho\colon r \Rightarrow (A\otimes r)\circ a$. 
\item Four equations of modifications relating the unit modifications. 
\end{itemize}
\end{itemize}
\end{defn}

In this section it is useful to keep the example of the  bicategory of rings, bimodules and homomorphisms in mind.  This bicategory is monoidal under the tensor product over $\bZ$. We will give more examples of monoidal bicategories in \cref{sec:enriched}.  The examples in that section are generalizations of this bicategory.  A less familiar monoidal bicategory is 
the parameterized stable homotopy bicategory \cite{may_sigurdson}.  It is monoidal under the smash product.

In \cref{sec:duality} we described a generalization of duality for objects in a symmetric monoidal category  to duality for 1-cells in a bicategory.  There is also a generalization to duality of 0-cells in a  monoidal bicategory.

\begin{defn}\label{defn:1_dualizable}
A 0-cell $A$ is {\bf 1-dualizable} if there is
\begin{itemize}
\item a zero cell $\zdual{A}$, 
\item  1-cells $C\in \sB(I, A\otimes \zdual{A})$ and $E\in \sB(\zdual{A}\otimes A, I)$ and 
\item invertible 2-cells 
\[U_A\xto{\triangle} {L_A}^{-1} \odot
	(C\otimes U_A)\odot
	(U_A\otimes E) \odot
	{R_A} \]
\[U_{\zdual{A}}\xto{\triangle} {R_{\zdual{A}}^{-1}}\odot
	(U_{\zdual{A}}\otimes C)\odot
	(E\otimes U_{\zdual{A}})\odot
	L_{\zdual{A}}\]
\end{itemize}
\end{defn}

\begin{example}
A ring $A$ is 1-dualizable and $\zdual{A}$ is $A^{\mathrm{op}}$.  The 1-cells $C$ and $E$ are $A$ regarded as a $(\bZ,A\otimes A^{\mathrm{op}})$-bimodule and 
$(A^{\mathrm{op}}\otimes A,\bZ)$-bimodule.  
\end{example}

The expressions for the targets of the invertible 2-cells in this definition are unwieldy, and since these are just the first of many unwieldy expressions we will make extensive use of graphical calculi following \cite{douglas_schommer_pries_snyder, schommer-pries,ps:indexed}.  We call these {\bf circuit diagrams}.

\begin{example}
As a first example, \cref{fig:dualizable_1_cell} represents the maps 
and compatibility for a dual pair of 1-cells in a bicategory.  We represent 0-cells by shaded lines and 
 1-cells by colored boxes. 
Bicategorical composition $\odot$ is represented by horizontal concatenation.
The unlabeled gray boxes may contain any composite of 1-cells.  2-cells are represented as arrows between nodes that contain a bicategorical composite of 1-cells.

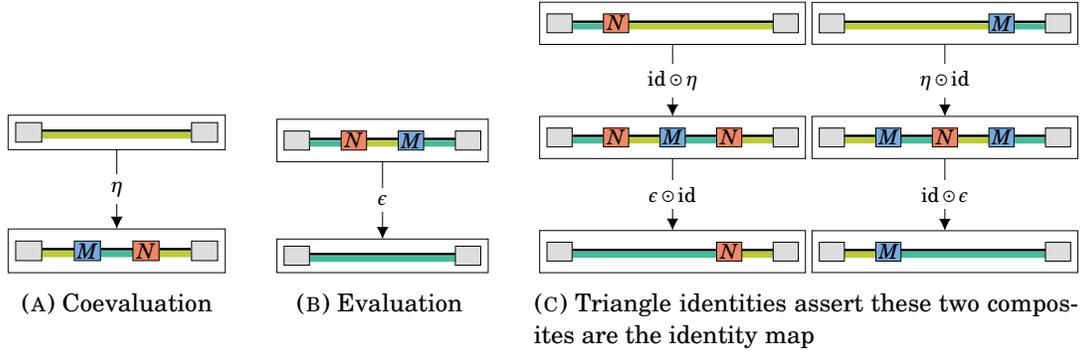
\begin{figure}
\input{II-dual_pairs_1_cell}
\caption{Circuit diagrams for dualizable 1-cells}\label{fig:dualizable_1_cell}
\end{figure}

\end{example}

The circuit diagrams associated to a monoidal bicategory can have more complicated structures than those associated to a bicategory. 
In a monoidal setting, we use vertical stacking to indicate the monoidal composition $\otimes$.  For 0-cells $A$ and $B$ the diagram for $A\otimes B$ has the edge for $A$ below the edge for $B$.  With this convention a 1-cell can have zero or any finite number of edges attached to the right and left sides.  If there are no edges attached to a side of a 1-cell then the 0-cell on that side is $I$.  If there is more than one edge attached to one side of a 1-cell, the 0-cell on that side is the monoidal product of those 0-cells.

\begin{example}\cref{fig:dualizable_0_cell} contains a first example of these features.  Note that the 1-cells $C$ and $E$ both have one side that has no inputs and one side with two inputs.  This means $C\in \sB(I, A\otimes B)$  and $E\in \sB(B\otimes A,I)$ for 0-cells $A$ and $B$.
Following \cite{douglas_schommer_pries_snyder}
the relative placement of the black line and the yellow line in the 0-cells in \cref{fig:dualizable_0_cell} indicates which 0-cells correspond to $A$ and which correspond to $\zdual{A}$.  In this case green above black is $\zdual{A}$ and black above green is $A$.
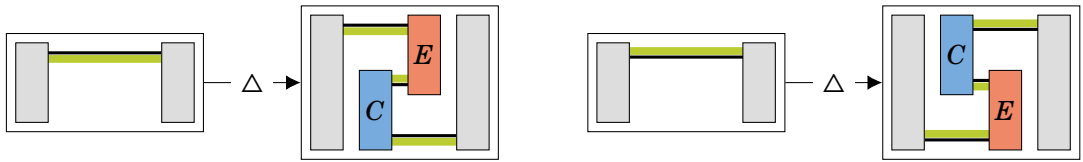
\begin{figure}[h]
\input{II-dual_coeval_eval_for_0_cells}
\caption{Invertible 2-cells completing the definition of a 1-dualizable 0-cell}\label{fig:dualizable_0_cell}
\end{figure}

\end{example}

Following \cite{schommer-pries,ps:indexed} we could view our diagrams of 0-cells and 1-cells as horizontal slices of a three dimensional diagram where 2-cells are represented vertically as surfaces.  
We choose to not use surface diagrams since  the projections of 3-dimensional diagrams onto a page can be difficult to interpret.  We also choose not to prove our results using geometric reasoning since we have no desire to prove the relevant coherence theorem here.

In what follows we will give composites of 1-cells as circuit diagrams and we will not provide a translation similar to the original description of the invertible 2-cells in \cref{defn:1_dualizable}.  We will also suppress associativity and unit 1-cells.

The monoidal bicategories we are interested in are also symmetric and this structure is necessary for many of the results we need.
\begin{defn}\cite{stay} 
A {\bf symmetric monoidal bicategory} consists of the following:
\begin{itemize}
	\item A monoidal bicategory $\mc{B}$.
	\item An adjoint equivalence 
			\[b: A\otimes B \to B\otimes A\]
		pseudonatural in $A$ and  $B$.
	\item Invertible modifications $R$ and $S$ filling the hexagons
			\[\xymatrix{A\otimes (B\otimes C)\ar[r]^b&(B\otimes C)\otimes A\ar[d]^a
				\\
				(A\otimes B)\otimes C\ar[u]^a\ar[d]_{b\otimes C}&B\otimes (C\otimes A)
				\\
				(B\otimes A)\otimes C\ar[r]^a&B\otimes (A\otimes C)\ar[u]_{B\otimes b}} 
	\hfill 
			\xymatrix{(A\otimes B)\otimes C\ar[r]^b&C\otimes (A\otimes B)\ar[d]_a
				\\
				A\otimes (B\otimes C)\ar[u]^a\ar[d]_{A\otimes b}&(C\otimes A)\otimes B
				\\
				A\otimes (C\otimes B)\ar[r]^a&(A\otimes C)\otimes B\ar[u]^{b\otimes B}}\]
	\begin{itemize}
		\item Two equations shuffling one object into three objects.
		\item An equation shuffling two objects into two objects 
		\item An equation relating multiple applications of $b$.
	\end{itemize}
	\item An invertible modification 
		\[\nu\colon b\to b\]
		satisfying two equations relating $\nu$ and the modifications $R$ and $S$ above and an equation relating $\nu$ applied to $A\otimes B$ and $B\otimes A$.
\end{itemize}
\end{defn}

Monoidal bicategories have significant structure.  In particular, the shadows on the bicategories in \cref{def:enriched} are induced by the monoidal structure.

\begin{prop}\cite{ben_zvi_nadler_1}\label{prop:shadow} If $\mc{B}$ is a symmetric monoidal bicategory and all 0-cells of $\sB$ are 1-dualizable $\sB$ has a shadow that takes values in the 
category $\sB(I,I)$.
\end{prop}

\begin{proof} 
If $M\in \sB(A,A)$ and $A$ is 1-dualizable, the shadow of $M$ is the following bicategorical composition.
\begin{center}
\begin{tikzpicture}

\oc{t1}{C}{_1}{0}{5}{4}{\et}
\oc{x}{M}{}{1.5}{4}{4}{\en}

\oc{ex}{E}{_x}{4}{5}{4}{\ep}
\gc{g1}{3}{5}{4}

  \begin{scope}[on background layer]
\dlp{({t1}t)\gpt{g1}({ex}b)}

\alp{ ({t1}b)--({x}b)\gpb{g1}({ex}t)}
\end{scope}
\end{tikzpicture}
\end{center}
The shadow isomorphism is defined in \cref{fig:shadow_isomorphism}. 

The diagrams relating the shadow isomorphism and the unit and associativity maps are very large and can be found on page \pageref{fig:shadow_isomorphism_assoc}.
See \cref{fig:shadow_isomorphism_assoc} for  the associativity condition and \cref{fig:shadow_isomorphism_unit} one of the unit conditions.  All regions in these diagram commute by the naturality of the monoidal symmetry map and the unit maps.
\end{proof}

\afterpage{\clearpage
\begin{figure}
\resizebox{.85\textwidth}{!}
{\input{II-shadow_isomorphism}}
\caption{The shadow isomorphism}\label{fig:shadow_isomorphism}
\end{figure}
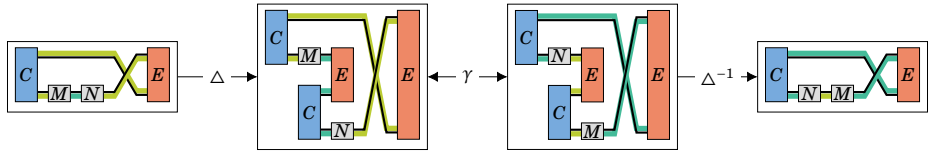

\begin{figure}
\resizebox{.75\textwidth}{!}
{\input{II-shadow_isomorphism_assoc}}
\caption{The associativity condition for the shadow isomorphism}\label{fig:shadow_isomorphism_assoc}
\end{figure}
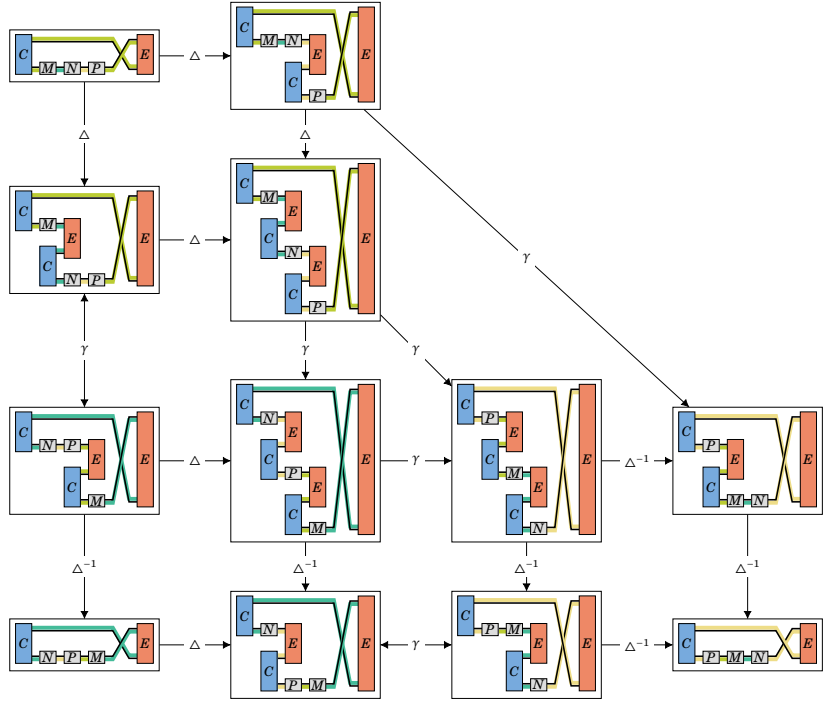

\begin{figure}
\resizebox{.75\textwidth}{!}
{\input{II-shadow_isomorphism_unit}}
\caption{The unit condition for the shadow isomorphism}\label{fig:shadow_isomorphism_unit}
\end{figure}
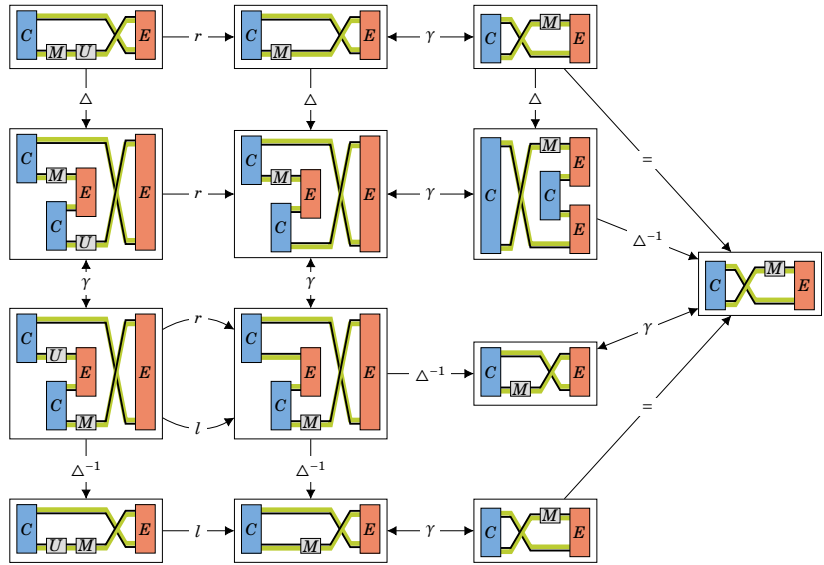
\clearpage}

In addition to defining a shadow, 1-cells $(C,E)$ witnessing the 1-dualizability of a 0-cell satisfy some strong compatibility results.
\begin{lem}
For a 1-dualizable 0-cell $A$ with witnessing 1-cells $(C,E)$  the 1-cell $C$ is left (respectively right) dualizable if and only if $E$ is right  (respectively left) dualizable. 
\end{lem}

Before defining the relevant coevaluation and evaluations it is helpful to first observe the shapes of  circuit diagrams of the coevaluation, evaluation and triangle identities for 1-cells in $\sB(1,A\otimes B)$ and $\sB(A\otimes B,1)$.  These are in \cref{fig:coeval_1_cell,fig:coeval_1_cell_2}.  The empty nodes in this diagram should be regarded as filled by the unit 1-cell for the monoidal unit $I$.

\afterpage
{\clearpage
\begin{figure}
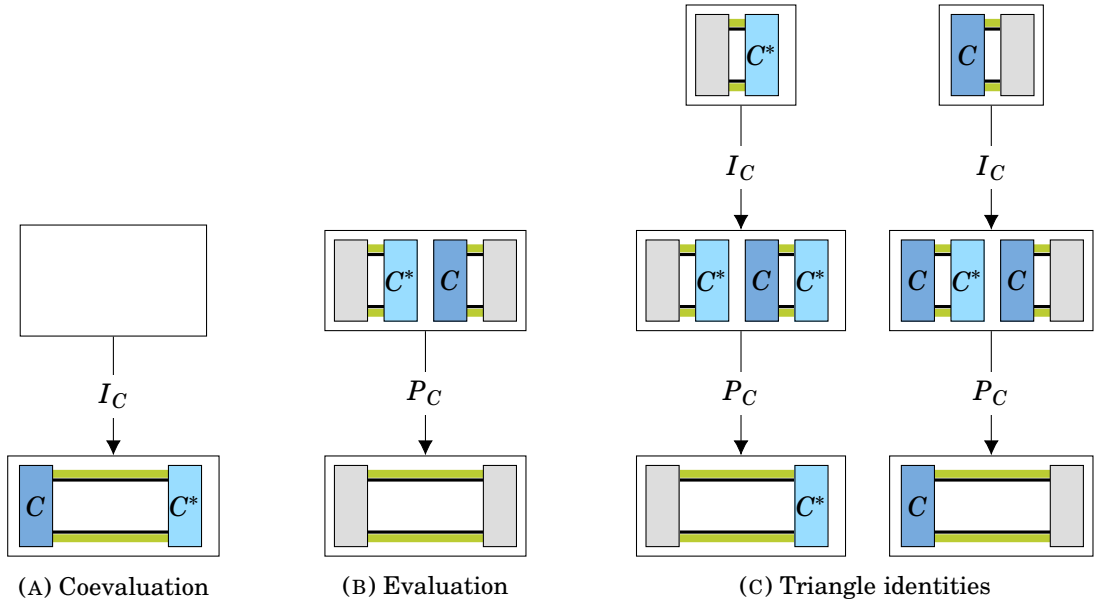

\include{II-dual_pair_2_cell}
\caption{Coevaluation, evaluation and triangle identities for right dualizability of a 1-cell $C$ in $\sB(1, A\otimes B)$.  Compare to \cref{fig:dualizable_1_cell}.}\label{fig:coeval_1_cell}
\end{figure}
\begin{figure}
\include{II-dual_pair_2_cell_flip}
\caption{Coevaluation, evaluation and triangle identities for right dualizability of a 1-cell $E$ in $\sB(A\otimes B,1)$.  Compare to \cref{fig:dualizable_1_cell}.}\label{fig:coeval_1_cell_2}
\end{figure}
\clearpage}

\begin{proof}
We will show that if $C$ is left dualizable then $E$ is right dualizable.  The other statement is dual.

If there are 2-cells as in the first two subfigures of \cref{fig:coeval_1_cell} then the dual of $E$ is the following composite.
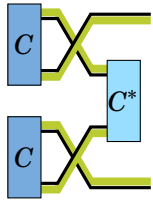
\begin{center}
\begin{tikzpicture}
\oc{ta}{C}{_a}{2}{2}{1}{\et}
\oc{t2}{\rdual{C}}{_2}{4}{1}{0}{\etr}
\oc{t3}{C}{_3}{2}{0}{-1}{\et\lt}

\gc{g1}{3}{2}{1}
\gc{g2}{3}{0}{-1}
  \begin{scope}[on background layer]
\dlp{ ({ta}t)\gpt{g1}({t2}t)}
\dlp{({t3}t)\gpt{g2}(3,-.75)}
\alp{({t3}b)\gpb{g2}({t2}b)}
\alp{({ta}b)\gpb{g1}(3,1.5)}
\end{scope}
\end{tikzpicture}
\end{center}
The coevaluation and evaluation for $E$ are defined in  \cref{fig:coeval_eval_E}.  The triangle diagrams for this evaluation and coevaluation are large and so appear on page \pageref{fig:triangle_for_E}.  See \cref{fig:triangle_for_E} for one of the triangle diagrams for $E$.  The other triangle diagram is similar.
\end{proof}

\afterpage{\clearpage
\begin{figure}
\input{II-dual_coeval_eval_for_1_cells}
\caption{The coevaluation and evaluation for $E$}\label{fig:coeval_eval_E}
\end{figure}
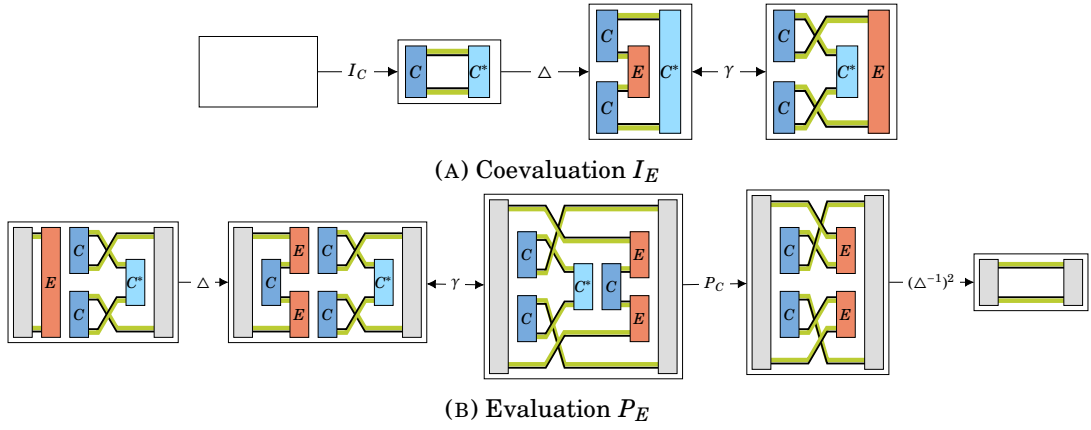

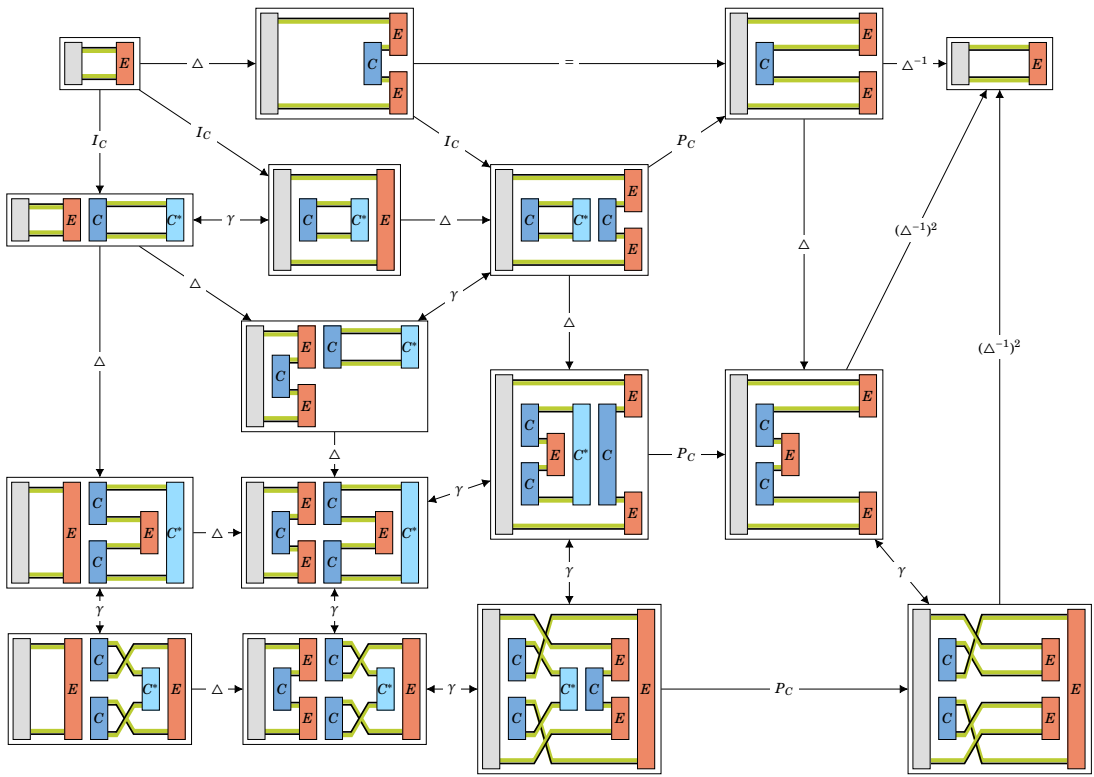
\begin{figure}
\resizebox{\textwidth}{!}
{\input{II-compatibility_left_right_duals}}
\caption{One of the triangle diagrams for $E$}\label{fig:triangle_for_E}
\end{figure}
\clearpage}

This lemma motivates the following definition.

\begin{defn}\label{defn:2_dualizable}
A 1-dualizable $0$-cell $A$ is {\bf 2-dualizable} if the witnessing 1-cells $(C,E)$ satisfy any of the following equivalent conditions 
\begin{enumerate}
\item $C$ and $E$ are left dualizable
\item $C$ is left and right dualizable
\item $C$ and $E$ are right dualizable 
\item $E$ is left and right dualizable
\end{enumerate} 
\end{defn}
This is a special case of \emph{fully dualizable} objects in an $n$-category \cite{schommer-pries,lurie_cobordism}.

\begin{example}
  For dg-algebras, 2-dualizability is the familiar condition of  smoothness and properness. For example, $A$ is a dg-algebra, and the derived category $\mc{D}(A)$ is equivalent to the derived category $\mc{D}(X)$ of the category of quasicoherent sheaves of a smooth, proper scheme $X$, then $A$ is smooth and proper. Thus, there is a ready supply of them.

  Smoothness and properness amount to the same conditions for algebras in spectra, but the conditions seem harder to satisfy, and we know of only $K(n)$-local examples. This will be the subject of a sequel to this paper. 
\end{example}

We now come to the heart of this paper and complete the proof of \cref{thm:main_bzn}. The rest of this section is occupied with the diagrams verifying this statement. 

\begin{thm}\label{thm:monoidal_to_umbra}
If $\sB$ is a symmetric monoidal bicategory where each 0-cell is 2-dualizable then $\sB$ is an umbra.
\end{thm}

\begin{proof}
 If $A$ is a 2-dualizable 0-cell there are maps as in \cref{fig:dual_triangle_2}.  
These are inverses.  See \cref{fig:triangle_star_invertible} on page \pageref{fig:triangle_star_invertible}.

\afterpage{\clearpage

\begin{figure}
\resizebox{.4\textwidth}{!}
{\input{II-dual_triangle}}
\hfill 
\resizebox{.4\textwidth}{!}
{\input{II-dual_triangle_1}}
\caption{The triangle 2-cells for $\rdual{C}/\rdual{E}$}\label{fig:dual_triangle_2}
\end{figure}

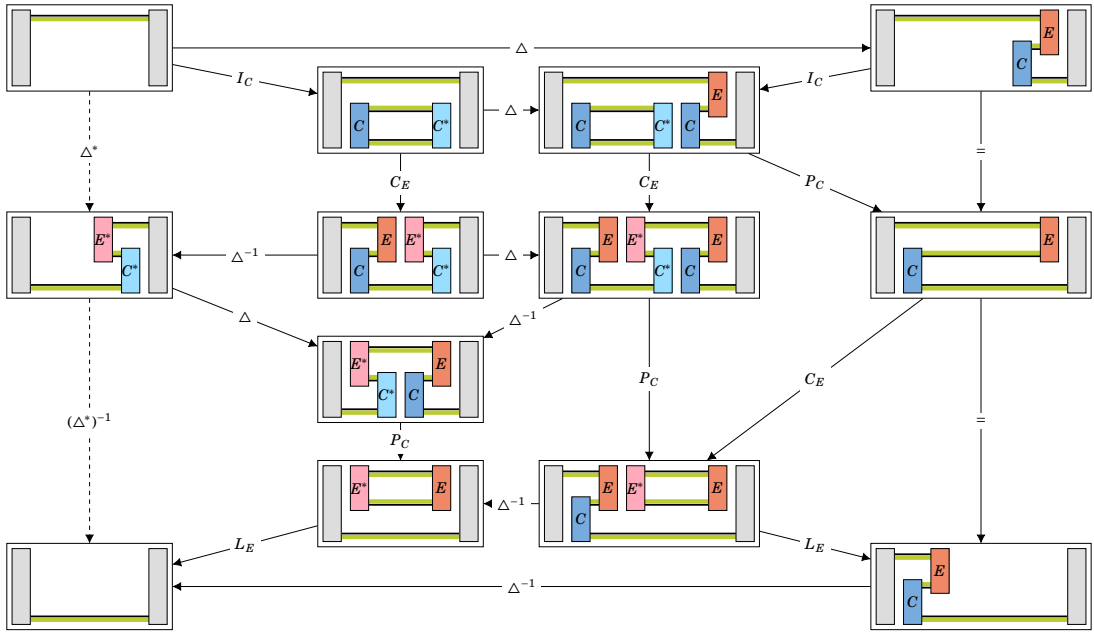
\begin{figure}
\resizebox{\textwidth}{!}{\input{II-dual_triangle_composites}}
\caption{Demonstrating the 2-cells for $\rdual{C}$ and $\rdual{E}$ are invertible}\label{fig:triangle_star_invertible}
\end{figure}

\begin{figure}
\resizebox{\textwidth}{!}{\input{II-dual_triangle_compare}}
\caption{Further compatibility of $\triangle^*$ and $\triangle$}\label{fig:dual_triangle_compare}
\end{figure}
\clearpage}

Then \cref{prop:shadow} implies there is a second shadow defined on $\sB(A,A)$ as the bicategorical composition 
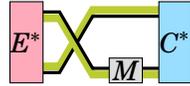
\begin{center}
\begin{tikzpicture}
\oc{e1}{\rdual{E}}{_1}{0}{5}{4}{\epr}
\oc{x}{M}{}{2}{4}{4}{\en}

\oc{tx}{\rdual{C}}{_x}{3}{5}{4}{\etr}
\gc{g1}{1}{5}{4}
  \begin{scope}[on background layer]
\dlp{ ({e1}b)\gpb{g1}({tx}t)}
\alp{({e1}t)\gpt{g1}({x}b)--({tx}b)}
\end{scope}
\end{tikzpicture}
\end{center}
There are two other bicategorical compositions 
defined using $C$, $E$ and their duals that define functors $\sB(A,A)\to \sB(I, I)$.  These two constructions as well as the two shadows defined above are compared in \cref{fig:penumbra_functors_B}.

These define the four functors of a penumbra.  The maps are defined in \cref{fig:penumbra_maps_B} and the required commutative diagrams follow from conditions on the bicategory and the triangle identities in \cref{fig:coeval_1_cell,fig:coeval_1_cell_2}.
\begin{figure}
\input{II-four_shadows}
\input{II-four_shadows_maps}
\caption{Shadows,  shadows like constructions and their maps on $\sB$}\label{fig:penumbra_functors_B}
\label{fig:penumbra_maps_B}
\end{figure}
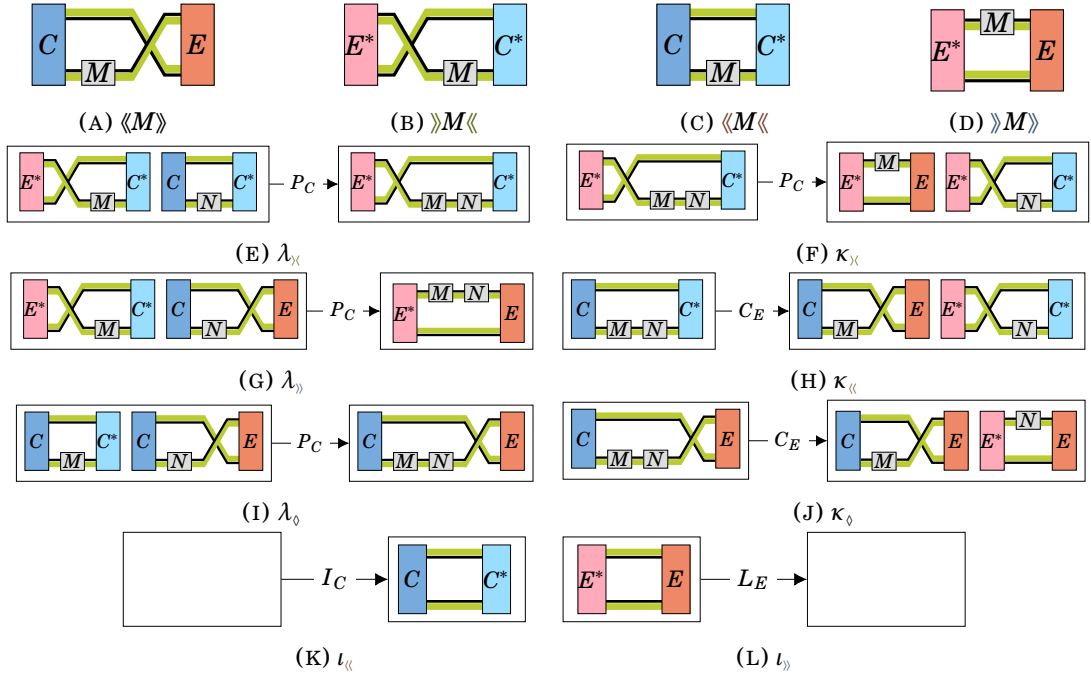
 
In a monoidal bicategory the composite in \cref{fig:symmetric_shadow}, called the {\bf 3-fold twisting map}, switches the order of the 1-cells as in \cref{eq:big_square} but treats them all more symmetrically.  
We verify the large diagram in the definition of an umbra commutes by showing the composites agree with the 3-fold twisting map.  One of these comparison is in 
 \cref{fig:compare_3_shadows} on page \pageref{fig:compare_3_shadows}.  The small regions commute by naturality.  All but one are immediate.  The remaining square commutes by  \cref{fig:dual_triangle_compare} on page \pageref{fig:dual_triangle_compare}.

The comparison of the 3-fold twisting map and the other composite in \cref{eq:big_square} is similar.
\end{proof}

\afterpage{\clearpage
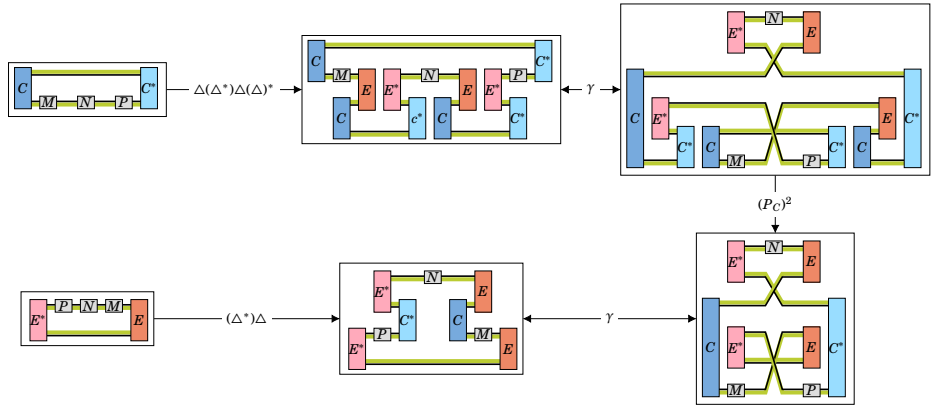
\begin{figure}
\resizebox{.85\textwidth}{!}
{\input{II-symmetric_shadow}}
\caption{The 3-fold twisting map}\label{fig:symmetric_shadow}
\end{figure}
\begin{figure}
\resizebox{\textwidth}{!}{\input{II-comparing_sided_trace_to_symmetric}}
\caption{Comparing left composite to symmetric shadow}\label{fig:compare_3_shadows}
\end{figure}
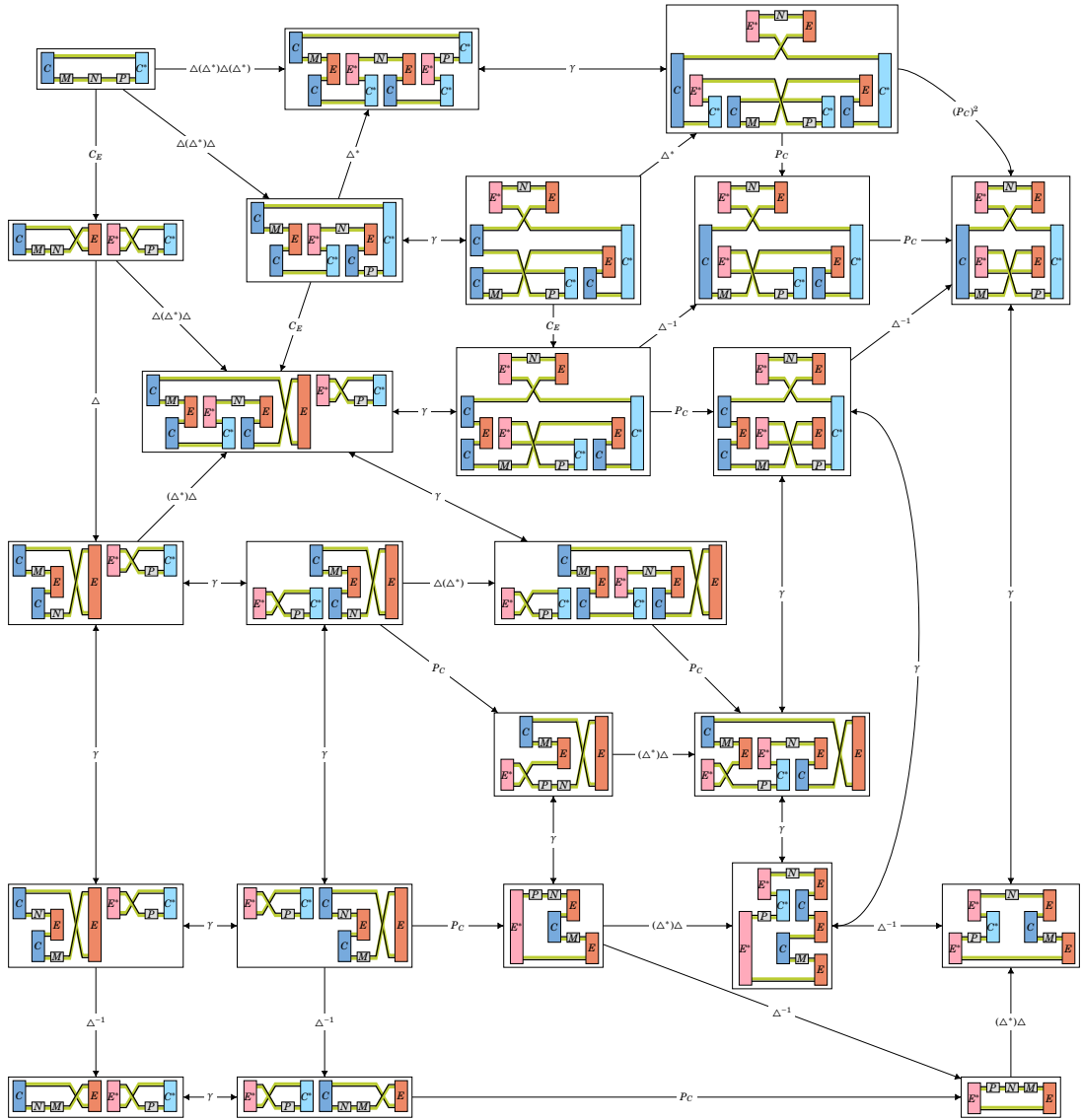
\clearpage}

\section{Applications: 2-Characters}\label{sec:two_characters}

One of the major motivations for iterated traces is categorical character theory. The notion of a categorical 2-character seems to have been introduced by Ganter-Kapranov in \cite{ganter_kapranov}, with motivation from the theory of characters in Hopkins-Kuhn-Ravenel \cite{Hopkins-Kuhn-Ravenel}. 
Ben-Zvi--Nadler \cite[Thm. 1.4]{ben_zvi_nadler_1} establish the modular invariance of the 2-categorical character.  Since this is an easy consequence of the techniques we use here, we include the proof as an introduction to the power of these ideas. 

In our development, the modular invariance of  the 2-categorical character does not have much to do with group theory, and so we start with a few more results in the flavor of \cref{sec:bicat_duals} and then recall the relevant definitions of \cite{ganter_kapranov}.  It seems that the most natural place to discuss 2-characters is not the context of 2-categories, but instead the context of double categories. This context makes the action by the usual generators $S, T$ of the modular group apparent. Given this, much of the rest of this section is setting up the appropriate double categorical language. Once this is in place, the modular invariance of 2-characters follows formally from work above.

\subsection{Double shadows}
A {\bf double category} is a category internal to $Cat$.  We think of a double category $\mathbb{D}$ as 0-cells, vertical 1-cells, horizontal 1-cells, and 2-cells that fill squares
\begin{equation*}\begin{tikzcd}
 \arrow[d, "X" description] \arrow[r, "f" description] &  \arrow[d, "Y" description] \\
 \arrow[r, "g" description]                            & {}                      
\end{tikzcd}\end{equation*} 

\begin{example}  For a bicategory $\sB$, let $D(\sB)$ be the double category 
whose 0-cells are the 0-cells of $\sB$.  
The vertical  1-cells are tuples $(X,X^*,\eta,\epsilon)$ consisting of a right dualizable 1-cell and a choice of right dual, coevaluation and evaluation.  The horizontal 1-cells are tuples $(X,^*X,\eta,\epsilon)$ consisting of a left dualizable 1-cell and a choice of left dual, coevaluation and evaluation.  The 2-cells filling a square 
\begin{equation*}\begin{tikzcd}
 \arrow[dd, "{(Z,Z^*,\eta_Z,\epsilon_Z)}" description] \arrow[rr, "{(X,^*X,\eta_X,\epsilon_X)}"] &&  \arrow[dd, "{(W,W^*,\eta_W,\epsilon_W)}" description] \\ \\
 \arrow[rr, "{(Y,^*Y,\eta_Y,\epsilon_Y)}" ]                            & &{}                          
\end{tikzcd}\label{fig:standard_duals}
\end{equation*}
are maps $\alpha \colon X\odot W\to Z\odot Y $.  
\end{example}

In a double category $\mathbb{D}$  a vertical 1-cell $ X:A\to B$ and a horizontal 1-cell $\cm{X}:A\to B$ are a {\bf companion pair} if they come equipped with 2-cells
\[\begin{tikzcd}
 \arrow[d, "X" description] \arrow[r, "\cm{X}" description] & \arrow[ld,shorten <>=10pt,Rightarrow]  \arrow[d, dotted,] \\
 \arrow[r,dotted]                            & {}                          
\end{tikzcd}
\hspace{1cm}\begin{tikzcd}
 \arrow[d, dotted,]  \arrow[r,dotted]  & \arrow[ld,shorten <>=10pt,Rightarrow]   \arrow[d, "X" description] \\
  \arrow[r, "\cm{X}" description]                        & {}                          
\end{tikzcd}\]
so the pastings along $X$ and $\cm{X}$ are identity 2-cells.
A vertical 1-cell $ X:A\to B$ and a horizontal 1-cell $\cn{X}:A\to B$ are a 
{\bf conjoint pair} if they come equipped with 2-cells
\[\begin{tikzcd}
 \arrow[d, "X" description]  \arrow[r,dotted]  &  \arrow[d, dotted,] \arrow[dl,shorten <>=10pt,Rightarrow] \\
       \arrow[r, "\cn{X}" description]                   & {}                          
\end{tikzcd}
\hspace{1cm}\begin{tikzcd}
\arrow[d, dotted,] \arrow[r, "\cn{X}" description]            &  \arrow[dl,shorten <>=10pt,Rightarrow] \arrow[d, "X" description]      \\
        \arrow[r,dotted]     & {}                          
\end{tikzcd}\]
so the pastings along $X$ and $\cn{X}$ are identity 2-cells.

\begin{example}
 The double category $D(\sB)$ has functorial  conjoints that take 
$(X,X^*,\eta,\epsilon)$ to $(X^*,X,\eta,\epsilon)$.  The 2-cells that demonstrate this is a conjoint pair are the maps $\eta$ and $\epsilon$. 
\end{example}

In this bicategory companions require additional definitions.
\begin{defn}
A 1-cell $X$ in a bicategory $\sB$ is {\bf invertible} if there is a 1-cell $X^{-1}$ and invertible 2-cells
$U_A \cong X \odot X^{-1}$ and $X^{-1} \odot X\cong U_B $ satisfying the triangle identities for a dual pair.
\end{defn}

\begin{lem}
If $X$ is an invertible 1-cell with inverse $X^{-1}$, then $X^{-1}$ is both the left and right dual of $X$. 
\end{lem}
\begin{proof}
  Since there are isomorphisms $U_A \cong X \odot X^{-1}$ and $X^{-1} \odot X\cong U_B $ 
the maps 
	\[(U_A\to X\odot X^{-1}, X^{-1}\odot X\to U_B)\]
 demonstrate $X$ is right dualizable.  The maps 
	\[(U_B\to X^{-1}\odot X, X\odot X^{-1}\to U_A)\]
demonstrate $X$ is left dualizable.
\end{proof}

\begin{example}
The 1-cells $(X,X^*,\eta,\epsilon)$ where $X^*$ is an inverse for $X$ and $\eta$ and $\epsilon$ are isomorphisms have companions. 
The companion of $(X,X^*,\eta,\epsilon)$  is $(X,X^*,\epsilon^{-1},\eta^{-1})$.   The 2-cells that demonstrate this is a companion pair are the identity 2-cells.

Let $I(\sB)$ be the subdouble category of $D(\sB)$ consisting of invertible 1-cells.
\end{example}

A 2-cell $\alpha$ in a double category $\mathbb{D}$ filling a square of the form 
\begin{equation*}\begin{tikzcd}
 \arrow[d, "X" description] \arrow[r, "f" description] &  \arrow[d, "X" description] \\
 \arrow[r, "f" description]                            & {}                          
\end{tikzcd}\end{equation*} 
is an {\bf endomorphism 2-cell} in $\mathbb{D}$.
A {\bf double shadow} on a double category $\mathbb{D}$ is a function from the set of endomorphism 2-cells to a set $S$ so that the images of
\[\begin{tikzcd}
 \arrow[r, "f" description] \arrow[d, "X" description] &   \arrow[ld,shorten <>=10pt,Rightarrow,"\alpha" above]\arrow[d, "Y" description] \arrow[r, "g" description]  & \arrow[ld, shorten <>=10pt,Rightarrow,"\beta"above ] \arrow[d, "X" description] \\
 \arrow[r, "f" description]                                                  &  \arrow[r, "g" description]                                                  & {}                          
\end{tikzcd}\text{ and } \begin{tikzcd}
 \arrow[r, "g" description] \arrow[d, "Y" description]& \arrow[ld,shorten <>=10pt,Rightarrow, "\beta" above] \arrow[d, "X" description] \arrow[r, "f" description]  &\arrow[ld,shorten <>=10pt,Rightarrow, "\alpha" above]  \arrow[d, "Y" description] \\
 \arrow[r, "g" description]                                                  &  \arrow[r, "f" description]                                                  & {}                          
\end{tikzcd}\]
agree and the same for the images of 
\[ \begin{tikzcd}
 \arrow[d, "X" description] \arrow[r, "f" description]  &\arrow[ld,shorten <>=10pt,Rightarrow, "\alpha" above]  \arrow[d, "X" description] \\
 \arrow[d, "Y" description] \arrow[r, "g" description] &\arrow[ld,shorten <>=10pt,Rightarrow, "\beta" above ]  \arrow[d, "Y" description] \\
 \arrow[r, "f" description]                                                  & {}                          
\end{tikzcd}\text{ and } \begin{tikzcd}
 \arrow[d, "Y" description] \arrow[r, "g" description] &\arrow[ld,shorten <>=10pt,Rightarrow, "\beta" above]  \arrow[d, "Y" description] \\
 \arrow[d, "X" description] \arrow[r, "f" description] &\arrow[ld,shorten <>=10pt,Rightarrow, "\alpha" above]  \arrow[d, "X" description] \\
 \arrow[r, "g" description]                                                  & {}                        
\end{tikzcd}.\]
(This is similar but not identical to the definition in \cite[13.7]{ps:linearity}.)

\begin{thm}
The iterated trace is a double shadow on $D(\sB)$.
\end{thm}

\begin{proof}
The iterated trace applied to a 2-cell filling  a square of the form 
\begin{equation*}\begin{tikzcd}
 \arrow[d, "X" description] \arrow[r, "f" description] &  \arrow[d, "X" description] \\
 \arrow[r, "f" description]                            & {}                          
\end{tikzcd}\end{equation*} 
is $\tr_{\sh{f}}(\tr_X(\alpha))=\tr_{\sh{X}}(\tr_f(\alpha))$.

Given 2-cells 
\[\begin{tikzcd}
 \arrow[r, "f" description] \arrow[d, "X"  description]  &  \arrow[ld,shorten <>=10pt,Rightarrow, "\alpha"above ]\arrow[d, "Y" description] \arrow[r, "g" description] & \arrow[ld,shorten <>=10pt,Rightarrow, "\beta" above] \arrow[d, "X" description] \\
 \arrow[r, "f" description]                                      &  \arrow[r, "g" description]                                                  & {}                          
\end{tikzcd}\]
Then 
\[\xymatrix@C=80pt{
\sh{f\odot g}\ar[r]^{\tr_X\left((\alpha\odot 1_g)\circ (1_f\odot \beta)\right)}\ar[d]^\sim &\sh{f\odot g}\ar[d]^\sim
\\
\sh{g\odot f}\ar[r]^{\tr_Y\left((\beta\odot 1_f)\circ (1_g\circ \alpha)\right)}&\sh{g\odot f}
}\]
by \cite[7.2]{ps:bicat}.
Then 
\[\tr_{\sh{f\odot g}}(\tr_X\left((\alpha\odot 1_g)\circ (1_f\odot \beta)\right))=\tr_{\sh{g\odot f}}(\tr_Y\left((\beta\odot 1_f)\circ (1_g\circ \alpha)\right))\]
by \cite[2.4]{ps:symmetric}. 

The other condition for a double shadow is similar.
\end{proof}

The next subsection is devoted to
describing symmetry properties of the endomorphisms of a double category and double shadow.

\subsection{Action of $SL_2(\mathbb{Z})$}
For a double category $\mathbb{D}$ let $2_\mathbb{D}$ be the set of 2-cells  subject to the relation 
$\alpha\sim\beta$ if there exist  isomorphisms 2-cells 
$\gamma_1, \gamma_2,\gamma_3,\gamma_4$ so that 
\[\begin{tikzcd}
 \arrow[d, dotted] \arrow[r, dotted]          &  \arrow[d, dotted] \arrow[r, "f" description]                          &  \arrow[d, dotted] \arrow[r, dotted]          &  \arrow[d, dotted]          \\
 \arrow[d, "X" description] \arrow[r, dotted] &  \arrow[d, "X" description] \arrow[r, "f" description]  &  \arrow[ld,shorten <>=10pt,Rightarrow,"\alpha" above]\arrow[d, "Y" description] \arrow[r, dotted] &  \arrow[d, "Y" description] \\
 \arrow[d, dotted] \arrow[r, dotted]          &  \arrow[r, "g" description] \arrow[d, dotted]                          & \arrow[r, dotted] \arrow[d, dotted]         &  \arrow[d, dotted]          \\
 \arrow[r, dotted]                            &  \arrow[r, "g" description]                                            &  \arrow[r, dotted]                            & {}                          
\end{tikzcd} = 
\begin{tikzcd}
 \arrow[d, dotted] \arrow[r, dotted]                          &  \arrow[d, dotted] \arrow[r, "f" description]     & \arrow[ld,shorten <>=10pt,Rightarrow,"\gamma_1" above] \arrow[d, dotted] \arrow[r, dotted]                          &  \arrow[d, dotted]          \\
 \arrow[d, "X" description] \arrow[r, dotted] &  \arrow[ld,shorten <>=10pt,Rightarrow,"\gamma_2" above]\arrow[d, "X'" description] \arrow[r, "f'" description]  & \arrow[ld,shorten <>=10pt,Rightarrow,"\beta" above] \arrow[d, "Y'" description] \arrow[r, dotted] &\arrow[ld,shorten <>=10pt,Rightarrow,"\gamma_3" above]  \arrow[d, "Y" description] \\
 \arrow[d, dotted] \arrow[r, dotted]                          &  \arrow[r, "g'" description] \arrow[d, dotted]      & \arrow[ld,shorten <>=10pt,Rightarrow, "\gamma_4" above] \arrow[r, dotted] \arrow[d, dotted]                         &  \arrow[d, dotted]          \\
 \arrow[r, dotted]                                            &  \arrow[r, "g" description]                                            &  \arrow[r, dotted]                                            & {}                          
\end{tikzcd}
\]
The dotted arrows are identity 1-cells and empty squares are filled by identity 2-cells.

\begin{lem}
There are group actions on $2_\mathbb{D}$ as specified in  \cref{table:group_act_2D}.
\begin{table}[]
\begin{tabular}{|l|p{8cm}|}
\hline
Group&Assumption\\ \hline
$\mathbb{Z}=\langle S\rangle$                  & There are companions and conjoint for all 1-cells                                                                                 \\ \hline
$\mathbb{Z}/4=\langle S| S^4\rangle$                                   & $\cn{\cm{X}}=\cm{\cn{X}}$ for all 1-cells and the diagrams in \cref{fig:big_z4_1,fig:big_z4_2} define the same 2-cells               \\ \hline
$\mathbb{N}=\langle T\rangle$                                & There are companions and conjoints for all horizontal 1-cells                                                                  \\ \hline
$\mathbb{Z}=\langle T\rangle$                          & There are companions and conjoints for all horizontal 1-cells.  There are  2-cells as in \cref{fig:inverse_for_t} and the 2-cells in \cref{fig:inverse_for_t_2} are identities    \\ \hline
$\mathbb{Z}/4*\mathbb{Z}=\langle S,T|S^4\rangle $ & All of the above                                                                                                             \\ \hline
\end{tabular}
\caption{Group actions on $2_\mathbb{D}$}\label{table:group_act_2D}
\end{table}
\end{lem}

\begin{proof}
Let $\tend{D}{X}fYg$ be the subset of $2_\mathbb{D}$ that fills squares of the form 
\begin{equation}\label{eq:shape}
\begin{tikzcd}[ampersand replacement=\&]
 {}\arrow[d, "X" description] \arrow[r, "f" description] \&  \arrow[d, "Y" description] \\
 {}\arrow[r, "g" description] \&{}                          
\end{tikzcd}
\end{equation}

The diagrams in \cref{fig:actions_of_generators} define functions
\begin{align}S\colon \tend{D}{X}{f}{Y}{g}&\to \tend{D}{\cn{f}}{\cm{Y}}{\cn{g}}{\cm{X}} \label{eq:map_s}
\\
S^{-1}\colon \tend{D}{X}{f}{Y}{g}&\to \tend{D}{\cm{g}}{\cn{X}}{\cm{f}}{\cn{Y}}\label{eq:map_s_inv}
\\
T\colon \tend{D}{X}{f}{Y}{g}&\to \tend{D}{(X\odot \cm{g})}{f}{(Y\odot \cm{g})}{g} \label{eq:map_t}
\\
T^{-1}\colon \tend{D}{X}{f}{Y}{g}&\to \tend{D}{(X\odot \cn{g})}{f}{(Y\odot \cn{g})}{g} \label{eq:map_t_inv}
\end{align}
See \cref{fig:s_t_inverses} for the verification that $S^{-1}$ acts as an inverse for $S$ and $T^{-1}$ acts as the inverse of $T$.
\begin{figure}
  \begin{subfigure}[b]{0.27\textwidth}
	\adjustbox{max width=\textwidth}{\begin{tikzcd}
 \arrow[d, "\cn{f}" description] \arrow[r, dotted]              &  \arrow[d, "\cn{f}" description] \arrow[r, dotted] & \arrow[ld,shorten <>=10pt,Rightarrow] \arrow[d, dotted] \arrow[r, "\cm{Y}" description]                                &  \arrow[d, dotted]          \\
\arrow[r, dotted] \arrow[d, dotted] &\arrow[ld,shorten <>=10pt,Rightarrow]  \arrow[d, "X" description] \arrow[r, "f" description]              &\arrow[ld,shorten <>=10pt,Rightarrow, "\alpha" above ]  \arrow[d, "Y" description] \arrow[r, "\cm{Y}" description]  & \arrow[ld,shorten <>=10pt,Rightarrow]  \arrow[d, dotted]          \\
 \arrow[r, "\cm{X}" description] \arrow[d, dotted]              &  \arrow[r, "g" description]  \arrow[d, dotted] & \arrow[ld,shorten <>=10pt,Rightarrow]\arrow[d, "\cn{g}" description] \arrow[r, dotted]                               &  \arrow[d, "\cn{g}" description] \\
 \arrow[r, "\cm{X}" description]                                &  \arrow[r, dotted]                                                  &  \arrow[r, dotted]                                                           & {}
\end{tikzcd}}
\caption{Action of $S$
 \eqref{eq:map_s}}
\end{subfigure}
\hfill 
  	\begin{subfigure}[b]{0.27\textwidth}
\adjustbox{max width=\textwidth}{\begin{tikzcd}
 \arrow[d, dotted] \arrow[r, "\cn{X}" description]                            &  \arrow[d, dotted] \arrow[r, dotted]        &\arrow[ld,shorten <>=10pt,Rightarrow]  \arrow[d, "\cm{f}" description] \arrow[r, dotted]                            &  \arrow[d, "\cm{f}" description] \\
 \arrow[r, "\cn{X}" description] \arrow[d, dotted] &\arrow[ld,shorten <>=10pt,Rightarrow]  \arrow[d, "X" description] \arrow[r, "f" description]                 &  \arrow[ld,shorten <>=10pt,Rightarrow, "\alpha" above]\arrow[d, "Y" description]  \arrow[r, dotted] & \arrow[ld,shorten <>=10pt,Rightarrow] \arrow[d, dotted]          \\
 \arrow[r, dotted] \arrow[d, "\cm{g}" description]                            &  \arrow[r, "g" description] \arrow[d, "\cm{g}" description]  & \arrow[ld,shorten <>=10pt,Rightarrow] \arrow[r, "\cn{Y}" description] \arrow[d, dotted]                            &  \arrow[d, dotted]          \\
 \arrow[r, dotted]                                                       &  \arrow[r, dotted]                                                     &  \arrow[r, "\cn{Y}" description]                                              & {}                         
\end{tikzcd}}
\caption{Action of $S^{-1}$
 \eqref{eq:map_s_inv}}
\end{subfigure}
\hfill 
  \begin{subfigure}[b]{0.19\textwidth}
\adjustbox{max width=\textwidth}{\begin{tikzcd}
 \arrow[d, "X" description] \arrow[r, "f" description] &  \arrow[ld,shorten <>=10pt,Rightarrow, "\alpha" above ]\arrow[d, "Y" description] \arrow[r, dotted]              &  \arrow[d, "Y" description] \\
 \arrow[r, "g" description]  \arrow[d, "\cm{g}" description]                  & \arrow[ld,shorten <>=10pt,Rightarrow] \arrow[r, dotted] \arrow[d, dotted] &\arrow[ld,shorten <>=10pt,Rightarrow]  \arrow[d, "\cm{g}" description] \\
 \arrow[r, dotted]                                      & \arrow[r, "g" description]                               & {}          
\end{tikzcd}}
	\caption{Action of $T$ \eqref{eq:map_t}}
  \end{subfigure}
\hfill
  \begin{subfigure}[b]{0.19\textwidth}
\adjustbox{max width=\textwidth}{\begin{tikzcd}
 \arrow[d, "X" description] \arrow[r, dotted]  &  \arrow[d, "X" description] \arrow[r, "f" description]              &\arrow[ld,shorten <>=10pt,Rightarrow,"\alpha" above]  \arrow[d, "Y" description] \\
 \arrow[r, dotted] \arrow[d, "\cn{g}" description]      &  \arrow[ld,shorten <>=10pt,Rightarrow]\arrow[r, "g" description] \arrow[d, dotted]  & \arrow[ld,shorten <>=10pt,Rightarrow] \arrow[d, "\cn{g}" description] \\
 \arrow[r, "g" description]                &  \arrow[r, dotted]                                                  & {}                          
\end{tikzcd}}
	\caption{Action of $T^{-1}$ \eqref{eq:map_t_inv}}
  \end{subfigure}
\caption{The actions of generators and their inverses on a 2-cell filling \cref{eq:shape}}\label{fig:actions_of_generators}
\end{figure}
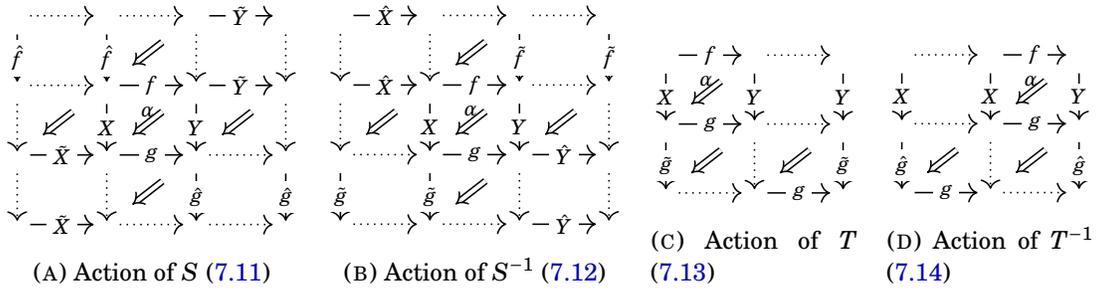
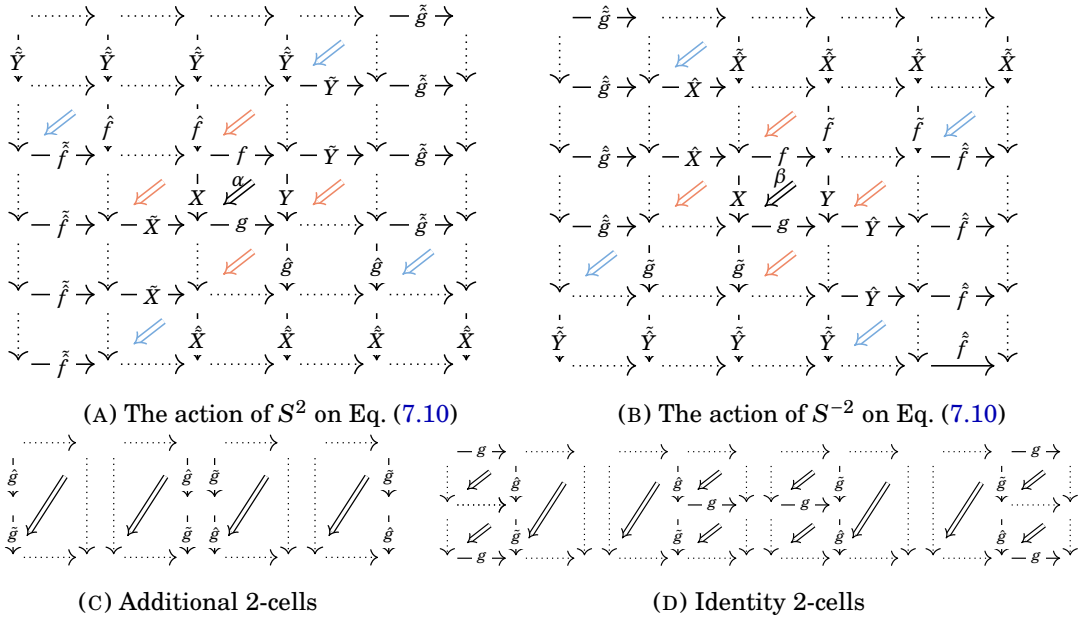
\begin{figure}
  \begin{subfigure}[b]{0.49\textwidth}
\begin{tikzcd}
 \arrow[d, "\cn{\cm{Y}}" description] \arrow[r, dotted]    &  \arrow[d, "\cn{\cm{Y}}" description] \arrow[r, dotted]                  &  \arrow[r, dotted] \arrow[d, "\cn{\cm{Y}}" description]                  &  \arrow[d, "\cn{\cm{Y}}" description] \arrow[r, dotted]   &\arrow[\et, ld,shorten <>=10pt,Rightarrow]\arrow[d, dotted] \arrow[r, "\cm{\cn{g}}" description]                                          &  \arrow[d, dotted]                    \\
 \arrow[d, dotted] \arrow[r, dotted] &  \arrow[\et, ld,shorten <>=10pt,Rightarrow]\arrow[d, "\cn{f}" description] \arrow[r, dotted]                       &  \arrow[d, "\cn{f}" description] \arrow[r, dotted]  & \arrow[\ep, ld,shorten <>=10pt,Rightarrow] \arrow[d, dotted] \arrow[r, "\cm{Y}" description]                                &  \arrow[d, dotted] \arrow[r, "\cm{\cn{g}}" description]                                          &  \arrow[d, dotted]                    \\
 \arrow[r, "\cm{\cn{f}}" description] \arrow[d, dotted]    &  \arrow[r, dotted] \arrow[d, dotted]               & \arrow[\ep, ld,shorten <>=10pt,Rightarrow] \arrow[d, "X" description] \arrow[r, "f" description]                   & \arrow[ld,shorten <>=10pt,Rightarrow,"\alpha" above] \arrow[d, "Y" description] \arrow[r, "\cm{Y}" description] &\arrow[ld,shorten <>=10pt,Rightarrow, \ep]  \arrow[d, dotted] \arrow[r, "\cm{\cn{g}}" description]                                          &  \arrow[d, dotted]                    \\
 \arrow[r, "\cm{\cn{f}}" description] \arrow[d, dotted]    &  \arrow[r, "\cm{X}" description] \arrow[d, dotted]                       &  \arrow[r, "g" description]  \arrow[d, dotted]      & \arrow[ld,shorten <>=10pt,Rightarrow, \ep] \arrow[d, "\cn{g}" description] \arrow[r, dotted]                                &  \arrow[d, "\cn{g}" description] \arrow[r, "\cm{\cn{g}}" description]& \arrow[ld,shorten <>=10pt,Rightarrow, \et] \arrow[d, dotted]                    \\
 \arrow[r, "\cm{\cn{f}}" description] \arrow[d, dotted]    &  \arrow[r, "\cm{X}" description] \arrow[d, dotted] &  \arrow[ld,shorten <>=10pt,Rightarrow, \et ]\arrow[r, dotted] \arrow[d, "\cn{\cm{X}}" description]                  &  \arrow[r, dotted] \arrow[d, "\cn{\cm{X}}" description]                           & {} \arrow[r, dotted] \arrow[d, "\cn{\cm{X}}" description]                                         &  \arrow[d, "\cn{\cm{X}}" description] \\
 \arrow[r, "\cm{\cn{f}}" description]                      &  \arrow[r, dotted]                                                       &  \arrow[r, dotted]                                                       &  \arrow[r, dotted]                                                                &  \arrow[r, dotted]                                                                               & {}                                    
\end{tikzcd}
\caption{The action of $S^2$ on \cref{eq:shape}}\label{fig:big_z4_1}
\end{subfigure}
  \begin{subfigure}[b]{0.49\textwidth}
\begin{tikzcd}
 \arrow[d, dotted] \arrow[r, "\cn{\cm{g}}" description]                       &  \arrow[r, dotted] \arrow[d, dotted]& \arrow[ld,shorten <>=10pt,Rightarrow,\et ] \arrow[d, "\cm{\cn{X}}" description] \arrow[r, dotted]                          &  \arrow[d, "\cm{\cn{X}}" description] \arrow[r, dotted]                       &  \arrow[d, "\cm{\cn{X}}" description] \arrow[r, dotted]                       &  \arrow[d, "\cm{\cn{X}}" description]    \\
 \arrow[d, dotted] \arrow[r, "\cn{\cm{g}}" description]                       &  \arrow[d, dotted] \arrow[r, "\cn{X}" description]                          &  \arrow[d, dotted] \arrow[r, dotted]& \arrow[ld,shorten <>=10pt,Rightarrow, \ep ] \arrow[d, "\cm{f}" description] \arrow[r, dotted]                       &  \arrow[d, "\cm{f}" description] \arrow[r, dotted]  & \arrow[ld,shorten <>=10pt,Rightarrow,\et] \arrow[d, dotted] \\
 \arrow[d, dotted] \arrow[r, "\cn{\cm{g}}" description]                       &  \arrow[r, "\cn{X}" description] \arrow[d, dotted]    & \arrow[ld,shorten <>=10pt,Rightarrow, \ep] \arrow[d, "X" description] \arrow[r, "f" description]                             & \arrow[ld,shorten <>=10pt,Rightarrow,"\beta" above] \arrow[d, "Y" description] \arrow[r, dotted] & \arrow[ld,shorten <>=10pt,Rightarrow, \ep] \arrow[d, dotted] \arrow[r, "\cn{\cm{f}}" description]                       &  \arrow[d, dotted] \\
 \arrow[r, "\cn{\cm{g}}" description] \arrow[d, dotted] & \arrow[ld,shorten <>=10pt,Rightarrow, \et] \arrow[r, dotted] \arrow[d, "\cm{g}" description]                          &  \arrow[r, "g" description] \arrow[d, "\cm{g}" description]   &\arrow[ld,shorten <>=10pt,Rightarrow, \ep]  \arrow[r, "\cn{Y}" description] \arrow[d, dotted]                       &  \arrow[d, dotted] \arrow[r, "\cn{\cm{f}}" description]                       &  \arrow[d, dotted] \\
 \arrow[r, dotted] \arrow[d, "\cm{\cn{Y}}" description]                       &  \arrow[r, dotted] \arrow[d, "\cm{\cn{Y}}" description]                          &  \arrow[r, dotted] \arrow[d, "\cm{\cn{Y}}" description]                          &  \arrow[r, "\cn{Y}" description] \arrow[d, "\cm{\cn{Y}}" description]    & \arrow[ld,shorten <>=10pt,Rightarrow, \et] \arrow[d, dotted] \arrow[r, "\cn{\cm{f}}" description]                       &  \arrow[d, dotted] \\
 \arrow[r, dotted]                                      &  \arrow[r, dotted]                                         &  \arrow[r, dotted]                                         &  \arrow[r, dotted]                                      &  \arrow[r, "\cn{\cm{f}}"]                                         & {}       
\end{tikzcd}\caption{The action of $S^{-2}$ on \cref{eq:shape}}\label{fig:big_z4_2}
\end{subfigure}
  \begin{subfigure}[b]{0.37\textwidth}
\adjustbox{max width=\textwidth}{
\begin{tikzcd}
 \arrow[d, "\cn{g}" description] \arrow[r, dotted] &  \arrow[ldd,shorten <>=10pt,Rightarrow]\arrow[dd, dotted] \\
 \arrow[d, "\cm{g}" description]                               &                     \\
 \arrow[r, dotted]                                             & {}            
\end{tikzcd}
\begin{tikzcd}
 \arrow[r, dotted] \arrow[dd, dotted] &  \arrow[ldd,shorten <>=10pt,Rightarrow]\arrow[d, "\cn{g}" description] \\
                                                             &  \arrow[d, "\cm{g}" description] \\
 \arrow[r, dotted]                                           & {}                          
\end{tikzcd}
\begin{tikzcd}
 \arrow[d, "\cm{g}" description] \arrow[r, dotted]&  \arrow[ldd,shorten <>=10pt,Rightarrow]\arrow[dd, dotted] \\
 \arrow[d, "\cn{g}" description]                                          &                     \\
\arrow[r, dotted]                                                       & {}              
\end{tikzcd}
\begin{tikzcd}
 \arrow[r, dotted] \arrow[dd, dotted] & \arrow[ldd,shorten <>=10pt,Rightarrow] \arrow[d, "\cm{g}" description] \\
                                                             &  \arrow[d, "\cn{g}" description] \\
 \arrow[r, dotted]                                           & {}                          
\end{tikzcd}}
\caption{Additional 2-cells}\label{fig:inverse_for_t}
\end{subfigure}
\hfill 
  \begin{subfigure}[b]{0.6\textwidth}
\adjustbox{max width=\textwidth}{
\begin{tikzcd}
 \arrow[r, "g" description] \arrow[d, dotted]      &\arrow[ld,shorten <>=10pt,Rightarrow]\arrow[d, "\cn{g}" description] \arrow[r, dotted]  & \arrow[ldd,shorten <>=10pt,Rightarrow] \arrow[dd, dotted] \\
 \arrow[r, dotted] \arrow[d, dotted] \arrow[r, dotted]  &  \arrow[ld,shorten <>=10pt,Rightarrow]\arrow[d, "\cm{g}" description]                               &                     \\
 \arrow[r, "g" description]                                                  &  \arrow[r, dotted]                                             & {}                 
\end{tikzcd}
\begin{tikzcd}
 \arrow[r, dotted] \arrow[dd, dotted] & \arrow[ldd,shorten <>=10pt,Rightarrow] \arrow[r, dotted] \arrow[d, "\cn{g}" description]      & \arrow[ld,shorten <>=10pt,Rightarrow] \arrow[d, dotted] \\
                                                             &  \arrow[r, "g" description] \arrow[d, "\cm{g}" description]  &  \arrow[ld,shorten <>=10pt,Rightarrow]\arrow[d, dotted] \\
 \arrow[r, dotted]                                           &  \arrow[r, dotted]                                                                & {}          
\end{tikzcd}
\begin{tikzcd}
 \arrow[d, dotted] \arrow[r, dotted]    &  \arrow[ld,shorten <>=10pt,Rightarrow]\arrow[d, "\cm{g}" description] \arrow[r, dotted]  & \arrow[ldd,shorten <>=10pt,Rightarrow] \arrow[dd, dotted] \\
 \arrow[r, "g" description] \arrow[d, dotted]& \arrow[ld,shorten <>=10pt,Rightarrow] \arrow[d, "\cn{g}" description]                                          &                     \\
 \arrow[r, dotted]                                                  & \arrow[r, dotted]                                                       & {}           
\end{tikzcd}
\begin{tikzcd}
 \arrow[r, dotted] \arrow[dd, dotted] & \arrow[ldd,shorten <>=10pt,Rightarrow] \arrow[r, "g" description] \arrow[d, "\cm{g}" description] & \arrow[ld,shorten <>=10pt,Rightarrow] \arrow[d, dotted] \\
                                                             &  \arrow[r, dotted] \arrow[d, "\cn{g}" description]       &  \arrow[ld,shorten <>=10pt,Rightarrow]\arrow[d, dotted] \\
 \arrow[r, dotted]                                           &  \arrow[r, "g" description]                                                       & {}          
\end{tikzcd}
}
\caption{Identity 2-cells}\label{fig:inverse_for_t_2}
\end{subfigure}
\caption{Conditions on companions and conjoints}
\end{figure}
\begin{figure}

  \begin{subfigure}[b]{0.9\textwidth}
\begin{tikzcd}
 \arrow[d, "X"  description] \arrow[r, dotted]                               &  \arrow[d, "X" description] \arrow[r, dotted]     & \arrow[ld,shorten <>=10pt,Rightarrow,\et] \arrow[d, dotted] \arrow[r, dotted]                              &  \arrow[d, dotted] \arrow[r, dotted]                                &  \arrow[d, dotted] \arrow[r, "f" description]                                     &  \arrow[d, dotted] 
\\
 \arrow[r, dotted] \arrow[d, dotted]                            &  \arrow[d, dotted] \arrow[r, "\cn{X}" description]                       &  \arrow[d, dotted] \arrow[r, dotted]  &\arrow[ld,shorten <>=10pt,Rightarrow, \ep]  \arrow[d, "\cm{f}" description] \arrow[r, dotted]                       &  \arrow[d, "\cm{f}" description] \arrow[r, "f" description] & \arrow[ld,shorten <>=10pt,Rightarrow,\et] \arrow[d, dotted] 
\\
 \arrow[r, dotted] \arrow[d, dotted]                            &  \arrow[r, "\cn{X}" description] \arrow[d, dotted]  & \arrow[ld,shorten <>=10pt,Rightarrow, \ep] \arrow[d, "X" description] \arrow[r, "f" description]            & \arrow[ld,shorten <>=10pt,Rightarrow, "\alpha" above ] \arrow[d, "Y" description] \arrow[r, dotted] & \arrow[ld,shorten <>=10pt,Rightarrow, \ep] \arrow[d, dotted] \arrow[r, dotted]                                  &  \arrow[d, dotted] 
\\
 \arrow[r, dotted] \arrow[d, dotted]  &  \arrow[ld,shorten <>=10pt,Rightarrow, \et]\arrow[r, dotted] \arrow[d, "\cm{g}" description]                       &  \arrow[r, "g" description] \arrow[d, "\cm{g}" description] & \arrow[ld,shorten <>=10pt,Rightarrow,\ep] \arrow[r, "\cn{Y}" description] \arrow[d, dotted]                       &  \arrow[d, dotted] \arrow[r, dotted]                                  &  \arrow[d, dotted] \\
 \arrow[r, "g" description] \arrow[d, dotted]                               &  \arrow[r, dotted] \arrow[d, dotted]                                &  \arrow[r, dotted] \arrow[d, dotted]                              &  \arrow[r, "\cn{Y}" description] \arrow[d, dotted]  & \arrow[ld,shorten <>=10pt,Rightarrow, \et] \arrow[d, "Y" description] \arrow[r, dotted]                                     &  \arrow[d, "Y" description]    
\\
 \arrow[r, "g" description]                                                 &  \arrow[r, dotted]                                                  &  \arrow[r, dotted]                                                &  \arrow[r, dotted]                                                  &  \arrow[r, dotted]                                                    & {}               
\end{tikzcd}
=
\begin{tikzcd}
 \arrow[d, "X" description, dotted] \arrow[r, dotted]   &  \arrow[ld,shorten <>=10pt,Rightarrow,\et ]\arrow[d, dotted] \arrow[r, dotted]             &  \arrow[ld,shorten <>=10pt,Rightarrow, \ep]\arrow[d, "\cm{f}" description] \arrow[r, "f" description] & \arrow[ld,shorten <>=10pt,Rightarrow,\et] \arrow[d, dotted]          \\
 \arrow[r, "\cn{X}" description] \arrow[d, dotted] &  \arrow[ld,shorten <>=10pt,Rightarrow, \ep]\arrow[d, "X" description] \arrow[r, "f" description]                            &\arrow[ld,shorten <>=10pt,Rightarrow, "\alpha" above ]  \arrow[d, "Y" description]  \arrow[r, dotted]          & \arrow[ld,shorten <>=10pt,Rightarrow, \ep] \arrow[d, dotted]          \\
 \arrow[r, dotted] \arrow[d, dotted]          &\arrow[ld,shorten <>=10pt,Rightarrow,\et]  \arrow[r, "g" description] \arrow[d, "\cm{g}" description] &  \arrow[ld,shorten <>=10pt,Rightarrow, \ep]\arrow[r, "\cn{Y}" description] \arrow[d, dotted]      & \arrow[ld,shorten <>=10pt,Rightarrow,\et] \arrow[d, "Y" description] \\
 \arrow[r, "g" description]                                              &  \arrow[r, dotted]                                                                &  \arrow[r, dotted]                                                                & {}                          
\end{tikzcd}
\caption{Verifying $SS^{-1}$ is the identity on \cref{eq:shape}}
\end{subfigure}
  \begin{subfigure}[b]{0.48\textwidth}
\adjustbox{max width=\textwidth}{
\begin{tikzcd}
 \arrow[d, "X" description] \arrow[r, dotted]                                &  \arrow[d, "X" description] \arrow[r, "f" description]              &  \arrow[ld,shorten <>=10pt,Rightarrow, "\alpha" above ]\arrow[d, "Y" description] \arrow[r, dotted]               &  \arrow[d, "Y" description] \\
 \arrow[r, dotted] \arrow[d, "\cn{g}" description]       & \arrow[ld,shorten <>=10pt,Rightarrow, \ep] \arrow[r, "g" description] \arrow[d, dotted]& \arrow[ld,shorten <>=10pt,Rightarrow,\ep] \arrow[d, "\cn{g}" description] \arrow[r, dotted]               &  \arrow[d, "\cn{g}" description] \\
 \arrow[r, "g" description] \arrow[d, "\cm{g}" description] &  \arrow[ld,shorten <>=10pt,Rightarrow,\et]\arrow[r, dotted] \arrow[d, dotted]                                &  \arrow[d, dotted] \arrow[r, dotted] &  \arrow[ld,shorten <>=10pt,Rightarrow,\et ]\arrow[d, "\cm{g}" description] \\
 \arrow[r, dotted]                                                           &  \arrow[r, dotted]                                                  &  \arrow[r, "g" description]                                 & {}                         
\end{tikzcd}=
\begin{tikzcd}
 \arrow[d, "X" description] \arrow[r, dotted]                                &  \arrow[d, "X" description] \arrow[r, "f" description]                       &  \arrow[ld,shorten <>=10pt,Rightarrow, "\alpha" above ]\arrow[d, "Y" description]  \\
 \arrow[r, dotted] \arrow[d, "\cn{g}" description]       & \arrow[ld,shorten <>=10pt,Rightarrow,\ep] \arrow[r, "g" description] \arrow[d, dotted]      &  \arrow[ld,shorten <>=10pt,Rightarrow, \ep]\arrow[d, "\cn{g}" description]  \\
 \arrow[r, "g" description] \arrow[d, "\cm{g}" description] &  \arrow[ld,shorten <>=10pt,Rightarrow,\et ]\arrow[r, dotted] \arrow[d, dotted] \arrow[r, dotted]  & \arrow[ld,shorten <>=10pt,Rightarrow,\et ]\arrow[d, "\cm{g}" description] \\
 \arrow[r, dotted]                                                           &  \arrow[r, "g" description]                                                  & {}                           
\end{tikzcd}}
	\caption{Verifying $TT^{-1}$ is the identity on \cref{eq:shape}}
  \end{subfigure}
  \begin{subfigure}[b]{0.48\textwidth}
\adjustbox{max width=\textwidth}{
\begin{tikzcd}
 \arrow[d, "X" description] \arrow[r, dotted]                       &  \arrow[d, "X" description] \arrow[r, "f" description]                       &  \arrow[ld,shorten <>=10pt,Rightarrow, "\alpha" above ]\arrow[d, "Y" description] \arrow[r, dotted]                       &  \arrow[d, "Y" description] \\
 \arrow[d, "\cm{g}" description] \arrow[r, dotted]                       &  \arrow[r, "g" description] \arrow[d, "\cm{g}" description] &  \arrow[ld,shorten <>=10pt,Rightarrow, \ep ]\arrow[d, dotted] \arrow[r, dotted]        & \arrow[ld,shorten <>=10pt,Rightarrow,\ep] \arrow[d, "\cm{g}" description] \\
\arrow[d, "\cn{g}" description] \arrow[r, dotted] &  \arrow[ld,shorten <>=10pt,Rightarrow,\et ]\arrow[r, dotted] \arrow[d, dotted]                                         &  \arrow[r, "g" description] \arrow[d, dotted]  & \arrow[ld,shorten <>=10pt,Rightarrow,\et ] \arrow[d, "\cn{g}" description] \\
 \arrow[r, "g" description]                                         &  \arrow[r, dotted]                                                           &  \arrow[r, dotted]                                                  & {}                          
\end{tikzcd}=\begin{tikzcd}
 \arrow[d, "X" description] \arrow[r, "f" description]                       &\arrow[ld,shorten <>=10pt,Rightarrow, "\alpha" above ]   \arrow[d, "Y" description] \arrow[r, dotted]                       & \arrow[d, "Y" description] \\
 \arrow[r, "g" description] \arrow[d, "\cm{g}" description] & \arrow[ld,shorten <>=10pt,Rightarrow,\ep] \arrow[d, dotted] \arrow[r, dotted]     &  \arrow[ld,shorten <>=10pt,Rightarrow, \ep]\arrow[d, "\cm{g}" description] \\
 \arrow[r, dotted] \arrow[d, "\cn{g}" description]       & \arrow[ld,shorten <>=10pt,Rightarrow,\et ] \arrow[r, "g" description] \arrow[d, dotted]  &  \arrow[ld,shorten <>=10pt,Rightarrow,\et ]\arrow[d, "\cn{g}" description] \\
 \arrow[r, "g" description]                                                  &  \arrow[r, dotted]                                                  & {}                          
\end{tikzcd}}
	\caption{Verifying $T^{-1}T$ is the identity on \cref{eq:shape}}
  \end{subfigure}
\caption{Verifying that $S$ and $T$ have inverses}\label{fig:s_t_inverses}
\end{figure}
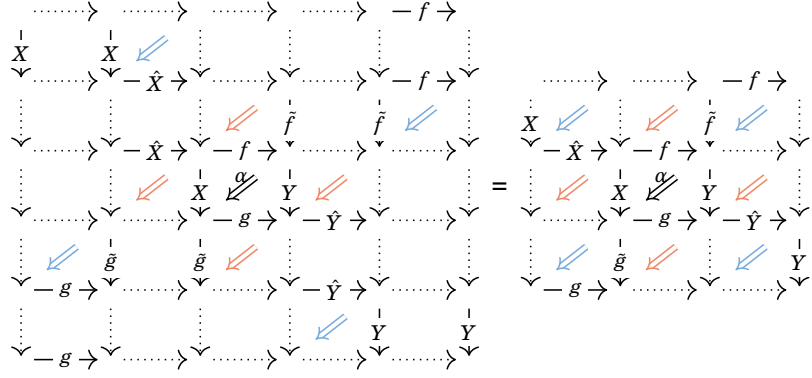
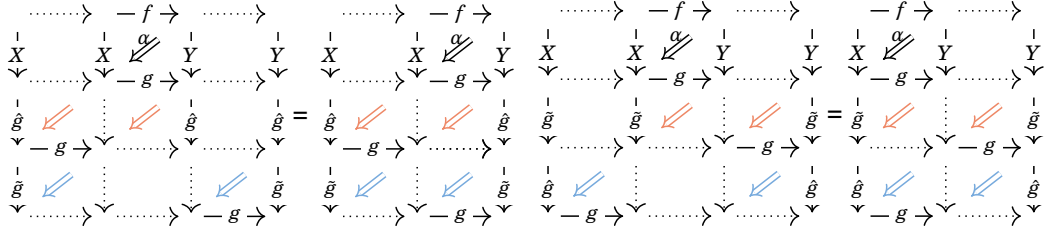  

Together these define an action of $\langle S,T\rangle$ on $2_\mathbb{D}$.  If the diagrams in \cref{fig:big_z4_1,fig:big_z4_2} define the same 2-cell then $S^4$ acts as the identity.
\end{proof}

\begin{prop}\label{prop:st_inv} If $\alpha$ is an endomorphism 2-cell in $\mathbb{D}$
 and  $\dosh{-}$ is a double shadow then \[\dosh{S\alpha }=\dosh{\alpha}=\dosh{T\alpha}.\]
\end{prop}

In particular, the action of $\langle S,T\rangle$ is trivial under the image of a double shadow.  The proof is a series of pasting arguments.  See \cref{fig:triv_s_sh}.

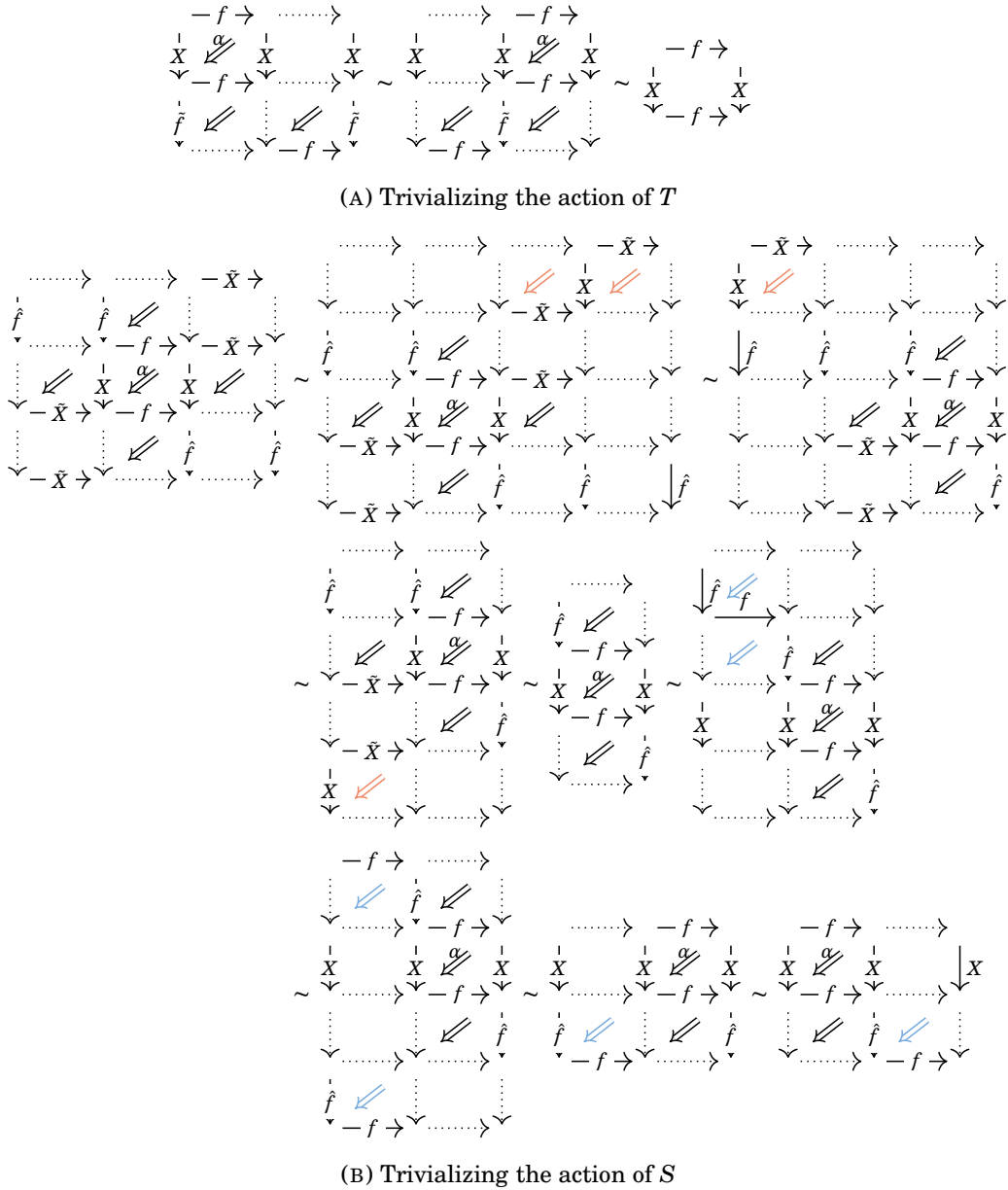
\begin{figure}  
\begin{subfigure}[b]{.65\textwidth}
\begin{tikzcd}
 \arrow[d, "X" description] \arrow[r, "f" description] &\arrow[ld,shorten <>=10pt,Rightarrow, "\alpha" above ]   \arrow[d, "X" description] \arrow[r, dotted]              &  \arrow[d, "X" description] \\
 \arrow[r, "f" description] \arrow[d, "\cm{f}" description]                  & \arrow[ld,shorten <>=10pt,Rightarrow] \arrow[r, dotted] \arrow[d, dotted]  &  \arrow[ld,shorten <>=10pt,Rightarrow]\arrow[d, "\cm{f}" description] \\
 \arrow[r, dotted]                                      &  \arrow[r, "f" description]                               & {}                       
\end{tikzcd}
$\sim$ 
\begin{tikzcd}
 \arrow[r, dotted] \arrow[d, "X" description]                          &  \arrow[d, "X" description] \arrow[r, "f" description]                            & \arrow[ld,shorten <>=10pt,Rightarrow, "\alpha" above ]  \arrow[d, "X" description] \\
 \arrow[r, dotted] \arrow[d, dotted]  & \arrow[ld,shorten <>=10pt,Rightarrow] \arrow[r, "f" description] \arrow[d, "\cm{f}" description] & \arrow[ld,shorten <>=10pt,Rightarrow] \arrow[d, dotted]          \\
 \arrow[r, "f" description]                                            &  \arrow[r, dotted]                                                                & {}                         
\end{tikzcd}
$\sim $
\begin{tikzcd}
 \arrow[d, "X" description] \arrow[r, "f" description] &  \arrow[d, "X" description] \\
 \arrow[r, "f" description]                            & {}                        
\end{tikzcd}
\caption{Trivializing the action of $T$}
\end{subfigure}

  \begin{subfigure}[b]{\textwidth}
\resizebox{\textwidth}{!}
{  \begin{minipage}{\textwidth}
\begin{align*}
	\begin{tikzcd}[ampersand replacement=\&]
 \arrow[d, "\cn{f}" description] \arrow[r, dotted]              \&  \arrow[d, "\cn{f}" description] \arrow[r, dotted]  \&\arrow[ld,shorten <>=10pt,Rightarrow]  \arrow[d, dotted] \arrow[r, "\cm{X}" description]                                \&  \arrow[d, dotted]          \\
\arrow[r, dotted] \arrow[d, dotted] \& \arrow[ld,shorten <>=10pt,Rightarrow] \arrow[d, "X" description] \arrow[r, "f" description]              \& \arrow[ld,shorten <>=10pt,Rightarrow, "\alpha" above ]  \arrow[d, "X" description] \arrow[r, "\cm{X}" description]  \&\arrow[ld,shorten <>=10pt,Rightarrow]  \arrow[d, dotted]          \\
 \arrow[r, "\cm{X}" description] \arrow[d, dotted]              \&  \arrow[r, "f" description]  \arrow[d, dotted] \& \arrow[ld,shorten <>=10pt,Rightarrow]\arrow[d, "\cn{f}" description] \arrow[r, dotted]                               \&  \arrow[d, "\cn{f}" description] \\
 \arrow[r, "\cm{X}" description]                                \&  \arrow[r, dotted]                                                  \&  \arrow[r, dotted]                                                           \& {}                   
\end{tikzcd}
&\sim
\begin{tikzcd}[ampersand replacement=\&]
 \arrow[r, dotted] \arrow[d, dotted]                       \&  \arrow[d, dotted] \arrow[r, dotted]                                     \&  \arrow[d, dotted] \arrow[r, dotted]                   \& \arrow[ld,shorten <>=10pt,Rightarrow,\ep] \arrow[d, "X" description] \arrow[r, "\cm{X}" description]  \& \arrow[ld,shorten <>=10pt,Rightarrow,\ep] \arrow[d, dotted] \\
 \arrow[d, "\cn{f}" description] \arrow[r, dotted]         \&  \arrow[d, "\cn{f}" description] \arrow[r, dotted] \& \arrow[ld,shorten <>=10pt,Rightarrow] \arrow[d, dotted] \arrow[r, "\cm{X}" description]                                \&  \arrow[d, dotted] \arrow[r, dotted]                 \&  \arrow[d, dotted] \\
\arrow[r, dotted] \arrow[d, dotted] \&  \arrow[ld,shorten <>=10pt,Rightarrow]\arrow[d, "X" description] \arrow[r, "f" description]                   \&  \arrow[ld,shorten <>=10pt,Rightarrow, "\alpha" above ] \arrow[d, "X" description] \arrow[r, "\cm{X}" description] \& \arrow[ld,shorten <>=10pt,Rightarrow] \arrow[d, dotted] \arrow[r, dotted]                 \&  \arrow[d, dotted] 
\\
 \arrow[r, "\cm{X}" description] \arrow[d, dotted]         \&  \arrow[r, "f" description]  \arrow[d, dotted]      \& \arrow[ld,shorten <>=10pt,Rightarrow]\arrow[d, "\cn{f}" description] \arrow[r, dotted]                               \&  \arrow[d, "\cn{f}" description] \arrow[r, dotted]   \&  \arrow[d, "\cn{f}"]    
\\
 \arrow[r, "\cm{X}" description]                           \&  \arrow[r, dotted]                                                       \&  \arrow[r, dotted]                                                                \&  \arrow[r, dotted]                                  \&                   {}
\end{tikzcd}
\sim
\begin{tikzcd}[ampersand replacement=\&]
 \arrow[d, "X" description] \arrow[r, "\cm{X}" description] \&  \arrow[ld,shorten <>=10pt,Rightarrow,\ep]\arrow[r, dotted] \arrow[d, dotted]                       \&  \arrow[d, dotted] \arrow[r, dotted]                                     \&  \arrow[d, dotted]                \\
 \arrow[r, dotted] \arrow[d, "\cn{f}"]                    \&  \arrow[d, "\cn{f}" description] \arrow[r, dotted]         \&  \arrow[d, "\cn{f}" description] \arrow[r, dotted] \&\arrow[ld,shorten <>=10pt,Rightarrow]  \arrow[d, dotted]                \\
 \arrow[d, dotted] \arrow[r, dotted]                 \&  \arrow[r, dotted] \arrow[d, dotted] \& \arrow[ld,shorten <>=10pt,Rightarrow] \arrow[d, "X" description] \arrow[r, "f" description]                   \& \arrow[ld,shorten <>=10pt,Rightarrow, "\alpha" above ]  \arrow[d, "X" description]       \\
 \arrow[r, dotted] \arrow[d, dotted]                 \&  \arrow[r, "\cm{X}" description] \arrow[d, dotted]         \&  \arrow[r, "f" description]  \arrow[d, dotted]      \& \arrow[ld,shorten <>=10pt,Rightarrow]\arrow[d, "\cn{f}" description] \\
 \arrow[r, dotted]                                   \&  \arrow[r, "\cm{X}" description]                           \&  \arrow[r, dotted]                                                       \& {}                                
\end{tikzcd}
\\
&\sim
\begin{tikzcd}[ampersand replacement=\&]
 \arrow[d, "\cn{f}" description] \arrow[r, dotted]                    \&  \arrow[d, "\cn{f}" description] \arrow[r, dotted]  \& \arrow[ld,shorten <>=10pt,Rightarrow] \arrow[d, dotted]                \\
\arrow[r, dotted] \arrow[d, dotted]            \&\arrow[ld,shorten <>=10pt,Rightarrow]  \arrow[d, "X" description] \arrow[r, "f" description]                   \& \arrow[ld,shorten <>=10pt,Rightarrow, "\alpha" above ]  \arrow[d, "X" description]       \\
 \arrow[r, "\cm{X}" description] \arrow[d, dotted]                    \&  \arrow[r, "f" description] \arrow[d, dotted]      \& \arrow[ld,shorten <>=10pt,Rightarrow] \arrow[d, "\cn{f}" description] \\
 \arrow[r, "\cm{X}" description] \arrow[d, "X" description]  \& \arrow[ld,shorten <>=10pt,Rightarrow,\ep,] \arrow[r, dotted] \arrow[d, dotted]                                     \& {} \arrow[d, dotted]               \\
 \arrow[r, dotted]                                                    \&  \arrow[r, dotted]                                                       \& {}                                
\end{tikzcd}
\sim
\begin{tikzcd}[ampersand replacement=\&]
 \arrow[d, "\cn{f}" description] \arrow[r, dotted] \&\arrow[ld,shorten <>=10pt,Rightarrow]  \arrow[d, dotted]                \\
 \arrow[d, "X" description] \arrow[r, "f" description]                   \& \arrow[ld,shorten <>=10pt,Rightarrow, "\alpha" above ]  \arrow[d, "X" description]       \\
 \arrow[r, "f" description]  \arrow[d, dotted]      \& \arrow[ld,shorten <>=10pt,Rightarrow] \arrow[d, "\cn{f}" description] \\
 \arrow[r, dotted]                                                       \& {}                              
\end{tikzcd}
\sim
\begin{tikzcd}[ampersand replacement=\&]
 \arrow[r, dotted] \arrow[d, "\cn{f}"]  \&  \arrow[ld,shorten <>=10pt,Rightarrow,\et]\arrow[d, dotted] \arrow[r, dotted]                                     \&  \arrow[d, dotted]                \\
 \arrow[r, "f"] \arrow[d, dotted]  \& \arrow[ld,shorten <>=10pt,Rightarrow,\et] \arrow[d, "\cn{f}" description] \arrow[r, dotted] \& \arrow[ld,shorten <>=10pt,Rightarrow] \arrow[d, dotted]                \\
 \arrow[r, dotted] \arrow[d, "X" description]                       \&  \arrow[d, "X" description] \arrow[r, "f" description]                   \& \arrow[ld,shorten <>=10pt,Rightarrow, "\alpha" above ]  \arrow[d, "X" description]       \\
 \arrow[r, dotted] \arrow[d, dotted]                    \&  \arrow[r, "f" description]  \arrow[d, dotted]      \& \arrow[ld,shorten <>=10pt,Rightarrow] \arrow[d, "\cn{f}" description] \\
 \arrow[r, dotted]                                      \&  \arrow[r, dotted]                                                       \& {}                               
\end{tikzcd}
\\
&\sim
\begin{tikzcd}[ampersand replacement=\&]
 \arrow[r, "f" description] \arrow[d, dotted] \&  \arrow[ld,shorten <>=10pt,Rightarrow,\et ]\arrow[d, "\cn{f}" description] \arrow[r, dotted]\&  \arrow[ld,shorten <>=10pt,Rightarrow]\arrow[d, dotted]                \\
 \arrow[r, dotted] \arrow[d, "X" description]                       \&  \arrow[d, "X" description] \arrow[r, "f" description]                   \& \arrow[ld,shorten <>=10pt,Rightarrow, "\alpha" above ]  \arrow[d, "X" description]       \\
 \arrow[r, dotted] \arrow[d, dotted]                    \&  \arrow[r, "f" description]  \arrow[d, dotted]      \& \arrow[ld,shorten <>=10pt,Rightarrow] \arrow[d, "\cn{f}" description] \\
 \arrow[r, dotted] \arrow[d, "\cn{f}" description] \&  \arrow[ld,shorten <>=10pt,Rightarrow,\et ]\arrow[r, dotted] \arrow[d, dotted]                                     \&  \arrow[d, dotted]               \\
 \arrow[r, "f" description]                                         \&  \arrow[r, dotted]                                                       \&                                  {}
\end{tikzcd}
\sim
\begin{tikzcd}[ampersand replacement=\&]
 \arrow[r, dotted] \arrow[d, "X" description]                       \&  \arrow[d, "X" description] \arrow[r, "f" description]              \& \arrow[ld,shorten <>=10pt,Rightarrow, "\alpha" above ]  \arrow[d, "X" description]       \\
 \arrow[r, dotted] \arrow[d, "\cn{f}" description]  \&\arrow[ld,shorten <>=10pt,Rightarrow,\et ]  \arrow[r, "f" description] \arrow[d, dotted] \& \arrow[ld,shorten <>=10pt,Rightarrow]\arrow[d, "\cn{f}" description] \\
 \arrow[r, "f"  description]                                 \&  \arrow[r, dotted]                                                  \& {}                                
\end{tikzcd}
\sim
\begin{tikzcd}[ampersand replacement=\&]
 \arrow[d, "X" description] \arrow[r, "f" description]              \&  \arrow[ld,shorten <>=10pt,Rightarrow, "\alpha" above ] \arrow[d, "X" description] \arrow[r, dotted]                             \&  \arrow[d, "X"]    \\
 \arrow[r, "f" description] \arrow[d, dotted] \& \arrow[ld,shorten <>=10pt,Rightarrow] \arrow[d, "\cn{f}" description] \arrow[r, dotted] \&  \arrow[ld,shorten <>=10pt,Rightarrow,\et ]\arrow[d, dotted] \\
 \arrow[r, dotted]                                                  \&  \arrow[r, "f"  description]                                                          \& {}     
\end{tikzcd}
\end{align*}
\end{minipage}
}
\caption{Trivializing the action of $S$}
\end{subfigure}

\caption{Trivializing the actions of $T$ and $S$ under the action of a double shadow}\label{fig:triv_s_sh}
\end{figure}

Let $I_\mathbb{D}$ be the set of endomorphism 2-cells
so that 
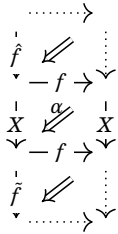
\[\begin{tikzcd}
\arrow[d, "\cn{f}" description ] \arrow[r, dotted]      & \arrow[ld,shorten <>=10pt,Rightarrow] \arrow[d, dotted] 
\\
 \arrow[d, "X" description] \arrow[r, "f" description]            & \arrow[ld,shorten <>=10pt,Rightarrow,"\alpha" above] \arrow[d, "X" description]    
\\
\arrow[d, "\cm{f}" description] \arrow[r, "f" description]  & \arrow[ld,shorten <>=10pt,Rightarrow] \arrow[d, dotted] 
\\
\arrow[r, dotted]                                                    &                 {}
\end{tikzcd}\] is an isomorphism.

\begin{lem}\label{lem:compareststst}
The $\langle S,T|S^4=1\rangle$ action above on endomorphism 2-cells induces an action of $SL_2(\mathbb{Z})$ on $I_\mathbb{D}$.
\end{lem}

\begin{proof}
See \cref{fig:s1t1s,fig:tst} for the actions of  $S^{-1}T^{-1}S$ and $TST$ on a 2-cell of the form 
\cref{eq:shape}.  
Then restrict to endomorphism 2-cells
and expand as in \cref{fig:expanded}.
\end{proof}

\begin{defn}
  If we regard a discrete group $G$ as a 2-category with a single 0-cell and only identity 2-cells,  a \textbf{2-representation} of $G$ on $\sB$ is  a (strong) 2-functor 
	\[\rho\colon G \Longrightarrow \sB.\] 
\end{defn}
This data is spelled out fully in \cite{ganter_kapranov}.

Continuing to think of a discrete group $G$ as  a 2-category with a single 0-cell and only identity 2-cells, $I_{\mathbb{D}}$.   Under a 2-representation $\rho$ a pair of commuting elements $(g, h)$ pick out 
\begin{itemize}
\item  invertible 1-cells $X, Y$ and 
\item 
an invertible 2-morphism
  \[X \odot Y \to Y \odot X.\]
\end{itemize}

That is, a 2-representation $G \to \sB$ picks out a subset of $D(\sB)$, indeed, a subset of $I_{D_{\sB}}$. It is easy to see that this subset admits an $SL_2 (\mathbb{Z})$-action. The categorical 2-character of Ganter-Kapranov is by definition the iterated trace applied to this subset. 

Then  the main theorem of \cite{ben_zvi_nadler_1} follows from \cref{prop:st_inv}.
\begin{thm}\label{thm:bzn_invariance}
  Categorical 2-characters are invariant under the action of $SL_2 (\Z)$ in \cref{lem:compareststst}. 
\end{thm}

\begin{rmk}
  A curious feature of this theorem is that 2-characters, defined invariants of 2-cells in double categories, are in fact invariant under the action of the free group $\Z \ast \Z = \langle S, T \rangle$. However, this action only descends to an $\operatorname{SL}_2 (\Z)$ action under very stringent conditions on the double category. 
\end{rmk}

\begin{figure}
 \begin{subfigure}[b]{\textwidth}
\adjustbox{max width=\textwidth}{
\begin{tikzcd}
 \arrow[r, "\cm{X}" description] \arrow[d, dotted]                       &  \arrow[r, "f" description] \arrow[d, dotted]                       &  \arrow[r, dotted] \arrow[d, dotted]                                         &  \arrow[r, dotted] \arrow[d, dotted]                                &  \arrow[r, dotted] \arrow[d, dotted]                                &  \arrow[d, dotted] \arrow[r, dotted]        &  \arrow[ld,shorten <>=10pt,Rightarrow,\et ]\arrow[d, "Y" description] \arrow[r, dotted]                        &  \arrow[r, dotted] \arrow[d, "Y" description]                       &  \arrow[d, "Y" description] 
\\
 \arrow[r, "\cm{X}" description] \arrow[d, dotted]                       &  \arrow[r, "f" description] \arrow[d, dotted] & \arrow[ld,shorten <>=10pt,Rightarrow,\et ] \arrow[d, "\cn{f}" description] \arrow[r, dotted]                                &  \arrow[r,  dotted] \arrow[d,  "\cn{f}" description]                       &  \arrow[d ,"\cn{f}" description] \arrow[r,  dotted] &  \arrow[ld,shorten <>=10pt,Rightarrow, \ep ]\arrow[d, dotted] \arrow[r, "\cm{Y}" description, ]                                &  \arrow[d, dotted, ] \arrow[r, dotted]                                 &  \arrow[r, dotted] \arrow[d, dotted]                                &  \arrow[d, dotted]          
\\
 \arrow[r, "\cm{X}" description] \arrow[d, dotted]                       &  \arrow[r, dotted] \arrow[d, dotted]                                &  \arrow[d, dotted] \arrow[r, dotted]                                         &  \arrow[r, dotted] \arrow[d, dotted]        &\arrow[ld,shorten <>=10pt,Rightarrow, \ep]  \arrow[d, "X" description] \arrow[r, "f" description]              &  \arrow[ld,shorten <>=10pt,Rightarrow, "\alpha" above ] \arrow[d, "Y" description] \arrow[r, "\cm{Y}" description] &  \arrow[ld,shorten <>=10pt,Rightarrow, \ep]\arrow[d, dotted] \arrow[r, dotted]                                 &  \arrow[r, dotted] \arrow[d, dotted]                                &  \arrow[d, dotted]         
 \\
 \arrow[r, "\cm{X}" description] \arrow[d, dotted]                       &  \arrow[r, dotted] \arrow[d, dotted]                                &  \arrow[d, dotted] \arrow[r, dotted]                                         &  \arrow[r, "\cm{X}" description] \arrow[d, dotted]                       &  \arrow[r, "g" description] \arrow[d, dotted] & \arrow[ld,shorten <>=10pt,Rightarrow, \ep] \arrow[d, "\cn{g}" description] \arrow[r, dotted]                                &  \arrow[d, "\cn{g}" description] \arrow[r, dotted]                        &  \arrow[d, "\cn{g}" description] \arrow[r, dotted] & \arrow[ld,shorten <>=10pt,Rightarrow,\et ] \arrow[d, dotted]          \\
 \arrow[r, "\cm{X}" description] \arrow[d, dotted] & \arrow[ld,shorten <>=10pt,Rightarrow,\et ] \arrow[r, dotted] \arrow[d, "\cn{\cm{X}}" description]                       &  \arrow[d, "\cn{\cm{X}}" description] \arrow[r, dotted]      &\arrow[ld,shorten <>=10pt,Rightarrow,\etr ]  \arrow[r, "\cm{X}" description] \arrow[d, dotted] & \arrow[ld,shorten <>=10pt,Rightarrow,\etr ] \arrow[r, dotted] \arrow[d, "\cn{\cm{X}}" description]                       &  \arrow[r, dotted] \arrow[d, "\cn{\cm{X}}" description]                                & {} \arrow[d, "\cn{\cm{X}}" description] \arrow[r, dotted]  &  \arrow[ld,shorten <>=10pt,Rightarrow,\et ]\arrow[d, dotted] \arrow[r, "g" description]                       &  \arrow[d, dotted]          \\
 \arrow[r, dotted] \arrow[d, "X" description]                       &  \arrow[r, dotted] \arrow[d, "X" description]                       &  \arrow[r, "\cm{X}" description] \arrow[d, "X" description] & \arrow[ld,shorten <>=10pt,Rightarrow,\et ] \arrow[r, dotted] \arrow[d, dotted]                                &  \arrow[r, dotted] \arrow[d, dotted]                                &  \arrow[r, dotted] \arrow[d, dotted]                                         & {} \arrow[r, "\cm{X}" description] \arrow[d, dotted]                       &  \arrow[r, "g" description] \arrow[d, dotted]                       &  \arrow[d, dotted]          \\
 \arrow[r, dotted]                                                  &  \arrow[r, dotted]                                                  &  \arrow[r, dotted]                                                           &  \arrow[r, dotted]                                                  &  \arrow[r, dotted]                                                  &  \arrow[r, dotted]                                                           &  \arrow[r, "\cm{X}" description]                                          &  \arrow[r, "g" description]                                         & {}                         
\end{tikzcd}

=\begin{tikzcd}
    \arrow[r, "\cm{X}" description] \arrow[d, "X" description]           
 &  \arrow[ld,shorten <>=10pt,Rightarrow,\et ]\arrow[r, dotted] \arrow[d, dotted] & \arrow[ld,shorten <>=10pt,Rightarrow,\ep] \arrow[d, "X" description] \arrow[r, "f" description] &\arrow[ld,shorten <>=10pt,Rightarrow, "\alpha" above ]   \arrow[d, "Y" description] 
\\        \arrow[r, dotted]              
 &  \arrow[r, "\cm{X}" description]      &  \arrow[r, "g" description]  &   {}        
\end{tikzcd}
}
\caption{The action of $S^{-1}T^{-1}{S}$ on \cref{eq:shape}}\label{fig:s1t1s}
\end{subfigure}

  \begin{subfigure}[b]{\textwidth}
\adjustbox{max width=\textwidth}{
\begin{tikzcd}
 \arrow[d, "\cn{f}" description] \arrow[r, dotted]                          &  \arrow[d, "\cn{f}" description] \arrow[r, dotted]                          &  \arrow[d, "\cn{f}" description] \arrow[r, dotted] & \arrow[ld,shorten <>=10pt,Rightarrow, \etr] \arrow[d, dotted] \arrow[r, dotted]                       &  \arrow[d, dotted] \arrow[r, "\cm{Y}" description]                    &  \arrow[d, dotted] \arrow[r, "g" description]                    &  \arrow[d, dotted] \arrow[r, dotted]                                          &  \arrow[r, dotted] \arrow[d, dotted]                       &  \arrow[d, dotted] 
\\
 \arrow[d, dotted] \arrow[r, dotted] &  \arrow[ld,shorten <>=10pt,Rightarrow,\etr ]\arrow[d, "X" description] \arrow[r, dotted]                          &  \arrow[d, "X" description] \arrow[r, "f" description]                          &\arrow[ld,shorten <>=10pt,Rightarrow, "\alpha" above ]   \arrow[d, "Y" description] \arrow[r, dotted]                          &  \arrow[d, "Y" description] \arrow[r, "\cm{Y}" description]  & \arrow[ld,shorten <>=10pt,Rightarrow,\etr] \arrow[r, "g" description] \arrow[d, dotted]                    &  \arrow[d, dotted] \arrow[r, dotted]                                          &  \arrow[r, dotted] \arrow[d, dotted]                       &  \arrow[d, dotted] 
\\
 \arrow[r, "\cm{X}" description] \arrow[d, dotted]                          &  \arrow[d, dotted] \arrow[r, dotted] & \arrow[ld,shorten <>=10pt,Rightarrow,\etr ] \arrow[r, "g" description] \arrow[d, "\cm{g}" description]  &  \arrow[ld,shorten <>=10pt,Rightarrow,\ep]\arrow[d, dotted] \arrow[r, dotted] &\arrow[ld,shorten <>=10pt,Rightarrow,\ep]  \arrow[d, "\cm{g}" description] \arrow[r, dotted]                    &  \arrow[r, "g" description] \arrow[d, "\cm{g}" description] & \arrow[ld,shorten <>=10pt,Rightarrow,\etr] \arrow[d, dotted] \arrow[r, dotted]                                          &  \arrow[r, dotted] \arrow[d, dotted]                       &  \arrow[d, dotted] 
\\
 \arrow[r, "\cm{X}" description] \arrow[d, dotted]                          &  \arrow[r, "g" description] \arrow[d, dotted]                          &  \arrow[r, dotted] \arrow[d, dotted]                    &  \arrow[r, "g" description] \arrow[d, dotted]  &  \arrow[ld,shorten <>=10pt,Rightarrow,\etr ]\arrow[r, dotted] \arrow[d, "\cn{g}" description]                    &  \arrow[r, dotted] \arrow[d, "\cn{g}" description]                    &  \arrow[d, "\cn{g}" description] \arrow[r, dotted]                                             &  \arrow[r, dotted] \arrow[d, "\cn{g}" description]                          &  \arrow[d, "\cn{g}" description]   
 \\
 \arrow[r, "\cm{X}" description] \arrow[d, "X" description]    &  \arrow[ld,shorten <>=10pt,Rightarrow,\et]\arrow[r, "g" description] \arrow[d, dotted]                          &  \arrow[r, dotted] \arrow[d, dotted]                    &  \arrow[r, dotted] \arrow[d, dotted]                       &  \arrow[r, dotted] \arrow[d, dotted]                 &  \arrow[r, dotted] \arrow[d, dotted]                 & {} \arrow[r, dotted] \arrow[d, dotted] \arrow[r, dotted]  &  \arrow[ld,shorten <>=10pt,Rightarrow,\et ]\arrow[r, dotted] \arrow[d, "X" description]                          &  \arrow[d, "X" description]    
\\
 \arrow[d, "\cm{g}" description] \arrow[r, dotted]               &  \arrow[r, "g" description] \arrow[d, "\cm{g}" description]       &  \arrow[ld,shorten <>=10pt,Rightarrow,\et]\arrow[d, dotted] \arrow[r, dotted]                    &  \arrow[d, dotted] \arrow[r, dotted]                       &  \arrow[r, dotted] \arrow[d, dotted]                 &  \arrow[d, dotted] \arrow[r, dotted]                 &  \arrow[d, dotted] \arrow[r, "\cm{X}" description]                                             &  \arrow[r, dotted] \arrow[d, dotted] &  \arrow[ld,shorten <>=10pt,Rightarrow,\et ]\arrow[d, "\cm{g}" description]   
 \\
 \arrow[r, dotted]                                         &  \arrow[r, dotted]                                         &  \arrow[r, dotted]                                      &  \arrow[r, dotted]                                         &  \arrow[r, dotted]                                   &  \arrow[r, dotted]                                   &  \arrow[r, "\cm{X}" description]                                                               &  \arrow[r, "g" description]                                            & {}                
\end{tikzcd}
=\begin{tikzcd}
 \arrow[d, "\cn{f}" description] \arrow[r, dotted]& \arrow[ld,shorten <>=10pt,Rightarrow,\etr] \arrow[d, dotted] \arrow[r, "\cm{Y}" description]                    &  \arrow[d, dotted] \arrow[r, dotted]                                 &  \arrow[d, dotted] \arrow[r, "g" description]    &  \arrow[ld,shorten <>=10pt,Rightarrow,\etr ]\arrow[d, "\cn{g}" description]              \\
 \arrow[d, "X" description] \arrow[r, "f" description]                               &\arrow[ld,shorten <>=10pt,Rightarrow, "\alpha" above ]   \arrow[d, "Y" description] \arrow[r, "\cm{Y}" description] & \arrow[ld,shorten <>=10pt,Rightarrow,\etr ] \arrow[d, dotted] \arrow[r, dotted]           & \arrow[ld,shorten <>=10pt,Rightarrow,\et ] \arrow[d, "X" description] \arrow[r, dotted]               &  \arrow[d, "X" description]                   \\
 \arrow[r, "g" description] \arrow[d, "\cm{g}" description]  &\arrow[ld,shorten <>=10pt,Rightarrow,\ep]  \arrow[d, dotted] \arrow[r, dotted]                                  &  \arrow[r, "\cm{X}" description] \arrow[d, dotted] &  \arrow[d, dotted] \arrow[r, dotted]  & \arrow[ld,shorten <>=10pt,Rightarrow,\et ] \arrow[d, "\cm{g}" description] 
\\
 \arrow[r, dotted]                                          &  \arrow[r, dotted]                                                    &  \arrow[r, "\cm{X}" description]                                     &  \arrow[r, "g" description]                                 & {}                                
\end{tikzcd}
}
\caption{The action of $TST$ on \cref{{eq:shape}}}\label{fig:tst}
\end{subfigure}

  \begin{subfigure}[b]{.6\textwidth}
\begin{tikzcd}
 \arrow[d, dotted] \arrow[r, dotted]            &  \arrow[lddd,shorten <>=10pt,Rightarrow,\epr]\arrow[d, "\cn{f}"] \arrow[r, dotted] & \arrow[ld,shorten <>=10pt,Rightarrow,\etr ] \arrow[d, dotted] \arrow[r, "\cm{X}" description]                              &  \arrow[d, dotted] \arrow[r, dotted]                                 &  \arrow[d, dotted] \arrow[r, "f" description]    & \arrow[ld,shorten <>=10pt,Rightarrow,\etr ] \arrow[d, "\cn{f}" description] \arrow[r, dotted]   &  \arrow[ld,shorten <>=10pt,Rightarrow,\epr ]\arrow[d, dotted] \\
\arrow[d, "X" description]                          &  \arrow[d, "X" description] \arrow[r, "f" description]       & \arrow[ld,shorten <>=10pt,Rightarrow,"\alpha" above] \arrow[d, "X" description] \arrow[r, "\cm{X}" description] &  \arrow[ld,shorten <>=10pt,Rightarrow,\etr ]\arrow[d, dotted] \arrow[r, dotted]           & \arrow[ld,shorten <>=10pt,Rightarrow,\et ] \arrow[d, "X" description] \arrow[r, dotted]               &  \arrow[d, "X" description] \arrow[r, "f" description]      &  \arrow[ld,shorten <>=10pt,Rightarrow,\epr, "\alpha" above]\arrow[d, "X" description]    
\\
\arrow[d, dotted]& \arrow[r, "f" description] \arrow[d, "\cm{f}" description]  & \arrow[ld,shorten <>=10pt,Rightarrow,\ep] \arrow[d, dotted] \arrow[r, dotted]                                            &  \arrow[r, "\cm{X}" description] \arrow[d, dotted] &  \arrow[d, dotted] \arrow[r, dotted] & \arrow[ld,shorten <>=10pt,Rightarrow,\et] \arrow[d, "\cm{f}" description] \arrow[r, "f" description] & \arrow[ld,shorten <>=10pt,Rightarrow,\epr] \arrow[d, dotted]
 \\
 \arrow[r, dotted]                                                & \arrow[r, dotted]                                          &  \arrow[r, dotted]                                                              &  \arrow[r, "\cm{X}" description]                                     &  \arrow[r, "f" description]                                 &  \arrow[r, dotted]                                                    &                   {}
\end{tikzcd}
\caption{The expanded action of $TST$ following \cref{lem:compareststst}}\label{fig:expanded}
\end{subfigure}
\caption{Comparing the actions of $S^{-1}T^{-1}S$ and $TST$}
\end{figure}
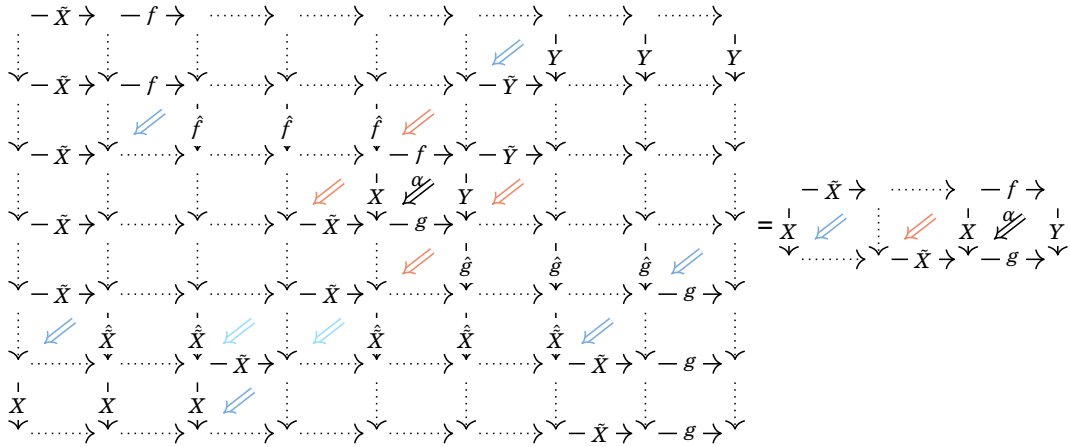
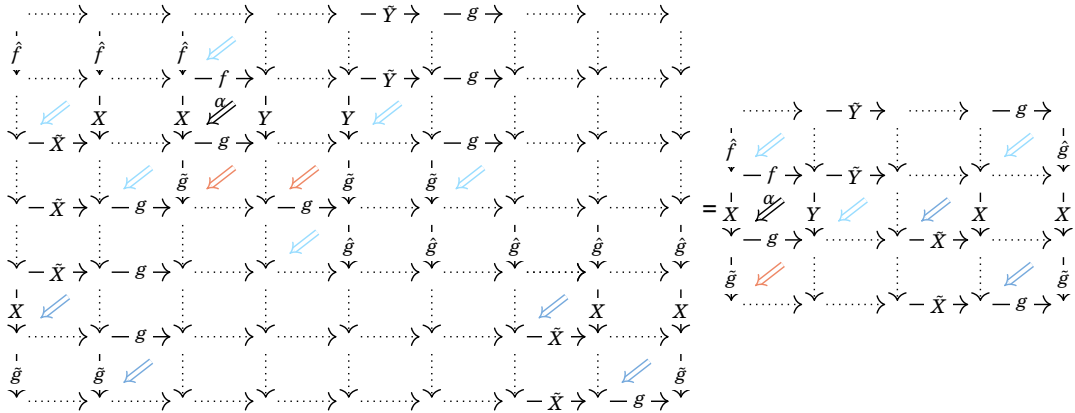
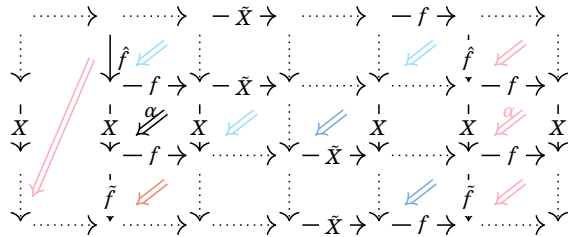

\section{Enriched homotopical categories to monoidal bicategories}\label{sec:enriched}
We now turn to the verification that the bicategories of dg-categories and their bimodules and homomorphisms and spectral categories and their bimodules and homomorphisms as in \cref{ex:enriched_bicat} are symmetric monoidal bicategories.  This follows from work of Shulman \cite{shulman, shulman_symmetric} that we summarize here since it does not seem to be nearly as well known as it deserves to be.

The first step is the verification that they appropriately assemble into bicategories. This relies almost exclusively on work of \cite{shulman} and this section can hopefully serve as an introduction to that lovely paper.  A textbook treatment of some of that material can also be found in \cite{riehl_homotopical}. The second step verifies that these bicategories are symmetric monoidal.  It is completed in \cref{sec:smb} and relies on \cite{shulman_symmetric}.

We are most interested in dg or spectral categories, but the constructions in this section apply to other categories with similar formal properties.  These are  $\mc{V}$-enriched categories (\cref{def:enriched}) where $\mc{V}$ is a symmetric monoidal category satisfying some additional conditions  (\cref{ass:homotopy}).  These are settings where bar resolutions (\cref{def:bar}) and cyclic bar resolutions (\cref{def:cyclic_bar}) are defined and give us the bicategorical composition and shadow.

An alternative would be to work in 
enriched $(\infty,2)$-categories (this technology is developed by Haugseng in \cite{haugseng}) but that seems unnecessary here. However, we point out that the formalism we work in has almost exactly the same power that that set-up has. Any dualizability statement on an $(\infty,2)$-category is reflected by the underlying homotopy 2-category. We are essentially building the homotopy 2-category of a symmetric monoidal $(\infty,2)$-category by hand. 

\subsection{Enriched categories and bimodules}\label{sec:enriched_bimod}
From this point on $\mc{V}$ is a closed symmetric monoidal category. We denote such a category by $(\mc{V}, \otimes, \hom, I)$ when we need to include the data of the category, the tensor product, internal hom, and monoidal unit.

\begin{defn}\label{def:enriched}
  A {\bf $\mc{V}$-enriched category} $\mc{C}$   is  
\begin{itemize}
  \item a {collection of objects} $a, b, c, \dots$
  \item {a morphism object} $\mc{C}(a,b) \in \mc{V}$ for each pair of objects 
  \item unit maps $I\to \mc{C}(a,a)$ for each object in $\mc{C}$ and 
  \item associative and unital { composition maps} 
    \[
    \mc{C}(a,b) \otimes \mc{C} (b, c) \to \mc{C}(a,c).
    \]
  \end{itemize}
\end{defn}

\begin{defn}
  Let $\mc{C}$ be an $\mc{V}$-enriched category. The \textbf{underlying category} of $\mc{C}$, denoted $\mc{C}_0$, has the same objects as $\mc{C}$ and 
  \[
  \mc{C}_0 (a, b)\coloneqq \mc{C}(I, \mc{C}(a, b)) 
  \]
\end{defn}

  If $\mc{A}$ and $\mc{B}$ are $\mc{V}$-enriched categories the {\bf pointwise tensor product} of 
$\mc{A}$ and $\mc{B}$, denoted $\mc{A} \otimes \mc{B}$, is a $\mc{V}$-enriched category whose objects are 
pairs of elements $(a, b)$ with $a \in \operatorname{ob} \mc{A}$ and $b \in \operatorname{ob} \mc{B}$.
The morphism spaces are given by 
    \[
    \mc{A} \otimes \mc{B} ((a,b), (a',b')) := \mc{A}(a,a') \otimes \mc{B}(b,b') 
    \]

    The following definition is critical and useful in what follows.
    
\begin{defn}
  A ($\mc{V}$-){\bf two-variable adjunction}
  \[
  (\oast, \hom^l, \hom_r): \mc{M} \otimes \mc{N} \to \mc{P} 
  \]
  is a triple of functors
  \begin{align*}
    \oast &: \mc{M} \otimes \mc{N} \to \mc{P}\\
    \hom^l &: \mc{M}^{\text{op}} \otimes \mc{P} \to \mc{N}\\
    \hom^r &: \mc{N}^{\text{op}} \otimes \mc{P} \to \mc{M}
  \end{align*}
  and natural isomorphisms
  \[
  \mc{P}(M \oast N, P) \cong \mc{N}(N, \hom^l (M, P)) \cong \mc{M}(M, \hom^r(N, P))
  \]
  for $M \in \mc{M}, N \in \mc{N}$ and $P \in \mc{P}$.
\end{defn}

\begin{eg} There are two very standard examples of 2-variable adjunctions.
\begin{itemize}
\item 
If $\mc{V}$ is a symmetric monoidal closed category, then $\otimes : \mc{V} \times \mc{V} \to \mc{V}$ participates in a 2-variable adjunction $(\otimes,\hom ,\{,\})$. This is a crucial example.
\item  Given three rings $A, B, C$ there is a 2-variable adjunction
  \[
  (\otimes_B, F_A, F_C): _A \Mod_B \otimes _B \Mod_C \to _A \Mod_C 
  \]
\end{itemize}
\end{eg}

\begin{eg}\label{eg:enrichmenet_2_var}
  Enrichment can also be encoded as a 2-variable adjunction \cite[Defn~14.3]{shulman}. That is, an unenriched category $\mc{C}_0$ can be given the structure of a $\mc{V}$-enriched, tensored, and cotensored category if there is an (unenriched) 2-variable adjunction
  \[
  (\odot, \{\}, \hom): \mc{V}_0 \times \mc{C}_0 \to \mc{C}_0 
  \]
  together with associatiity and unit isomorphisms
  \[
  V_1 \odot (V_2 \odot C) \cong (V_1 \otimes V_2) \odot C \qquad I \odot C \cong C. 
  \]
  
\end{eg}

We recall a generalization of the Morita bicategory where  rings are replaced by $\mc{V}$-enriched categories. The core definition is the following. 

\begin{defn}
  Let $\mc{A}, \mc{B}$ be $\mc{V}$-categories. An $(\mc{A},\mc{B})$-{\bf bimodule} is a $\sV$-functor $\mc{M}: \mc{A}^{\text{op}} \otimes \mc{B} \to \mc{V}$. 
\end{defn}

An enriched category  $\mc{A}$ itself is naturally an $\mc{A}^{\text{op}} \otimes \mc{A}$-bimodule 
and $\mc{V}$-functors generalize bimodule homomorphisms. 

The most basic operation with bimodules is tensor product. We have that same operation in its ``many objects'' version here.

\begin{defn}\label{def:odot}
  Let $\mc{M}$ be an $(\mc{A}, \mc{B})$-module and $\mc{N}$ an $(\mc{B}, \mc{C})$-module. We define $\mc{M} \odot_{\mc{B}} \mc{N}$ to be the following $(\mc{A},\mc{C})$-module
  \[
  \mc{M}\odot_{\mc{B}} \mc{N} := \left(
  \begin{tikzcd}
    \displaystyle\coprod_{b,b'} \mc{M}(-,b) \otimes \mc{B}(b,b') \otimes \mc{N}(b', -) \ar[r,shift left=.5ex]\ar[r,shift right=.5ex]&  \displaystyle\coprod_{b''} \mc{M}(-,b'') \otimes \mc{N}(b'',-)
  \end{tikzcd}
  \right)
  \]
\end{defn}

The relative tensor product $\odot_{\mc{B}}$ participates in a 2-variable adjunction, just as in the case of rings and bimodules.

\begin{defn}
  Let $\mc{A},\mc{B},\mc{C}$ be $\mc{V}$-categories. Let $\mc{M}$ be an $(\mc{A}, \mc{B})$-module, $\mc{N}$ a $(\mc{B},\mc{C})$-module and $\mc{P}$ an $(\mc{A}, \mc{C})$-module. We define a $(\mc{B},\mc{C})$-module $  \hom^{\mc{A}} (\mc{M}, \mc{P})(b,c) $ by 
  \[
\operatorname{eq}\left(  \begin{tikzcd}
    \displaystyle\prod_{a} \hom^{\mc{V}}( \mc{M}(a,b), \mc{P}(a,c))\ar[r,shift left=.5ex] \ar[r, shift right=.5ex] & \displaystyle\prod_{a,a'} \hom^{\mc{V}}(\mc{A}(a',a) \otimes \mc{M}(a,b), \mc{P}(a', c))
  \end{tikzcd}
  \right)
  \]
  and similarly a $(\mc{A},\mc{B})$-module  $\hom^{\mc{C}} (\mc{N}, \mc{P})(a,b)$ by 
\begin{equation*}
{
  \operatorname{eq}\left(
  \begin{tikzcd}
      \displaystyle\prod_{c} \hom^{\mc{V}}(\mc{N}(a,c), \mc{P}(b,c))\ar[r,shift left = .5ex] \ar[r, shift right = .5ex] & 
	  \displaystyle\prod_{c,c'} \hom^{\mc{V}}\left(\mc{C}(c,c') \otimes \mc{N}(a,c), \mc{P}(b,c') \right)
  \end{tikzcd}
 \right)}
\end{equation*}
\end{defn}

If $\mc{A}$ and $\mc{B}$ are $\mc{V}$-categories, we let  $_{\mc{A}} \Mod_{\mc{B}}$ denote the category of $(\mc{A},\mc{B})$-bimodules and their homomorphisms.
\begin{prop}\label{prop:two_var}
  The three functors above assemble into a two variable adjunction
  \[
  (\odot, \hom^{\mc{A}}, \hom^{\mc{C}}) : _{\mc{A}} \Mod_{\mc{B}} \otimes _{\mc{B}} \Mod_{\mc{C}} \to _{\mc{A}} \Mod_{\mc{C}} 
  \]
\end{prop}

\subsection{Homotopical categories and enriched categories}
All of the categories $\mc{V}$ that we are interested in are fundamentally homotopical. In this section we review the necessary machinery to that the Morita bicategory is appropriately homotopical.  

We work in a homotopical situation that is slightly more general than a Quillen model category. This is desirable because it 
more clearly illustrates the required technical assumptions and 
  gives added flexibility in computations. 

\begin{defn}\cite{DHKS}
  A \textbf{homotopical category} is a category $\mc{C}$ equipped with a subcategory of weak equivalences $\mc{W}$ that satisfy the 2/6 property: If $h \circ g$ and $g \circ f$ in the diagram below are in $\mc{W}$, then so are the remaining four. 
  
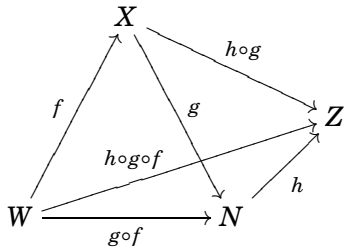
\[
  \xymatrix{
    & X \ar[ddr]^g\ar[drr]^{h \circ g} & & \\
    & & &Z \\
    W\ar[uur]^f\ar[rr]_{g \circ f}\ar[urrr]^(.4){h \circ g \circ f} &\ & N\ar[ur]_h 
  }
  \]
A functor $F: \mc{C} \to \mc{D}$ between two homotopical categories is {\bf homotopical} if $F$ preserves weak equivalences. 
\end{defn}

To  define left and right derived functors, we need  left and right deformations. These are  formal analogues of cofibrant and fibrant replacement and  capture the homotopically well-behaved objects. This weakening is useful since some homotopically well-behaved objects do not participate in a model category structure (e.g. flat resolutions), but are useful tools for homotopical control. 

\begin{defn}\cite{DHKS}
Let $\mc{C}$ be a homotopical category. A \textbf{left deformation} is an endofunctor $Q: \mc{C} \to \mc{C}$ together with a natural weak equivalence $Q \Rightarrow \id_\mc{C}$. Dually, a \textbf{right deformation} is an endofunctor $R :\mc{C} \to \mc{C}$ together with a natural weak equivalence $\id_{\mc{C}} \Rightarrow R$. 

The \textbf{left deformation retract} of $\mc{C}$ is the full subcategory on objects in the image of $Q$.  We denote this by $\mc{C}_Q$. Similarly for a \textbf{right deformation retract}, $\mc{C}_R$. If $\mc{C}$ has both and left and right deformation we call the pair $(\mc{C}_Q, \mc{C}_R)$ a {\bf deformation retract} of $\mc{C}$. 
\end{defn}

\begin{rmk}
At this point the sources of deformation retracts are not important.  As we will see later, the main lesson of \cite{shulman} is that bar and cobar constructions almost always provide excellent models for them. 
\end{rmk}

The categories we are interested in are both homotopical and enriched and so our next step is to define what it means for those structures to be compatible.  This requires the notion of a deformation of a 2-variable adjunction. In turn, this requires that each of the categories participating in a 2-variable adjunction be homotopical.  

\begin{defn}\cite[Defn.~15.1]{shulman}
  A \textbf{deformation} for a 2-variable adjunction \[(\oast, \hom^l, \hom_r): \mc{M} \times \mc{N} \to \mc{P}\] consists of left deformation retracts $\mc{M}_Q$, $\mc{N}_Q$ and their associated deformations, and a right deformation retract $\mc{P}_R$ and its associated deformation, such that: 
  \begin{itemize}
  \item \textbf{(tensor homotopical)} $\oast$ is homotopical on $\mc{M}_Q \times \mc{N}_Q$
  \item \textbf{(left hom homotopical)} $\hom^l$ is homotopical on $\mc{M}^{\text{op}} \times \mc{P}_R$
  \item \textbf{(right hom homotopical)} $\hom^r$ is homotopical on $\mc{N}^{\text{op}}_Q \times \mc{P}_R$
  \end{itemize}
\end{defn}

Among other uses, this definition provides us with a context for discussing a deformation of an enrichment, since an enrichment is simply a 2-variable adjunction  (\cref{eg:enrichmenet_2_var}). This leads to the following, absolutely crucial, definition. The definition gives the conditions under which the enrichment and the homotopical structures ``play well'' together. 

\begin{defn}\label{defn:v_homotopical}
  Let $\mc{V}$ be a homotopical category.  
 An enriched, tensored and cotensored $\mc{V}$-category $\mc{M}$ is \textbf{$\mc{V}$-homotopical} if there are left and right deformations $\mc{M}_Q, \mc{M}_R$ of $\mc{M}$ such that 
	\[(\mc{V}_Q, \mc{M}_Q, \mc{M}_R)\]
 is a deformation of the two variable adjunction for the enrichment and the following hold
  \begin{itemize}
  \item \textbf{(cofibrant-fibrant preservation)} $\odot$ takes $\mc{V}_Q \times \mc{M}_Q$ to $\mc{M}_Q$ and $\{-,-\}$ takes $\mc{V}^{\text{op}}_Q \times \mc{M}_R$ to $\mc{M}_R$.
  \item \textbf{(unit conditions)} Let $M \in \mc{M}_Q$ and $N \in \mc{M}_R$, then the natural maps
\[
      QI \odot M \to I \odot M \qquad \{I,N\} \to \{QI, N\}
\]
    are weak equivalences 
  \end{itemize}
\end{defn}
We think of the first of these conditions as asserting that the tensor preserves cofibrant objects and the cotensor preserves fibrant objects.  A consequence is the following.

\begin{prop}\cite[Prop.16.2]{shulman}
If $\mc{M}$ is $\mc{V}$-homotopical, then $\operatorname{Ho}(\mc{M}_0)$ is a tensored and cotensored $\operatorname{Ho}(\mc{V})$-category. 
\end{prop}

That is, the enriched structure descends to homotopy categories.

Finally, with the definition of $\mc{V}$-homotopical categories in hand,  we can discuss 2-variable adjunctions \textit{between} homotopical $\mc{V}$-categorical categories.

\begin{defn}\cite[16.11]{shulman}
  Let $(\oast, \hom^l, \hom^r): \mc{M} \otimes \mc{N} \to \mc{P}$ be a $\mc{V}$-two variable adjunction between $\mc{V}$-homotopical categories and  $(\mc{M}_Q, \mc{N}_Q, \mc{P}_R)$ be a deformation retract for $\mc{M}_0 \times \mc{N}_0 \to \mc{P}_0$. 
  \begin{itemize}
  \item The adjunction is $\oast$-$\mc{V}$-\textbf{deformable} if $\oast$ takes $\mc{M}_Q \times \mc{N}_Q$ to $\mc{P}_Q$.
  \item The adjunction is $\hom^l$-$\mc{V}$-\textbf{deformable} if $\hom^l$ takes $\mc{M}^{\text{op}} \times \mc{P}_R$ to $\mc{M}_R$
  \item The adjunction is $\hom^r$-$\mc{V}$-\textbf{deformable} if $\hom^r$ takes $\mc{N}^{\text{op}} \times \mc{P}_R$ to $\mc{N}_R$ 
  \end{itemize}
\end{defn}

The key point is the following proposition

\begin{prop}\cite[16.13]{shulman}
  A $\mc{V}$-deformable two-variable adjunction $\mc{M} \otimes \mc{N} \to \mc{P}$ descends to a $\operatorname{Ho}(\mc{V})$-2-variable adjunction
  \[
  \operatorname{Ho}(\mc{M}) \otimes^{\mbf{L}} \operatorname{Ho}(\mc{N}) \to \operatorname{Ho}(\mc{P}) 
  \]
\end{prop}

\subsection{Bimodules with homotopy}
Our goal is now to combine the bicategory of bimodules in a symmetric monoidal category $\mc{V}$ with a homotopical structure on $\mc{V}$.  The primary challenge is to show that for 
the 2-variable adjunction
\[
_{\mc{A}} \Mod_{\mc{B}} \otimes _{\mc{B}} \Mod_{\mc{C}} \to _{\mc{A}} \Mod_{\mc{C}}
\]
the functor $-\odot N$ preserves weak equivalences. 
Shulman's  \cite{shulman} key observation is that under good conditions bar and cobar constructions  compute an enriched deformation of the tensor product defined in \cref{def:odot}.

At this point, we make the following technical assumption about $\mc{V}$.

\begin{ass}\label{ass:homotopy}
   $\mc{V}$ is a closed, symmetric monoidal homotopical category with a symmetric monoidal adjunction $\mbf{Set}_{\Delta} \leftrightarrows \mc{V}$ where the left adjoint is strong symmetric monoidal. 
\end{ass}
This equips $\mc{V}$ with a canonical cosimplicial object associated to $\Delta^\bullet$ in $\mbf{Set}_{\Delta}$. We denote this cosimplicial object by $\Delta^\bullet: \Delta \to \mc{V}$.

\begin{defn}\label{def:bar}
  Let $(\oast, \hom^l, \hom^r): \mc{M} \otimes \mc{N} \to \mc{P}$ be a 2-variable adjunction of $\mc{V}$-categories, $\mc{D}$ be a small $\mc{V}$-category and $F: \mc{D}^{\text{op}} \to \mc{M}$ and $G: \mc{D} \to \mc{N}$ be $\mc{V}$-functors. The \textbf{enriched simplicial two-sided bar construction} is the simplicial object in $\mc{P}$ given in degree $n$ by
  \[
  B_n (F, \mc{D},G):= \coprod_{d_0 \to d_1 \to \cdots \to d_n} \left(\mc{D}(d_{n-1},d_n) \otimes \cdots \otimes \mc{D}(d_0, d_1) \right)\odot \left(\left (G(d_n) \oast F(d_0)\right)\right)
  \]
  The \textbf{enriched two-sided bar construction} is the geometric realization of the above simplicial object in $\mc{P}$
  \[
  B(F, \mc{D}, G) := \Delta^\bullet \odot_{\Delta} B_n (F, \mc{D}, G). 
  \]

\end{defn}

Bar constructions define homotopy colimits and preserve pointwise weak equivalences in certain cases because the simplicial objects they produce are Reedy cofibrant \cite[Prop.~23.6]{shulman}. In the $\mc{V}$-homotopical situation, a more comprehensive condition is needed. 
\begin{defn}\cite[Defn. 20.1]{shulman}
  Let $\mc{D}$ be a small $\mc{V}$-category and
  \[
  (\oast, \hom^l, \hom^r): \mc{M} \otimes \mc{N} \to \mc{P}
  \]
  be a deformable two variable $\mc{V}$-adjunction between $\mc{V}$-homotopical categories. We say $\mc{D}$ is \textbf{very good for $\oast$} if
  \begin{itemize}
\item $B(-,\mc{D},-)$ is homotopical on $\mc{M}^{\mc{D}^{\text{op}}}_Q \times \mc{N}^{\mc{D}}_Q$. This condition is the ``pointwise preservation of weak equivalences'' condition.
\item $B(\mc{D},\mc{D},F) \in \mc{N}^{\mc{D}}_Q$ and $B(G, \mc{D},\mc{D}) \in \mc{M}^{\mc{D}^{\text{op}}}_Q$
\item $B(G, \mc{D}, F) \in \mc{P}_Q$
\item $\mc{D}(x,y)$ is in $\mc{V}_Q$ for each $x, y \in \operatorname{ob}(\mc{D})$. This is ``pointwise cofibrancy''.
  \end{itemize}
\end{defn}

This  condition  ensures that the enriched bar construction descends to a derived functor. It can be thought of as a generalization of Reedy cofibrancy.  Verifying the condition in practice is another matter, which we take up later in this section. We most frequently use the condition for the hom-tensor-cotensor 2-variable adjunction for $\mc{V}$. For convenience, we fully state that.

\begin{defn}\label{defn:very_good}
  A small $\mc{V}$-category $\mc{D}$ is {\bf very good} for $\otimes$ if
  \begin{itemize}
  \item $B(-, \mc{D}, -)$ is homotopical on $\mc{V}_Q^{\mc{D}^{\text{op}}} \otimes \mc{V}_Q^{\mc{D}}$
  \item $B(\mc{D}, \mc{D}, F) \in \mc{V}^{\mc{D}}_Q$ and $B(G, \mc{D}, \mc{D}) \in \mc{V}^{\mc{D}^{\text{op}}}_Q$
  \item $B(G, \mc{D}, F) \in \mc{V}_Q$
  \item $\mc{D}(x, y) \in \mc{V}_Q$ for all $x, y \in \operatorname{ob} \mc{D}$. 
  \end{itemize}
\end{defn}

The $\mc{V}$-categories that are good for $(\mc{V}, \otimes)$ are the objects of a bicategory (in fact, double category, as we will show below).

\begin{thm}\cite[Prop. 22.11]{shulman}\label{thm:main_construction}
  Let $\mc{V}$ be a closed symmetric monoidal homotopical category with strong monoidal adjunction $\mbf{Set}_{\Delta} \leftrightarrows \mc{V}$. Assume  that all simplicial homotopy equivalences in $\mc{V}$ are weak equivalences and that $\mc{V}$ is saturated. Let $\mc{B}(\Cat_{\mc{V}})$ be the bicategory with
  \begin{itemize}
  \item 0-cells: very good small $\mc{V}$-categories.
  \item 1-cells: bimodules, $\mc{A}^{\text{op}} \otimes \mc{B} \to \mc{V}$. 
  \item 2-cells: $\mc{V}$-natural transformations. 
  \end{itemize}
  $\mc{B}(\Cat_{\mc{V}})$ has a homotopy bicategory $\operatorname{Ho}(\mc{B}(\Cat_{\mc{V}}))$ that is $\operatorname{Ho}(\mc{V})$ enriched.
\end{thm}

The associativity of the categorical composition is proved in \cite[p.65]{shulman}.

There are many examples of categories $\mc{V}$ for which we can identify large classes of very good $\mc{V}$-categories. The following two examples are the most useful for us. 

\begin{eg}
  Let $\mc{V}$ be the symmetric monoidal category of orthogonal spectra $(\Sp^{O}, \sma)$, symmetric spectra $(\Sp^{\Sigma}, \sma)$ with the stable model structure \cite{mandell_may_shipley_schwede} or the category $\mc{M}_S$ of $S$-modules \cite{EKMM}. Then a small $\mc{V}$-category $\mc{C}$ such that
  \begin{itemize}
  \item $\mc{C}(a,b)$ is cofibrant for each $a, b \in \operatorname{ob} \mc{C}$
  \item The unit inclusion $S \to \mc{C}(c,c)$ is a cofibration for each $c \in \mc{C}$ 
  \end{itemize}
  is very good for $\otimes$ by \cite[Thm.~23.12]{shulman}. The verification amounts to a proof of Reedy cofibrancy for the necessary bar constructions.
\end{eg}

\begin{eg}
Let $\operatorname{Ch}_k$ denote the category of chain complexes equipped with the projective model structure. Then $\operatorname{Ch}_k$-enriched categories that are pointwise cofibrant are very good. 
\end{eg}

\subsection{Symmetric Monoidal Structure}\label{sec:smb}

In this subsection, we discuss the symmetric monoidal structure on $\mc{B}(\Cat_{\mc{V}})$. Putting a symmetric monoidal structure on a bicategory requires checking many coherence conditions (to see this, consult \cite{stay}). However, in particularly nice situations, such as bicategories that come from double categories \cite{ehresmann}, there are shortcuts available \cite{shulman_symmetric}. Luckily, we are in this situation, and we describe that structure now.

Throughout this section we freely use the language of double categories and double categorical notation. Horizontal morphisms will be marked with a vertical bar $|$ for clarity. We also recall that there are two ways of taking an ``opposite'' double category: one horizontally, which we call $\mc{D}^{\text{hop}}$ and one vertically, which we call $\mc{D}^{\text{vop}}$. Finally, we denote the underlying bicategory of a double category by $\underline{\mc{D}}$. 

The workhorse theorem in this section is a theorem of Shulman's \cite[Thm.~1.1]{shulman_symmetric} that states that the underlying bicategory of a symmetric monoidal double category is also symmetric monoidal. 

\begin{thm}\cite[Thm~1.1]{shulman_symmetric}
If $\mc{D}$ is a fibrant (see \cref{defn:fibrant_double_category}) symmetric monoidal double category, then the underlying bicategory $\underline{\mc{D}}$ is symmetric monoidal. 
\end{thm}

There are a number of conditions to be introduced and checked in order to be able to use this theorem. Although the following list may seem overwhelming, each of the verifications is trivial, and almost all of them follow from a simple statement about the category $\mc{V}$. 

\begin{itemize}
\item Show that $\operatorname{Ho}(\mc{B}(\Cat_{\mc{V}}))$ is the underlying bicategory of a double category. 
\item Define the tensor product and show that tensor product descends to homotopy category. 
\item Show that the tensor product is a $\operatorname{Ho}(\mc{V})$-functor.
\item Construct the ``globular morphisms'' or ``tensorators'' and unit morphisms.
\item Show that the globular morphisms and unit morphisms satisfy the required axioms.
\item Show that the double category possesses companions and conjoints (i.e. forms a \textit{fibrant} double category).
\item Apply Shulman's theorem.
\end{itemize}

We define the main double category of interest.

\begin{defn}\label{defn:double_cat}
  Let $\mc{D}(\Cat_{\mc{V}})$ be  a double category with
  \begin{itemize}
  \item \textbf{category of objects, $\mc{D}_0$}: very good $\mc{V}$-categories $\mc{A},\mc{B},\mc{C}...$ with morphisms given by functors between very good categories. 
  \item \textbf{category of morphisms, $\mc{D}_1$:} the category $(\mc{A},\mc{B})$-bimodules and morphisms homotopy classes of $\mc{V}$-natural transformations,  $\operatorname{Ho}(_{\mc{A}} \Mod_{\mc{B}})$
  \item \textbf{horizontal composition:} Horizontal composition is given by derived tensor product
    \[
    \odot^{\mbf{L}}: \operatorname{Ho}(_{\mc{A}} \Mod_{\mc{B}}) \times  \operatorname{Ho}(_{\mc{B}} \Mod_{\mc{C}}) \to  \operatorname{Ho}(_{\mc{A}} \Mod_{\mc{C}})
    \]
  \item \textbf{unit, source, target:} The unit $\mc{D}_0 \to \mc{D}_1$ is given by $\mc{A}\mapsto U_{\mc{A}}$. The source and target functors are the left and right categories. 
  \item \textbf{associators, left unit, right unit}. The associator expresses the associativity of composition:
    \[
    \begin{tikzcd}
      (M \odot^{\mbf{L}} N) \odot^{\mbf{L}} P \xrightarrow{\cong} M \odot^{\mbf{L}} (N \odot^{\mbf{L}} P) 
    \end{tikzcd}
    \]
    and the units are typical the maps expressing that tensoring by the unit is an isomorphism:
    \[
    \begin{tikzcd}
      l : U_A \odot^{\mbf{L}} M \ar{r}{\cong} & M & r: M \odot^{\mbf{L}} U_B \ar{r}{\cong} & M 
    \end{tikzcd}
    \]
  \end{itemize}
\end{defn}

We now need to produce a symmetric monoidal structure on this double category. This will involve showing that $\mc{D}_0$ and $\mc{D}_1$ are symmetric monoidal categories that satisfy a number of conditions. For the moment, we assume the following, and we later verify it in cases of interest.

\begin{ass}\label{ass:very_good}
Suppose $\mc{A}, \mc{B}$ are very good $\mc{V}$-categories. Then $\mc{A} \otimes \mc{B}$ is very good.   
\end{ass}

\begin{rmk}
With the exception of the final requirement, it turns out that the requirements of \cref{defn:very_good} are \textit{not} formal, and cannot be transferred automatically. However, in all of the cases of interest, \cref{ass:very_good} does hold, and given how very goodness is verified, it is difficult to imagine a case where it would not. 
\end{rmk}

\begin{eg}
For any of the categories of spectra the tensor product of pointwise cofibrant spectral categories is pointwise cofibrant, so \cref{ass:very_good} is satisfied. 
\end{eg}

\begin{defn}
  Define the symmetric monoidal structure on $\mc{D}_0$ (see \cref{defn:double_cat}) via pointwise tensor: Given $\mc{A}, \mc{B}$ very good small $\mc{V}$-categories, define
  \[
  \mc{A} \otimes \mc{B} ((a_1, b_1), (a_2, b_2)) = \mc{A}(a_1, a_2) \otimes \mc{B}(b_1, b_2)
  \]
\end{defn}

\begin{defn}
  The symmetric monoidal structure on $\mc{D}_1$  (see \cref{defn:double_cat}) is defined pointwise: Given $\mc{M} \in \,_\mc{A}\operatorname{Mod}_{\mc{B}}$, $\mc{N} \in \,_{\mc{C}} \Mod_{\mc{D}}$, define $\mc{M} \otimes \mc{N} \in \,_{\mc{A} \otimes \mc{B}} \Mod_{\mc{C} \otimes \mc{D}}$ by
  \[
  \mc{M} \otimes \mc{N} ((a, b), (c, d)) := \mc{M}(a,b) \otimes \mc{N}(c,d) 
  \]
\end{defn}

\begin{rmk}
The definition above descends to homotopy categories since $\otimes$ is homotopical on $\mc{V}_Q$. Furthermore, since $\otimes$ is an $\operatorname{Ho}(\mc{V})$-functor, so are the tensor products. 
\end{rmk}

We now define (part of) the structure of a symmetric monoidal double category. For a full definition, see \cite[Defn.~2.9]{shulman_symmetric}. The main interesting structure in a symmetric monoidal double category is the following. 

\begin{defn}
  The \textbf{globular morphism} or \textbf{tensorator}, $\operatorname{tns}$, on $\operatorname{Ho}(\mc{B}(\Cat_{\mc{V}}))$  is defined as follows. Let $_{\mc{A}}\mc{M}_{\mc{A}'}$, $_{\mc{A}'}\mc{M}'_{\mc{A}''}$, $_{\mc{B}} \mc{N}_{\mc{B}'}$, $_{\mc{B}'} \mc{N}'_{\mc{B}''}$ be objects in $\mc{B}(\Cat_{\mc{V}})$. Define a map 
  \[
  \begin{tikzcd}
    (\mc{M} \otimes \mc{N}) \odot (\mc{M}' \otimes \mc{N}') := |B_\bullet(\mc{M} \otimes \mc{N}, \mc{A} \otimes \mc{B}, \mc{N} \otimes \mc{N}')| \ar{d}{\operatorname{tns}}[swap]{\cong}\\
    (\mc{M} \odot \mc{M}') \otimes (\mc{N} \odot \mc{N}') := |B_\bullet (\mc{M}, \mc{A}', \mc{M}')| \otimes |B_\bullet (\mc{N}, \mc{B}', \mc{N}')| 
  \end{tikzcd}
  \]
  via the shuffle homomorphisms. This descends to the homotopy categories.

  The \textbf{unitor} $U_{\mc{A} \otimes \mc{B}} \to (U_{\mc{A}} \otimes U_{\mc{B}})$ is defined via the maps
  \[
  U_{\mc{A} \otimes \mc{B}} := _{\mc{A} \otimes \mc{B}} (\mc{A}\otimes \mc{B})_{\mc{A} \otimes \mc{B}} \xrightarrow{\cong} _{\mc{A}} \mc{A}_{\mc{A}} \otimes _{\mc{B}} \mc{B}_{\mc{B}} 
  \]
\end{defn}

As per \cite[Defn~2.9]{shulman_symmetric} there are a number of axioms to check. All of them can be checked using the same technique and we give one example. The key point is that the tensor product $\otimes: \mc{V} \times \mc{V} \to \mc{V}$ is homotopical, and the horizontal composition passes to the homotopy category. All of the conditions in \cite[Defn~2.9]{shulman_symmetric} can be verified in the same way. 

\begin{prop} (see \cite[Defn~2.9(iv)]{shulman_symmetric}) 
  Let $\mc{M}_1 \in _{\mc{A}_1} \Mod_{\mc{A}_2}$, $\mc{N}_1 \in _{\mc{B}_1} \Mod_{\mc{B}_2}$,  $\mc{M}_2 \in _{\mc{A}_2} \Mod_{\mc{A}_3}$, $\mc{N}_2 \in _{\mc{B}_2} \Mod_{\mc{B}_3}$,  $\mc{M}_3 \in _{\mc{A}_3} \Mod_{\mc{A}_4}$, $\mc{N}_3 \in _{\mc{B}_3} \Mod_{\mc{B}_4}$. Then the following diagram commutes
  \[
  \begin{tikzcd}
    ((\mc{M}_1 \otimes \mc{N}_1) \odot (\mc{M}_2 \otimes \mc{N}_2)) \odot (\mc{M}_3 \otimes \mc{N}_3) \ar{r}{t \odot 1}\ar{d}{\alpha}  & ((\mc{M}_1 \odot \mc{M}_2) \otimes (\mc{N}_1 \odot \mc{N}_2)) \odot (\mc{M}_3 \otimes \mc{N}_3) \ar{d}{t}\\
    (\mc{M}_1 \otimes \mc{N}_1) \odot ((\mc{M}_2 \otimes \mc{N}_2) \odot (\mc{M}_3 \otimes \mc{N}_3)) \ar{d}{1 \odot t}  & ((\mc{M}_1 \odot \mc{M}_2) \odot \mc{M}_3) \otimes ((\mc{N}_1 \odot \mc{N}_2) \odot \mc{N}_3)\ar{d}{\alpha \otimes \alpha}\\
    (\mc{M}_1 \otimes \mc{N}_1) \odot ((\mc{M}_2 \odot \mc{M}_3) \otimes (\mc{N}_2 \odot \mc{N}_3)) \ar{r} & (\mc{M}_1 \odot (\mc{M}_2 \odot \mc{M}_3)) \otimes (\mc{N}_1 \odot (\mc{N}_2 \odot \mc{N}_3))
  \end{tikzcd}
  \]
\end{prop}
\begin{proof}
  We check that these maps commute at each level of the geometric realizations that define the bicategorical composition. They are induced by the following diagram, where all of the maps are appropriate shuffle homomorphisms. 
  \[
  \begin{tikzpicture}[baseline= (a).base]
    \node[scale=.8] (a) at (0,0){
      \begin{tikzcd}
           (\mc{M}_1 \otimes \mc{N}_1 \otimes (\mc{A}_2 \otimes \mc{B}_2)^{\otimes n} \otimes \mc{M}_2 \otimes \mc{N}_2) \otimes (\mc{A}_3 \otimes \mc{B}_3)^{\otimes n}  \otimes (\mc{M}_3 \otimes \mc{N}_3) \ar{dd}{\alpha} \ar{dr}{t \otimes 1}  & [-150pt]
	\\
	 & ((\mc{M}_1 \otimes \mc{A}^{\otimes n}_2 \otimes \mc{M}_2) \otimes (\mc{N}_1 \otimes \mc{B}^{\otimes n}_2 \otimes \mc{N}_2)) \otimes (\mc{A}_3 \otimes \mc{B}_3)^{\otimes n} \otimes (\mc{M}_3 \otimes \mc{N}_3) \ar{dd}
\\
        \mc{M}_1 \otimes \mc{N}_1 \otimes (\mc{A}_2 \otimes \mc{B}_2)^{\otimes n} \otimes (\mc{M}_2 \otimes \mc{N}_2 \otimes (\mc{A}_3 \otimes \mc{B}_3)^{\otimes n} \otimes (\mc{M}_3 \otimes \mc{N}_3)) \ar{dd} 
	\\
& ((\mc{M}_1 \otimes \mc{A}^{\otimes n}_2 \otimes \mc{M}_2) \otimes \mc{A}^{\otimes n}_3 \otimes \mc{M}_3) \otimes ((\mc{N}_1 \otimes \mc{B}^{\otimes n}_2 \otimes \mc{N}_2) \otimes \mc{B}^{\otimes n}_3 \otimes  \mc{N}_3 ) \ar{dd}\\
        \mc{M}_1 \otimes \mc{N}_1 \otimes (\mc{A}_2 \otimes \mc{B}_2)^{\otimes n} \otimes ((\mc{M}_2 \otimes \mc{A}^{\otimes n}_3 \otimes \mc{M}_3 ) \otimes (\mc{N}_2 \otimes \mc{B}^{\otimes n}_3 \otimes  \mc{N}_3)\ar{dr} 
	& 
	\\
         &(\mc{M}_1 \otimes \mc{A}^{\otimes n}_2 \otimes (\mc{M}_2 \otimes \mc{A}^{\otimes n}_3 \otimes \mc{M}_3)) \otimes (\mc{N}_1 \otimes \mc{B}^{\otimes n}_2 \otimes (\mc{N}_2 \otimes \mc{B}^{\otimes n}_3 \otimes  \mc{N}_3))   
      \end{tikzcd}
    };
  \end{tikzpicture}
  \]
  This diagram commutes by assumption. 
\end{proof}

\begin{rmk}
The other verifications are  as tedious and unilluminating as this one and so we leave them to the highly motivated reader.  They are all  consequences of the corresponding statements for $\mc{V}$. 
\end{rmk}

In order to apply \cite[Thm.~1.2]{shulman_symmetric} we need one more concept. Essentially, we need to be able to flip vertial morphisms into horizontal morphisms. 

\begin{defn}\cite[Defn.~3.1]{shulman_symmetric}
  Let $\mc{D}$ be a double category and let $f: A \to B$ be a horizontal morphism. A \textbf{companion} for $f$ is a horizontal morphism $\widehat{f}: A \to B$ together with 2-morphisms 
  \[
  \begin{tikzcd}
    A \arrow[r, mid vert, "\widehat{f}",  ""{name = U, below}] \arrow[d, swap,  "f"] & B \ar[d,equals]\\
    B \arrow[r,mid vert,swap, "U_B", ""{name = D, above}] \arrow[Rightarrow, from = U, to = D] & B
  \end{tikzcd}
  \qquad
  \begin{tikzcd}
    A\ar[d,equals] \arrow[r, mid vert,  "U_A ", ""{name = U, below}] & A\ar{d}{f} \\
    A\ar[r, mid vert,swap, " \widehat{f} ", ""{name = D, above}] \arrow[Rightarrow, from = U, to = D] & B 
  \end{tikzcd}
  \]
  such that the conditions in \cite[Eqn.~3.2]{shulman} hold.  A \textbf{conjoint}  for $f$ is a horizontal morphism $\check{f}: B \to A$ that is a companion for $f$ in $\mc{D}^{hop}$. 
\end{defn}

\begin{defn}\cite[Defn.~3.4]{shulman_symmetric}\label{defn:fibrant_double_category}
A \textbf{fibrant} double category is one where every vertical morphism $f: A \to B$ has a companion $\widehat{f}$ and a conjoint $\check{f}$. 
\end{defn}

\begin{lem}
$\mc{D}(\Cat_{\mc{V}})$ is a fibrant double category. 
\end{lem}
\begin{proof}
The companions and conjoints are the base change objects of  \cref{defn:base_change}. 
\end{proof}

With all of the requirements verified, we can now state the following theorem. 

\begin{thm}\label{main_categorical_construction}
$\mc{D}(\Cat_{\mc{V}})$ is a symmetric monoidal double category, and so its underlying bicategory, $\operatorname{Ho}(\mc{B}(\Cat_{\mc{V}}))$ is symmetric monoidal. 
\end{thm}

The symmetric monoidal bicategory 
$\operatorname{Ho}(\mc{B}(\Cat_{\mc{V}}))$
has a shadow and it coincides with Hochschild homology in cases of interest.

\begin{defn}\label{def:cyclic_bar}
  Let $\mc{C}$ be a very good $\mc{V}$-category and $\mc{M}$ be a $(\mc{C},\mc{C})$-module in $\mc{B}(\Cat_{\mc{V}})$.
	We define $\sh{\mc{M}}$ to be the coend
  \begin{align*}
    \sh{\mc{M}} &:= \mbf{L}\int^{c} \mc{M}(c,c) \\
    &=\biggl| \coprod_{c_0, c_1, \dots, c_n} \mc{C}(c_0,c_1) \otimes \cdots \otimes \mc{C}(c_{n-1},c_n) \otimes \mc{M}(c_n, c_0) \biggr|
  \end{align*}
  This descends to the homotopy category to define a functor \[\sh{-}: \coprod_{\mc{C}} \operatorname{Ho}(\Cat_{\mc{V}}) \to \operatorname{Ho}(\mc{V}).\]
\end{defn}

 The equality follows since a coend is a colimit and derived homotopy colimits are computed via the bar construction. 

\begin{eg}
  If $\mc{M}$ is a spectral $(\mc{C},\mc{C})$-bimodule 
	then 
  \[
  \sh{\mc{M}} \cong \thh(\mc{C}; \mc{M})
  \]

  If $\mc{M}$ is a dg-$(\mc{A},\mc{A})$-bimodule then
  \[
  \sh{\mc{M}} \cong \hh(\mc{A};\mc{M}) 
  \]
\end{eg}

\begin{rmk} In this paper we ignore the $S^1$-equivariance of the Hochschild and topological Hochschild constructions. Dealing with the $S^1$-equivariance requires working exclusively with point-set models rather than homotopy categories. It seems possible that some of the results in this paper could be stated for topological restriction homology if the traces were handled with care. 
\end{rmk}

\bibliographystyle{amsalpha2}
\bibliography{fixed_points}
\end{document}

%% file: II-motivation_dual_pair.tex
    \begin{subfigure}[t]{.3\textwidth}
	\hfill
	\begin{tikzpicture}
	\draw [color = \et, line width= 4pt] (0,0)--(.5,0);
	\fill (0,0) circle[radius=2pt] ;
	\end{tikzpicture}
	\hfill
	\begin{tikzpicture}
	\draw [color = \et, line width= 4pt] (0,0)--(-.5,0);
	\fill (0,0) circle[radius=2pt] ;
	\end{tikzpicture}
	\hspace{1cm}
        \caption{Framed 0-manifolds}\label{fig:famed_zero_mfld}
    \end{subfigure}
    \qquad\qquad 
    \begin{subfigure}[t]{.40\textwidth}
	\resizebox{\textwidth}{!}{
	\begin{tikzpicture}
	\fill[color = \et] (0,0)to [out = -90, in = -90, looseness = 2] (-2,0)--(-2.5,0)to [out = -90, in = -90, looseness = 2] ( .5,0);
	\draw [ line width= 4pt] (0,0)to [out = -90, in = -90, looseness = 2] (-2,0);
	\end{tikzpicture}
		\hspace{1cm}
	\begin{tikzpicture}

	\fill[color = \et] (0,0)--(0,-3)--(.5,-3)-- ( .5,0);
	\draw [ line width= 4pt] (0,0)--(0,-3);
	\end{tikzpicture}
		\hspace{1cm}
	\begin{tikzpicture}

	\fill[color = \et] (0,0)to [out = 90, in = 90, looseness = 2] (-2,0)--(-2.5,0)to [out = 90, in = 90, looseness = 2] ( .5,0);
	\draw [ line width= 4pt] (.5,0)to [out = 90, in = 90, looseness = 2] (-2.5,0);
	\end{tikzpicture}}
        \caption{Framed 1-manifolds}\label{fig:famed_one_mfld}
    \end{subfigure}
    \begin{subfigure}[t]{.5\textwidth}
\resizebox{\textwidth}{!}{
\begin{tikzpicture}

\fill[color = \et] (0,0)to [out = 90, in = 90, looseness = 2] (-2,0)--(-2.5,0)to [out = 90, in = 90, looseness = 2] ( .5,0);
\draw [ line width= 4pt] (.5,0)to [out = 90, in = 90, looseness = 2] (-2.5,0);

\fill[color = \et] (.5,0)to [out = -90, in = -90, looseness = 2] (2.5,0)--(3,0)to [out = -90, in = -90, looseness = 2] (0, 0);
\draw [ line width= 4pt] (.5,0)to [out = -90, in = -90, looseness = 2] (2.5,0);

\fill[color = \et] (3,3)--(3,0)--( 2.5,0)--( 2.5,3);
\draw[line width = 4pt] (2.5,3)--( 2.5,0);

\fill[color = \et] (-2,-3)--(-2,0)--( -2.5,0)--( -2.5,-3);
\draw[line width = 4pt] ( -2.5,-3)--( -2.5,0);

\draw[dotted, line width= 2pt, white] (0,0)--(.5,0);
\draw[dotted, line width= 2pt, white] (3,0)--(2.5,0);
\draw[dotted, line width= 2pt, white] (-2,0)--(-2.5,0);
\end{tikzpicture}
\hspace{.5cm}
\begin{tikzpicture}

\fill[color = \et] (2,0)to [out = 90, in = 90, looseness = 2] (0,0)--(-.5,0)to [out = 90, in = 90, looseness = 2] ( 2.5,0);
\draw [ line width= 4pt] (2.5,0)to [out = 90, in = 90, looseness = 2] (-.5,0);

\fill[color = \et,] (-2.5,0)to [out = -90, in = -90, looseness = 2] (-.5,0)--(0,0)to [out = -90, in = -90, looseness = 2] ( -3,0);
\draw [ line width= 4pt] (-2.5,0)to [out = -90, in = -90, looseness = 2] (-.5,0);

\fill[color =\et] (-3,3)--(-3,0)--( -2.5,0)--( -2.5,3);
\draw[line width = 4pt] ( -2.5,3)--(-2.5,0);

\fill[color = \et] (2,-3)--(2,0)--( 2.5,0)--( 2.5,-3);
\draw[line width = 4pt] ( 2.5,-3)--( 2.5,0);

\draw[dotted, line width= 2pt, white] (2.5,0)--(2,0);
\draw[dotted, line width= 2pt, white] (0,0)--(-.5,0);
\draw[dotted, line width= 2pt, white] (-3,0)--(-2.5,0);
\end{tikzpicture}
}
\caption{Triangle diagrams}\label{fig:motivation_triangle}
\end{subfigure}
\qquad\qquad
\begin{subfigure}[t]{.3\textwidth}
\begin{center}
\resizebox{.45\textwidth}{!}{
\begin{tikzpicture}

\fill[color = \et,] (2,0)to [out = 90, in = 90, looseness = 2] (0,0)--(-.5,0)to [out = 90, in = 90, looseness = 2] (2.5,0);
\draw [ line width= 4pt] (2.5,0)to [out = 90, in = 90, looseness = 2] (-.5,0);

\fill[color = \et] (2,-4)to [out =-90, in = -90, looseness = 2] (0,-4)--(-.5,-4)to [out = -90, in = -90, looseness = 2] (2.5,-4);
\draw [ line width= 4pt] (2,-4)to [out = -90, in = -90, looseness = 2] (0,-4);

\fill[color = \et] (2.5,0)to [out = -90, in =90, looseness =.5](0,-4)--( -.5,-4)to [out = 90, in =-90, looseness =.5]( 2,0);
\draw[line width = 4pt] ( 2.5,0)to [out = -90, in =90, looseness =.5]( 0,-4);

\fill[color =\et,] (2.5,-4)to [out = 90, in =-90, looseness =.5](0,0)--( -.5,0)to [out =- 90, in =90, looseness =.5](2,-4);
\draw[line width = 4pt] (2,-4)to [out = 90, in =-90, looseness =.5]( -.5,0);

\draw[dotted, line width= 2pt, white] (2.5,0)--(2,0);
\draw[dotted, line width= 2pt, white] (0,0)--(-.5,0);
\draw[dotted, line width= 2pt, white] (2.5,-4)--(2,-4);
\draw[dotted, line width= 2pt, white] (0,-4)--(-.5,-4);
\end{tikzpicture}}
\end{center}
       \caption{Euler characteristic}\label{fig:motivation_euler_char}
    \end{subfigure}

%% file: II-framed_small_mflds.tex
    \begin{subfigure}[t]{.30\textwidth}

\begin{centering}
\resizebox{!}{.3\textheight}{
\begin{tikzpicture}

\fill[color = \et] (2,0)to [out = 90, in = 90, looseness = 2] (0,0)--(-.5,0)to [out = 90, in = 90, looseness = 2] ( 2.5,0);
\draw [ line width= 4pt] (2.5,0)to [out = 90, in = 90, looseness = 2] (-.5,0);

\fill[color = \et] (2,-6)to [out = -90, in = -90, looseness = 2] (0,-6)--(-.5,-6)to [out = -90, in = -90, looseness = 2] (2.5,-6);
\draw [ line width= 4pt] (2,-6)to [out = -90, in = -90, looseness = 2] (0,-6);

\fill[color = \et] (2.5,-2)to [out =-90, in =90, looseness =.5](0,-6)--( -.5,-6)to [out = 90, in =-90, looseness =.5]( 2,-2);
\draw[line width = 4pt] ( 2.5,-2)to [out=-90, in =90, looseness =.5](0,-6);

\fill[color = \et] (2.5,-6)to [out = 90, in =-90, looseness =.5](0,-2)--( -.5,-2)to [out = -90, in =90, looseness =.5]( 2,-6);
\draw[line width = 4pt] ( 2,-6)to [out = 90, in =-90, looseness =.5]( -.5,-2);

\fill[color = \ep] (0,0)--(0,-2)--(-.5,-2)--( -.5,0);
\draw[line width = 4pt] (-.5,0)--( -.5, -2);

\fill[color = \et] (2,0)--(2,-2)--( 2.5,-2)--( 2.5,0);
\draw[line width = 4pt] (2.5,0)--( 2.5,-2);

\draw[dotted, line width= 2pt, white] (2.5,0)--(2,0);
\draw[dotted, line width= 2pt, white] (0,0)--(-.5,0);
\draw[dotted, line width= 2pt, white] (2.5,-2)--(2,-2);
\draw[dotted, line width= 2pt, white] (0,-2)--(-.5,-2);
\draw[dotted, line width= 2pt, white] (2.5,-6)--(2,-6);
\draw[dotted, line width= 2pt, white] (0,-6)--(-.5,-6);
\end{tikzpicture}}

\end{centering}
       \caption{Symmetric monoidal trace}\label{fig:motivation_trace_1}
\end{subfigure}


%% file: II-motivation_iterated_traces.tex
\tdplotsetmaincoords{90}{90}
\tdplotsetrotatedcoords{0}{10}{80}


\def\r{3.5}
\def\sr{2}
\def\tr{1}
\def\h{-2}
\def\inc{-40}
\def\startangle{-30}

\def\hf{4.75}
\def\hg{6.75}
\def\hh{7.5}

\def\ax{\r*cos(\startangle+0*\inc))}
\def\ay{\r*sin(\startangle+0*\inc))}

\def\bx{\r*cos(\startangle+.4*\inc)}
\def\by{\r*sin(\startangle+.4*\inc)}

\def\cx{\r*cos(\startangle+1.3*\inc)}
\def\cy{\r*sin(\startangle+1.3*\inc)}

\def\dx{\r*cos(\startangle+1.55*\inc)}
\def\dy{\r*sin(\startangle+1.55*\inc)}

\def\ex{\r-\tr+\tr*cos(\startangle+.75*\inc)}
\def\ey{\tr*sin(\startangle+.75*\inc)}
\def\exa{-\r+\tr+\tr*cos(\startangle+.75*\inc)}

\def\fx{\r-\tr+\tr*cos(\startangle+1.75*\inc)}
\def\fy{\tr*sin(\startangle+1.75*\inc)}
\def\fxa{-\r+\tr+\tr*cos(\startangle+1.75*\inc)}

\def\gx{-\r+\sr+\sr*cos(\startangle+0*\inc)}
\def\gy{\sr*sin(\startangle+0*\inc)}
\def\gxa{\r-\sr+\sr*cos(\startangle+0*\inc)}

\def\hx{-\r+\sr+\sr*cos(\startangle+.5*\inc)}
\def\hy{\sr*sin(\startangle+.5*\inc)}
\def\hxa{\r-\sr+\sr*cos(\startangle+.5*\inc)}

\def\ix{-\r+\sr+\sr*cos(\startangle+1*\inc)}
\def\iy{\sr*sin(\startangle+1*\inc)}
\def\ixa{\r-\sr+\sr*cos(\startangle+1*\inc)}

\def\jx{-\r+\sr+\sr*cos(\startangle+1.5*\inc)}
\def\jy{\sr*sin(\startangle+1.5*\inc)}
\def\jxa{\r-\sr+\sr*cos(\startangle+1.5*\inc)}

\def\kx{-\r+\sr+\sr*cos(\startangle+2*\inc)}
\def\ky{\sr*sin(\startangle+2*\inc)}
\def\kxa{\r-\sr+\sr*cos(\startangle+2*\inc)}

\def\mx{-\r+\sr+\sr*cos(\startangle+.75*\inc)}
\def\my{\sr*sin(\startangle+1.75*\inc)}
\def\mxa{\r-\sr+\sr*cos(\startangle+.75*\inc)}

\def\lx{-\r+\sr+\sr*cos(\startangle+2.5*\inc)}
\def\ly{\sr*sin(\startangle+2.5*\inc)}
\def\lxa{\r-\sr+\sr*cos(\startangle+2.5*\inc)}

\def\px{-\r+\tr+\tr*cos(\startangle+.8*\inc)}
\def\py{\tr*sin(\startangle+.8*\inc)}
\def\pxa{\r-\tr+\tr*cos(\startangle+.8*\inc)}

\def\qx{-\r+\tr+\tr*cos(\startangle+1.8*\inc)}
\def\qy{\tr*sin(\startangle+1.8*\inc)}
\def\qxa{\r-\tr+\tr*cos(\startangle+1.8*\inc)}

\def\rx{\r-\sr+\sr*cos(\startangle+2*\inc)}
\def\ry{\sr*sin(\startangle+.3*\inc)}
\def\rxa{-\r\sr+\sr*cos(\startangle+2*\inc)}

\def\sx{\r-\sr+\sr*cos(\startangle+2.5*\inc)}
\def\sy{\sr*sin(\startangle+.9*\inc)}
\def\sxa{-\r+\sr+\sr*cos(\startangle+2.5*\inc)}

\def\tx{\r*cos(\startangle+1*\inc))}
\def\ty{\r*sin(\startangle+1*\inc))}

\def\ux{\r*cos(\startangle+1.3*\inc)}
\def\uy{\r*sin(\startangle+1.3*\inc)}

\def\vx{\r*cos(\startangle+2.2*\inc)}
\def\vy{\r*sin(\startangle+2.2*\inc)}

\def\wx{\r*cos(\startangle+2.6*\inc)}
\def\wy{\r*sin(\startangle+2.6*\inc)}
\begin{centering}
\resizebox{!}{.67\textheight}{\begin{tikzpicture} [tdplot_rotated_coords,
    mystyle/.style={%
      label={right:\pgfkeysvalueof{/pgf/minimum width}},
    },
   my style/.style={%
     append after command={
       \pgfextra{\node [right] at (\tikzlastnode.mid east) {\tikzlastnode};}
     },
   },
  ]\draw[dotted] (0,0,-.5*\h) circle(\r);

\coordinate
	(A1) at  ({\ax},{\ay},-.5*\h) ;

\coordinate
	(A2) at ({\bx},{\by},-.5*\h);
\coordinate
	(A3) at ({\cx},{\cy},-.5*\h);
\coordinate
	(A4) at ({\dx},{\dy},-.5*\h);

\draw[dotted] (\r-\tr,0,.25*\h) circle(\tr);
\draw[dotted] (-\r+\sr,0,.25*\h) circle(\sr);
\coordinate
	(C1) at ({\ex},{\ey},.25*\h) ;
\coordinate
	 (C2) at ({\fx},{\fy},.25*\h) ;

\coordinate
	(C7) at ({\gx},{\gy},.25*\h) ;
\coordinate
	(C8) at ({\hx},{\hy},.25*\h) ;

\draw[dotted] (\r-\tr,0,1*\h) circle(\tr);
\draw[dotted] (-\r+\sr,0,1*\h) circle(\sr);
\coordinate
	(B1) at ({\ex},{\ey},1*\h) ;
\coordinate
	(B2) at ({\fx},{\fy},1*\h) ;

\coordinate
	(B3) at ({\lx},{\ly},1*\h);
\coordinate
	 (B4) at ({\kx},{\ky},1*\h) ;
\coordinate
	(B5) at ({\jx},{\jy},1*\h);
\coordinate
	(B6) at ({\ix},{\iy},1*\h) ;
\coordinate
	(B7) at ({\gx},{\gy},1*\h) ;
\coordinate
	(B8) at ({\hx},{\hy},1*\h);

\draw [dotted](\r-\tr,0,2.25*\h) circle(\tr);
\draw[dotted] (-\r+\sr,0,2.25*\h) circle(\sr);
\coordinate
	(D1) at ({\ex},{\ey},2.25*\h) ;
\coordinate
	 (D2) at ({\fx},{\fy},2.25*\h) ;

\coordinate 
	(D3) at ({\lx},{\ly},2.25*\h);
\coordinate
	(D4) at ({\kx},{\ky},2.25*\h) ;
\coordinate
	 (D5) at ({\ix},{\iy},2.25*\h);
\coordinate
	 (D6) at ({\jx},{\jy},2.25*\h) ;
\coordinate
	(D7) at ({\hx},{\hy},2.25*\h) ;
\coordinate
	 (D8) at ({\gx},{\gy},2.25*\h);

\node[draw, fill =white, circle] at  ({\mx},{\my},1.625*\h) {$f$};

\draw[dotted] (\r-\tr,0,4*\h) circle(\tr);
\draw[dotted] (-\r+\sr,0,4*\h) circle(\sr);
\coordinate
	(E1) at ({\ex},{\ey},4*\h) ;
\coordinate
	(E2) at ({\fx},{\fy},4*\h) ;

\coordinate
	(E3) at ({\kx},{\ky},4*\h);
\coordinate
	(E4) at ({\lx},{\ly},4*\h) ;
\coordinate
	(E5) at ({\jx},{\jy},4*\h);
\coordinate
	 (E6) at ({\ix},{\iy},4*\h) ;
\coordinate
	 (E7) at ({\hx},{\hy},4*\h) ;
\coordinate
	 (E8) at ({\gx},{\gy},4*\h);

\draw[dotted] (\r-\tr,0,\hf*\h) circle(\tr);
\draw[dotted] (-\r+\sr,0,\hf*\h) circle(\sr);
\coordinate
	(F1) at ({\ex},{\ey},\hf*\h) ;
\coordinate
	(F2) at ({\fx},{\fy},\hf*\h) ;

\coordinate
	(F3) at ({\kx},{\ky},\hf*\h);
\coordinate
	(F4) at ({\lx},{\ly},\hf*\h) ;

\draw[dotted] (-\r+\tr,0,\hg*\h) circle(\tr);
\draw[dotted] (\r-\sr,0,\hg*\h) circle(\sr);
\coordinate
	(H1) at ({\px},{\py},\hg*\h) ;
\coordinate
	(H2) at ({\qx},{\qy},\hg*\h) ;

\coordinate
	(H3) at ({\rx},{\ry},\hg*\h);
\coordinate
	(H4) at ({\sx},{\sy},\hg*\h) ;

\draw[dotted] (0,0,\hh*\h) circle(\r);

\coordinate
	(G1) at  ({\tx},{\ty},\hh*\h) ;
\coordinate
	(G2) at ({\ux},{\uy},\hh*\h);
\coordinate
	(G3) at ({\vx},{\vy},\hh*\h);
\coordinate
	(G4) at ({\wx},{\wy},\hh*\h);

\coordinate
	 (co1) at ({-\r+\sr+\sr*cos(15)},{\sr*sin(15)},.25*\h) ;
\coordinate
	(co2) at ({\r-\tr+\tr*cos(195)},{\tr*sin(195)},.25*\h) ;
\coordinate
	(co3) at ({\r-\tr+\tr*cos(195)},{\tr*sin(195)},\hf*\h) ;
\coordinate
	(co4) at ({-\r+\sr+\sr*cos(15)},{\sr*sin(15)},\hf*\h) ;
\coordinate
	(co3c) at ({\r-\tr+\tr*cos(15)},{\tr*sin(15)},\hf*\h) ;
\coordinate
	(co4c) at ({-\r+\sr+\sr*cos(195)},{\sr*sin(195)},\hf*\h) ;
\coordinate
	(co3a) at ({\r-\sr+\sr*cos(195)},{\sr*sin(195)},\hg*\h) ;
\coordinate
	(co4a) at ({-\r+\tr+\tr*cos(15)},{\tr*sin(15)},\hg*\h) ;
\coordinate
	(co3b) at ({\r-\sr+\sr*cos(15)},{\sr*sin(15)},\hg*\h) ;
\coordinate
	(co4b) at ({-\r+\tr+\tr*cos(195)},{\tr*sin(195)},\hg*\h) ;

\coordinate
	 (co5) at ({\r*cos(15)},{\r*sin(15)},-.5*\h) ;
\coordinate
	 (co6) at ({\r*cos(15)},{\r*sin(15)},\hh*\h) ;
\coordinate
	 (co7) at ({\r*cos(195)},{\r*sin(195)},\hh*\h) ;
\coordinate
	 (co8) at ({\r*cos(195)},{\r*sin(195)},-.5*\h) ;

\coordinate
	(tw1) at ({-\r+\sr+\sr*cos(195)},{\sr*sin(195)},2.6*\h) ;
\coordinate
	(tw2) at ({-\r+\sr+\sr*cos(15)},{\tr*sin(15)},3.25*\h) ;
\coordinate
	(tw3) at ({-\r+\sr+\sr*cos(195)},{\tr*sin(195)},2.85*\h) ;
\coordinate
	 (tw4) at ({-\r+\sr+\sr*cos(15)},{\sr*sin(15)},3.5*\h) ;

\begin{scope}[on background layer]

\fill[black!10!white,opacity =.8](co5)--(co3c)--(co3)--(co2)to [out = 90, in =90,looseness = 2] (co1) --(co4)--(co4c)  --(co8)to [out = 90, in =90,looseness = 1] (co5);

\fill[black!10!white,opacity =.8](co3b)--(co6)  to [out = -90, in =-90,looseness = 1.5](co7)--(co4b)--(co4a) to [out = -90, in =-90,looseness = 2] (co3a);

\fill[black!10!white,opacity =.8] (co3c)to [out = -90, in =90, looseness =.5](co4a)--(co4b)to [out = 90, in =-90, looseness =.5](co3);

\fill[\et] (B8) -- (C8) to [out = 90, in = -90, looseness =.5]
(A4) to [out = 90, in =90, looseness = 1] (A1)to [out = -90, in = 90, looseness =.75](C1)--(B1)--(D1)--(E1)--(F1)to [out = -90, in = 90, looseness =.5](H1)--
(H2)to [out = 90, in = -90, looseness =.5](F2)
--(E2)--(D2)--(B2)--(C2)to [out = 90, in = -90, looseness =.75](A2)to [out = 90, in =90, looseness = 1] (A3)to [out = -90, in = 90, looseness =.75](C7)--(B7);


\draw[very thick]
(H2)to [out = 90, in = -90, looseness =.5](F2);

\fill[black!10!white,opacity =.8] (co4c)to [out = -90, in =90, looseness =.75](co3a)--(co3b)to [out = 90, in =-90, looseness =.75](co4);

\fill[\ep!5!white] (tw1) to [ out =180, in =0,looseness =1](tw2)--(tw4)to [ out =0, in =180,looseness =1](tw3);

\fill[\et](G2)to [out = 90, in =-90, looseness = 1]
	(H4)to [out = 90, in = -90, looseness =.75]
	(F4)-- (E4)to [out = 90, in = -90, looseness =.75](D6) to [out = 90, in = -90, looseness =.75](B8)--(C8)--(C7)--(B7)to [out = -90, in = 90, looseness =.75](D5)to [out = -90, in = 90, looseness =.75](E3) --(F3)to [out = -90, in = 90, looseness =.75](H3)to [out = -90, in =90, looseness = 1](G1) to [out = -90, in =-90, looseness = 1](G4)to [out = 90, in =-90, looseness = 1](H2)--(H1)to [out = -90, in =90, looseness = 1](G3)to [out = -90, in =-90, looseness = 1](G2);

\fill[\ep] (E6)to [out = 90, in = -90, looseness =.75](D8)to [out = 90, in = -90, looseness =.75] (B6)to [out = 90, in =90, looseness = 2] (B3)--(D3)to [out = -90, in = 0, looseness =.75](tw1)--(tw3)to [out = 0, in = -90, looseness =.75](D4)--(B4)to [out = 90, in =90, looseness = 2] (B5)to [out = -90, in = 90, looseness =.75](D7)to [out = -90, in = 90, looseness =.75] (E5)to [out = -90, in =-90, looseness = 2] (E8)to [out = 90, in = 180, looseness =.75](tw4)--(tw2)to [out = 180, in = 90, looseness =.75](E7)to [out = -90, in =-90, looseness = 2] (E6);

\draw[very thick](C7)--(B7)to [out = -90, in = 90, looseness =.75](D5)to [out = -90, in = 90, looseness =.75](E3) --
(F3);
\draw[very thick](F3)to [out = -90, in = 90, looseness =.75](H3);
\draw[very thick](H3)to [out = -90, in = 90, looseness =.75](G1)to [out = -90, in =-90, looseness = 1](G4)to [out = 90, in = -90, looseness =.75](H2);
\draw[very thick](F2)
--(E2)--(D2)--(B2)--(C2)to [out = 90, in = -90, looseness =.75](A2)to [out = 90, in =90, looseness = 1] (A3)to [out = -90, in = 90, looseness =.75](C7)--(B7);

\draw[very thick] (tw3)to [out = 0, in = -90, looseness =.75](D4)--(B4)to [out = 90, in =90, looseness = 2] (B5)to [out = -90, in = 90, looseness =.75](D7)to [out = -90, in = 90, looseness =.75] (E5)to [out = -90, in =-90, looseness = 2] (E8)to [out = 90, in = 180, looseness =.75](tw4);

\draw[dotted, line width = .5pt] (co7)--(co4b)to [out = 90, in =-90, looseness =.5](co3)--(co2)to [out = 90, in =90,looseness = 2] (co1) --(co4)to [out = -90, in =90, looseness =.75] (co3b)--(co6) to [out = -90, in =-90,looseness = 1.5] (co7);

\draw[dotted, line width = .5pt] (co5)--(co3c)to [out = -90, in =90, looseness =.5](co4a)to [out = -90, in =-90,looseness = 2] (co3a)to [out = 90, in =-90, looseness =.75](co4c)  -- (co8)to [out = 90, in =90,looseness = 1] (co5);

\end{scope}

\end{tikzpicture}
}
\end{centering}

%% file: II-motivation-bicat_trace.tex
\tdplotsetmaincoords{90}{90}
\tdplotsetrotatedcoords{0}{10}{80}

\def\r{3.5}
\def\sr{2}
\def\tr{1}
\def\h{-2}
\def\inc{-40}
\def\startangle{-30}

\def\ax{\r*cos(\startangle+0*\inc))}
\def\ay{\r*sin(\startangle+0*\inc))}

\def\bx{\r*cos(\startangle+.4*\inc)}
\def\by{\r*sin(\startangle+.4*\inc)}

\def\cx{\r*cos(\startangle+2.35*\inc)}
\def\cy{\r*sin(\startangle+2.35*\inc)}

\def\dx{\r*cos(\startangle+2.75*\inc)}
\def\dy{\r*sin(\startangle+2.75*\inc)}

\def\ex{\r-\tr+\tr*cos(\startangle+.75*\inc)}
\def\ey{\sr*sin(\startangle+.75*\inc)}

\def\fx{\r-\tr+\tr*cos(\startangle+1.75*\inc)}
\def\fy{\sr*sin(\startangle+1.75*\inc)}

\def\gx{-\r+\sr+\sr*cos(\startangle+0*\inc)}
\def\gy{\sr*sin(\startangle+0*\inc)}

\def\hx{-\r+\sr+\sr*cos(\startangle+.5*\inc)}
\def\hy{\sr*sin(\startangle+.5*\inc)}

\def\ix{-\r+\sr+\sr*cos(\startangle+1*\inc)}
\def\iy{\sr*sin(\startangle+1*\inc)}

\def\jx{-\r+\sr+\sr*cos(\startangle+1.5*\inc)}
\def\jy{\sr*sin(\startangle+1.5*\inc)}

\def\kx{-\r+\sr+\sr*cos(\startangle+2*\inc)}
\def\ky{\sr*sin(\startangle+2*\inc)}

\def\mx{-\r+\sr+\sr*cos(\startangle+1.75*\inc)}
\def\my{\sr*sin(\startangle+1.75*\inc)}

\def\lx{-\r+\sr+\sr*cos(\startangle+2.5*\inc)}
\def\ly{\sr*sin(\startangle+2.5*\inc)}

\def\px{-\r+\tr+\tr*cos(\startangle+.8*\inc)}
\def\py{\sr*sin(\startangle+.8*\inc)}

\def\qx{-\r+\tr+\tr*cos(\startangle+1.8*\inc)}
\def\qy{\sr*sin(\startangle+1.8*\inc)}

\def\rx{\r-\sr+\sr*cos(\startangle+.3*\inc)}
\def\ry{\sr*sin(\startangle+.3*\inc)}

\def\sx{\r-\sr+\sr*cos(\startangle+.9*\inc)}
\def\sy{\sr*sin(\startangle+.9*\inc)}

\resizebox{!}{.3\textheight}{
\begin{tikzpicture} [tdplot_rotated_coords,
    mystyle/.style={%
      label={right:\pgfkeysvalueof{/pgf/minimum width}},
    },
   my style/.style={%
     append after command={
       \pgfextra{\node [right] at (\tikzlastnode.mid east) {\tikzlastnode};}
     },
   },
  ]

\draw[dotted] (-\r+\sr,0,0*\h) circle(\sr);

\coordinate
	(C7) at ({\kx},{\ky},0*\h) ;
\coordinate
	(C8) at ({\lx},{\ly},0*\h) ;

\draw[dotted] (-\r+\sr,0,1*\h) circle(\sr);

\coordinate
	(B3) at ({\gx},{\gy},1*\h);
\coordinate
	 (B4) at ({\hx},{\hy},1*\h) ;
\coordinate
	(B5) at ({\ix},{\iy},1*\h);
\coordinate
	(B6) at ({\jx},{\jy},1*\h) ;
\coordinate
	(B7) at ({\kx},{\ky},1*\h) ;
\coordinate
	(B8) at ({\lx},{\ly},1*\h);

\draw[dotted] (-\r+\sr,0,2.25*\h) circle(\sr);

\coordinate 
	(D3) at ({\gx},{\gy},2.25*\h);
\coordinate
	(D4) at ({\hx},{\hy},2.25*\h) ;
\coordinate
	 (D5) at ({\ix},{\iy},2.25*\h);
\coordinate
	 (D6) at ({\jx},{\jy},2.25*\h) ;
\coordinate
	(D7) at ({\kx},{\ky},2.25*\h) ;
\coordinate
	 (D8) at ({\lx},{\ly},2.25*\h);

\node[draw, fill =white, circle] at  ({\mx},{\my},1.625*\h) {$f$};

\draw[dotted] (-\r+\sr,0,4*\h) circle(\sr);

\coordinate
	(E3) at ({\gx},{\gy},4*\h);
\coordinate
	(E4) at ({\hx},{\hy},4*\h) ;
\coordinate
	(E5) at ({\ix},{\iy},4*\h);
\coordinate
	 (E6) at ({\jx},{\jy},4*\h) ;
\coordinate
	 (E7) at ({\kx},{\ky},4*\h) ;
\coordinate
	 (E8) at ({\lx},{\ly},4*\h);

\draw[dotted] (-\r+\sr,0,5*\h) circle(\sr);
\coordinate
	(F3) at ({\gx},{\gy},5*\h);
\coordinate
	(F4) at ({\hx},{\hy},5*\h) ;

\coordinate
	 (co1) at ({-\r+\sr+\sr*cos(15)},{\sr*sin(15)},0*\h) ;
\coordinate
	(co4) at ({-\r+\sr+\sr*cos(15)},{\sr*sin(15)},5*\h) ;

\coordinate
	 (co7) at ({\r*cos(195)},{\r*sin(195)},5*\h) ;
\coordinate
	 (co8) at ({\r*cos(195)},{\r*sin(195)},0*\h) ;

\coordinate
	(tw1) at ({-\r+\sr+\sr*cos(15)},{\sr*sin(15)},2.6*\h) ;
\coordinate
	(tw2) at ({-\r+\sr+\sr*cos(195)},{\tr*sin(195)},3.25*\h) ;
\coordinate
	(tw3) at ({-\r+\sr+\sr*cos(15)},{\tr*sin(15)},2.85*\h) ;
\coordinate
	 (tw4) at ({-\r+\sr+\sr*cos(195)},{\sr*sin(195)},3.5*\h) ;

\begin{scope}[on background layer]
\fill[black!10!white](co1)--(co4)--
(co7)--(co8);

\fill [black!5!white] (-\r+\sr,0,5*\h) circle(\sr);
\fill[\et!5!white] (tw1) to [ out =180, in =0,looseness =1](tw2)--(tw4)to [ out =0, in =180,looseness =1](tw3);

\fill[\ep] (B8) -- (C8) --(C7)--(B7);

\fill[\ep](F4)-- (E4)to [out = 90, in = -90, looseness =.75](D6) to [out = 90, in = -90, looseness =.75](B8)--(C8)--(C7)--(B7)to [out = -90, in = 90, looseness =.75](D5)to [out = -90, in = 90, looseness =.75](E3) --(F3);

\draw[very thick](C7)--(B7)to [out = -90, in = 90, looseness =.75](D5)to [out = -90, in = 90, looseness =.75](E3) --(F3);

\fill[\et] (E6)to [out = 90, in = -90, looseness =.75](D8)to [out = 90, in = -90, looseness =.75] (B6)to [out = 90, in =90, looseness = 2] (B3)--(D3)to [out = -90, in = 180, looseness =.75](tw1)--(tw3)to [out = 180, in = -90, looseness =.75](D4)--(B4)to [out = 90, in =90, looseness = 2] (B5)to [out = -90, in = 90, looseness =.75](D7)to [out = -90, in = 90, looseness =.75] (E5)to [out = -90, in =-90, looseness = 2] (E8)to [out = 90, in = 0, looseness =.75](tw4)--(tw2)to [out = 0, in = 90, looseness =.75](E7)to [out = -90, in =-90, looseness = 2] (E6);

\draw[very thick] (tw3)to [out = 180, in = -90, looseness =.75](D4)--(B4)to [out = 90, in =90, looseness = 2] (B5)to [out = -90, in = 90, looseness =.75](D7)to [out = -90, in = 90, looseness =.75] (E5)to [out = -90, in =-90, looseness = 2] (E8)to [out = 90, in = 0, looseness =.75](tw4);

\fill [black!20!white] (-\r+\sr,0,0*\h) circle(\sr);

\end{scope}

\end{tikzpicture}
}


%% file: II-motivation_iterated_traces_main_thm.tex
\tdplotsetmaincoords{90}{90}
\tdplotsetrotatedcoords{0}{10}{80}


\def\r{3.5}
\def\sr{2}
\def\tr{1}
\def\h{-2}
\def\inc{-40}
\def\startangle{-30}

\def\hf{4.75}
\def\hg{6.75}
\def\hh{7.5}

\def\ax{\r*cos(\startangle+0*\inc))}
\def\ay{\r*sin(\startangle+0*\inc))}
\def\axa{\r*cos(\startangle+2.5*\inc))}
\def\aya{\r*sin(\startangle+2.5*\inc))}

\def\bx{\r*cos(\startangle+.4*\inc)}
\def\by{\r*sin(\startangle+.4*\inc)}
\def\bxa{\r*cos(\startangle+2.1*\inc)}
\def\bya{\r*sin(\startangle+2.1*\inc)}

\def\cx{\r*cos(\startangle+1.3*\inc)}
\def\cy{\r*sin(\startangle+1.3*\inc)}
\def\cxa{\r*cos(\startangle+1.3*\inc)}
\def\cya{\r*sin(\startangle+1.3*\inc)}

\def\dx{\r*cos(\startangle+1.55*\inc)}
\def\dy{\r*sin(\startangle+1.55*\inc)}
\def\dxa{\r*cos(\startangle+1.05*\inc)}
\def\dya{\r*sin(\startangle+1.05*\inc)}

\def\ex{\r-\tr+\tr*cos(\startangle+.75*\inc)}
\def\ey{\tr*sin(\startangle+.75*\inc)}
\def\exa{-\r+\tr+\tr*cos(\startangle+1.75*\inc)}
\def\eya{\tr*sin(\startangle+1.75*\inc)}

\def\fx{\r-\tr+\tr*cos(\startangle+1.75*\inc)}
\def\fy{\tr*sin(\startangle+1.75*\inc)}
\def\fxa{-\r+\tr+\tr*cos(\startangle+.75*\inc)}
\def\fya{\tr*sin(\startangle+75*\inc)}

\def\gx{-\r+\sr+\sr*cos(\startangle+0*\inc)}
\def\gy{\sr*sin(\startangle+0*\inc)}
\def\gxa{\r-\sr+\sr*cos(\startangle+.6*\inc)}
\def\gya{\sr*sin(\startangle+.6*\inc)}

\def\hx{-\r+\sr+\sr*cos(\startangle+.5*\inc)}
\def\hy{\sr*sin(\startangle+.5*\inc)}
\def\hxa{\r-\sr+\sr*cos(\startangle+0*\inc)}
\def\hya{\sr*sin(\startangle+.0\inc)}

\def\ix{-\r+\sr+\sr*cos(\startangle+1*\inc)}
\def\iy{\sr*sin(\startangle+1*\inc)}
\def\ixa{\r-\sr+\sr*cos(\startangle+1*\inc)}
\def\iya{\sr*sin(\startangle+1*\inc)}

\def\jx{-\r+\sr+\sr*cos(\startangle+1.5*\inc)}
\def\jy{\sr*sin(\startangle+1.5*\inc)}
\def\jxa{\r-\sr+\sr*cos(\startangle+1.5*\inc)}
\def\jya{\sr*sin(\startangle+1.5*\inc)}

\def\kx{-\r+\sr+\sr*cos(\startangle+2*\inc)}
\def\ky{\sr*sin(\startangle+2*\inc)}
\def\kxa{\r-\sr+\sr*cos(\startangle+2*\inc)}
\def\kya{\sr*sin(\startangle+2*\inc)}

\def\mx{-\r+\sr+\sr*cos(\startangle+.75*\inc)}
\def\my{\sr*sin(\startangle+.75*\inc)}
\def\mxa{\r-\sr+\sr*cos(\startangle+1.75*\inc)}
\def\mya{\sr*sin(\startangle+1.75*\inc)}

\def\lx{-\r+\sr+\sr*cos(\startangle+2.5*\inc)}
\def\ly{\sr*sin(\startangle+2.5*\inc)}
\def\lxa{\r-\sr+\sr*cos(\startangle+2.5*\inc)}
\def\lya{\sr*sin(\startangle+2.5*\inc)}

\def\px{-\r+\tr+\tr*cos(\startangle+.8*\inc)}
\def\py{\tr*sin(\startangle+.8*\inc)}
\def\pxa{\r-\tr+\tr*cos(\startangle+.8*\inc)}
\def\pya{\tr*sin(\startangle+.8*\inc)}

\def\qx{-\r+\tr+\tr*cos(\startangle+1.8*\inc)}
\def\qy{\tr*sin(\startangle+1.8*\inc)}
\def\qxa{\r-\tr+\tr*cos(\startangle+1.8*\inc)}
\def\qya{\tr*sin(\startangle+1.8*\inc)}

\def\rx{\r-\sr+\sr*cos(\startangle+2*\inc)}
\def\ry{\sr*sin(\startangle+2*\inc)}
\def\rxa{-\r+\sr+\sr*cos(\startangle+.5*\inc)}
\def\rya{\sr*sin(\startangle+.5*\inc)}

\def\sx{\r-\sr+\sr*cos(\startangle+2.5*\inc)}
\def\sy{\sr*sin(\startangle+2.5*\inc)}
\def\sxa{-\r+\sr+\sr*cos(\startangle+1*\inc)}
\def\sya{\sr*sin(\startangle+1*\inc)}

\def\tx{\r*cos(\startangle+1*\inc))}
\def\ty{\r*sin(\startangle+1*\inc))}
\def\txa{\r*cos(\startangle+0*\inc))}
\def\tya{\r*sin(\startangle+0*\inc))}

\def\ux{\r*cos(\startangle+1.3*\inc)}
\def\uy{\r*sin(\startangle+1.3*\inc)}
\def\uxa{\r*cos(\startangle+.3*\inc)}
\def\uya{\r*sin(\startangle+.3*\inc)}

\def\vx{\r*cos(\startangle+2.2*\inc)}
\def\vy{\r*sin(\startangle+2.2*\inc)}
\def\vxa{\r*cos(\startangle+1.3*\inc)}
\def\vya{\r*sin(\startangle+1.3*\inc)}

\def\wx{\r*cos(\startangle+2.6*\inc)}
\def\wy{\r*sin(\startangle+2.6*\inc)}
\def\wxa{\r*cos(\startangle+1.6*\inc)}
\def\wya{\r*sin(\startangle+1.6*\inc)}

  \begin{subfigure}[t]{.20\textwidth}
\begin{centering}
\resizebox{\textwidth}{!}{
\begin{tikzpicture} [tdplot_rotated_coords,
    mystyle/.style={%
      label={right:\pgfkeysvalueof{/pgf/minimum width}},
    },
   my style/.style={%
     append after command={
       \pgfextra{\node [right] at (\tikzlastnode.mid east) {\tikzlastnode};}
     },
   },
  ]
\draw[dotted] (0,0,-.5*\h) circle(\r);

\coordinate
	(A1) at  ({\axa},{\aya},-.5*\h) ;

\coordinate
	(A2) at ({\bxa},{\bya},-.5*\h);
\coordinate
	(A3) at ({\cxa},{\cya},-.5*\h);
\coordinate
	(A4) at ({\dxa},{\dya},-.5*\h);

\draw[dotted] (-\r+\tr,0,.25*\h) circle(\tr);
\draw[dotted] (\r-\sr,0,.25*\h) circle(\sr);
\coordinate
	(C1) at ({\exa},{\eya},.25*\h) ;
\coordinate
	 (C2) at ({\fxa},{\fya},.25*\h) ;

\coordinate
	(C7) at ({\lxa},{\lya},.25*\h) ;
\coordinate
	(C8) at ({\kxa},{\kya},.25*\h) ;

\draw[dotted] (-\r+\tr,0,1*\h) circle(\tr);
\draw[dotted] (\r-\sr,0,1*\h) circle(\sr);
\coordinate
	(B1) at ({\exa},{\eya},1*\h) ;
\coordinate
	(B2) at ({\fxa},{\fya},1*\h) ;

\coordinate
	(B3) at ({\hxa},{\hya},1*\h);
\coordinate
	 (B4) at ({\gxa},{\gya},1*\h) ;
\coordinate
	(B5) at ({\ixa},{\iya},1*\h);
\coordinate
	(B6) at ({\jxa},{\jya},1*\h) ;
\coordinate
	(B7) at ({\lxa},{\lya},1*\h) ;
\coordinate
	(B8) at ({\kxa},{\kya},1*\h);

\draw [dotted](-\r+\tr,0,2.25*\h) circle(\tr);
\draw[dotted] (\r-\sr,0,2.25*\h) circle(\sr);
\coordinate
	(D1) at ({\exa},{\eya},2.25*\h) ;
\coordinate
	 (D2) at ({\fxa},{\fya},2.25*\h) ;

\coordinate 
	(D3) at ({\hxa},{\hya},2.25*\h);
\coordinate
	(D4) at ({\gxa},{\gya},2.25*\h) ;
\coordinate
	 (D5) at ({\jxa},{\jya},2.25*\h);
\coordinate
	 (D6) at ({\ixa},{\iya},2.25*\h) ;
\coordinate
	(D7) at ({\kxa},{\kya},2.25*\h) ;
\coordinate
	 (D8) at ({\lxa},{\lya},2.25*\h);

\node[draw, fill =white, circle] at  ({\mxa},{\my},1.625*\h) {$f$};

\draw[dotted] (-\r+\tr,0,4*\h) circle(\tr);
\draw[dotted] (\r-\sr,0,4*\h) circle(\sr);
\coordinate
	(E1) at ({\exa},{\eya},4*\h) ;
\coordinate
	(E2) at ({\fxa},{\fya},4*\h) ;

\coordinate
	(E3) at ({\gxa},{\gya},4*\h);
\coordinate
	(E4) at ({\hxa},{\hya},4*\h) ;
\coordinate
	(E5) at ({\ixa},{\iya},4*\h);
\coordinate
	 (E6) at ({\jxa},{\jya},4*\h) ;
\coordinate
	 (E7) at ({\kxa},{\kya},4*\h) ;
\coordinate
	 (E8) at ({\lxa},{\lya},4*\h);

\draw[dotted] (-\r+\tr,0,\hf*\h) circle(\tr);
\draw[dotted] (\r-\sr,0,\hf*\h) circle(\sr);
\coordinate
	(F1) at ({\exa},{\eya},\hf*\h) ;
\coordinate
	(F2) at ({\fxa},{\fya},\hf*\h) ;

\coordinate
	(F3) at ({\gxa},{\gya},\hf*\h);
\coordinate
	(F4) at ({\hxa},{\hya},\hf*\h) ;

\draw[dotted] (\r-\tr,0,\hg*\h) circle(\tr);
\draw[dotted] (-\r+\sr,0,\hg*\h) circle(\sr);
\coordinate
	(H1) at ({\qxa},{\qya},\hg*\h) ;
\coordinate
	(H2) at ({\pxa},{\pya},\hg*\h) ;

\coordinate
	(H3) at ({\sxa},{\sya},\hg*\h);
\coordinate
	(H4) at ({\rxa},{\rya},\hg*\h) ;

\draw[dotted] (0,0,\hh*\h) circle(\r);

\coordinate
	(G1) at  ({\wxa},{\wya},\hh*\h) ;
\coordinate
	(G2) at ({\vxa},{\vya},\hh*\h);
\coordinate
	(G3) at ({\uxa},{\uya},\hh*\h);
\coordinate
	(G4) at ({\txa},{\tya},\hh*\h);

\coordinate
	 (co1) at ({\r-\sr+\sr*cos(195)},{\sr*sin(195)},.25*\h) ;
\coordinate
	(co2) at ({-\r+\tr+\tr*cos(15)},{\tr*sin(15)},.25*\h) ;
\coordinate
	(co3) at ({-\r+\tr+\tr*cos(15)},{\tr*sin(15)},\hf*\h) ;
\coordinate
	(co4) at ({\r-\sr+\sr*cos(195)},{\sr*sin(195)},\hf*\h) ;
\coordinate
	(co3c) at ({-\r+\tr+\tr*cos(195)},{\tr*sin(195)},\hf*\h) ;
\coordinate
	(co4c) at ({\r-\sr+\sr*cos(15)},{\sr*sin(15)},\hf*\h) ;
\coordinate
	(co3a) at ({-\r+\sr+\sr*cos(15)},{\sr*sin(15)},\hg*\h) ;
\coordinate
	(co4a) at ({\r-\tr+\tr*cos(195)},{\tr*sin(195)},\hg*\h) ;
\coordinate
	(co3b) at ({-\r+\sr+\sr*cos(195)},{\sr*sin(195)},\hg*\h) ;
\coordinate
	(co4b) at ({\r-\tr+\tr*cos(15)},{\tr*sin(15)},\hg*\h) ;

\coordinate
	 (co5) at ({\r*cos(195)},{\r*sin(195)},-.5*\h) ;
\coordinate
	 (co6) at ({\r*cos(195)},{\r*sin(195)},\hh*\h) ;
\coordinate
	 (co7) at ({\r*cos(15)},{\r*sin(15)},\hh*\h) ;
\coordinate
	 (co8) at ({\r*cos(15)},{\r*sin(15)},-.5*\h) ;

\coordinate
	(tw1) at ({\r-\sr+\sr*cos(15)},{\sr*sin(15)},2.6*\h) ;
\coordinate
	(tw3) at ({\r-\sr+\sr*cos(15)},{\tr*sin(15)},2.85*\h) ;
\coordinate
	(tw2) at ({\r-\sr+\sr*cos(195)},{\tr*sin(195)},3.25*\h) ;
\coordinate
	 (tw4) at ({\r-\sr+\sr*cos(195)},{\sr*sin(195)},3.5*\h) ;

\begin{scope}[on background layer]

\fill[black!20!white,opacity =.8](co5)--(co3c)--(co3)--(co2)to [out = 90, in =90,looseness = 2] (co1) --(co4)--(co4c)  --(co8)to [out = 90, in =90,looseness = 1] (co5);

\fill[black!20!white,opacity =.8](co3b)--(co6)  to [out = -90, in =-90,looseness = 1.5](co7)--(co4b)--(co4a) to [out = -90, in =-90,looseness = 2] (co3a);

\fill[black!20!white,opacity =.8] (co3c)to [out = -90, in =90, looseness =.5](co4a)--(co4b)to [out = 90, in =-90, looseness =.5](co3);

\fill[\ep] (B8) -- (C8) to [out = 90, in = -90, looseness =.5]
(A4) to [out = 90, in =90, looseness = 1] (A1)to [out = -90, in = 90, looseness =.75](C1)--(B1)--(D1)--(E1)--(F1)to [out = -90, in = 90, looseness =.5](H1)--
(H2)to [out = 90, in = -90, looseness =.5](F2)
--(E2)--(D2)--(B2)--(C2)to [out = 90, in = -90, looseness =.75](A2)to [out = 90, in =90, looseness = 1] (A3)to [out = -90, in = 90, looseness =.75](C7)--(B7);


\draw[very thick]
(H2)to [out = 90, in = -90, looseness =.5](F2);

\fill[black!20!white,opacity =.8] (co4c)to [out = -90, in =90, looseness =.75](co3a)--(co3b)to [out = 90, in =-90, looseness =.75](co4);

\fill[\et!5!white] (tw1) to [ out =180, in =0,looseness =1](tw2)--(tw4)to [ out =0, in =180,looseness =1](tw3);

\fill[\ep](G2)to [out = 90, in =-90, looseness = 1]
	(H4)to [out = 90, in = -90, looseness =.75]
	(F4)-- (E4)to [out = 90, in = -90, looseness =.75](D6) to [out = 90, in = -90, looseness =.75](B8)--(C8)--(C7)--(B7)to [out = -90, in = 90, looseness =.75](D5)to [out = -90, in = 90, looseness =.75](E3) --(F3)to [out = -90, in = 90, looseness =.75](H3)to [out = -90, in =90, looseness = 1](G1) to [out = -90, in =-90, looseness = 1](G4)to [out = 90, in =-90, looseness = 1](H2)--(H1)to [out = -90, in =90, looseness = 1](G3)to [out = -90, in =-90, looseness = 1](G2);

\fill[\et] (E6)to [out = 90, in = -90, looseness =.75](D8)to [out = 90, in = -90, looseness =.75] (B6)to [out = 90, in =90, looseness = 2] (B3)--(D3)to [out = -90, in = 180, looseness =.75](tw1)--(tw3)to [out = 180, in = -90, looseness =.75](D4)--(B4)to [out = 90, in =90, looseness = 2] (B5)to [out = -90, in = 90, looseness =.75](D7) to [out = -90, in = 90, looseness =.75] (E5)to [out = -90, in =-90, looseness = 2] (E8)to [out = 90, in = 0, looseness =.75](tw4)--
(tw2)to [out = 0, in = 90, looseness =.75](E7)to [out = -90, in =-90, looseness = 2] (E6);

\draw[very thick](C7)--(B7)to [out = -90, in = 90, looseness =.75](D5)to [out = -90, in = 90, looseness =.75](E3) --
(F3);
\draw[very thick](F3)to [out = -90, in = 90, looseness =.75](H3);
\draw[very thick](H3)to [out = -90, in = 90, looseness =.75](G1)to [out = -90, in =-90, looseness = 1](G4)to [out = 90, in = -90, looseness =.75](H2);
\draw[very thick](F2)
--(E2)--(D2)--(B2)--(C2)to [out = 90, in = -90, looseness =.75](A2)to [out = 90, in =90, looseness = 1] (A3)to [out = -90, in = 90, looseness =.75](C7)--(B7);

\draw[very thick] (tw3)to [out = 180, in = -90, looseness =.75](D4)--(B4)to [out = 90, in =90, looseness = 2] (B5)to [out = -90, in = 90, looseness =.75](D7)to [out = -90, in = 90, looseness =.75] (E5)to [out = -90, in =-90, looseness = 2] (E8)to [out = 90, in = 0, looseness =.75](tw4);

\draw[dotted, line width = .5pt] (co7)--(co4b)to [out = 90, in =-90, looseness =.5](co3)--(co2)to [out = 90, in =90,looseness = 2] (co1) --(co4)to [out = -90, in =90, looseness =.75] (co3b)--(co6) to [out = -90, in =-90,looseness = 1.5] (co7);

\draw[dotted, line width = .5pt] (co5)--(co3c)to [out = -90, in =90, looseness =.5](co4a)to [out = -90, in =-90,looseness = 2] (co3a)to [out = 90, in =-90, looseness =.75](co4c)  -- (co8)to [out = 90, in =90,looseness = 1] (co5);

\end{scope}

\end{tikzpicture}
}
\end{centering}
\caption{$\mathrm{tr}_{\sh{X}}(\mathrm{tr}_Y(f))$}
\end{subfigure}
  \begin{subfigure}[t]{.23\textwidth}
\begin{centering}
\resizebox{\textwidth}{!}{
\begin{tikzpicture} [tdplot_rotated_coords,
    mystyle/.style={%
      label={right:\pgfkeysvalueof{/pgf/minimum width}},
    },
   my style/.style={%
     append after command={
       \pgfextra{\node [right] at (\tikzlastnode.mid east) {\tikzlastnode};}
     },
   },
  ]\draw[dotted] (0,0,-.5*\h) circle(\r);

\coordinate
	(A1) at  ({\ax},{\ay},-.5*\h) ;

\coordinate
	(A2) at ({\bx},{\by},-.5*\h);
\coordinate
	(A3) at ({\cx},{\cy},-.5*\h);
\coordinate
	(A4) at ({\dx},{\dy},-.5*\h);

\draw[dotted] (\r-\tr,0,.25*\h) circle(\tr);
\draw[dotted] (-\r+\sr,0,.25*\h) circle(\sr);
\coordinate
	(C1) at ({\ex},{\ey},.25*\h) ;
\coordinate
	 (C2) at ({\fx},{\fy},.25*\h) ;

\coordinate
	(C7) at ({\gx},{\gy},.25*\h) ;
\coordinate
	(C8) at ({\hx},{\hy},.25*\h) ;

\draw[dotted] (\r-\tr,0,1*\h) circle(\tr);
\draw[dotted] (-\r+\sr,0,1*\h) circle(\sr);
\coordinate
	(B1) at ({\ex},{\ey},1*\h) ;
\coordinate
	(B2) at ({\fx},{\fy},1*\h) ;

\coordinate
	(B3) at ({\lx},{\ly},1*\h);
\coordinate
	 (B4) at ({\kx},{\ky},1*\h) ;
\coordinate
	(B5) at ({\jx},{\jy},1*\h);
\coordinate
	(B6) at ({\ix},{\iy},1*\h) ;
\coordinate
	(B7) at ({\gx},{\gy},1*\h) ;
\coordinate
	(B8) at ({\hx},{\hy},1*\h);

\draw [dotted](\r-\tr,0,2.25*\h) circle(\tr);
\draw[dotted] (-\r+\sr,0,2.25*\h) circle(\sr);
\coordinate
	(D1) at ({\ex},{\ey},2.25*\h) ;
\coordinate
	 (D2) at ({\fx},{\fy},2.25*\h) ;

\coordinate 
	(D3) at ({\lx},{\ly},2.25*\h);
\coordinate
	(D4) at ({\kx},{\ky},2.25*\h) ;
\coordinate
	 (D5) at ({\ix},{\iy},2.25*\h);
\coordinate
	 (D6) at ({\jx},{\jy},2.25*\h) ;
\coordinate
	(D7) at ({\hx},{\hy},2.25*\h) ;
\coordinate
	 (D8) at ({\gx},{\gy},2.25*\h);

\node[draw, fill =white, circle] at  ({\mx},{\my},1.625*\h) {$f$};

\draw[dotted] (\r-\tr,0,4*\h) circle(\tr);
\draw[dotted] (-\r+\sr,0,4*\h) circle(\sr);
\coordinate
	(E1) at ({\ex},{\ey},4*\h) ;
\coordinate
	(E2) at ({\fx},{\fy},4*\h) ;

\coordinate
	(E3) at ({\kx},{\ky},4*\h);
\coordinate
	(E4) at ({\lx},{\ly},4*\h) ;
\coordinate
	(E5) at ({\jx},{\jy},4*\h);
\coordinate
	 (E6) at ({\ix},{\iy},4*\h) ;
\coordinate
	 (E7) at ({\hx},{\hy},4*\h) ;
\coordinate
	 (E8) at ({\gx},{\gy},4*\h);

\draw[dotted] (\r-\tr,0,\hf*\h) circle(\tr);
\draw[dotted] (-\r+\sr,0,\hf*\h) circle(\sr);
\coordinate
	(F1) at ({\ex},{\ey},\hf*\h) ;
\coordinate
	(F2) at ({\fx},{\fy},\hf*\h) ;

\coordinate
	(F3) at ({\kx},{\ky},\hf*\h);
\coordinate
	(F4) at ({\lx},{\ly},\hf*\h) ;

\draw[dotted] (-\r+\tr,0,\hg*\h) circle(\tr);
\draw[dotted] (\r-\sr,0,\hg*\h) circle(\sr);
\coordinate
	(H1) at ({\px},{\py},\hg*\h) ;
\coordinate
	(H2) at ({\qx},{\qy},\hg*\h) ;

\coordinate
	(H3) at ({\rx},{\ry},\hg*\h);
\coordinate
	(H4) at ({\sx},{\sy},\hg*\h) ;

\draw[dotted] (0,0,\hh*\h) circle(\r);

\coordinate
	(G1) at  ({\tx},{\ty},\hh*\h) ;
\coordinate
	(G2) at ({\ux},{\uy},\hh*\h);
\coordinate
	(G3) at ({\vx},{\vy},\hh*\h);
\coordinate
	(G4) at ({\wx},{\wy},\hh*\h);

\coordinate
	 (co1) at ({-\r+\sr+\sr*cos(15)},{\sr*sin(15)},.25*\h) ;
\coordinate
	(co2) at ({\r-\tr+\tr*cos(195)},{\tr*sin(195)},.25*\h) ;
\coordinate
	(co3) at ({\r-\tr+\tr*cos(195)},{\tr*sin(195)},\hf*\h) ;
\coordinate
	(co4) at ({-\r+\sr+\sr*cos(15)},{\sr*sin(15)},\hf*\h) ;
\coordinate
	(co3c) at ({\r-\tr+\tr*cos(15)},{\tr*sin(15)},\hf*\h) ;
\coordinate
	(co4c) at ({-\r+\sr+\sr*cos(195)},{\sr*sin(195)},\hf*\h) ;
\coordinate
	(co3a) at ({\r-\sr+\sr*cos(195)},{\sr*sin(195)},\hg*\h) ;
\coordinate
	(co4a) at ({-\r+\tr+\tr*cos(15)},{\tr*sin(15)},\hg*\h) ;
\coordinate
	(co3b) at ({\r-\sr+\sr*cos(15)},{\sr*sin(15)},\hg*\h) ;
\coordinate
	(co4b) at ({-\r+\tr+\tr*cos(195)},{\tr*sin(195)},\hg*\h) ;

\coordinate
	 (co5) at ({\r*cos(15)},{\r*sin(15)},-.5*\h) ;
\coordinate
	 (co6) at ({\r*cos(15)},{\r*sin(15)},\hh*\h) ;
\coordinate
	 (co7) at ({\r*cos(195)},{\r*sin(195)},\hh*\h) ;
\coordinate
	 (co8) at ({\r*cos(195)},{\r*sin(195)},-.5*\h) ;

\coordinate
	(tw1) at ({-\r+\sr+\sr*cos(195)},{\sr*sin(195)},2.6*\h) ;
\coordinate
	(tw2) at ({-\r+\sr+\sr*cos(15)},{\tr*sin(15)},3.25*\h) ;
\coordinate
	(tw3) at ({-\r+\sr+\sr*cos(195)},{\tr*sin(195)},2.85*\h) ;
\coordinate
	 (tw4) at ({-\r+\sr+\sr*cos(15)},{\sr*sin(15)},3.5*\h) ;

\begin{scope}[on background layer]

\fill[black!20!white,opacity =.8](co5)--(co3c)--(co3)--(co2)to [out = 90, in =90,looseness = 2] (co1) --(co4)--(co4c)  --(co8)to [out = 90, in =90,looseness = 1] (co5);

\fill[black!20!white,opacity =.8](co3b)--(co6)  to [out = -90, in =-90,looseness = 1.5](co7)--(co4b)--(co4a) to [out = -90, in =-90,looseness = 2] (co3a);

\fill[black!20!white,opacity =.8] (co3c)to [out = -90, in =90, looseness =.5](co4a)--(co4b)to [out = 90, in =-90, looseness =.5](co3);

\fill[\et] (B8) -- (C8) to [out = 90, in = -90, looseness =.5]
(A4) to [out = 90, in =90, looseness = 1] (A1)to [out = -90, in = 90, looseness =.75](C1)--(B1)--(D1)--(E1)--(F1)to [out = -90, in = 90, looseness =.5](H1)--
(H2)to [out = 90, in = -90, looseness =.5](F2)
--(E2)--(D2)--(B2)--(C2)to [out = 90, in = -90, looseness =.75](A2)to [out = 90, in =90, looseness = 1] (A3)to [out = -90, in = 90, looseness =.75](C7)--(B7);


\draw[very thick]
(H2)to [out = 90, in = -90, looseness =.5](F2);

\fill[black!20!white,opacity =.8] (co4c)to [out = -90, in =90, looseness =.75](co3a)--(co3b)to [out = 90, in =-90, looseness =.75](co4);

\fill[\ep!5!white] (tw1) to [ out =180, in =0,looseness =1](tw2)--(tw4)to [ out =0, in =180,looseness =1](tw3);

\fill[\et](G2)to [out = 90, in =-90, looseness = 1]
	(H4)to [out = 90, in = -90, looseness =.75]
	(F4)-- (E4)to [out = 90, in = -90, looseness =.75](D6) to [out = 90, in = -90, looseness =.75](B8)--(C8)--(C7)--(B7)to [out = -90, in = 90, looseness =.75](D5)to [out = -90, in = 90, looseness =.75](E3) --(F3)to [out = -90, in = 90, looseness =.75](H3)to [out = -90, in =90, looseness = 1](G1) to [out = -90, in =-90, looseness = 1](G4)to [out = 90, in =-90, looseness = 1](H2)--(H1)to [out = -90, in =90, looseness = 1](G3)to [out = -90, in =-90, looseness = 1](G2);

\fill[\ep] (E6)to [out = 90, in = -90, looseness =.75](D8)to [out = 90, in = -90, looseness =.75] (B6)to [out = 90, in =90, looseness = 2] (B3)--(D3)to [out = -90, in = 0, looseness =.75](tw1)--(tw3)to [out = 0, in = -90, looseness =.75](D4)--(B4)to [out = 90, in =90, looseness = 2] (B5)to [out = -90, in = 90, looseness =.75](D7)to [out = -90, in = 90, looseness =.75] (E5)to [out = -90, in =-90, looseness = 2] (E8)to [out = 90, in = 180, looseness =.75](tw4)--(tw2)to [out = 180, in = 90, looseness =.75](E7)to [out = -90, in =-90, looseness = 2] (E6);

\draw[very thick](C7)--(B7)to [out = -90, in = 90, looseness =.75](D5)to [out = -90, in = 90, looseness =.75](E3) --
(F3);
\draw[very thick](F3)to [out = -90, in = 90, looseness =.75](H3);
\draw[very thick](H3)to [out = -90, in = 90, looseness =.75](G1)to [out = -90, in =-90, looseness = 1](G4)to [out = 90, in = -90, looseness =.75](H2);
\draw[very thick](F2)
--(E2)--(D2)--(B2)--(C2)to [out = 90, in = -90, looseness =.75](A2)to [out = 90, in =90, looseness = 1] (A3)to [out = -90, in = 90, looseness =.75](C7)--(B7);

\draw[very thick] (tw3)to [out = 0, in = -90, looseness =.75](D4)--(B4)to [out = 90, in =90, looseness = 2] (B5)to [out = -90, in = 90, looseness =.75](D7)to [out = -90, in = 90, looseness =.75] (E5)to [out = -90, in =-90, looseness = 2] (E8)to [out = 90, in = 180, looseness =.75](tw4);

\draw[dotted, line width = .5pt] (co7)--(co4b)to [out = 90, in =-90, looseness =.5](co3)--(co2)to [out = 90, in =90,looseness = 2] (co1) --(co4)to [out = -90, in =90, looseness =.75] (co3b)--(co6) to [out = -90, in =-90,looseness = 1.5] (co7);

\draw[dotted, line width = .5pt] (co5)--(co3c)to [out = -90, in =90, looseness =.5](co4a)to [out = -90, in =-90,looseness = 2] (co3a)to [out = 90, in =-90, looseness =.75](co4c)  -- (co8)to [out = 90, in =90,looseness = 1] (co5);

\end{scope}

\end{tikzpicture}
}
\end{centering}
\caption{$\mathrm{tr}_{\sh{Y}}(\mathrm{tr}_X(f))$}
\end{subfigure}

%% file: II-main_umbra_diagram.tex
\resizebox{\textwidth}{!}
{
\begin{tikzpicture}
\node (B1) at (0,13.75){$I$};
\node (C1) at (0,12.25){$\csh{U_A}$};

\node (D1) at (-8,10.75){$\sh{N}\otimes \dsh{\rdual{N}}$};
\node (D2) at (-2,10.75){$\csh{N\odot \rdual{N}}$};
\node (D3) at (2,10.75){$\csh{\ldual{M}\odot M}$};
\node (D4) at (8,10.75){$\sh{\ldual{M}}\otimes \dsh{M}$};

\node (E1) at (-4.25,9){$\sh{\ldual{M}\odot M\odot N}\otimes \dsh{\rdual{N}}$};
\node (E2) at (0,9){$\csh{\ldual{M}\odot M\odot N\odot \rdual{N}}$};
\node (E3) at (4.25,9){$\sh{\ldual{M}}\otimes \dsh{M\odot N\odot \rdual{N}}$};

\node (F1) at (-4.25,7.5){$\sh{\ldual{M}\odot N\odot M}\otimes \dsh{\rdual{N}}$};
\node (F2) at (0,7.5){$\csh{\ldual{M}\odot N\odot M\odot \rdual{N}}$};
\node (F3) at (4.25,7.5){$\sh{\ldual{M}}\otimes \dsh{N\odot M\odot \rdual{N}}$};

\node (G1) at (-4.25,6){$\sh{N\odot M\odot \ldual{M}}\otimes \dsh{\rdual{N}}$};
\node (G2) at (4.25,6){$\sh{\ldual{M}}\otimes \dsh{\rdual{N}\odot N\odot M}$};

\node (H1) at (-8,4.25){$\sh{N}\otimes \dsh{\rdual{N}}$};
\node (H2) at (-4.25,4.25){$\dsh{\rdual{N}}\otimes \sh{N\odot M\odot \ldual{M}}$};
\node (H3) at (0,4.25){$\esh{\rdual{N}\odot N\odot M\odot \ldual{M}}$};
\node (H4) at (4.25,4.25){$\dsh{\rdual{N}\odot N\odot M}\otimes \sh{\ldual{M}}$};
\node (H5) at (8,4.25){$\sh{\ldual{M}}\otimes \dsh{M}$};

\node (I1) at (-8,2.5){$\dsh{\rdual{N}}\otimes \sh{N}$};
\node (I2) at (-2,2.5){$\esh{\rdual{N}\odot N}$};
\node (I3) at (2,2.5){$\esh{M\odot \ldual{M}}$};
\node (I4) at (8,2.5){$\dsh{M}\otimes \sh{\ldual{M}}$};

\node (J1) at (0,1){$\esh{U_A}$};

\node (K1) at (0, -.25){$I$};

\draw [->] (B1)--(C1) node[midway,right]{$\iunit$};
\draw [dashed, ->] (B1)--(D1) node[midway,right]{};
\draw [dashed,->] (B1)--(D4) node[midway,right]{};

\draw [->] (C1)--(D2) node[midway,left]{$\csh{\eta_N}$};
\draw [->] (C1)--(D3) node[midway,right]{$\csh{\eta_M}$};

\draw [->] (D1)--(E1) node[midway,fill=white]{$\sh{\eta_M\odot \id}\otimes \id$};
\draw [->] (D2)--(D1) node[midway,above]{$\spl$};
\draw [->] (D2)--(E2) node[midway,fill=white]{$\csh{\eta_M\odot \id}$};
\draw [->] (D3)--(E2) node[midway,fill=white]{$\csh{\id\odot \eta_N}$};
\draw [->] (D3)--(D4) node[midway,above]{$\spl$};
\draw [->] (D4)--(E3) node[midway, fill=white]{$\id\otimes\dsh{\id\odot \eta_N}$};

\draw [dashed,->] (D1)--(H1) node[midway,fill=white]{$\tr_M(\phi)\otimes \id$};
\draw [dashed,->] (D4)--(H5) node[midway,fill=white]{$\id\otimes \tr_N(\phi)$};

\draw [->] (E1)--(F1) node[midway,right]{$\sh{\id\odot\phi}\otimes \id$};
\draw [->] (E2)--(E1) node[midway,above]{$\spl$};
\draw [->] (E2)--(F2) node[midway,right]{$\csh{\id\odot\phi\odot \id}$};
\draw [->] (E2)--(E3) node[midway,above]{$\spl$};
\draw [->] (E3)--(F3) node[midway,right]{$\id\otimes \dsh{\phi\odot \id}$};

\draw [->] (F1)--(G1) node[midway,right]{$\theta\otimes \id$};
\draw [->] (F2)--(F1) node[midway,above]{$\spl$};
\draw [->] (F2)--(F3) node[midway,above]{$\spl$};
\draw [->] (F3)--(G2) node[midway,right]{$\id\otimes \theta$};

\draw [->] (G1)--(H1) node[midway, fill=white]{$\sh{\id \otimes \epsilon_M}\otimes \id$};
\draw [->] (G1)--(H2) node[midway,right]{$\gamma$};
\draw [->] (G2)--(H4) node[midway,right]{$\gamma$};
\draw [->] (G2)--(H5) node[midway,fill=white]{$\id \otimes \dsh{\epsilon_N\odot \id}$};

\draw [->] (H1)--(I1) node[midway,right]{$\gamma$};
\draw [->] (H2)--(I1) node[midway,fill=white]{$\id \otimes \sh{\id\odot \epsilon_M}$};
\draw [->] (H2)--(H3) node[midway,above]{$\uspl$};
\draw [->] (H3)--(I2) node[midway,fill=white]{$\esh{\id\odot \epsilon_M}$};
\draw [->] (H3)--(I3) node[midway,fill=white]{$\esh{\epsilon_N\odot \id}$};
\draw [->] (H4)--(I4) node[midway,fill=white]{$\dsh{\epsilon_N\odot \id}\odot \id$};
\draw [->] (H4)--(H3) node[midway,above]{$\uspl$};
\draw [->] (H5)--(I4) node[midway,right]{$\gamma$};
\draw [->] (I1)--(I2) node[midway,above]{$\uspl$};
\draw [dashed, ->] (I1)--(K1) node[midway,right]{};
\draw [->] (I2)--(J1) node[midway,left]{$\esh{\epsilon_N}$};
\draw [->] (I3)--(J1) node[midway,right]{$\esh{\epsilon_M}$};
\draw [dashed,->] (I4)--(K1) node[midway,right]{$$};
\draw [->] (I4)--(I3) node[midway,above]{$\uspl$};

\draw [->] (J1)--(K1) node[midway,right]{$\ounit$};

\end{tikzpicture}}

%% file: II-dual_pairs_1_cell.tex
    \begin{subfigure}[t]{0.2\textwidth}
        \centering
\resizebox{\textwidth}{!}{\begin{tikzpicture}

\node[draw] (A1) at (0, 2){\begin{tikzpicture}
\oc{y}{}{Y}{4.5}{0}{0}{\en}
\oc{x}{}{X}{0}{0}{0}{\en}

  \begin{scope}[on background layer]
\alp{ ({x}t)--({y}t)}
\end{scope}
\end{tikzpicture}
};
\node[draw] (A2) at (0, 0){\begin{tikzpicture}
\oc{y}{}{Y}{4.5}{0}{0}{\en}
\oc{x}{}{X}{0}{0}{0}{\en}
\oc{m}{M}{}{1.5}{0}{0}{\et}
\oc{n}{N}{}{3}{0}{0}{\ep}
  \begin{scope}[on background layer]
\blp{ ({m}b)--({n}t)}
\alp{ ({x}b)--({m}b)}
\alp{ ({n}t)--({y}t)}
\end{scope}

\end{tikzpicture}
};

\draw[ ->](A1)--(A2)node [midway , fill=white] {$\eta$};
\end{tikzpicture}}

        \caption{Coevaluation}
    \end{subfigure}%
\hspace{.5cm}
    \begin{subfigure}[t]{0.2\textwidth}
        \centering
\resizebox{\textwidth}{!}{
\begin{tikzpicture}

\node[draw] (A1) at (0, 2){\begin{tikzpicture}
\oc{y}{}{Y}{4.5}{0}{0}{\en}
\oc{x}{}{X}{0}{0}{0}{\en}
\oc{n}{N}{}{1.5}{0}{0}{\ep}
\oc{m}{M}{}{3}{0}{0}{\et}

  \begin{scope}[on background layer]
\alp{({m}b)--({n}t)}
\blp{ ({x}b)--({n}b)}
\blp{({m}t)--({y}t)}
\end{scope}
\end{tikzpicture}
};

\node[draw] (A2) at (0, 0){\begin{tikzpicture}
\oc{y}{}{Y}{4.5}{0}{0}{\en}
\oc{x}{}{X}{0}{0}{0}{\en}
  \begin{scope}[on background layer]
\blp{ ({x}t)--({y}t)}
\end{scope}
\end{tikzpicture}
};

\draw[ ->](A1)--(A2)node [midway , fill=white] {$\epsilon$};
\end{tikzpicture}}

        \caption{Evaluation}
    \end{subfigure}%
 \hspace{.5cm}
    \begin{subfigure}[t]{0.5\textwidth}
        \centering
\resizebox{\textwidth}{!}{
\begin{tikzpicture}

\node[draw] (A3) at (0, 4){\begin{tikzpicture}
\oc{y}{}{Y}{6}{0}{0}{\en}
\oc{x}{}{X}{0}{0}{0}{\en}
\oc{n}{N}{}{1.5}{0}{0}{\ep}

  \begin{scope}[on background layer]
\alp{({n}b)--({y}t)}
\blp{({x}b)--({n}b)}
\end{scope}
\end{tikzpicture}
};

\node[draw] (A1) at (0, 2){\begin{tikzpicture}
\oc{y}{}{Y}{6}{0}{0}{\en}
\oc{x}{}{X}{0}{0}{0}{\en}
\oc{n}{N}{}{1.5}{0}{0}{\ep}
\oc{m}{M}{}{3}{0}{0}{\et}
\oc{n2}{N}{}{4.5}{0}{0}{\ep}
  \begin{scope}[on background layer]
\alp{ ({m}b)--({n}t)}
\alp{ ({n2}b)--({y}t)}
\blp{ ({x}b)--({n}b)}
\blp{({m}t)--({n2}t)}
\end{scope}
\end{tikzpicture}
};

\node[draw] (A2) at (0, 0){\begin{tikzpicture}
\oc{y}{}{Y}{6}{0}{0}{\en}
\oc{x}{}{X}{0}{0}{0}{\en}
\oc{n2}{N}{}{4.5}{0}{0}{\ep}
  \begin{scope}[on background layer]
\alp{ ({n2}b)--({y}t)}
\blp{({x}t)--({n2}b)}
\end{scope}
\end{tikzpicture}
};
\draw[ ->](A3)--(A1)node [midway , fill=white] {$\id\odot \eta$};
\draw[ ->](A1)--(A2)node [midway , fill=white] {$\epsilon\odot \id$};
\end{tikzpicture}
\hfill
%
\begin{tikzpicture}

\node[draw] (A3) at (0, 4){\begin{tikzpicture}
\oc{y}{}{Y}{6}{0}{0}{\en}
\oc{x}{}{X}{0}{0}{0}{\en}
\oc{m}{M}{}{4.5}{0}{0}{\et}
  \begin{scope}[on background layer]
\blp{ ({m}b)--({y}t)}
\alp{ ({x}b)--({m}b)}
\end{scope}
\end{tikzpicture}
};

\node[draw] (A1) at (0, 2){\begin{tikzpicture}
\oc{y}{}{Y}{6}{0}{0}{\en}
\oc{x}{}{X}{0}{0}{0}{\en}
\oc{m}{M}{}{1.5}{0}{0}{\et}
\oc{n}{N}{}{3}{0}{0}{\ep}
\oc{m2}{M}{}{4.5}{0}{0}{\et}

  \begin{scope}[on background layer]
\blp{({n}b)--({m}t)}
\blp{ ({m2}b)--({y}t)}
\alp{({x}b)--({m}b)}
\alp{({n}t)--({m2}t)}
\end{scope}

\end{tikzpicture}
};

\node[draw] (A2) at (0, 0){\begin{tikzpicture}
\oc{y}{}{Y}{6}{0}{0}{\en}
\oc{x}{}{X}{0}{0}{0}{\en}
\oc{m}{M}{}{1.5}{0}{0}{\et}
  \begin{scope}[on background layer]
\blp{ ({m}b)--({y}t)}
\alp{({x}b)--({m}b)}
\end{scope}
\end{tikzpicture}
};
\draw[ ->](A3)--(A1)node [midway , fill=white] {$\eta\odot \id$};
\draw[ ->](A1)--(A2)node [midway , fill=white] {$\id\odot \epsilon$};
\end{tikzpicture}}

        \caption{Triangle identities assert these two composites are the identity map}
    \end{subfigure}%

%% file: II-dual_coeval_eval_for_0_cells.tex
\resizebox{\textwidth}{!}
{
\begin{tikzpicture}

\node[draw] (A2) at (0, 0){\begin{tikzpicture}
\oc{ta}{}{X}{2}{-1}{-2}{\en}
\oc{t2}{}{Y}{5}{-1}{-2}{\en}

  \begin{scope}[on background layer]
\alp{({ta}t)--({t2}t)}
\end{scope}
\end{tikzpicture}
};

\node[draw] (A3) at (4, 0){\begin{tikzpicture}
\oc{ta}{}{X}{1}{1}{-1}{\en}
\oc{t2}{}{Y}{4}{1}{-1}{\en}
\oc{t3}{C}{_3}{2}{0}{-1}{\et}
\oc{e4}{E}{_4}{3}{1}{0}{\ep}
  \begin{scope}[on background layer]
\dlp{({t3}t)--({e4}b)}
\alp{({t3}b)--({t2}b)}
\alp{({ta}t)--({e4}t)}
\end{scope}
\end{tikzpicture}
};

\draw[->](A2)--(A3)node [midway , fill=white] {$\triangle$};
\end{tikzpicture}
\hspace{1cm}
\begin{tikzpicture}

\node[draw] (A2) at (0, 0){\begin{tikzpicture}
\oc{ta}{}{X}{2}{-1}{-2}{\en}
\oc{t2}{}{Y}{5}{-1}{-2}{\en}
  \begin{scope}[on background layer]
\dlp{({ta}t)--({t2}t)}
\end{scope}
\end{tikzpicture}
};

\node[draw] (A3) at (4, 0){\begin{tikzpicture}
\oc{ta}{}{X}{1}{1}{-1}{\en}
\oc{t2}{}{Y}{4}{1}{-1}{\en}
\oc{t3}{C}{_3}{2}{1}{0}{\et}
\oc{e4}{E}{_4}{3}{0}{-1}{\ep}
  \begin{scope}[on background layer]
\alp{ ({t3}b)--({e4}t)}
\dlp{({t3}t)--({t2}t)}
\dlp{({ta}b)--({e4}b)}
\end{scope}
\end{tikzpicture}
};

\draw[->](A2)--(A3)node [midway , fill=white] {$\triangle$};
\end{tikzpicture}
}

%% file: II-shadow_isomorphism.tex
\begin{tikzpicture}

\node[draw] (Z2) at (0, 4){\begin{tikzpicture}

\oc{t1}{C}{_1}{0}{5}{4}{\et}
\oc{x}{M}{}{1}{4}{4}{\en}
\oc{y}{N}{}{2}{4}{4}{\en}

\oc{ex}{E}{_x}{4}{5}{4}{\ep}
\gc{g1}{3}{5}{4}

  \begin{scope}[on background layer]
\dlp{({t1}t)\gpt{g1}({ex}b)}

\alp{ ({t1}b)--({x}b)}
\blp{({x}b)--({y}b)}
\alp{({y}b)\gpb{g1}({ex}t)}
\end{scope}
\end{tikzpicture}};

\node[draw] (Z3) at (5, 4){\begin{tikzpicture}

\oc{t1}{C}{_1}{1}{5}{4}{\et}
\oc{x}{M}{}{2}{4}{4}{\en}
\oc{y}{N}{}{3}{2}{2}{\en}
\oc{t5}{C}{_5}{2}{3}{2}{\et\dk}
\oc{e6}{E}{_6}{3}{4}{3}{\ep\dk}

\oc{ex}{E}{_x}{5}{5}{2}{\ep}
\gc{g1}{4}{5}{2}

  \begin{scope}[on background layer]
\dlp{({t1}t)\gpt{g1}({ex}b)}
\dblp{({t5}t)--({e6}b)}

\alp{({t1}b)--({x}b)}
\blp{({x}b)--({e6}t)}
\blp{ ({t5}b)--({y}b)}
\alp{({y}b)\gpb{g1}({ex}t)}
\end{scope}
\end{tikzpicture}};

\node[draw] (X5) at (15, 4){\begin{tikzpicture}

\oc{x}{M}{}{3}{4}{4}{\en}
\oc{y}{N}{}{2}{4}{4}{\en}
\oc{t5}{C}{_5}{1}{5}{4}{\et\dk}
\oc{e6}{E}{_6}{5}{5}{4}{\ep\dk}
\gc{g1}{4}{5}{4}
  \begin{scope}[on background layer]
\dblp{({t5}t)\gpt{g1}({e6}b)}
\blp{({t5}b)--({y}b)}
\alp{({y}b)--({x}b)}
\blp{({x}b)\gpb{g1}({e6}t)}
\end{scope}
\end{tikzpicture}};

\node[draw] (Y3) at (10, 4){\begin{tikzpicture}

\oc{t1}{C}{_1}{2}{3}{2}{\et}
\oc{x}{M}{}{3}{2}{2}{\en}
\oc{y}{N}{}{2}{4}{4}{\en}
\oc{t5}{C}{_5}{1}{5}{4}{\et\dk}
\oc{e6}{E}{_6}{5}{5}{2}{\ep\dk}

\oc{ex}{E}{_x}{3}{4}{3}{\ep}
\gc{g1}{4}{5}{2}
  \begin{scope}[on background layer]
\dlp{ ({t1}t)--({ex}b)}
\dblp{({t5}t)\gpt{g1}({e6}b)}

\alp{ ({t1}b)--({x}b)}
\blp{({x}b)\gpb{g1}({e6}t)}
\blp{ ({t5}b)--({y}b)}
\alp{({y}b)--({ex}t)}
\end{scope}
\end{tikzpicture}};

\draw[->](Z2)--(Z3)node [midway , fill=white] {$\triangle$};
\draw[<->](Z3)--(Y3)node [midway , fill=white] {$\gamma$};
\draw[->](Y3)--(X5)node [midway , fill=white] {$\triangle^{-1}$};
\end{tikzpicture}

%% file: II-shadow_isomorphism_assoc.tex
\begin{tikzpicture}

\node[draw] (Z2) at (0, 11){\begin{tikzpicture}

\oc{t1}{C}{_1}{0}{5}{4}{\et}
\oc{m}{M}{}{1}{4}{4}{\en}
\oc{n}{N}{}{2}{4}{4}{\en}
\oc{p}{P}{}{3}{4}{4}{\en}

\oc{ex}{E}{_x}{5}{5}{4}{\ep}
\gc{g1}{4}{5}{4}

  \begin{scope}[on background layer]
\dlp{({t1}t)\gpt{g1}({ex}b)}
\alp{({t1}b)--({m}b)}
\blp{({m}b)--({n}b)}
\clp{({n}b)--({p}b)}
\alp{({p}b)\gpb{g1}({ex}t)}
\end{scope}
\end{tikzpicture}};

\node[draw] (Z3) at (0, 6){\begin{tikzpicture}

\oc{t1}{C}{_1}{1}{5}{4}{\et}
\oc{m}{M}{}{2}{4}{4}{\en}
\oc{n}{N}{}{3}{2}{2}{\en}
\oc{p}{P}{}{4}{2}{2}{\en}
\oc{t5}{C}{_5}{2}{3}{2}{\et\dk}
\oc{e6}{E}{_6}{3}{4}{3}{\ep\dk}

\oc{ex}{E}{_x}{6}{5}{2}{\ep}
\gc{g1}{5}{5}{2}

  \begin{scope}[on background layer]
\dlp{({t1}t)\gpt{g1}({ex}b)}
\dblp{({t5}t)--({e6}b)}

\alp{({t1}b)--({m}b)}
\blp{({m}b)--({e6}t)}
\blp{({t5}b)--({n}b)}
\clp{({n}b)--({p}b)}
\alp{({p}b)\gpb{g1}({ex}t)}
\end{scope}
\end{tikzpicture}};

\node[draw] (X5) at (0, -5){\begin{tikzpicture}

\oc{m}{M}{}{3}{4}{4}{\en}
\oc{n}{N}{}{1}{4}{4}{\en}
\oc{p}{P}{}{2}{4}{4}{\en}
\oc{t5}{C}{_5}{0}{5}{4}{\et\dk}
\oc{e6}{E}{_6}{5}{5}{4}{\ep\dk}
\gc{g1}{4}{5}{4}

  \begin{scope}[on background layer]
\dblp{({t5}t)\gpt{g1}({e6}b)}
\blp{({t5}b)--({n}b)}
\clp{({n}b)--({p}b)}
\alp{({p}b)-- ({m}b)}
\blp{({m}b)\gpb{g1}({e6}t)}
	\end{scope}
\end{tikzpicture}};

\node[draw] (Y3) at (0, 0){\begin{tikzpicture}

\oc{t1}{C}{_1}{2}{3}{2}{\et}
\oc{m}{M}{}{3}{2}{2}{\en}
\oc{n}{N}{}{1}{4}{4}{\en}
\oc{p}{P}{}{2}{4}{4}{\en}
\oc{t5}{C}{_5}{0}{5}{4}{\et\dk}
\oc{e6}{E}{_6}{5}{5}{2}{\ep\dk}

\oc{ex}{E}{_x}{3}{4}{3}{\ep}
\gc{g1}{4}{5}{2}

  \begin{scope}[on background layer]
\dlp{({t1}t)--({ex}b)}
\dblp{({t5}t)\gpt{g1}({e6}b)}

\alp{({t1}b)--({m}b)}
\blp{({m}b)\gpb{g1}({e6}t)}
\blp{ ({t5}b)--({n}b)}
\clp{({n}b)--({p}b)}
\alp{({p}b)--({ex}t)}
\end{scope}
\end{tikzpicture}};

\draw[->](Z2)--(Z3)node [midway , fill=white] {$\triangle$};
\draw[<->](Z3)--(Y3)node [midway , fill=white] {$\gamma$};
\draw[->](Y3)--(X5)node [midway , fill=white] {$\triangle^{-1}$};

\node[draw] (A1) at (6, -5){\begin{tikzpicture}

\oc{m}{M}{}{3}{2}{2}{\en}
\oc{n}{N}{}{1}{4}{4}{\en}
\oc{p}{P}{}{2}{2}{2}{\en}
\oc{t5}{C}{_5}{0}{5}{4}{\et\dk}
\oc{e6}{E}{_6}{5}{5}{2}{\ep\dk}
\oc{ta}{C}{_a}{1}{3}{2}{\et\lt}
\oc{eb}{E}{_b}{2}{4}{3}{\ep\lt}
\gc{g1}{4}{5}{2}

  \begin{scope}[on background layer]
\dblp{({t5}t)\gpt{g1}({e6}b)}
\blp{ ({t5}b)--({n}b)}
\clp{({n}b)-- ({eb}t)}
\clp{ ({ta}b)--({p}b)}
\alp{({p}b)--({m}b)}
\blp{({m}b)\gpb{g1}({e6}t)}
\dclp{({ta}t)--({eb}b)}
\end{scope}
\end{tikzpicture}};

\node[draw] (A2) at (12, -5){\begin{tikzpicture}

\oc{m}{M}{}{2}{4}{4}{\en}
\oc{n}{N}{}{3}{2}{2}{\en}
\oc{p}{P}{}{1}{4}{4}{\en}
\oc{t5}{C}{_5}{2}{3}{2}{\et\dk}
\oc{e6}{E}{_6}{3}{4}{3}{\ep\dk}
\oc{ta}{C}{_a}{0}{5}{4}{\et\lt}
\oc{eb}{E}{_b}{5}{5}{2}{\ep\lt}

\gc{g1}{4}{5}{2}

  \begin{scope}[on background layer]
\dblp{({t5}t)--({e6}b)}
\blp{ ({t5}b)--({n}b)}
\clp{({n}b)\gpb{g1} ({eb}t)}
\clp{ ({ta}b)--({p}b)}
\alp{({p}b)--({m}b)}
\blp{({m}b)--({e6}t)}
\dclp{({ta}t)\gpt{g1}({eb}b)}
\end{scope}
\end{tikzpicture}};

\node[draw] (A3) at (18, -5){\begin{tikzpicture}

\oc{m}{M}{}{2}{4}{4}{\en}
\oc{n}{N}{}{3}{4}{4}{\en}
\oc{p}{P}{}{1}{4}{4}{\en}
\oc{ta}{C}{_a}{0}{5}{4}{\et\lt}
\oc{eb}{E}{_b}{5}{5}{4}{\ep\lt}
\gc{g1}{4}{5}{4}

  \begin{scope}[on background layer]
\clp{ ({ta}b)--({p}b)}
\alp{({p}b)--({m}b)}
\blp{({m}b)--({n}b)}
\clp{({n}b)\gpb{g1} ({eb}t)}
\dclp{({ta}t)\gpt{g1}({eb}b)}
\end{scope}
\end{tikzpicture}};

\draw[->](X5)--(A1)node [midway , fill=white] {$\triangle$};

\draw[<->](A1)--(A2)node [midway , fill=white] {$\gamma$};
\draw[->](A2)--(A3)node [midway , fill=white] {$\triangle^{-1}$};

\node[draw] (B1) at (6, 0){\begin{tikzpicture}

\oc{t1}{C}{_1}{2}{1}{0}{\et}
\oc{m}{M}{}{3}{0}{0}{\en}
\oc{n}{N}{}{1}{4}{4}{\en}
\oc{p}{P}{}{2}{2}{2}{\en}
\oc{t5}{C}{_5}{0}{5}{4}{\et\dk}
\oc{e6}{E}{_6}{5}{5}{0}{\ep\dk}

\oc{ex}{E}{_x}{3}{2}{1}{\ep}
\gc{g1}{4}{5}{0}

\oc{ta}{C}{_a}{1}{3}{2}{\et\lt}
\oc{eb}{E}{_b}{2}{4}{3}{\ep\lt}

  \begin{scope}[on background layer]
\dclp{ ({ta}t)--({eb}b)}
\dlp{({t1}t)--({ex}b)}
\dblp{({t5}t)\gpt{g1}({e6}b)}

\alp{({t1}b)--({m}b)}
\blp{({m}b)\gpb{g1}({e6}t)}
\blp{ ({t5}b)--({n}b)}
\clp{({n}b)--({eb}t)}
\clp{ ({ta}b)--({p}b)}
\alp{({p}b)--({ex}t)}
\end{scope}
\end{tikzpicture}};

\node[draw] (B2) at (12, 0){\begin{tikzpicture}

\oc{t1}{C}{_1}{1}{1}{0}{\et}
\oc{m}{M}{}{2}{0}{0}{\en}
\oc{n}{N}{}{3}{-2}{-2}{\en}
\oc{p}{P}{}{1}{2}{2}{\en}
\oc{t5}{C}{_5}{2}{-1}{-2}{\et\dk}
\oc{e6}{E}{_6}{3}{0}{-1}{\ep\dk}

\oc{ex}{E}{_x}{2}{2}{1}{\ep}
\gc{g1}{4}{3}{-2}

\oc{ta}{C}{_a}{0}{3}{2}{\et\lt}
\oc{eb}{E}{_b}{5}{3}{-2}{\ep\lt}

  \begin{scope}[on background layer]
\dclp{ ({ta}t)\gpt{g1}({eb}b)}
\dlp{({t1}t)--({ex}b)}
\dblp{({t5}t)-- ({e6}b)}

\alp{({t1}b)--({m}b)}
\blp{({m}b)-- ({e6}t)}
\blp{ ({t5}b)--({n}b)}
\clp{({n}b)\gpb{g1}({eb}t)}
\clp{({ta}b)--({p}b)}
\alp{({p}b)--({ex}t)}
\end{scope}
\end{tikzpicture}};

\node[draw] (B3) at (18, 0){\begin{tikzpicture}
\oc{t1}{C}{_1}{1}{1}{0}{\et}
\oc{m}{M}{}{2}{0}{0}{\en}
\oc{n}{N}{}{3}{0}{0}{\en}
\oc{p}{P}{}{1}{2}{2}{\en}

\oc{ex}{E}{_x}{2}{2}{1}{\ep}
\gc{g1}{4}{3}{0}

\oc{ta}{C}{_a}{0}{3}{2}{\et\lt}
\oc{eb}{E}{_b}{5}{3}{0}{\ep\lt}

  \begin{scope}[on background layer]
\dclp{ ({ta}t)\gpt{g1}({eb}b)}
\dlp{ ({t1}t)--({ex}b)}

\alp{({t1}b)--({m}b)}
\blp{({m}b)--({n}b)}
\clp{({n}b)\gpb{g1}({eb}t)}
\clp{ ({ta}b)--({p}b)}
\alp{({p}b)--({ex}t)}
\end{scope}
\end{tikzpicture}};

\node[draw] (C1) at (6, 6){\begin{tikzpicture}

\oc{t1}{C}{_1}{-2}{5}{4}{\et}
\oc{m}{M}{}{-1}{4}{4}{\en}
\oc{n}{N}{}{0}{2}{2}{\en}
\oc{p}{P}{}{1}{0}{0}{\en}
\oc{t5}{C}{_5}{-1}{3}{2}{\et\dk}
\oc{e6}{E}{_6}{0}{4}{3}{\ep\dk}

\oc{ex}{E}{_x}{3}{5}{0}{\ep}
\gc{g1}{2}{5}{0}

\oc{ta}{C}{_a}{0}{1}{0}{\et\lt}
\oc{eb}{E}{_b}{1}{2}{1}{\ep\lt}

  \begin{scope}[on background layer]
\dclp{({ta}t)--({eb}b)}
\dlp{({t1}t)\gpt{g1}({ex}b)}
\dblp{({t5}t)--({e6}b)}

\alp{({t1}b)--({m}b)}
\blp{({m}b)--({e6}t)}
\blp{({t5}b)--({n}b)}
\clp{({n}b)--({eb}t)}
\clp{({ta}b)--({p}b)}
\alp{({p}b)\gpb{g1}({ex}t)}
\end{scope}
\end{tikzpicture}};

\node[draw] (C2) at (6, 11){\begin{tikzpicture}

\oc{t1}{C}{_1}{-2}{3}{2}{\et}
\oc{m}{M}{}{-1}{2}{2}{\en}
\oc{n}{N}{}{0}{2}{2}{\en}
\oc{p}{P}{}{1}{0}{0}{\en}

\oc{ex}{E}{_x}{3}{3}{0}{\ep}
\gc{g1}{2}{3}{0}

\oc{ta}{C}{_a}{0}{1}{0}{\et\lt}
\oc{eb}{E}{_b}{1}{2}{1}{\ep\lt}

  \begin{scope}[on background layer]
\dclp{ ({ta}t)--({eb}b)}
\dlp{({t1}t)\gpt{g1}({ex}b)}

\alp{({t1}b)--({m}b)}
\blp{({m}b)--({n}b)}
\clp{({n}b)--({eb}t)}
\clp{ ({ta}b)--({p}b)}
\alp{({p}b)\gpb{g1}({ex}t)}
\end{scope}
\end{tikzpicture}};

\draw[->](Y3)--(B1)node [midway , fill=white] {$\triangle$};
\draw[->](B1)--(A1)node [midway , fill=white] {$\triangle^{-1}$};
\draw[->](B1)--(B2)node [midway , fill=white] {$\gamma$};
\draw[->](B2)--(A2)node [midway , fill=white] {$\triangle^{-1}$};
\draw[->](C1)--(B1)node [midway , fill=white] {$\gamma$};
\draw[->](C1)--(B2)node [midway , fill=white] {$\gamma$};
\draw[->](C2)--(B3)node [midway , fill=white] {$\gamma$};

\draw[->](Z3)--(C1)node [midway , fill=white] {$\triangle$};
\draw[->](B2)--(B3)node [midway , fill=white] {$\triangle^{-1}$};
\draw[->](B3)--(A3)node [midway , fill=white] {$\triangle^{-1}$};
\draw[->](C2)--(C1)node [midway , fill=white] {$\triangle$};
\draw[->](Z2)--(C2)node [midway , fill=white] {$\triangle$};
\end{tikzpicture}

%% file: II-shadow_isomorphism_unit.tex
\begin{tikzpicture}

\node[draw] (Z2) at (0, 7.5){\begin{tikzpicture}

\oc{t1}{C}{_1}{0}{5}{4}{\et}
\oc{x}{M}{}{1}{4}{4}{\en}
\oc{u}{U}{}{2}{4}{4}{\en}

\oc{ex}{E}{_x}{4}{5}{4}{\ep}
\gc{g1}{3}{5}{4}
  \begin{scope}[on background layer]
\dlp{({t1}t)\gpt{g1}({ex}b)}

\alp{ ({t1}b)--({x}b)--({u}b)\gpb{g1}({ex}t)}
\end{scope}
\end{tikzpicture}};

\node[draw] (Z3) at (0, 4){\begin{tikzpicture}

\oc{t1}{C}{_1}{1}{5}{4}{\et}
\oc{x}{M}{}{2}{4}{4}{\en}
\oc{u}{U}{}{3}{2}{2}{\en}
\oc{t5}{C}{_5}{2}{3}{2}{\et\dk}
\oc{e6}{E}{_6}{3}{4}{3}{\ep\dk}

\oc{ex}{E}{_x}{5}{5}{2}{\ep}
\gc{g1}{4}{5}{2}
  \begin{scope}[on background layer]
\dlp{ ({t1}t)\gpt{g1}({ex}b)}
\dlp{({t5}t)--({e6}b)}

\alp{ ({t1}b)--({x}b)--({e6}t)}
\alp{({t5}b)--({u}b)\gpb{g1}({ex}t)}
\end{scope}
\end{tikzpicture}};

\node[draw] (Y3) at (0, 0){\begin{tikzpicture}

\oc{t1}{C}{_1}{2}{3}{2}{\et}
\oc{x}{M}{}{3}{2}{2}{\en}
\oc{u}{U}{}{2}{4}{4}{\en}
\oc{t5}{C}{_5}{1}{5}{4}{\et\dk}
\oc{e6}{E}{_6}{5}{5}{2}{\ep\dk}

\oc{ex}{E}{_x}{3}{4}{3}{\ep}
\gc{g1}{4}{5}{2}
  \begin{scope}[on background layer]
\dlp{ ({t1}t)--({ex}b)}
\dlp{({t5}t)\gpt{g1}({e6}b)}

\alp{ ({t1}b)--({x}b)\gpb{g1}({e6}t)}
\alp{ ({t5}b)--({u}b)--({ex}t)}
\end{scope}
\end{tikzpicture}};

\node[draw] (X5) at (0, -3.5){\begin{tikzpicture}

\oc{x}{M}{}{3}{4}{4}{\en}
\oc{u}{U}{}{2}{4}{4}{\en}
\oc{t5}{C}{_5}{1}{5}{4}{\et\dk}
\oc{e6}{E}{_6}{5}{5}{4}{\ep\dk}
\gc{g1}{4}{5}{4}
  \begin{scope}[on background layer]
\dlp{({t5}t)\gpt{g1}({e6}b)}
\alp{ ({t5}b)--({u}b)--({x}b)\gpb{g1}({e6}t)}
\end{scope}
\end{tikzpicture}};

\node[draw] (A1) at (5, 7.5){\begin{tikzpicture}

\oc{t1}{C}{_1}{0}{5}{4}{\et}
\oc{x}{M}{}{1}{4}{4}{\en}

\oc{ex}{E}{_x}{4}{5}{4}{\ep}
\gc{g1}{3}{5}{4}
  \begin{scope}[on background layer]
\dlp{ ({t1}t)\gpt{g1}({ex}b)}

\alp{({t1}b)--({x}b)\gpb{g1}({ex}t)}
\end{scope}
\end{tikzpicture}};

\node[draw] (A2) at (5, 4){\begin{tikzpicture}

\oc{t1}{C}{_1}{1}{5}{4}{\et}
\oc{x}{M}{}{2}{4}{4}{\en}
\oc{t5}{C}{_5}{2}{3}{2}{\et\dk}
\oc{e6}{E}{_6}{3}{4}{3}{\ep\dk}

\oc{ex}{E}{_x}{5}{5}{2}{\ep}
\gc{g1}{4}{5}{2}
  \begin{scope}[on background layer]
\dlp{({t1}t)\gpt{g1}({ex}b)}
\dlp{({t5}t)--({e6}b)}

\alp{({t1}b)--({x}b)--({e6}t)}
\alp{({t5}b)\gpb{g1}({ex}t)}
\end{scope}
\end{tikzpicture}};

\node[draw] (A3) at (5, 0){\begin{tikzpicture}

\oc{t1}{C}{_1}{2}{3}{2}{\et}
\oc{x}{M}{}{3}{2}{2}{\en}
\oc{t5}{C}{_5}{1}{5}{4}{\et\dk}
\oc{e6}{E}{_6}{5}{5}{2}{\ep\dk}

\oc{ex}{E}{_x}{3}{4}{3}{\ep}
\gc{g1}{4}{5}{2}
  \begin{scope}[on background layer]
\dlp{({t1}t)--({ex}b)}
\dlp{({t5}t)\gpt{g1}({e6}b)}

\alp{({t1}b)--({x}b)\gpb{g1}({e6}t)}
\alp{({t5}b)--({ex}t)}
\end{scope}
\end{tikzpicture}};

\node[draw] (A4) at (5, -3.5){\begin{tikzpicture}

\oc{x}{M}{}{3}{4}{4}{\en}
\oc{t5}{C}{_5}{1}{5}{4}{\et\dk}
\oc{e6}{E}{_6}{5}{5}{4}{\ep\dk}

\gc{g1}{4}{5}{4}
  \begin{scope}[on background layer]
\dlp{({t5}t)\gpt{g1}({e6}b)}
\alp{({t5}b)--({x}b)\gpb{g1}({e6}t)}
\end{scope}
\end{tikzpicture}};

\node[draw] (B1) at (10, 7.5){\begin{tikzpicture}

\oc{t1}{C}{_1}{0}{5}{4}{\et}
\oc{x}{M}{}{2}{5}{5}{\en}

\oc{ex}{E}{_x}{3}{5}{4}{\ep}
\gc{g1}{1}{5}{4}
  \begin{scope}[on background layer]
\dlp{ ({t1}t)\gpt{g1}({ex}b)}
\alp{ ({t1}b)\gpb{g1}({x}b)--({ex}t)}
\end{scope}
\end{tikzpicture}};

\node[draw] (B2) at (10, 4){\begin{tikzpicture}

\oc{t1}{C}{_1}{2}{5}{2}{\et}
\oc{x}{M}{}{4}{5}{5}{\en}
\oc{t5}{C}{_5}{4}{4}{3}{\et\dk}
\oc{e6}{E}{_6}{5}{5}{4}{\ep\dk}

\oc{ex}{E}{_x}{5}{3}{2}{\ep}
\gc{g1}{3}{5}{2}
  \begin{scope}[on background layer]
\dlp{({t1}t)\gpt{g1}({ex}b)}
\dlp{({t5}t)--({e6}b)}

\alp{ ({t1}b)\gpb{g1}({x}b)--({e6}t)}
\alp{({t5}b)--({ex}t)}
\end{scope}
\end{tikzpicture}};

\node[draw] (B3) at (10, 0){\begin{tikzpicture}

\oc{t1}{C}{_1}{2}{3}{2}{\et}
\oc{x}{M}{}{3}{2}{2}{\en}
\oc{e6}{E}{_6}{5}{3}{2}{\ep\dk}

\gc{g1}{4}{3}{2}
  \begin{scope}[on background layer]
\dlp{({t1}t)\gpt{g1}({e6}b)}
\alp{ ({t1}b)--({x}b)\gpb{g1}({e6}t)}
\end{scope}
\end{tikzpicture}};

\node[draw] (B4) at (10, -3.5){\begin{tikzpicture}

\oc{x}{M}{}{3}{5}{5}{\en}
\oc{t5}{C}{_5}{1}{5}{4}{\et\dk}
\oc{e6}{E}{_6}{4}{5}{4}{\ep\dk}

\gc{g1}{2}{5}{4}
  \begin{scope}[on background layer]
\dlp{({t5}t)\gpt{g1}({e6}b)}
\alp{({t5}b)\gpb{g1}({x}b)--({e6}t)}
\end{scope}
\end{tikzpicture}};

\node[draw] (C2) at (15, 2){\begin{tikzpicture}

\oc{t1}{C}{_1}{2}{5}{4}{\et}
\oc{x}{M}{}{4}{5}{5}{\en}
\oc{e6}{E}{_6}{5}{5}{4}{\ep\dk}

\gc{g1}{3}{5}{4}
  \begin{scope}[on background layer]
\dlp{ ({t1}t)\gpt{g1}({e6}b)}
\alp{({t1}b)\gpb{g1}({x}b)--({e6}t)}
\end{scope}
\end{tikzpicture}};

\draw[->](Z2)--(Z3)node [midway , fill=white] {$\triangle$};
\draw[<->](Z3)--(Y3)node [midway , fill=white] {$\gamma$};
\draw[->](Y3)--(X5)node [midway , fill=white] {$\triangle^{-1}$};
\draw[->](A1)--(A2)node [midway , fill=white] {$\triangle$};
\draw[<->](A2)--(A3)node [midway , fill=white] {$\gamma$};
\draw[->](A3)--(A4)node [midway , fill=white] {$\triangle^{-1}$};
\draw[->](Z2)--(A1)node [midway , fill=white] {$r$};
\draw[->](Z3)--(A2)node [midway , fill=white] {$r$};
\draw[->,out=30,in=150](Y3)to node[midway , fill=white] {$r$} (A3);
\draw[->,out=-30,in=-150](Y3)to node[midway , fill=white] {$l$} (A3);
\draw[->](X5)--(A4)node [midway , fill=white] {$l$};

\draw[->](B1)--(B2)node [midway , fill=white] {$\triangle$};
\draw[<->](A1)--(B1)node [midway , fill=white] {$\gamma$};
\draw[<->](A2)--(B2)node [midway , fill=white] {$\gamma$};
\draw[<->](A4)--(B4)node [midway , fill=white] {$\gamma$};
\draw[->](B2)--(C2)node [midway , fill=white] {$\triangle^{-1}$};
\draw[->](B1)--(C2)node [midway , fill=white] {$=$};

\draw[->](B4)--(C2)node [midway , fill=white] {$=$};
\draw[<->](B3)--(C2)node [midway , fill=white] {$\gamma$};

\draw[->](A3)--(B3)node [midway , fill=white] {$\triangle^{-1}$};
\end{tikzpicture}

%% file: II-dual_pair_2_cell.tex
    \centering
    \begin{subfigure}[t]{0.2\textwidth}
        \centering
		\begin{tikzpicture}

		\node[draw] (A1) at (0, 3){\begin{tikzpicture}
		\node (ta) at (2,-1){};
		\node (ta) at (4,-2){};
		\end{tikzpicture}
		};

		\node[draw] (A2) at (0, 0){\begin{tikzpicture}
		\oc{ta}{C}{_a}{2}{-1}{-2}{\et}
		\oc{t2}{\rdual{C}}{_2}{5}{-1}{-2}{\etr}
		  \begin{scope}[on background layer]
		\dlp{ ({ta}t)--({t2}t)}
		\alp{({ta}b)--({t2}b)}
		\end{scope}
\end{tikzpicture}
};

\draw[->](A1)--(A2)node [midway , fill=white] {$I_C$};
\end{tikzpicture}

        \caption{Coevaluation}
    \end{subfigure}%
\hfill 
    \begin{subfigure}[t]{0.2\textwidth}
        \centering
\begin{tikzpicture}

\node[draw] (A1) at (0, 3){\begin{tikzpicture}
\oc{t1}{C}{_1}{6}{-1}{-2}{\et}
\oc{y}{}{Y}{7}{-1}{-2}{\en}
\oc{x}{}{X}{4}{-1}{-2}{\en}

\oc{t2}{\rdual{C}}{_2}{5}{-1}{-2}{\etr}
  \begin{scope}[on background layer]
\dlp{({x}t)--({t2}t)}
\dlp{({t1}t)--({y}t)}
\alp{({t1}b)--({y}b)}

\alp{({x}b)--({t2}b)}
\end{scope}
\end{tikzpicture}
};

\node[draw] (A2) at (0, 0){\begin{tikzpicture}
\oc{y}{}{Y}{7}{-1}{-2}{\en}
\oc{x}{}{X}{4}{-1}{-2}{\en}
  \begin{scope}[on background layer]
\dlp{({x}t)--({y}t)}
\alp{({x}b)--({y}b)}
\end{scope}
\end{tikzpicture}
};

\draw[->](A1)--(A2)node [midway , fill=white] {$P_C$};
\end{tikzpicture}
        \caption{Evaluation}
    \end{subfigure}
\hfill 
\begin{subfigure}[t]{0.42\textwidth}
        \centering
\begin{tikzpicture}

\node[draw] (A1) at (0, 6){\begin{tikzpicture}
\oc{x}{}{X}{4}{-1}{-2}{\en}

\oc{t2}{\rdual{C}}{_2}{5}{-1}{-2}{\etr}
  \begin{scope}[on background layer]
\dlp{({x}t)--({t2}t)}
\alp{({x}b)--({t2}b)}
\end{scope}
\end{tikzpicture}
};

\node[draw] (A2) at (0, 3){\begin{tikzpicture}
\oc{t1}{C}{_1}{6}{-1}{-2}{\et}
\oc{t3}{\rdual{C}}{_3}{7}{-1}{-2}{\etr\lt}
\oc{x}{}{X}{4}{-1}{-2}{\en}

\oc{t2}{\rdual{C}}{_2}{5}{-1}{-2}{\etr}
  \begin{scope}[on background layer]
\dlp{ ({x}t)--({t2}t)}
\dlp{ ({t1}t)--({t3}t)}
\alp{({t1}b)--({t3}b)}
\alp{({x}b)--({t2}b)}
\end{scope}
\end{tikzpicture}
};

\node[draw] (A3) at (0, 0){\begin{tikzpicture}
\oc{t3}{\rdual{C}}{_3}{7}{-1}{-2}{\etr\lt}
\oc{x}{}{X}{4}{-1}{-2}{\en}
  \begin{scope}[on background layer]
\dlp{({x}t)--({t3}t)}
\alp{({x}b)--({t3}b)}
\end{scope}
\end{tikzpicture}
};

\draw[->](A1)--(A2)node [midway , fill=white] {$I_C$};
\draw[->](A2)--(A3)node [midway , fill=white] {$P_C$};
\end{tikzpicture}
\hfill 
\begin{tikzpicture}

\node[draw] (A1) at (0, 6){\begin{tikzpicture}
\oc{t1}{C}{_1}{6}{-1}{-2}{\et}
\oc{y}{}{Y}{7}{-1}{-2}{\en}
  \begin{scope}[on background layer]
\dlp{({t1}t)--({y}t)}
\alp{({t1}b)--({y}b)}
\end{scope}
\end{tikzpicture}
};

\node[draw] (A2) at (0, 3){\begin{tikzpicture}
\oc{t1}{C}{_1}{6}{-1}{-2}{\et}
\oc{y}{}{Y}{7}{-1}{-2}{\en}
\oc{t3}{C}{_3}{4}{-1}{-2}{\et\lt}

\oc{t2}{\rdual{C}}{_2}{5}{-1}{-2}{\etr}
  \begin{scope}[on background layer]
\dlp{ ({t3}t)--({t2}t)}
\dlp{({t1}t)--({y}t)}
\alp{({t1}b)--({y}b)}
\alp{({t3}b)--({t2}b)}
\end{scope}
\end{tikzpicture}
};

\node[draw] (A3) at (0, 0){\begin{tikzpicture}
\oc{y}{}{Y}{7}{-1}{-2}{\en}
\oc{t3}{C}{_3}{4}{-1}{-2}{\et\lt}
  \begin{scope}[on background layer]
\dlp{ ({t3}t)--({y}t)}
\alp{({t3}b)--({y}b)}
\end{scope}
\end{tikzpicture}
};

\draw[->](A1)--(A2)node [midway , fill=white] {$I_C$};
\draw[->](A2)--(A3)node [midway , fill=white] {$P_C$};
\end{tikzpicture}
        \caption{Triangle identities}
    \end{subfigure}

%% file: II-dual_pair_2_cell_flip.tex
    \begin{subfigure}[t]{0.2\textwidth}
        \centering
\begin{tikzpicture}

\node[draw] (A1) at (0, -3){\begin{tikzpicture}
\oc{e1}{\rdual{E}}{_1}{6}{-1}{-2}{\epr}
\oc{y}{}{Y}{7}{-1}{-2}{\en}
\oc{x}{}{X}{4}{-1}{-2}{\en}

\oc{e2}{{E}}{_2}{5}{-1}{-2}{\ep}
  \begin{scope}[on background layer]
\dlp{({x}b)--({t2}b)}
\dlp{ ({t1}b)--({y}b)}
\alp{ ({t1}t)--({y}t)}
\alp{({x}t)--({t2}t)}
\end{scope}
\end{tikzpicture}
};

\node[draw] (A2) at (0, 0){\begin{tikzpicture}
\oc{y}{}{Y}{7}{-1}{-2}{\en}
\oc{x}{}{X}{4}{-1}{-2}{\en}
  \begin{scope}[on background layer]
\dlp{({x}b)--({y}b)}
\alp{({x}t)--({y}t)}
\end{scope}
\end{tikzpicture}
};

\draw[->](A2)--(A1)node [midway , fill=white] {$C_E$};
\end{tikzpicture}
        \caption{Coevaluation}
    \end{subfigure}
\hfill 
    \centering
    \begin{subfigure}[t]{0.2\textwidth}
        \centering
		\begin{tikzpicture}

		\node[draw] (A1) at (0, -3){\begin{tikzpicture}
		\node (ta) at (2,-1){};
		\node (ta) at (4,-2){};
		\end{tikzpicture}
		};

		\node[draw] (A2) at (0, 0){\begin{tikzpicture}
		\oc{ea}{\rdual{E}}{_a}{2}{-1}{-2}{\epr}
		\oc{e2}{{E}}{_2}{5}{-1}{-2}{\ep}
	  \begin{scope}[on background layer]
		\dlp{ ({ea}b)--({e2}b)}
		\alp{({ea}t)--({e2}t)}
	\end{scope}
\end{tikzpicture}
};

\draw[->](A2)--(A1)node [midway , fill=white] {$L_E$};
\end{tikzpicture}

        \caption{Evaluation}
    \end{subfigure}%
\hfill 
\begin{subfigure}[t]{0.42\textwidth}
        \centering
\begin{tikzpicture}
\node[draw] (A2) at (0, 0){\begin{tikzpicture}
\oc{e3}{E}{_3}{7}{-1}{-2}{\ep\lt}
\oc{x}{}{X}{4}{-1}{-2}{\en}
  \begin{scope}[on background layer]
\dlp{({x}b)--({e3}b)}
\alp{({x}t)--({e3}t)}
\end{scope}
\end{tikzpicture}
};

\node[draw] (A1) at (0, -3){\begin{tikzpicture}
\oc{e1}{\rdual{E}}{_1}{6}{-1}{-2}{\epr}
\oc{e3}{E}{_3}{7}{-1}{-2}{\ep\lt}
\oc{x}{}{X}{4}{-1}{-2}{\en}

\oc{e2}{{E}}{_2}{5}{-1}{-2}{\ep}
  \begin{scope}[on background layer]
\dlp{({x}b)--({t2}b)}
\dlp{ ({t1}b)--({e3}b)}
\alp{ ({t1}t)--({e3}t)}
\alp{({x}t)--({t2}t)}
\end{scope}
\end{tikzpicture}
};

\node[draw] (A3) at (0, -6){\begin{tikzpicture}
\oc{x}{}{X}{4}{-1}{-2}{\en}
\oc{e2}{{E}}{_2}{5}{-1}{-2}{\ep}

  \begin{scope}[on background layer]
\dlp{({x}b)--({t2}b)}
\alp{({x}t)--({t2}t)}
\end{scope}
\end{tikzpicture}
};

\draw[->](A2)--(A1)node [midway , fill=white] {$C_E$};
\draw[->](A1)--(A3)node [midway , fill=white] {$L_E$};
\end{tikzpicture}
\hfill 
\begin{tikzpicture}
\node[draw] (A2) at (0, 0){\begin{tikzpicture}
\oc{y}{}{Y}{7}{-1}{-2}{\en}
\oc{e3}{\rdual{E}}{_3}{4}{-1}{-2}{\epr\lt}

  \begin{scope}[on background layer]
\dlp{({e3}b)--({y}b)}
\alp{({e3}t)--({y}t)}
\end{scope}
\end{tikzpicture}
};

\node[draw] (A1) at (0, -3){\begin{tikzpicture}
\oc{e1}{\rdual{E}}{_1}{6}{-1}{-2}{\epr}
\oc{y}{}{Y}{7}{-1}{-2}{\en}
\oc{e3}{\rdual{E}}{_3}{4}{-1}{-2}{\epr\lt}

\oc{e2}{{E}}{_2}{5}{-1}{-2}{\ep}

  \begin{scope}[on background layer]
\dlp{({e3}b)--({t2}b)}
\dlp{({t1}b)--({y}b)}
\alp{({t1}t)--({y}t)}
\alp{({e3}t)--({t2}t)}
\end{scope}
\end{tikzpicture}
};

\node[draw] (A3) at (0, -6){\begin{tikzpicture}
\oc{e1}{\rdual{E}}{_1}{6}{-1}{-2}{\epr}
\oc{y}{}{Y}{7}{-1}{-2}{\en}

  \begin{scope}[on background layer]
\dlp{({t1}b)--({y}b)}
\alp{({t1}t)--({y}t)}
\end{scope}
\end{tikzpicture}
};

\draw[->](A2)--(A1)node [midway , fill=white] {$C_E$};
\draw[->](A1)--(A3)node [midway , fill=white] {$L_E$};
\end{tikzpicture}
        \caption{Triangle identities}
    \end{subfigure}

%% file: II-dual_coeval_eval_for_1_cells.tex
    \centering
    \begin{subfigure}[t]{.65\textwidth}
\resizebox{\textwidth}{!}
{
\begin{tikzpicture}
\node[draw] (A1) at (0, 3.5){\begin{tikzpicture}
\node (ta) at (2,-1){};
\node (ta) at (4,-2){};
\end{tikzpicture}
};

\node[draw] (A2) at (4, 3.5){\begin{tikzpicture}
\oc{ta}{C}{_a}{3}{-1}{-2}{\et}
\oc{t2}{\rdual{C}}{_2}{5}{-1}{-2}{\etr}
  \begin{scope}[on background layer]
\dlp{ ({ta}t)--({t2}t)}
\alp{({ta}b)--({t2}b)}
\end{scope}
\end{tikzpicture}
};

\node[draw] (A3) at (8, 3.5){\begin{tikzpicture}
\oc{ta}{C}{_a}{2}{2}{1}{\et}
\oc{t2}{\rdual{C}}{_2}{4}{2}{-1}{\etr}
\oc{t3}{C}{_3}{2}{0}{-1}{\et\lt}
\oc{e4}{E}{_4}{3}{1}{0}{\ep}

  \begin{scope}[on background layer]
\dlp{({ta}t)--({t2}t)}
\dlp{({t3}t)--({e4}b)}
\alp{({t3}b)--({t2}b)}
\alp{({ta}b)--({e4}t)}
\end{scope}
\end{tikzpicture}
};

\node[draw] (A4) at (12, 3.5){\begin{tikzpicture}
\oc{ta}{C}{_a}{2}{2}{1}{\et}
\oc{t2}{\rdual{C}}{_2}{4}{1}{0}{\etr}
\oc{t3}{C}{_3}{2}{0}{-1}{\et\lt}
\oc{e4}{E}{_4}{5}{2}{-1}{\ep}

\gc{g1}{3}{2}{1}
\gc{g2}{3}{0}{-1}
  \begin{scope}[on background layer]
\dlp{ ({ta}t)\gpt{g1}({t2}t)}
\dlp{({t3}t)\gpt{g2}({e4}b)}
\alp{({t3}b)\gpb{g2}({t2}b)}
\alp{({ta}b)\gpb{g1}({e4}t)}
\end{scope}
\end{tikzpicture}
};
\draw[->](A1)--(A2)node [midway , fill=white] {$I_C$};
\draw[->](A2)--(A3)node [midway , fill=white] {$\triangle$};
\draw[<->](A3)--(A4)node [midway , fill=white] {$\gamma$};
\end{tikzpicture}
}
        \caption{Coevaluation $I_E$}
    \end{subfigure}
    ~ 
    \begin{subfigure}[t]{\textwidth}
\resizebox{\textwidth}{!}{
\begin{tikzpicture}

\node[draw] (A1) at (-5.5, 6.5){\begin{tikzpicture}
\oc{ta}{C}{_a}{2}{2}{1}{\et}
\oc{t2}{\rdual{C}}{_2}{4}{1}{0}{\etr}
\oc{t3}{C}{_3}{2}{0}{-1}{\et\lt}
\oc{e4}{E}{_4}{1}{2}{-1}{\ep}
\oc{y}{}{Y}{5}{2}{-1}{\en}
\oc{x}{}{X}{0}{2}{-1}{\en}

\gc{g1}{3}{2}{1}
\gc{g2}{3}{0}{-1}

  \begin{scope}[on background layer]
\dlp{ ({ta}t)\gpt{g1}({t2}t)}
\dlp{({t3}t)\gpt{g2}({y}b)}
\dlp{({x}b)--		({e4}b)}
\alp{({t3}b)\gpb{g2}({t2}b)}
\alp{({ta}b)\gpb{g1} ({y}t)}	
\alp{({x}t)--({e4}t)}
\end{scope}
\end{tikzpicture}
};

\node[draw] (A2) at (0, 6.5){\begin{tikzpicture}
\oc{ta}{C}{_a}{2}{2}{1}{\et}
\oc{t2}{\rdual{C}}{_2}{4}{1}{0}{\etr}
\oc{t3}{C}{_3}{2}{0}{-1}{\et\lt}
\oc{e4}{E}{_4}{1}{0}{-1}{\ep}
\oc{y}{}{Y}{5}{2}{-1}{\en}
\oc{x}{}{X}{-1}{2}{-1}{\en}
\oc{t5}{C}{_5}{0}{1}{0}{\et\dk}
\oc{e6}{E}{_6}{1}{2}{1}{\ep\dk}

\gc{g1}{3}{2}{1}
\gc{g2}{3}{0}{-1}

  \begin{scope}[on background layer]
\dlp{({t5}t)--({e6}b)}
\dlp{({ta}t)\gpt{g1}({t2}t)}
\dlp{({t3}t)\gpt{g2}({y}b)}
\dlp{ ({x}b)--		({e4}b)}
\alp{({t3}b)\gpb{g2}({t2}b)}
\alp{({ta}b)\gpb{g1} ({y}t)}		
\alp{({x}t)--({e6}t)}
\alp{({t5}b)--({e4}t)}
\end{scope}
\end{tikzpicture}
};

\node[draw] (A3) at (6, 6.4){\begin{tikzpicture}
\oc{ta}{C}{_a}{0}{2}{1}{\et}
\oc{t2}{\rdual{C}}{_2}{2}{1}{0}{\etr}
\oc{t3}{C}{_3}{0}{0}{-1}{\et\lt}
\oc{e4}{E}{_4}{4}{0}{-1}{\ep}
\oc{y}{}{Y}{5}{3}{-2}{\en}
\oc{x}{}{X}{-1}{3}{-2}{\en}
\oc{t5}{C}{_5}{3}{1}{0}{\et\dk}
\oc{e6}{E}{_6}{4}{2}{1}{\ep\dk}
\gc{g1}{1}{3}{1}
\gc{g2}{1}{0}{-2}
 
  \begin{scope}[on background layer]
\dlp{({t5}t)--({e6}b)}
\dlp{({ta}t)--({g1}ml)--({g1}br)--({t2}t)}
\dlp{({t3}t)--({g2}tl)--({g2}br)--({y}b)}
\dlp{({x}b)--({g2}bl)--({g2}mr)--	({e4}b)}
\alp{({t3}b)--({g2}ml)--({g2}tr)--({t2}b)}
\alp{({ta}b)--({g1}bl)--({g1}tr)--({y}t)}
\alp{({x}t)--({g1}tl)--({g1}mr)--({e6}t)}
\alp{({t5}b)--({e4}t)}
\end{scope}
\end{tikzpicture}
};

\node[draw] (A4) at (11.5, 6.5){\begin{tikzpicture}
\oc{ta}{C}{_a}{0}{2}{1}{\et}
\oc{t3}{C}{_3}{0}{0}{-1}{\et\lt}
\oc{e4}{E}{_4}{2}{0}{-1}{\ep}
\oc{y}{}{Y}{3}{3}{-2}{\en}
\oc{x}{}{X}{-1}{3}{-2}{\en}
\oc{e6}{E}{_6}{2}{2}{1}{\ep\dk}
\gc{g1}{1}{3}{1}
\gc{g2}{1}{0}{-2}

  \begin{scope}[on background layer]
\dlp{({ta}t)--({g1}ml)--({g1}br)--({e6}b)}
\dlp{({t3}t)--({g2}tl)--({g2}br)--({y}b)}
\dlp{ ({x}b)--({g2}bl)--({g2}mr)--	({e4}b)}
\alp{({t3}b)--({g2}ml)--({g2}tr)--({e4}t)}
\alp{({ta}b)--({g1}bl)--({g1}tr)--({y}t)}
\alp{({x}t)--({g1}tl)--({g1}mr)--({e6}t)}
\end{scope}
\end{tikzpicture}
};

\node[draw] (A5) at ( 16.5,6.5){\begin{tikzpicture}
\oc{y}{}{Y}{7}{-1}{-2}{\en}
\oc{x}{}{X}{4}{-1}{-2}{\en}

  \begin{scope}[on background layer]
\dlp{({x}t)--({y}t)}
\alp{({x}b)--({y}b)}
\end{scope}
\end{tikzpicture}
};

\draw[->](A1)--(A2)node [midway , fill=white] {$\triangle$};
\draw[<->](A2)--(A3)node [midway , fill=white] {$\gamma$};
\draw[->](A3)--(A4)node [midway , fill=white] {$P_C$};
\draw[->](A4)--(A5)node [midway , fill=white] {$(\triangle^{-1})^2$};
\end{tikzpicture}
}
        \caption{Evaluation $P_E$}
    \end{subfigure}

%% file: II-compatibility_left_right_duals.tex
\begin{tikzpicture}

\node[draw] (B3) at (0, 20){\begin{tikzpicture}
\oc{e4}{E}{_4}{1}{2}{1}{\ep}
\oc{x}{}{X}{-1}{2}{1}{\en}

  \begin{scope}[on background layer]
\dlp{({x}b)--({e4}b)}
\alp{ ({x}t)--({e4}t)}
\end{scope}
\end{tikzpicture}
};

\node[draw] (B2) at (0, 16){\begin{tikzpicture}
\oc{ta}{C}{_a}{2}{2}{1}{\et}
\oc{t2}{\rdual{C}}{_2}{5}{2}{1}{\etr}
\oc{e4}{E}{_4}{1}{2}{1}{\ep}
\oc{x}{}{X}{-1}{2}{1}{\en}

  \begin{scope}[on background layer]
\dlp{({ta}t)--({t2}t)}
\dlp{({x}b)--		({e4}b)}
\alp{({ta}b)--({t2}b)}
\alp{ ({x}t)--({e4}t)}
\end{scope}
\end{tikzpicture}
};

\node[draw] (B1) at (0, 8){\begin{tikzpicture}
\oc{ta}{C}{_a}{2}{2}{1}{\et}
\oc{t2}{\rdual{C}}{_2}{5}{2}{-1}{\etr}
\oc{t3}{C}{_3}{2}{0}{-1}{\et\lt}
\oc{e4}{E}{_4}{1}{2}{-1}{\ep}
\oc{e1}{E}{_1}{4}{1}{0}{\ep\lt}
\oc{x}{}{X}{-1}{2}{-1}{\en}

  \begin{scope}[on background layer]
\dlp{({ta}t)--({t2}t)}
\dlp{ ({t3}t)--({e1}b)}
\dlp{({x}b)--		({e4}b)}
\alp{({t3}b)--({t2}b)}
\alp{({ta}b)--({e1}t)}
\alp{({x}t)--({e4}t)}
\end{scope}
\end{tikzpicture}
};

\node[draw] (A1) at (0, 4){\begin{tikzpicture}
\oc{ta}{C}{_a}{2}{2}{1}{\et}
\oc{t2}{\rdual{C}}{_2}{4}{1}{0}{\etr}
\oc{t3}{C}{_3}{2}{0}{-1}{\et\lt}
\oc{e4}{E}{_4}{1}{2}{-1}{\ep}
\oc{e1}{E}{_1}{5}{2}{-1}{\ep\lt}
\oc{x}{}{X}{-1}{2}{-1}{\en}

\gc{g1}{3}{2}{1}
\gc{g2}{3}{0}{-1}

  \begin{scope}[on background layer]
\dlp{({ta}t)\gpt{g1}({t2}t)}
\dlp{({t3}t)\gpt{g2}({e1}b)}
\dlp{({x}b)--		({e4}b)}
\alp{({t3}b)\gpb{g2}({t2}b)}
\alp{({ta}b)\gpb{g1} ({e1}t)}
\alp{({x}t)--({e4}t)}
\end{scope}
\end{tikzpicture}
};

\node[draw] (A2) at (6, 4){\begin{tikzpicture}
\oc{ta}{C}{_a}{2}{2}{1}{\et}
\oc{t2}{\rdual{C}}{_2}{4}{1}{0}{\etr}
\oc{t3}{C}{_3}{2}{0}{-1}{\et\lt}
\oc{e1}{E}{_!}{5}{2}{-1}{\ep\lt}

\gc{g1}{3}{2}{1}
\gc{g2}{3}{0}{-1}

\oc{e4}{E}{_4}{1}{0}{-1}{\ep}
\oc{x}{}{X}{-1}{2}{-1}{\en}
\oc{t5}{C}{_5}{0}{1}{0}{\et\dk}
\oc{e6}{E}{_6}{1}{2}{1}{\ep\dk}

  \begin{scope}[on background layer]
\alp{ ({x}t)--({e6}t)}
\alp{ ({t5}b)--({e4}t)}
\dlp{({x}b)--({e4}b)}
\dlp{({t5}t)--({e6}b)}

\dlp{({ta}t)\gpt{g1}({t2}t)}
\dlp{({t3}t)\gpt{g2}({e1}b)}
\alp{({t3}b)\gpb{g2}({t2}b)}
\alp{({ta}b)\gpb{g1} ({e1}t)}		
\end{scope}
\end{tikzpicture}
};

\node[draw] (A3) at (12, 4){\begin{tikzpicture}
\oc{ta}{C}{_a}{0}{2}{1}{\et}
\oc{t2}{\rdual{C}}{_2}{2}{1}{0}{\etr}
\oc{t3}{C}{_3}{0}{0}{-1}{\et\lt}
\oc{e4}{E}{_4}{4}{0}{-1}{\ep}
\oc{e1}{E}{_1}{5}{3}{-2}{\ep\lt}
\oc{x}{}{X}{-1}{3}{-2}{\en}
\oc{t5}{C}{_5}{3}{1}{0}{\et\dk}
\oc{e6}{E}{_6}{4}{2}{1}{\ep\dk}
\gc{g1}{1}{3}{1}
\gc{g2}{1}{0}{-2}

  \begin{scope}[on background layer]
\dlp{({t5}t)--({e6}b)}
\dlp{({ta}t)--({g1}ml)--({g1}br)--({t2}t)}
\dlp{({t3}t)--({g2}tl)--({g2}br)--({e1}b)}
\dlp{({x}b)--({g2}bl)--({g2}mr)--	({e4}b)}
\alp{({t3}b)--({g2}ml)--({g2}tr)--({t2}b)}
\alp{({ta}b)--({g1}bl)--({g1}tr)--({e1}t)}
\alp{ ({x}t)--({g1}tl)--({g1}mr)--({e6}t)}
\alp{ ({t5}b)--({e4}t)}
\end{scope}
\end{tikzpicture}
};

\node[draw] (A4) at (23, 4){\begin{tikzpicture}
\oc{ta}{C}{_a}{0}{2}{1}{\et}
\oc{t3}{C}{_3}{0}{0}{-1}{\et\lt}
\oc{e4}{E}{_4}{4}{0}{-1}{\ep}
\oc{e1}{E}{_1}{5}{3}{-2}{\ep\lt}
\oc{x}{}{X}{-1}{3}{-2}{\en}
\oc{e6}{E}{_6}{4}{2}{1}{\ep\dk}
\gc{g1}{1}{3}{1}
\gc{g2}{1}{0}{-2}

  \begin{scope}[on background layer]
\dlp{({ta}t)--({g1}ml)--({g1}br)--({e6}b)}
\dlp{ ({t3}t)--({g2}tl)--({g2}br)--({e1}b)}
\dlp{({x}b)--({g2}bl)--({g2}mr)--	({e4}b)}
\alp{({t3}b)--({g2}ml)--({g2}tr)--({e4}t)}
\alp{({ta}b)--({g1}bl)--({g1}tr)--({e1}t)}
\alp{ ({x}t)--({g1}tl)--({g1}mr)--({e6}t)}
\end{scope}
\end{tikzpicture}
};

\node[draw] (A5) at (23, 20){\begin{tikzpicture}
\oc{e1}{E}{_1}{7}{-1}{-2}{\ep\lt}
\oc{x}{}{X}{4}{-1}{-2}{\en}

  \begin{scope}[on background layer]
\dlp{({x}t)--({e1}t)}
\alp{({x}b)--({e1}b)}
\end{scope}
\end{tikzpicture}
};

\node[draw] (C1) at (6, 12){\begin{tikzpicture}
\oc{ta}{C}{_a}{2}{2}{1}{\et}
\oc{t2}{\rdual{C}}{_2}{5}{2}{1}{\etr}

\oc{e4}{E}{_4}{1}{0}{-1}{\ep}
\oc{x}{}{X}{-1}{2}{-1}{\en}
\oc{t5}{C}{_5}{0}{1}{0}{\et\dk}
\oc{e6}{E}{_6}{1}{2}{1}{\ep\dk}

  \begin{scope}[on background layer]
\alp{ ({x}t)--({e6}t)}
\alp{ ({t5}b)--({e4}t)}
\dlp{({x}b)--({e4}b)}
\dlp{({t5}t)--({e6}b)}

\dlp{({ta}t)--({t2}t)}
\alp{({ta}b)--({t2}b)}
\end{scope}
\end{tikzpicture}
};

\node[draw] (C2) at (6, 8){\begin{tikzpicture}
\oc{ta}{C}{_a}{2}{2}{1}{\et}
\oc{t2}{\rdual{C}}{_2}{5}{2}{-1}{\etr}
\oc{t3}{C}{_3}{2}{0}{-1}{\et\lt}
\oc{e1}{E}{_1}{4}{1}{0}{\ep\lt}

\oc{e4}{E}{_4}{1}{0}{-1}{\ep}
\oc{x}{}{X}{-1}{2}{-1}{\en}
\oc{t5}{C}{_5}{0}{1}{0}{\et\dk}
\oc{e6}{E}{_6}{1}{2}{1}{\ep\dk}

  \begin{scope}[on background layer]
\alp {({x}t)--({e6}t)}
\alp{({t5}b)--({e4}t)}
\dlp{({x}b)--({e4}b)}
\dlp{({t5}t)--({e6}b)}

\dlp{({ta}t)--({t2}t)}
\dlp{({t3}t)--({e1}b)} 
\alp{({t3}b)--({t2}b)}
\alp{({ta}b)--({e1}t)}
\end{scope}
\end{tikzpicture}
};

\node[draw] (D1) at (6, 16){\begin{tikzpicture}
\oc{ta}{C}{_a}{0}{2}{1}{\et}
\oc{t2}{\rdual{C}}{_2}{2}{2}{1}{\etr}
\oc{e4}{E}{_4}{3}{3}{0}{\ep}
\oc{x}{}{X}{-1}{3}{0}{\en}

  \begin{scope}[on background layer]
\dlp{({ta}t)--({t2}t)}
\dlp{({x}b)--({e4}b)}
\alp{({ta}b)--({t2}b)}
\alp{ ({x}t)--({e4}t)}
\end{scope}
\end{tikzpicture}
};

\node[draw] (D2) at (12, 16){\begin{tikzpicture}
\oc{ta}{C}{_a}{0}{1}{0}{\et}
\oc{t2}{\rdual{C}}{_2}{2}{1}{0}{\etr}

\oc{e4}{E}{_4}{4}{0}{-1}{\ep}
\oc{x}{}{X}{-1}{2}{-1}{\en}
\oc{t5}{C}{_5}{3}{1}{0}{\et\dk}
\oc{e6}{E}{_6}{4}{2}{1}{\ep\dk}

  \begin{scope}[on background layer]
\alp{ ({x}t)--({e6}t)}
\alp{ ({t5}b)--({e4}t)}
\dlp{ ({x}b)--({e4}b)}
\dlp{({t5}t)--({e6}b)}

\dlp{({ta}t)--({t2}t)}
\alp{({ta}b)--({t2}b)}
\end{scope}
\end{tikzpicture}
};

\node[draw] (D3) at (12, 10){\begin{tikzpicture}
\oc{ta}{C}{_a}{0}{2}{1}{\et}
\oc{t2}{\rdual{C}}{_2}{2}{2}{-1}{\etr}
\oc{t3}{C}{_3}{0}{0}{-1}{\et\lt}
\oc{e1}{E}{_1}{1}{1}{0}{\ep\lt}

\oc{e4}{E}{_4}{4}{-1}{-2}{\ep}
\oc{x}{}{X}{-1}{3}{-2}{\en}
\oc{t5}{C}{_5}{3}{2}{-1}{\et\dk}
\oc{e6}{E}{_6}{4}{3}{2}{\ep\dk}

  \begin{scope}[on background layer]
\alp{({x}t)--({e6}t)}
\alp{({t5}b)--({e4}t)}
\dlp{({x}b)--({e4}b)}
\dlp{({t5}t)--({e6}b)}

\dlp{({ta}t)--({t2}t)}
\dlp{({t3}t)--({e1}b)}
\alp{({t3}b)--({t2}b)}
\alp{({ta}b)--({e1}t)}
\end{scope}	
\end{tikzpicture}
};

\node[draw] (E1) at (18, 20){\begin{tikzpicture}
\oc{ta}{C}{_a}{0}{1}{0}{\et}

\oc{e4}{E}{_4}{4}{0}{-1}{\ep}
\oc{x}{}{X}{-1}{2}{-1}{\en}
\oc{e6}{E}{_6}{4}{2}{1}{\ep\dk}

  \begin{scope}[on background layer]
\alp{({x}t)--({e6}t)}
\alp{({ta}b)--({e4}t)}
\dlp{({x}b)--({e4}b)}
\dlp{({ta}t)--({e6}b)}
\end{scope}
\end{tikzpicture}
};

\node[draw] (E2) at (18, 10){\begin{tikzpicture}
\oc{ta}{C}{_a}{0}{2}{1}{\et}
\oc{t3}{C}{_3}{0}{0}{-1}{\et\lt}
\oc{e1}{E}{_1}{1}{1}{0}{\ep\lt}

\oc{e4}{E}{_4}{4}{-1}{-2}{\ep}
\oc{x}{}{X}{-1}{3}{-2}{\en}
\oc{e6}{E}{_6}{4}{3}{2}{\ep\dk}

  \begin{scope}[on background layer]
\alp{({x}t)--({e6}t)}
\alp{({t3}b)--({e4}t)}

\dlp{ ({ta}t)--({e6}b)}
\dlp{({x}b)--({e4}b)}
\dlp{({t3}t)--({e1}b)}
\alp{({ta}b)--({e1}t)}
\end{scope}
\end{tikzpicture}
};

\node[draw] (F1) at (6, 20){\begin{tikzpicture}
\oc{e4}{E}{_4}{4}{0}{-1}{\ep}
\oc{x}{}{X}{-1}{2}{-1}{\en}
\oc{t5}{C}{_5}{3}{1}{0}{\et\dk}
\oc{e6}{E}{_6}{4}{2}{1}{\ep\dk}

  \begin{scope}[on background layer]
\alp{ ({x}t)--({e6}t)}
\alp{ ({t5}b)--({e4}t)}
\dlp{({x}b)--({e4}b)}
\dlp{({t5}t)--({e6}b)}
\end{scope}
\end{tikzpicture}
};

\draw[->](B3)--(B2)node [midway , fill=white] {$I_C$};
\draw[->](B2)--(B1)node [midway , fill=white] {$\triangle$};
\draw[<->](B1)--(A1)node [midway , fill=white] {$\gamma$};
\draw[->](A1)--(A2)node [midway , fill=white] {$\triangle$};
\draw[<->](A2)--(A3)node [midway , fill=white] {$\gamma$};
\draw[->](A3)--(A4)node [midway , fill=white] {$P_C$};
\draw[->](A4)--(A5)node [midway , fill=white] {$(\triangle^{-1})^2$};
\draw[->](B2)--(C1)node [midway , fill=white] {$\triangle$};
\draw[->](C1)--(C2)node [midway , fill=white] {$\triangle$};
\draw[->](B1)--(C2)node [midway , fill=white] {$\triangle$};
\draw[<->](C2)--(A2)node [midway , fill=white] {$\gamma$};
\draw[->](D1)--(D2)node [midway , fill=white] {$\triangle$};
\draw[->](B3)--(D1)node [midway , fill=white] {$I_C$};
\draw[<->](B2)--(D1)node [midway , fill=white] {$\gamma$};
\draw[<->](C1)--(D2)node [midway , fill=white] {$\gamma$};
\draw[->](D2)--(D3)node [midway , fill=white] {$\triangle$};
\draw[<->](C2)--(D3)node [midway , fill=white] {$\gamma$};
\draw[<->](D3)--(A3)node [midway , fill=white] {$\gamma$};
\draw[->](D2)--(E1)node [midway , fill=white] {$P_C$};
\draw[->](D3)--(E2)node [midway , fill=white] {$P_C$};
\draw[<->](E2)--(A4)node [midway , fill=white] {$\gamma$};
\draw[->](E2)--(A5)node [midway , fill=white] {$(\triangle^{-1})^2$};
\draw[->](E1)--(E2)node [midway , fill=white] {$\triangle$};
\draw[->](B3)--(F1)node [midway , fill=white] {$\triangle$};
\draw[->](F1)--(D2)node [midway , fill=white] {$I_C$};
\draw[->](F1)--(E1)node [midway , fill=white] {$=$};
\draw[->](E1)--(A5)node [midway , fill=white] {$\triangle^{-1}$};
\end{tikzpicture}

%% file: II-dual_triangle.tex
\begin{tikzpicture}

\node[draw] (A1) at (5, 3.5){\begin{tikzpicture}
\oc{y}{}{Y}{6}{0}{-2}{\en}
\oc{x}{}{X}{1}{0}{-2}{\en}

  \begin{scope}[on background layer]
\alp{ ({x}t)--({y}t)}
\end{scope}
\end{tikzpicture}
};

\node[draw] (B1) at (10.5, 3.5){\begin{tikzpicture}
\oc{y}{}{Y}{6}{0}{-2}{\en}
\oc{x}{}{X}{1}{0}{-2}{\en}

\oc{ta}{C}{_a}{2}{-1}{-2}{\et}
\oc{t2}{\rdual{C}}{_2}{5}{-1}{-2}{\etr}

  \begin{scope}[on background layer]
\dlp{({ta}t)--({t2}t)}
\alp{({x}t)--({y}t)}

\alp{({ta}b)--({t2}b)}
\end{scope}
\end{tikzpicture}
};

\node[draw] (B2) at (10.5, 0){\begin{tikzpicture}
\oc{y}{}{Y}{6}{0}{-2}{\en}
\oc{x}{}{X}{1}{0}{-2}{\en}

\oc{ta}{C}{_a}{2}{-1}{-2}{\et}
\oc{t2}{\rdual{C}}{_2}{5}{-1}{-2}{\etr}

\oc{e4}{\rdual{E}}{_4}{4}{0}{-1}{\epr}
\oc{e3}{E}{_3}{3}{0}{-1}{\ep}

  \begin{scope}[on background layer]
\dlp{({ta}t)--({e3}b)}
\dlp{({e4}b)--({t2}t)}
\alp{({x}t)--({e3}t)}
\alp{ ({e4}t)--({y}t)}

\alp{({ta}b)--({t2}b)}
\end{scope}
\end{tikzpicture}
};

\node[draw] (B3) at (5, 0){\begin{tikzpicture}
\oc{y}{}{Y}{6}{0}{-2}{\en}
\oc{x}{}{X}{1}{0}{-2}{\en}
\oc{t2}{\rdual{C}}{_2}{5}{-1}{-2}{\etr}
\oc{e4}{\rdual{E}}{_4}{4}{0}{-1}{\epr}

  \begin{scope}[on background layer]
\dlp{({e4}b)--({t2}t)}
\alp{({x}b)--({t2}b)}
\alp{({e4}t)--({y}t)}
\end{scope}
\end{tikzpicture}
};

\draw[->](A1)--(B1)node [midway , fill=white] {$I_C$};
\draw[->](B1)--(B2)node [midway , fill=white] {$C_E$};

\draw[->](B2)--(B3)node [midway , fill=white] {$\triangle^{-1}$};
\draw[dashed, ->](A1)--(B3)node [midway , fill=white] {$\rdual{\triangle}$};

\end{tikzpicture}

%% file: II-dual_triangle_1.tex
\begin{tikzpicture}

\node[draw] (A2) at (0, 3.5){\begin{tikzpicture}
\oc{t1}{\rdual{C}}{_1}{5}{-1}{-2}{\etr}
\oc{x}{}{X}{3}{0}{-2}{\en}
\oc{y}{}{Y}{8}{0}{-2}{\en}
\oc{e3}{\rdual{E}}{_3}{4}{0}{-1}{\epr}

  \begin{scope}[on background layer]
\dlp{ ({t1}t)--({e3}b)}
\alp{({t1}b)--({x}b)}
\alp{ ({y}t)--({e3}t)}
\end{scope}
\end{tikzpicture}
};

\node[draw] (A3) at (0, 0){\begin{tikzpicture}
\oc{x}{}{X}{3}{0}{-2}{\en}
\oc{y}{}{Y}{8}{0}{-2}{\en}

  \begin{scope}[on background layer]
\alp{ ({y}b)--({x}b)}
\end{scope}
\end{tikzpicture}
};

\node[draw] (B1) at (5.5, 3.5){\begin{tikzpicture}
\oc{t1}{\rdual{C}}{_1}{5}{-1}{-2}{\etr}
\oc{x}{}{X}{3}{0}{-2}{\en}
\oc{y}{}{Y}{8}{0}{-2}{\en}
\oc{e3}{\rdual{E}}{_3}{4}{0}{-1}{\epr}

\oc{e5}{E}{_5}{7}{0}{-1}{\ep}
\oc{t6}{C}{_6}{6}{-1}{-2}{\et}

  \begin{scope}[on background layer]
\dlp{ ({t1}t)--({e3}b)}
\dlp{({t6}t)--({e5}b)}
\alp{({t1}b)--({x}b)}
\alp{({y}b)--({t6}b)}
\alp{ ({e5}t)--({e3}t)}
\end{scope}
\end{tikzpicture}
};

\node[draw] (B2) at (5.5,0){\begin{tikzpicture}
\oc{x}{}{X}{3}{0}{-2}{\en}
\oc{y}{}{Y}{8}{0}{-2}{\en}
\oc{e3}{\rdual{E}}{_3}{4}{0}{-1}{\epr}
\oc{e5}{E}{_5}{7}{0}{-1}{\ep}

  \begin{scope}[on background layer]
\dlp{({e3}b)--({e5}b)}
\alp{({y}b)--({x}b)}
\alp{ ({e5}t)--({e3}t)}
\end{scope}
\end{tikzpicture}
};

\draw[->](A2)--(B1)node [midway , fill=white] {$\triangle$};
\draw[->](B1)--(B2)node [midway , fill=white] {$P_C$};

\draw[dashed, ->](A2)--(A3)node [midway , fill=white] {$(\rdual{\triangle})^{-1}$};

\draw[->](B2)--(A3)node [midway , fill=white] {$L_E$}; 
\end{tikzpicture}

%% file: II-dual_triangle_composites.tex
\begin{tikzpicture}[
    my style/.style={%
      label={right:\pgfkeysvalueof{/pgf/minimum width}},
    },
   my style/.style={%
     append after command={
       \pgfextra{\node [right] at (\tikzlastnode.mid east) {\tikzlastnode};}
     },
   },
  ]

\node[draw 
	] (A1) at (2, 5){\begin{tikzpicture}
\oc{y}{}{Y}{6}{0}{-2}{\en}
\oc{x}{}{X}{1}{0}{-2}{\en}

  \begin{scope}[on background layer]
\alp{ ({x}t)--({y}t)}
\end{scope}
\end{tikzpicture}
};

\node[draw
	] (B1) at (9.5, 3.5){\begin{tikzpicture}
\oc{y}{}{Y}{6}{0}{-2}{\en}
\oc{x}{}{X}{1}{0}{-2}{\en}

\oc{ta}{C}{_a}{2}{-1}{-2}{\et}
\oc{t2}{\rdual{C}}{_2}{5}{-1}{-2}{\etr}

  \begin{scope}[on background layer]
\dlp{({ta}t)--({t2}t)}
\alp{({x}t)--({y}t)}

\alp{({ta}b)--({t2}b)}
\end{scope}
\end{tikzpicture}
};

\node[draw 
] (B2) at (9.5, 0){\begin{tikzpicture}
\oc{y}{}{Y}{6}{0}{-2}{\en}
\oc{x}{}{X}{1}{0}{-2}{\en}

\oc{ta}{C}{_a}{2}{-1}{-2}{\et}
\oc{t2}{\rdual{C}}{_2}{5}{-1}{-2}{\etr}

\oc{e4}{\rdual{E}}{_4}{4}{0}{-1}{\epr}
\oc{e3}{E}{_3}{3}{0}{-1}{\ep}

  \begin{scope}[on background layer]
\dlp{({ta}t)--({e3}b)}
\dlp{({e4}b)--({t2}t)}
\alp{({x}t)--({e3}t)}
\alp{ ({e4}t)--({y}t)}

\alp{({ta}b)--({t2}b)}
\end{scope}
\end{tikzpicture}
};

\node[draw
	] (B3) at (2, 0){\begin{tikzpicture}
\oc{y}{}{Y}{6}{0}{-2}{\en}
\oc{x}{}{X}{1}{0}{-2}{\en}
\oc{t2}{\rdual{C}}{_2}{5}{-1}{-2}{\etr}
\oc{e4}{\rdual{E}}{_4}{4}{0}{-1}{\epr}

  \begin{scope}[on background layer]
\dlp{({e4}b)--({t2}t)}
\alp{({x}b)--({t2}b)}
\alp{({e4}t)--({y}t)}
\end{scope}
\end{tikzpicture}
};

\draw[->](A1)--(B1)node [midway , fill=white] {$I_C$};
\draw[->](B1)--(B2)node [midway , fill=white] {$C_E$};

\draw[->](B2)--(B3)node [midway , fill=white] {$\triangle^{-1}$};
\draw[dashed, ->](A1)--(B3)node [midway , fill=white] {$\rdual{\triangle}$};

%
%

\node[draw
	] (A3) at (2, -8){\begin{tikzpicture}
\oc{x}{}{X}{3}{0}{-2}{\en}
\oc{y}{}{Y}{8}{0}{-2}{\en}

  \begin{scope}[on background layer]
\alp{ ({y}b)--({x}b)}
\end{scope}
\end{tikzpicture}
};

\node[draw
	] (B4) at (9.5, -3){\begin{tikzpicture}
\oc{t1}{\rdual{C}}{_1}{5}{-1}{-2}{\etr}
\oc{x}{}{X}{3}{0}{-2}{\en}
\oc{y}{}{Y}{8}{0}{-2}{\en}
\oc{e3}{\rdual{E}}{_3}{4}{0}{-1}{\epr}

\oc{e5}{E}{_5}{7}{0}{-1}{\ep}
\oc{t6}{C}{_6}{6}{-1}{-2}{\et}

  \begin{scope}[on background layer]
\dlp{ ({t1}t)--({e3}b)}
\dlp{({t6}t)--({e5}b)}
\alp{({t1}b)--({x}b)}
\alp{({y}b)--({t6}b)}
\alp{ ({e5}t)--({e3}t)}
\end{scope}
\end{tikzpicture}
};

\node[draw
	] (B5) at (9.5,-6){\begin{tikzpicture}
\oc{x}{}{X}{3}{0}{-2}{\en}
\oc{y}{}{Y}{8}{0}{-2}{\en}
\oc{e3}{\rdual{E}}{_3}{4}{0}{-1}{\epr}
\oc{e5}{E}{_5}{7}{0}{-1}{\ep}

  \begin{scope}[on background layer]
\dlp{({e3}b)--({e5}b)}
\alp{({y}b)--({x}b)}
\alp{ ({e5}t)--({e3}t)}
\end{scope}
\end{tikzpicture}
};

\draw[dashed, ->](B3)--(A3)node [midway , fill=white] {$(\rdual{\triangle})^{-1}$};
\draw[->](B3)--(B4)node [midway , fill=white] {$\triangle$};
\draw[->](B4)--(B5)node [midway , fill=white] {$P_C$};

\draw[->](B5)--(A3)node [midway , fill=white] {$L_E$}; 

\node[draw
	] (C1) at (15.5, 0){\begin{tikzpicture}
\oc{y}{}{Y}{8}{0}{-2}{\en}
\oc{x}{}{X}{1}{0}{-2}{\en}

\oc{ta}{C}{_a}{2}{-1}{-2}{\et}
\oc{t2}{\rdual{C}}{_2}{5}{-1}{-2}{\etr}

\oc{e4}{\rdual{E}}{_4}{4}{0}{-1}{\epr}
\oc{e3}{E}{_3}{3}{0}{-1}{\ep}

\oc{tb}{C}{_b}{6}{-1}{-2}{\et\lt}

\oc{ec}{E}{_c}{7}{0}{-1}{\ep\lt}

  \begin{scope}[on background layer]
\dlp{({ta}t)--({e3}b)}
\dlp{({e4}b)--({t2}t)}
\alp{({x}t)--({e3}t)}
\alp{ ({e4}t)--({ec}t)}

\dlp{({tb}t)--({ec}b)}
\alp{({ta}b)--({t2}b)}
\alp{({tb}b)--({y}b)}
\end{scope}
\end{tikzpicture}
};

\node[draw
	] (C2) at (15.5, 3.5){\begin{tikzpicture}
\oc{y}{}{Y}{8}{0}{-2}{\en}
\oc{x}{}{X}{1}{0}{-2}{\en}

\oc{ta}{C}{_a}{2}{-1}{-2}{\et}
\oc{t2}{\rdual{C}}{_2}{5}{-1}{-2}{\etr}

\oc{tb}{C}{_b}{6}{-1}{-2}{\et\lt}

\oc{ec}{E}{_c}{7}{0}{-1}{\ep\lt}

  \begin{scope}[on background layer]
\dlp{({ta}t)--({t2}t)}
\alp{({x}t)--({ec}t)}

\dlp{({tb}t)--({ec}b)}
\alp{({ta}b)--({t2}b)}
\alp{({tb}b)--({y}b)}
\end{scope}
\end{tikzpicture}
};

\node[draw
	] (C3) at (15.5, -6){\begin{tikzpicture}
\oc{y}{}{Y}{8}{0}{-2}{\en}
\oc{x}{}{X}{1}{0}{-2}{\en}

\oc{ta}{C}{_a}{2}{-1}{-2}{\et}

\oc{e4}{\rdual{E}}{_4}{4}{0}{-1}{\epr}
\oc{e3}{E}{_3}{3}{0}{-1}{\ep}

\oc{ec}{E}{_c}{7}{0}{-1}{\ep\lt}

  \begin{scope}[on background layer]
\dlp{({ta}t)--({e3}b)}
\dlp{({e4}b)--({ec}b)}
\alp{({x}t)--({e3}t)}
\alp{ ({e4}t)--({ec}t)}

\alp{({ta}b)--({y}b)}
\end{scope}
\end{tikzpicture}
};

\node[draw
	] (C4) at (23.5, 0){\begin{tikzpicture}
\oc{y}{}{Y}{8}{0}{-2}{\en}
\oc{x}{}{X}{1}{0}{-2}{\en}

\oc{ta}{C}{_a}{2}{-1}{-2}{\et}

\oc{ec}{E}{_c}{7}{0}{-1}{\ep\lt}

  \begin{scope}[on background layer]
\dlp{({ta}t)--({ec}b)}
\alp{({x}t)--({ec}t)}

\alp{({ta}b)--({y}b)}
\end{scope}
\end{tikzpicture}
};

\node[draw
	] (D1) at (23.5, 5){\begin{tikzpicture}
\oc{y}{}{Y}{8}{0}{-2}{\en}
\oc{x}{}{X}{1}{0}{-2}{\en}

\oc{tb}{C}{_b}{6}{-1}{-2}{\et\lt}

\oc{ec}{E}{_c}{7}{0}{-1}{\ep\lt}

  \begin{scope}[on background layer]
\alp{({x}t)--({ec}t)}

\dlp{({tb}t)--({ec}b)}
\alp{({tb}b)--({y}b)}
\end{scope}
\end{tikzpicture}
};

\node[draw
	] (D2) at (23.5, -8){\begin{tikzpicture}
\oc{y}{}{Y}{8}{0}{-2}{\en}
\oc{x}{}{X}{1}{0}{-2}{\en}

\oc{ta}{C}{_a}{2}{-1}{-2}{\et}
\oc{e3}{E}{_3}{3}{0}{-1}{\ep}

  \begin{scope}[on background layer]
\dlp{({ta}t)--({e3}b)}
\alp{({x}t)--({e3}t)}

\alp{({ta}b)--({y}b)}
\end{scope}
\end{tikzpicture}
};

\draw[->](B2)--(C1)node [midway , fill=white] {$\triangle$};
\draw[->](B1)--(C2)node [midway , fill=white] {$\triangle$};

\draw[->](C1)--(B4)node [midway , fill=white] {$\triangle^{-1}$};

\draw[->](C3)--(B5)node [midway , fill=white] {$\triangle^{-1}$};
\draw[->](C2)--(C1)node [midway , fill=white] {$C_E$};
\draw[->](C1)--(C3)node [midway , fill=white] {$P_C$};
\draw[->](C4)--(C3)node [midway , fill=white] {$C_E$};
\draw[->](C2)--(C4)node [midway , fill=white] {$P_C$};

\draw[->](A1)--(D1)node [midway , fill=white] {$\triangle$};
\draw[->](D1)--(C2)node [midway , fill=white] {$I_C$};

\draw[->](D1)--(C4)node [midway , fill=white] {$=$};
\draw[->](C4)--(D2)node [midway , fill=white] {$=$};

\draw[->](D2)--(A3)node [midway , fill=white] {$\triangle^{-1}$};
\draw[->](C3)--(D2)node [midway , fill=white] {$L_E$};

\end{tikzpicture}


%% file: II-dual_triangle_compare.tex
\begin{tikzpicture}[
    my style/.style={
      label={right:\pgfkeysvalueof{/pgf/minimum width}},
    },
   my style/.style={%
     append after command={
       \pgfextra{\node [right] at (\tikzlastnode.mid east) {\tikzlastnode};}
     },
   },
  ]

\node[
draw] (B2) at (-2, 10){\begin{tikzpicture}
\oc{t1}{C}{_1}{0}{5}{4}{\et}
\oc{ta}{}{ta}{3}{7}{4}{\en}

\oc{tc}{}{tc}{0}{7}{6}{\en}
  \begin{scope}[on background layer]
\dlp{ ({t1}t)--(3*\d, 5*\h)}

\dlp{({tc}t)--({ta}t)}

\alp{ ({t1}b)--({ta}b)}

\alp{ ({tc}b)--(3*\d, 6*\h)}
\end{scope}
\end{tikzpicture}};

\node[
draw] (D2) at (4, 8){\begin{tikzpicture}
\oc{t1}{C}{_1}{2}{5}{4}{\et}
\oc{ta}{}{ta}{3}{7}{4}{\en}

\oc{t2}{\rdual{C}}{_2}{1}{5}{4}{\etr\lt}
\oc{t3}{C}{_3}{-2}{5}{4}{\et\lt}

\oc{tc}{}{tc}{-2}{7}{6}{\en}
  \begin{scope}[on background layer]
\dlp{ ({t1}t)--(3*\d, 5*\h)}

\alp{ ({t1}b)--({ta}b)}

\dlp{ ({t2}t)--({t3}t)}

\alp{ ({t2}b)--({t3}b)}

\dlp{({tc}t)--({ta}t)}

\alp{ ({t1}b)--({ta}b)}

\alp{ ({tc}b)--(3*\d, 6*\h)}
\end{scope}
\end{tikzpicture}};

\node[
draw] (D6) at (10, 6){\begin{tikzpicture}
\oc{ta}{}{ta}{3}{7}{4}{\en}

\oc{t3}{C}{_3}{-2}{5}{4}{\et\lt}

\oc{tc}{}{tc}{-2}{7}{6}{\en}
  \begin{scope}[on background layer]

\alp{ ({t1}b)--({ta}b)}

\dlp{ (3*\d, 5*\h)--({t3}t)}

\dlp{({tc}t)--({ta}t)}

\alp{ ({t3}b)--({ta}b)}

\alp{ ({tc}b)--(3*\d, 6*\h)}
\end{scope}
\end{tikzpicture}};

\node[
draw] (D3) at (16, 8){\begin{tikzpicture}
\oc{t1}{C}{_1}{2}{5}{4}{\et}
\oc{ta}{}{ta}{3}{7}{4}{\en}

\oc{e4}{\rdual{E}}{_4}{0}{6}{5}{\epr}
\oc{e5}{E}{_5}{-1}{6}{5}{\ep}

\oc{t2}{\rdual{C}}{_2}{1}{5}{4}{\etr\lt}
\oc{t3}{C}{_3}{-2}{5}{4}{\et\lt}

\oc{tc}{}{tc}{-2}{7}{6}{\en}
  \begin{scope}[on background layer]
\dlp{ ({t1}t)--(3*\d, 5*\h)}

\alp{ ({t1}b)--({ta}b)}

\dlp{ ({t2}t)-- ({e4}b)}
\dlp{({e5}b)--({t3}t)}

\alp{ ({t2}b)--({t3}b)}

\dlp{({tc}t)--({ta}t)}

\alp{ ({t1}b)--({ta}b)}

\alp{ ({tc}b)--({e5}t)}
\alp{({e4}t)--(3*\d, 6*\h)}
\end{scope}
\end{tikzpicture}};

\node[
draw] (B3) at (-2, 1){\begin{tikzpicture}
\oc{t1}{C}{_1}{-1}{5}{4}{\et}
\oc{ta}{}{ta}{3}{7}{4}{\en}

\oc{tc}{}{t_c}{-1}{7}{6}{\en}

\oc{ex}{E}{_x}{0}{6}{5}{\ep}
\oc{e8}{\rdual{E}}{_8}{1}{6}{5}{\epr\dk}

  \begin{scope}[on background layer]
\dlp{({t1}t)--({ex}b)}
\dlp{({e8}b)--(3*\d, 5*\h)}
\dlp{({tc}t)--({ta}t)}

\alp{ ({t1}b)--({ta}b)}

\alp{ ({tc}b)--({ex}t)}
\alp{ ({e8}t)--(3*\d, 6*\h)}
\end{scope}
\end{tikzpicture}};

\node[
draw] (D4) at (4, 4){\begin{tikzpicture}
\oc{t1}{C}{_1}{-1}{5}{4}{\et}
\oc{ta}{}{ta}{3}{7}{4}{\en}

\oc{tc}{}{t_c}{-5}{7}{6}{\en}

\oc{ex}{E}{_x}{0}{6}{5}{\ep}
\oc{e8}{\rdual{E}}{_8}{1}{6}{5}{\epr\dk}

\oc{t2}{\rdual{C}}{_2}{-2}{5}{4}{\etr\lt}
\oc{t3}{C}{_3}{-5}{5}{4}{\et\lt}

  \begin{scope}[on background layer]
\dlp{({t1}t)--({ex}b)}
\dlp{({e8}b)--(3*\d, 5*\h)}
\dlp{({tc}t)--({ta}t)}

\alp{ ({t1}b)--({ta}b)}

\alp{ ({tc}b)--({ex}t)}
\alp{ ({e8}t)--(3*\d, 6*\h)}

\dlp{ ({t2}t)--({t3}t)}

\alp{ ({t2}b)--({t3}b)}
\end{scope}
\end{tikzpicture}};

%
%
%
%
%
%
%
%
%
%

\node[
draw] (C2) at (22, 10){\begin{tikzpicture}
\oc{t1}{C}{_1}{0}{5}{4}{\et}
\oc{ta}{}{t_a}{1}{7}{4}{\en}

\oc{tc}{}{t_c}{-3}{7}{4}{\en}
\oc{t7}{\rdual{C}}{_7}{-1}{5}{4}{\etr\dk}
\oc{e8}{\rdual{E}}{_8}{-2}{6}{5}{\epr\dk}

\gc{g1}{2.5}{8}{7}
\gc{g2}{2.5}{6}{4}
  \begin{scope}[on background layer]
\dlp{ ({t1}t)--(1*\d, 5*\h)}
\dlp{({tc}t)--({ta}t)}
\dlp{({t7}t)--({e8}b)}

\alp{ ({t1}b)--({ta}b)}

\alp{ ({tc}b)--({t7}b)}
\alp{ ({e8}t)--(1*\d, 6*\h)}
\end{scope}
\end{tikzpicture}};

\node[
draw] (C3) at (22, 1){\begin{tikzpicture}
\oc{ta}{}{t_a}{4}{7}{4}{\en}

\oc{tc}{}{t_c}{0}{7}{4}{\en}
\oc{e8}{\rdual{E}}{_8}{1}{6}{5}{\epr\dk}

\gc{g1}{2.5}{8}{7}
\gc{g2}{2.5}{6}{4}
  \begin{scope}[on background layer]

\dlp{({e8}b)--(4*\d, 5*\h)}
\dlp{({tc}t)--({ta}t)}

\alp{ ({tc}b)--({ta}b)}

\alp{ ({e8}t)--(4*\d, 6*\h)}
\end{scope}
\end{tikzpicture}};

\node[
draw] (D5) at (16, 4){\begin{tikzpicture}
\oc{ta}{}{ta}{3}{7}{4}{\en}

\oc{e4}{\rdual{E}}{_4}{0}{6}{5}{\epr}
\oc{e5}{E}{_5}{-1}{6}{5}{\ep}

\oc{t3}{C}{_3}{-2}{5}{4}{\et\lt}

\oc{tc}{}{tc}{-2}{7}{6}{\en}
  \begin{scope}[on background layer]

\alp{ ({t1}b)--({ta}b)}

\dlp{ (3*\d, 5*\h)-- ({e4}b)}
\dlp{({e5}b)--({t3}t)}

\alp{ ({ta}b)--({t3}b)}

\dlp{({tc}t)--({ta}t)}

\alp{ ({tc}b)--({e5}t)}
\alp{({e4}t)--(3*\d, 6*\h)}
\end{scope}
\end{tikzpicture}};

\draw[dashed,->](B2)to node [midway , fill=white] {$\triangle^*$} (C2);
\draw[->](B3)--(C3)node [midway , fill=white] {$\triangle^{-1}$};
\draw[->](B2)--(B3)node [midway , fill=white] {$C_E$};
\draw[->](C2)--(C3)node [midway , fill=white] {$P_C$};
\draw[->](B2)--(D2)node [midway , fill=white] {$I_C$};
\draw[->](D2)--(D3)node [midway , fill=white] {$C_E$};
\draw[->](D3)--(C2)node [midway , fill=white] {$\triangle^{-1}$};
\draw[->](D2)--(D4)node [midway , fill=white] {$C_E$};

\draw[->](B3)--(D4)node [midway , fill=white] {$I_C$};
\draw[->](D3)--(D5)node [midway , fill=white] {$P_C$};
\draw[->](D5)--(C3)node [midway , fill=white] {$\triangle^{-1}$};

\draw[->](D2)--(D6)node [midway , fill=white] {$P_C$};
\draw[->](D6)--(D5)node [midway , fill=white] {$C_E$};
\draw[->](D4)--(D5)node [midway , fill=white] {$P_C$};
\draw[->](B3)to [out = 10, in =200,looseness =.5] node [midway , fill=white] {$=$}(D5);
\end{tikzpicture}

%% file: II-four_shadows.tex
    \centering
    \begin{subfigure}[t]{0.22\textwidth}
        \centering
\begin{tikzpicture}

\oc{t1}{C}{_1}{1}{5}{4}{\et}
\oc{x}{M}{}{2}{4}{4}{\en}

\oc{ex}{E}{_x}{4}{5}{4}{\ep}
\gc{g1}{3}{5}{4}
  \begin{scope}[on background layer]
\dlp{ ({t1}t)\gpt{g1}({ex}b)}
\alp{ ({t1}b)--({x}b)\gpb{g1}({ex}t)}
\end{scope}
\end{tikzpicture}
        \caption{$\sh{M}$}
    \end{subfigure}%
\hfill
    \begin{subfigure}[t]{0.22\textwidth}
        \centering
\begin{tikzpicture}

\oc{e1}{\rdual{E}}{_1}{0}{5}{4}{\epr}
\oc{x}{M}{}{2}{4}{4}{\en}

\oc{tx}{\rdual{C}}{_x}{3}{5}{4}{\etr}
\gc{g1}{1}{5}{4}
  \begin{scope}[on background layer]
\dlp{ ({e1}b)\gpb{g1}({tx}t)}
\alp{({e1}t)\gpt{g1}({x}b)--({tx}b)}
\end{scope}
\end{tikzpicture}
        \caption{$\dsh{M}$}
    \end{subfigure}%
\hfill
    \begin{subfigure}[t]{0.18\textwidth}
        \centering
\begin{tikzpicture}

\oc{t1}{C}{_1}{0}{5}{4}{\et}
\oc{x}{M}{}{1}{4}{4}{\en}

\oc{tx}{\rdual{C}}{_x}{2}{5}{4}{\etr}
  \begin{scope}[on background layer]
\dlp{({t1}t)--({tx}t)}
\alp{({t1}b)--({x}b)--({tx}b)}
\end{scope}
\end{tikzpicture}
        \caption{$\csh{M}$}
    \end{subfigure}%
\hfill 
    \begin{subfigure}[t]{0.18\textwidth}
        \centering
\begin{tikzpicture}

\oc{e1}{\rdual{E}}{_1}{0}{5}{4}{\epr}
\oc{x}{M}{}{1}{5}{5}{\en}

\oc{ex}{E}{_x}{2}{5}{4}{\ep}
  \begin{scope}[on background layer]
\dlp{({e1}b)--({ex}b)}
\alp{({e1}t)--({x}t)--({ex}t)}
\end{scope}
\end{tikzpicture}
        \caption{$\esh{M}$}
    \end{subfigure}%

%% file: II-four_shadows_maps.tex
    \centering
    \begin{subfigure}[t]{0.48\textwidth}
        \centering
		\resizebox{\textwidth}{!}{\begin{tikzpicture}

	\node[draw] (A1) at (0, 0){\begin{tikzpicture}
		\oc{x}{M}{}{-1}{4}{4}{\en}
		\oc{t4}{\rdual{C}}{_4}{0}{5}{4}{\etr}
		\gc{g2}{-2}{5}{4}
		\oc{e3}{\rdual{E}}{_3}{-3}{5}{4}{\epr}
  \begin{scope}[on background layer]
		\dlp{ ({e3}b)\gpb{g2}({t4}t)}
		\alp{ ({e3}t)\gpt{g2}({x}b)--({t4}b)}
\end{scope}
		\oc{t1}{C}{_1}{1}{5}{4}{\et}
		\oc{y}{N}{}{2}{4}{4}{\en}
		\oc{tx}{\rdual{C}}{_x}{3}{5}{4}{\etr\lt}
  \begin{scope}[on background layer]
		\dlp{({t1}t)--({tx}t)}
		\alp{ ({t1}b)--({y}b)--({tx}b)}
\end{scope}
		\end{tikzpicture}
		};

		\node[draw] (A2) at (5.5, 0){\begin{tikzpicture}
		\oc{x}{M}{}{-1}{4}{4}{\en}
		\gc{g2}{-2}{5}{4}
		\oc{e3}{\rdual{E}}{_3}{-3}{5}{4}{\epr}

		\oc{y}{N}{}{0}{4}{4}{\en}
		\oc{tx}{\rdual{C}}{_x}{1}{5}{4}{\etr\lt}
  \begin{scope}[on background layer]
		\dlp{ ({e3}b)\gpb{g2}({tx}t)}
		\alp{ ({e3}t)\gpt{g2}({x}b)--({y}b)--({tx}b)}
\end{scope}
		\end{tikzpicture}
		};

		\draw[->](A1)--(A2)node [midway , fill=white] {$P_C$};
	\end{tikzpicture}}
        \caption{$\rspl$}
    \end{subfigure}
\hspace{.02\textwidth}
    \centering
    \begin{subfigure}[t]{0.48\textwidth}
        \centering
		\resizebox{\textwidth}{!}{\begin{tikzpicture}

	\node[draw] (A1) at (0, 0){\begin{tikzpicture}		
		\oc{x}{M}{}{-1}{4}{4}{\en}
		\gc{g2}{-2}{5}{4}
		\oc{e3}{\rdual{E}}{_3}{-3}{5}{4}{\epr}

		\oc{y}{N}{}{0}{4}{4}{\en}
		\oc{tx}{\rdual{C}}{_x}{1}{5}{4}{\etr}
  \begin{scope}[on background layer]
		\dlp{({e3}b)\gpb{g2}({tx}t)}
		\alp{ ({e3}t)\gpt{g2}({x}b)--({y}b)--({tx}b)}
\end{scope}
		\end{tikzpicture}
		};

		\node[draw] (A2) at (5.5, 0){\begin{tikzpicture}
		\oc{x}{M}{}{-4}{5}{5}{\en}
		\gc{g2}{-1}{5}{4}
		\oc{e3}{\rdual{E}}{_3}{-5}{5}{4}{\epr}
		\oc{e1}{E}{_1}{-3}{5}{4}{\ep}
		\oc{e2}{\rdual{E}}{_2}{-2}{5}{4}{\epr\lt}

		\oc{y}{N}{}{0}{4}{4}{\en}
		\oc{tx}{\rdual{C}}{_x}{1}{5}{4}{\etr}
  \begin{scope}[on background layer]
		\dlp{({e3}b)--({e1}b)}
		\dlp{({e2}b)\gpb{g2}({tx}t)}
		\alp{({e3}t)--({x}b)--({e1}t)}
		\alp{({e2}t)\gpt{g2}({y}b)--({tx}b)}
\end{scope}
		\end{tikzpicture}
		};

		\draw[->](A1)--(A2)node [midway , fill=white] {$P_C$};
	\end{tikzpicture}}
        \caption{$\luspl$}
    \end{subfigure}

    \centering
    \begin{subfigure}[t]{0.48\textwidth}
        \centering
		\resizebox{\textwidth}{!}{\begin{tikzpicture}

	\node[draw] (A1) at (0, 0){\begin{tikzpicture}
		\oc{e1}{\rdual{E}}{_1}{-1}{5}{4}{\epr}
		\oc{x}{M}{}{1}{4}{4}{\en}
		\oc{tx}{\rdual{C}}{_x}{2}{5}{4}{\etr}
		\gc{g1}{0}{5}{4}
  \begin{scope}[on background layer]
		\dlp{ ({e1}b)\gpb{g1}({tx}t)}
		\alp{({e1}t)\gpt{g1}({x}b)--({tx}b)}
\end{scope}
		\oc{t2}{C}{_2}{3}{5}{4}{\et}
		\oc{y}{N}{}{4}{4}{4}{\en}
		\gc{g2}{5}{5}{4}
		\oc{e3}{E}{_3}{6}{5}{4}{\ep}
  \begin{scope}[on background layer]
		\dlp{({t2}t)\gpt{g2}({e3}b)}
		\alp{ ({t2}b)--({y}b)\gpb{g2}({e3}t)}
\end{scope}
		\end{tikzpicture}
		};

		\node[draw] (A2) at (5.5, 0){\begin{tikzpicture}
		\oc{e1}{\rdual{E}}{_1}{-1}{5}{4}{\epr}
		\oc{x}{M}{}{0}{5}{5}{\en}
		\oc{y}{N}{}{1}{5}{5}{\en}
		\oc{e3}{E}{_3}{2}{5}{4}{\ep}
  \begin{scope}[on background layer]
		\dlp{ ({e1}b)--({e3}b)}
		\alp{ ({e1}t)--({x}b)--({y}b)--({e3}t)}
\end{scope}
		\end{tikzpicture}
		};

		\draw[->](A1)--(A2)node [midway , fill=white] {$P_C$};
	\end{tikzpicture}}
        \caption{$\uspl$}
    \end{subfigure}%
\hspace{.02\textwidth}
    \centering
    \begin{subfigure}[t]{0.48\textwidth}
        \centering
	\resizebox{\textwidth}{!}{\begin{tikzpicture}

	\node[draw] (A1) at (0, 0){\begin{tikzpicture}
		\oc{t1}{C}{_1}{0}{5}{4}{\et}
		\oc{x}{M}{}{1}{4}{4}{\en}
		\oc{y}{N}{}{2}{4}{4}{\en}
		\oc{tx}{\rdual{C}}{_x}{3}{5}{4}{\etr}
  \begin{scope}[on background layer]
		\dlp{({t1}t)--({tx}t)}
		\alp{ ({t1}b)--({x}b)--({y}b)--({tx}b)}
\end{scope}
		\end{tikzpicture}
		};

		\node[draw] (A2) at (5.5, 0){\begin{tikzpicture}
		\oc{t1}{C}{_1}{0}{5}{4}{\et}
		\oc{x}{M}{}{1}{4}{4}{\en}
		\oc{y}{N}{}{6}{4}{4}{\en}
		\oc{tx}{\rdual{C}}{_x}{7}{5}{4}{\etr}
		\gc{g1}{2}{5}{4}
		\gc{g2}{5}{5}{4}
		\oc{e2}{E}{_2}{3}{5}{4}{\ep}
		\oc{e3}{\rdual{E}}{_3}{4}{5}{4}{\epr}
  \begin{scope}[on background layer]
		\dlp{({t1}t)\gpt{g1}({e2}b)}
		\dlp{({e3}b)\gpb{g2}({tx}t)}
		\alp{ ({t1}b)--({x}b)\gpb{g1}({e2}t)}
		\alp{ ({e3}t)\gpt{g2}({y}b)--({tx}b)}
\end{scope}
		\end{tikzpicture}
		};

		\draw[->](A1)--(A2)node [midway , fill=white] {$C_E$};
	\end{tikzpicture}}
        \caption{$\spl$}
    \end{subfigure}
%

    \centering
    \begin{subfigure}[t]{0.48\textwidth}
        \centering
	\resizebox{\textwidth}{!}{	\begin{tikzpicture}

	\node[draw] (A1) at (0, 0){\begin{tikzpicture}
		\oc{t1}{C}{_1}{0}{5}{4}{\et}
		\oc{x}{M}{}{1}{4}{4}{\en}
		\oc{tx}{\rdual{C}}{_x}{2}{5}{4}{\etr}

		\oc{t2}{C}{_2}{3}{5}{4}{\et\lt}
		\oc{y}{N}{}{4}{4}{4}{\en}
		\gc{g1}{5}{5}{4}
		\oc{e3}{E}{_3}{6}{5}{4}{\ep}
  \begin{scope}[on background layer]
		\dlp{({t1}t)--({tx}t)}
		\alp{ ({t1}b)--({x}b)--({tx}b)}
		\dlp{ ({t2}t)\gpt{g1}({e3}b)}
		\alp{ ({t2}b)--({y}b)\gpb{g1}({e3}t)}
\end{scope}
		\end{tikzpicture}
		};

		\node[draw] (A2) at (5.5, 0){\begin{tikzpicture}
		\oc{t1}{C}{_1}{2}{5}{4}{\et}
		\oc{x}{M}{}{3}{4}{4}{\en}
		\oc{y}{N}{}{4}{4}{4}{\en}
		\gc{g1}{5}{5}{4}
		\oc{e3}{E}{_3}{6}{5}{4}{\ep}
  \begin{scope}[on background layer]
		\dlp{({t1}t)\gpt{g1}({e3}b)}
		\alp{({t1}b)--({x}b)--({y}b)\gpb{g1}({e3}t)}
\end{scope}
		\end{tikzpicture}
		};

		\draw[->](A1)--(A2)node [midway , fill=white] {$P_C$};
	\end{tikzpicture}}
        \caption{$\lspl$}
    \end{subfigure}%
\hspace{.02\textwidth}
 \centering
    \begin{subfigure}[t]{0.48\textwidth}
        \centering
		\resizebox{\textwidth}{!}{\begin{tikzpicture}

	\node[draw] (A1) at (0, 0){\begin{tikzpicture}
		\oc{x}{M}{}{3}{4}{4}{\en}
		\oc{t2}{C}{_2}{2}{5}{4}{\et}
		\oc{y}{N}{}{4}{4}{4}{\en}
		\gc{g2}{5}{5}{4}
		\oc{e3}{E}{_3}{6}{5}{4}{\ep}
  \begin{scope}[on background layer]
		\dlp{({t2}t)\gpt{g2}({e3}b)}
		\alp{({t2}b)--({x}b)--({y}b)\gpb{g2}({e3}t)}
\end{scope}
		\end{tikzpicture}
		};

		\node[draw] (A2) at (5.5, 0){\begin{tikzpicture}
		\oc{x}{M}{}{3}{4}{4}{\en}
		\oc{t2}{C}{_2}{2}{5}{4}{\et}
		\oc{y}{N}{}{7}{5}{5}{\en}
		\gc{g2}{4}{5}{4}
		\oc{e1}{E}{_1}{5}{5}{4}{\ep\lt}
		\oc{e2}{\rdual{E}}{_2}{6}{5}{4}{\epr}
		\oc{e3}{E}{_3}{8}{5}{4}{\ep}
  \begin{scope}[on background layer]
		\dlp{({t2}t)\gpt{g2}({e1}b)}
		\dlp{({e2}b)--({e3}b)}
		\alp{({t2}b)--({x}b)\gpb{g2}({e1}t)}
		\alp{ ({e2}t)--({y}b)--({e3}t)}
\end{scope}
		\end{tikzpicture}
		};

		\draw[->](A1)--(A2)node [midway , fill=white] {$C_E$};
	\end{tikzpicture}}
        \caption{$\ruspl$}
    \end{subfigure}%

    \centering
    \begin{subfigure}[t]{0.37\textwidth}
        \centering
		\resizebox{\textwidth}{!}{\begin{tikzpicture}

	\node[draw] (A1) at (0, 0){\begin{tikzpicture}		
		\node (ta) at (2,-1){};
		\node (ta) at (4,-2){};
		\end{tikzpicture}
		};

	\node[draw] (A2) at (4, 0){\begin{tikzpicture}
		\oc{t1}{C}{_1}{0}{5}{4}{\et}
		\oc{tx}{\rdual{C}}{_x}{2}{5}{4}{\etr}
  \begin{scope}[on background layer]
		\dlp{ ({t1}t)--({tx}t)}
		\alp{ ({t1}b)--({tx}b)}
\end{scope}
		\end{tikzpicture}
		};

		\draw[->](A1)--(A2)node [midway , fill=white] {$I_C$};
	\end{tikzpicture}}
        \caption{$\iunit$}
    \end{subfigure}
\hspace{.02\textwidth}
    \centering
    \begin{subfigure}[t]{0.37\textwidth}
        \centering
		\resizebox{\textwidth}{!}{\begin{tikzpicture}

	\node[draw] (A1) at (0, 0){\begin{tikzpicture}		
		\oc{e1}{\rdual{E}}{_1}{0}{5}{4}{\epr}
		\oc{ex}{E}{_x}{2}{5}{4}{\ep}
  \begin{scope}[on background layer]
		\dlp{({e1}t)--({ex}t)}
		\alp{({e1}b)--({ex}b)}
\end{scope}
		\end{tikzpicture}
		};

		\node[draw] (A2) at (4, 0){\begin{tikzpicture}
		\node (ta) at (2,-1){};
		\node (ta) at (4,-2){};
		\end{tikzpicture}
		};

		\draw[->](A1)--(A2)node [midway , fill=white] {$L_E$};
	\end{tikzpicture}}
        \caption{$\ounit$}
    \end{subfigure}

%% file: II-symmetric_shadow.tex
\begin{tikzpicture}

\node[draw] (Y6) at (0,-6){\begin{tikzpicture}
\oc{x}{M}{}{3}{4}{4}{\en}
\oc{y}{N}{}{2}{4}{4}{\en}
\oc{e6}{E}{_6}{4}{4}{3}{\ep\dk}

\oc{e8}{\rdual{E}}{_8}{0}{4}{3}{\epr\dk}

\oc{z}{P}{}{1}{4}{4}{\en}
\gc{g1}{-2}{4}{1}
\gc{g2}{3}{4}{1}

\begin{scope}[on background layer]
\dlp{({e8}b)--({e6}b)}
\alp{ ({e8}t)--({z}b)--({y}b)--({x}b)--({e6}t)}
\end{scope}
\end{tikzpicture}};

\node[draw] (B9) at (0,0){\begin{tikzpicture}
\oc{t1}{C}{_1}{0}{5}{4}{\et}
\oc{x}{M}{}{1}{4}{4}{\en}
\oc{y}{N}{}{2.5}{4}{4}{\en}
\oc{t2}{\rdual{C}}{_2}{5}{5}{4}{\etr}

\oc{z}{P}{}{4}{4}{4}{\en}
\gc{g1}{-2}{4}{1}
\gc{g2}{3}{4}{1}

\begin{scope}[on background layer]
\dlp{({t1}t)--({t2}t)}

\alp{ ({t1}b)--({x}b)--({y}b)--({z}b)--({t2}b)}
\end{scope}
\end{tikzpicture}};

\node[draw] (C1) at (9,0){\begin{tikzpicture}
\oc{t1}{C}{_1}{-2}{5}{4}{\et}
\oc{x}{M}{}{-1}{4}{4}{\en}
\oc{y}{N}{}{2.5}{4}{4}{\en}
\oc{t2}{\rdual{C}}{_2}{7}{5}{4}{\etr}
\oc{t5}{C}{_5}{-1}{3}{2}{\et\dk}
\oc{e6}{E}{_6}{0}{4}{3}{\ep\dk}
\oc{ta}{\rdual{c}}{_a}{2}{3}{2}{\etr\lt}
\oc{eb}{\rdual{E}}{_b}{1}{4}{3}{\epr\lt}

\oc{tc}{C}{_c}{3}{3}{2}{\et\lt}
\oc{ed}{E}{_d}{4}{4}{3}{\ep\lt}
\oc{t7}{\rdual{C}}{_7}{6}{3}{2}{\etr\dk}
\oc{e8}{\rdual{E}}{_8}{5}{4}{3}{\epr\dk}

\oc{z}{P}{}{6}{4}{4}{\en}
\gc{g1}{-2}{4}{1}
\gc{g2}{3}{4}{1}

\begin{scope}[on background layer]
\dlp{ ({t1}t)--({t2}t)}
\dlp{({t5}t)--({e6}b)}
\dlp{({ta}t)--({eb}b)}
\dlp{({tc}t)--({ed}b)}
\dlp{({t7}t)--({e8}b)}

\alp{ ({t1}b)--({x}b)--({e6}t)}
\alp{ ({t5}b)--({ta}b)}
\alp{ ({eb}t)--({y}b)--({ed}t)}
\alp{ ({tc}b)--({t7}b)}
\alp{ ({e8}t)--({z}b)--({t2}b)}
\end{scope}
\end{tikzpicture}};

\node[draw] (C2) at (18,0){\begin{tikzpicture}
\oc{t1}{C}{_1}{0}{5}{4}{\et}
\oc{x}{M}{}{1}{4}{4}{\en}
\oc{y}{N}{}{2.5}{9}{9}{\en}
\oc{t2}{\rdual{C}}{_2}{5}{5}{4}{\etr}
\oc{t5}{C}{_5}{6}{5}{4}{\et\dk}
\oc{e6}{E}{_6}{7}{6}{5}{\ep\dk}
\oc{ta}{\rdual{C}}{_a}{8}{7}{4}{\etr\lt}
\oc{eb}{\rdual{E}}{_b}{1}{9}{8}{\epr\lt}

\oc{tc}{C}{_c}{-3}{7}{4}{\et\lt}
\oc{ed}{E}{_d}{4}{9}{8}{\ep\lt}
\oc{t7}{\rdual{C}}{_7}{-1}{5}{4}{\etr\dk}
\oc{e8}{\rdual{E}}{_8}{-2}{6}{5}{\epr\dk}

\oc{z}{P}{}{4}{4}{4}{\en}
\gc{g1}{2.5}{8}{7}
\gc{g2}{2.5}{6}{4}

\begin{scope}[on background layer]
\dlp{ ({t1}t)--({t2}t)}
\dlp{({t5}t)--({e6}b)}
\dlp{({ta}t)--({g1}br)--({g1}tl)--({eb}b)}
\dlp{({tc}t)--({g1}bl)--({g1}tr)--({ed}b)}
\dlp{({t7}t)--({e8}b)}

\alp{ ({t1}b)--({x}b)--({g2}bl)--({g2}tr)--({e6}t)}
\alp{ ({t5}b)--({ta}b)}
\alp{ ({eb}t)--({y}b)--({ed}t)}
\alp{ ({tc}b)--({t7}b)}
\alp{ ({e8}t)--({g2}tl)--({g2}br)--({z}b)--({t2}b)}
\end{scope}
\end{tikzpicture}};

\node[draw] (C4) at ( 18,-6){\begin{tikzpicture}

\oc{x}{M}{}{1}{4}{4}{\en}
\oc{y}{N}{}{2.5}{9}{9}{\en}
\oc{e6}{E}{_6}{4}{6}{5}{\ep\dk}
\oc{ta}{\rdual{C}}{_a}{5}{7}{4}{\etr\lt}
\oc{eb}{\rdual{E}}{_b}{1}{9}{8}{\epr\lt}

\oc{tc}{C}{_c}{0}{7}{4}{\et\lt}
\oc{ed}{E}{_d}{4}{9}{8}{\ep\lt}
\oc{e8}{\rdual{E}}{_8}{1}{6}{5}{\epr\dk}

\oc{z}{P}{}{4}{4}{4}{\en}
\gc{g1}{2.5}{8}{7}
\gc{g2}{2.5}{6}{4}

\begin{scope}[on background layer]
\dlp{({e8}b)--({e6}b)}
\dlp{({ta}t)--({g1}br)--({g1}tl)--({eb}b)}
\dlp{({tc}t)--({g1}bl)--({g1}tr)--({ed}b)}

\alp{({tc}b)--({x}b)--({g2}bl)--({g2}tr)--({e6}t)}
\alp{ ({e8}t)--({g2}tl)--({g2}br)--({z}b)--({ta}b)}
\alp{ ({eb}t)--({y}b)--({ed}t)}
\end{scope}
\end{tikzpicture}};

\node[draw] (C5) at (9,-6){\begin{tikzpicture}

\oc{x}{M}{}{1}{4}{4}{\en}
\oc{y}{N}{}{-1}{6}{6}{\en}
\oc{e6}{E}{_6}{2}{4}{3}{\ep\dk}
\oc{ta}{\rdual{C}}{_a}{-2}{5}{4}{\etr\lt}
\oc{eb}{\rdual{E}}{_b}{-3}{6}{5}{\epr\lt}

\oc{tc}{C}{_c}{0}{5}{4}{\et\lt}
\oc{ed}{E}{_d}{1}{6}{5}{\ep\lt}
\oc{e8}{\rdual{E}}{_8}{-4}{4}{3}{\epr\dk}

\oc{z}{P}{}{-3}{4}{4}{\en}

\begin{scope}[on background layer]
\dlp{({e8}b)--({e6}b)}
\dlp{({ta}t)--({eb}b)}
\dlp{({tc}t)--({ed}b)}

\alp{ ({tc}b)--({x}b)--({e6}t)}
\alp{ ({e8}t)--({z}b)--({ta}b)}
\alp{ ({eb}t)--({y}b)--({ed}t)}
\end{scope}
\end{tikzpicture}};

\draw[<->](C1)--(C2)node [midway , fill=white] {$\gamma$};
\draw[<->](C4)--(C5)node [midway , fill=white] {$\gamma$};
\draw[->](B9)--(C1)node [midway , fill=white] {$\triangle(\triangle^*)\triangle (\triangle)^*$};
\draw[->](C2)--(C4)node [midway , fill=white] {$(P_C)^2$};
\draw[->](Y6)--(C5)node [midway , fill=white] {$(\triangle^*)\triangle $};
\end{tikzpicture}

%% file: II-comparing_sided_trace_to_symmetric.tex
\begin{tikzpicture}[
    my style/.style={
      label={right:\pgfkeysvalueof{/pgf/minimum width}},
    },
   my style/.style={%
     append after command={
       \pgfextra{\node [right] at (\tikzlastnode.mid east) {\tikzlastnode};}
     },
   },
  ]

\def\x{8}
\def\y{6}

\node[draw, ] (X4) at (\x, -6*\y){\begin{tikzpicture}

\oc{x}{M}{}{3}{4}{4}{\en}
\oc{y}{N}{}{2}{4}{4}{\en}
\oc{t2}{\rdual{C}}{_2}{0}{5}{4}{\etr}
\oc{t5}{C}{_5}{1}{5}{4}{\et\dk}
\oc{e6}{E}{_6}{5}{5}{4}{\ep\dk}
\oc{e8}{\rdual{E}}{_8}{-3}{5}{4}{\epr\dk}

\oc{z}{P}{}{-1}{4}{4}{\en}
\gc{g1}{4}{5}{4}
\gc{g2}{-2}{5}{4}

  \begin{scope}[on background layer]
\dlp{({e8}b)\gpb{g2}({t2}t)}
\dlp{({t5}t)\gpt{g1}({e6}b)}

\alp{ ({t5}b)--({y}b)--({x}b)\gpb{g1}({e6}t)}
\alp{ ({e8}t)\gpt{g2}({z}b)--({t2}b)}
\end{scope}
\end{tikzpicture}};

\node[draw, ] (X5) at (0*\x, -6*\y){\begin{tikzpicture}

\oc{x}{M}{}{3}{4}{4}{\en}
\oc{y}{N}{}{2}{4}{4}{\en}
\oc{t2}{\rdual{C}}{_2}{9}{5}{4}{\etr}
\oc{t5}{C}{_5}{1}{5}{4}{\et\dk}
\oc{e6}{E}{_6}{5}{5}{4}{\ep\dk}

\oc{e8}{\rdual{E}}{_8}{6}{5}{4}{\epr\dk}

\oc{z}{P}{}{8}{4}{4}{\en}
\gc{g1}{4}{5}{4}
\gc{g2}{7}{5}{4}
  \begin{scope}[on background layer]
\dlp{({e8}b)\gpb{g2}({t2}t)}
\dlp{({t5}t)\gpt{g1}({e6}b)}
\alp{ ({t5}b)--({y}b)--({x}b)\gpb{g1}({e6}t)}
\alp{({e8}t)\gpt{g2}({z}b)--({t2}b)}
\end{scope}
\end{tikzpicture}};

\node[draw, ] (Y3) at (0*\x, -5*\y){\begin{tikzpicture}

\oc{t1}{C}{_1}{2}{3}{2}{\et}
\oc{x}{M}{}{3}{2}{2}{\en}
\oc{y}{N}{}{2}{4}{4}{\en}
\oc{t2}{\rdual{C}}{_2}{9}{5}{4}{\etr}
\oc{t5}{C}{_5}{1}{5}{4}{\et\dk}
\oc{e6}{E}{_6}{5}{5}{2}{\ep\dk}

\oc{ex}{E}{_x}{3}{4}{3}{\ep}
\oc{e8}{\rdual{E}}{_8}{6}{5}{4}{\epr\dk}

\oc{z}{P}{}{8}{4}{4}{\en}
\gc{g1}{4}{5}{2}
\gc{g2}{7}{5}{4}
  \begin{scope}[on background layer]
\dlp{({t1}t)--({ex}b)}
\dlp{({e8}b)\gpb{g2}({t2}t)}
\dlp{({t5}t)\gpt{g1}({e6}b)}

\alp{ ({t1}b)--({x}b)\gpb{g1}({e6}t)}
\alp{ ({t5}b)--({y}b)--({ex}t)}
\alp{({e8}t)\gpt{g2}({z}b)--({t2}b)}
\end{scope}
\end{tikzpicture}};

\node[draw, ] (Y4) at (\x, -5*\y){\begin{tikzpicture}

\oc{t1}{C}{_1}{2}{3}{2}{\et}
\oc{x}{M}{}{3}{2}{2}{\en}
\oc{y}{N}{}{2}{4}{4}{\en}
\oc{t2}{\rdual{C}}{_2}{0}{5}{4}{\etr}
\oc{t5}{C}{_5}{1}{5}{4}{\et\dk}
\oc{e6}{E}{_6}{5}{5}{2}{\ep\dk}

\oc{ex}{E}{_x}{3}{4}{3}{\ep}
\oc{e8}{\rdual{E}}{_8}{-3}{5}{4}{\epr\dk}

\oc{z}{P}{}{-1}{4}{4}{\en}
\gc{g1}{4}{5}{2}
\gc{g2}{-2}{5}{4}
  \begin{scope}[on background layer]
\dlp{({t1}t)--({ex}b)}
\dlp{({e8}b)\gpb{g2}({t2}t)}
\dlp{({t5}t)\gpt{g1}({e6}b)}

\alp{ ({t1}b)--({x}b)\gpb{g1}({e6}t)}
\alp{ ({t5}b)--({y}b)--({ex}t)}
\alp{ ({e8}t)\gpt{g2}({z}b)--({t2}b)}
\end{scope}
\end{tikzpicture}};

\node[draw, ] (Y5) at (2*\x, -5*\y){\begin{tikzpicture}
\oc{t1}{C}{_1}{2}{3}{2}{\et}
\oc{x}{M}{}{3}{2}{2}{\en}
\oc{y}{N}{}{2}{4}{4}{\en}
\oc{e6}{E}{_6}{4}{2}{1}{\ep\dk}

\oc{ex}{E}{_x}{3}{4}{3}{\ep}
\oc{e8}{\rdual{E}}{_8}{0}{4}{1}{\epr\dk}

\oc{z}{P}{}{1}{4}{4}{\en}
\gc{g1}{-2}{4}{1}
\gc{g2}{3}{4}{1}
  \begin{scope}[on background layer]
\dlp{({t1}t)--({ex}b)}
\dlp{({e8}b)--({e6}b)}

\alp{ ({t1}b)--({x}b)--({e6}t)}
\alp{ ({e8}t)--({z}b)--({y}b)--({ex}t)}
\end{scope}
\end{tikzpicture}};

\node[draw, ] (Y6) at (4*\x, -6*\y){\begin{tikzpicture}
\oc{x}{M}{}{3}{4}{4}{\en}
\oc{y}{N}{}{2}{4}{4}{\en}
\oc{e6}{E}{_6}{4}{4}{3}{\ep\dk}

\oc{e8}{\rdual{E}}{_8}{0}{4}{3}{\epr\dk}

\oc{z}{P}{}{1}{4}{4}{\en}
\gc{g1}{-2}{4}{1}
\gc{g2}{3}{4}{1}
  \begin{scope}[on background layer]
\dlp{({e8}b)--({e6}b)}

\alp{ ({e8}t)--({z}b)--({y}b)--({x}b)--({e6}t)}
\end{scope}
\end{tikzpicture}};

\node[draw, ] (Z2) at (0, -1*\y){\begin{tikzpicture}

\oc{t1}{C}{_1}{0}{5}{4}{\et}
\oc{x}{M}{}{1}{4}{4}{\en}
\oc{y}{N}{}{2}{4}{4}{\en}
\oc{t2}{\rdual{C}}{_2}{8}{5}{4}{\etr}

\oc{ex}{E}{_x}{4}{5}{4}{\ep}
\oc{e8}{\rdual{E}}{_8}{5}{5}{4}{\epr\dk}

\oc{z}{P}{}{7}{4}{4}{\en}
\gc{g1}{3}{5}{4}
\gc{g2}{6}{5}{4}
  \begin{scope}[on background layer]
\dlp{ ({t1}t)\gpt{g1}({ex}b)}
\dlp{({e8}b)\gpb{g2}({t2}t)}

\alp{ ({t1}b)--({x}b)--({y}b)\gpb{g1}({ex}t)}
\alp{ ({e8}t)\gpt{g2}({z}b)--({t2}b)}
\end{scope}
\end{tikzpicture}};

\node[draw, ] (Z3) at (0, -3*\y){\begin{tikzpicture}

\oc{t1}{C}{_1}{1}{5}{4}{\et}
\oc{x}{M}{}{2}{4}{4}{\en}
\oc{y}{N}{}{3}{2}{2}{\en}
\oc{t2}{\rdual{C}}{_2}{9}{5}{4}{\etr}
\oc{t5}{C}{_5}{2}{3}{2}{\et\dk}
\oc{e6}{E}{_6}{3}{4}{3}{\ep\dk}

\oc{ex}{E}{_x}{5}{5}{2}{\ep}
\oc{e8}{\rdual{E}}{_8}{6}{5}{4}{\epr\dk}

\oc{z}{P}{}{8}{4}{4}{\en}
\gc{g1}{4}{5}{2}
\gc{g2}{7}{5}{4}
  \begin{scope}[on background layer]
\dlp{ ({t1}t)\gpt{g1}({ex}b)}
\dlp{({e8}b)\gpb{g2}({t2}t)}
\dlp{({t5}t)--({e6}b)}

\alp{ ({t1}b)--({x}b)--({e6}t)}
\alp{({t5}b)--({y}b)\gpb{g1}({ex}t)}
\alp{ ({e8}t)\gpt{g2}({z}b)--({t2}b)}
\end{scope}
\end{tikzpicture}};

\node[draw, ] (Z4) at (\x, -3*\y){\begin{tikzpicture}

\oc{t1}{C}{_1}{-2}{5}{4}{\et}
\oc{x}{M}{}{-1}{4}{4}{\en}
\oc{y}{N}{}{0}{2}{2}{\en}
\oc{t2}{\rdual{C}}{_2}{-2}{3}{2}{\etr}
\oc{t5}{C}{_5}{-1}{3}{2}{\et\dk}
\oc{e6}{E}{_6}{0}{4}{3}{\ep\dk}

\oc{ex}{E}{_x}{2}{5}{2}{\ep}
\oc{e8}{\rdual{E}}{_8}{-5}{3}{2}{\epr\dk}

\oc{z}{P}{}{-3}{2}{2}{\en}
\gc{g1}{1}{5}{2}
\gc{g2}{-4}{3}{2}
  \begin{scope}[on background layer]
\dlp{({t1}t)\gpt{g1}({ex}b)}
\dlp{({e8}b)\gpb{g2}({t2}t)}
\dlp{({t5}t)--({e6}b)}

\alp{ ({t1}b)--({x}b)--({e6}t)}
\alp{ ({t5}b)--({y}b)\gpb{g1}({ex}t)}
\alp{ ({e8}t)\gpt{g2}({z}b)--({t2}b)}
\end{scope}
\end{tikzpicture}};

\node[draw, ] (Z5) at (2*\x, -4*\y){\begin{tikzpicture}

\oc{t1}{C}{_1}{-2}{5}{4}{\et}
\oc{x}{M}{}{-1}{4}{4}{\en}
\oc{y}{N}{}{0}{2}{2}{\en}
\oc{e6}{E}{_6}{0}{4}{3}{\ep\dk}

\oc{ex}{E}{_x}{2}{5}{2}{\ep}
\oc{e8}{\rdual{E}}{_8}{-3}{3}{2}{\epr\dk}

\oc{z}{P}{}{-1}{2}{2}{\en}
\gc{g1}{1}{5}{2}
\gc{g2}{-2}{3}{2}
  \begin{scope}[on background layer]
\dlp{ ({t1}t)\gpt{g1}({ex}b)}
\dlp{({e8}b)\gpb{g2}({e6}b)}

\alp{ ({t1}b)--({x}b)--({e6}t)}
\alp{ ({e8}t)\gpt{g2}({z}b)--({y}b)\gpb{g1}({ex}t)}
\end{scope}

\end{tikzpicture}};

\node[draw, ] (A3) at (.75*\x, -2*\y){\begin{tikzpicture}

\oc{t1}{C}{_1}{-2}{5}{4}{\et}
\oc{x}{M}{}{-1}{4}{4}{\en}
\oc{y}{N}{}{2.5}{4}{4}{\en}
\oc{t2}{\rdual{C}}{_2}{10}{5}{4}{\etr}
\oc{t5}{C}{_5}{-1}{3}{2}{\et\dk}
\oc{e6}{E}{_6}{0}{4}{3}{\ep\dk}
\oc{ta}{\rdual{C}}{_a}{2}{3}{2}{\etr\lt}
\oc{eb}{\rdual{E}}{_b}{1}{4}{3}{\epr\lt}

\oc{tc}{C}{_c}{3}{3}{2}{\et\lt}
\oc{ed}{E}{_d}{4}{4}{3}{\ep\lt}

\oc{ex}{E}{_x}{6}{5}{2}{\ep}
\oc{e8}{\rdual{E}}{_8}{7}{5}{4}{\epr\dk}

\oc{z}{P}{}{9}{4}{4}{\en}
\gc{g1}{5}{5}{2}
\gc{g2}{8}{5}{4}
  \begin{scope}[on background layer]
\dlp{ ({t1}t)\gpt{g1}({ex}b)}
\dlp{({e8}b)\gpb{g2}({t2}t)}
\dlp{({t5}t)--({e6}b)}
\dlp{({ta}t)--({eb}b)}
\dlp{({tc}t)--({ed}b)}

\alp{ ({t1}b)--({x}b)--({e6}t)}
\alp{ ({t5}b)--({ta}b)}
\alp{({eb}t)--({y}b)--({ed}t)}
\alp{({tc}b)\gpb{g1}({ex}t)}
\alp{ ({e8}t)\gpt{g2}({z}b)--({t2}b)}
\end{scope}
\end{tikzpicture}};

\node[draw, ] (A4) at (2.25*\x, -3*\y){\begin{tikzpicture}

\oc{t1}{C}{_1}{-2}{5}{4}{\et}
\oc{x}{M}{}{-1}{4}{4}{\en}
\oc{y}{N}{}{2.5}{4}{4}{\en}
\oc{t2}{\rdual{C}}{_2}{-2}{3}{2}{\etr}
\oc{t5}{C}{_5}{-1}{3}{2}{\et\dk}
\oc{e6}{E}{_6}{0}{4}{3}{\ep\dk}
\oc{ta}{\rdual{C}}{_a}{2}{3}{2}{\etr\lt}
\oc{eb}{\rdual{E}}{_b}{1}{4}{3}{\epr\lt}

\oc{tc}{C}{_c}{3}{3}{2}{\et\lt}
\oc{ed}{E}{_d}{4}{4}{3}{\ep\lt}

\oc{ex}{E}{_x}{6}{5}{2}{\ep}
\oc{e8}{\rdual{E}}{_8}{-5}{3}{2}{\epr\dk}

\oc{z}{P}{}{-3}{2}{2}{\en}
\gc{g1}{5}{5}{2}
\gc{g2}{-4}{3}{2}
  \begin{scope}[on background layer]
\dlp{ ({t1}t)\gpt{g1}({ex}b)}
\dlp{({e8}b)\gpb{g2}({t2}t)}
\dlp{({t5}t)--({e6}b)}
\dlp{({ta}t)--({eb}b)}
\dlp{({tc}t)--({ed}b)}

\alp{({t1}b)--({x}b)--({e6}t)}
\alp{ ({t5}b)--({ta}b)}
\alp{ ({eb}t)--({y}b)--({ed}t)}
\alp{({tc}b)\gpb{g1}({ex}t)}
\alp{ ({e8}t)\gpt{g2}({z}b)--({t2}b)}
\end{scope}
\end{tikzpicture}};

\node[draw, ] (A5) at (3*\x, -4*\y){\begin{tikzpicture}

\oc{t1}{C}{_1}{-2}{5}{4}{\et}
\oc{x}{M}{}{-1}{4}{4}{\en}
\oc{y}{N}{}{2.5}{4}{4}{\en}
\oc{e6}{E}{_6}{0}{4}{3}{\ep\dk}
\oc{ta}{\rdual{C}}{_a}{2}{3}{2}{\etr\lt}
\oc{eb}{\rdual{E}}{_b}{1}{4}{3}{\epr\lt}

\oc{tc}{C}{_c}{3}{3}{2}{\et\lt}
\oc{ed}{E}{_d}{4}{4}{3}{\ep\lt}

\oc{ex}{E}{_x}{6}{5}{2}{\ep}
\oc{e8}{\rdual{E}}{_8}{-2}{3}{2}{\epr\dk}

\oc{z}{P}{}{1}{2}{2}{\en}
\gc{g1}{5}{5}{2}
\gc{g2}{-1}{3}{2}
  \begin{scope}[on background layer]
\dlp{ ({t1}t)\gpt{g1}({ex}b)}
\dlp{({e8}b)\gpb{g2}({e6}b)}
\dlp{({ta}t)--({eb}b)}
\dlp{({tc}t)--({ed}b)}

\alp{ ({t1}b)--({x}b)--({e6}t)}
\alp{ ({eb}t)--({y}b)--({ed}t)}
\alp{({tc}b)\gpb{g1}({ex}t)}
\alp{ ({e8}t)\gpt{g2}({z}b)--({ta}b)}
\end{scope}
\end{tikzpicture}};

\node[draw, ] (A6) at (3*\x, -5*\y){\begin{tikzpicture}

\oc{t1}{C}{_1}{2}{1}{0}{\et}
\oc{x}{M}{}{3}{0}{0}{\en}
\oc{y}{N}{}{2.5}{4}{4}{\en}
\oc{e6}{E}{_6}{4}{0}{-1}{\ep\dk}
\oc{ta}{\rdual{C}}{_a}{2}{3}{2}{\etr\lt}
\oc{eb}{\rdual{E}}{_b}{1}{4}{3}{\epr\lt}

\oc{tc}{C}{_c}{3}{3}{2}{\et\lt}
\oc{ed}{E}{_d}{4}{4}{3}{\ep\lt}

\oc{ex}{E}{_x}{4}{2}{1}{\ep}
\oc{e8}{\rdual{E}}{_8}{0}{2}{-1}{\epr\dk}

\oc{z}{P}{}{1}{2}{2}{\en}
\gc{g1}{-2}{4}{1}
\gc{g2}{3}{4}{1}
  \begin{scope}[on background layer]
\dlp{ ({t1}t)--({ex}b)}
\dlp{({e8}b)--({e6}b)}
\dlp{({ta}t)--({eb}b)}
\dlp{({tc}t)--({ed}b)}

\alp{ ({t1}b)--({x}b)--({e6}t)}
\alp{ ({eb}t)--({y}b)--({ed}t)}
\alp{ ({tc}b)--({ex}t)}
\alp{ ({e8}t)--({z}b)--({ta}b)}
\end{scope}
\end{tikzpicture}};

\node[draw, ] (B9) at (0, 0){\begin{tikzpicture}
\oc{t1}{C}{_1}{0}{5}{4}{\et}
\oc{x}{M}{}{1}{4}{4}{\en}
\oc{y}{N}{}{2.5}{4}{4}{\en}
\oc{t2}{\rdual{C}}{_2}{5}{5}{4}{\etr}

\oc{z}{P}{}{4}{4}{4}{\en}
\gc{g1}{-2}{4}{1}
\gc{g2}{3}{4}{1}
  \begin{scope}[on background layer]
\dlp{({t1}t)--({t2}t)}

\alp{ ({t1}b)--({x}b)--({y}b)--({z}b)--({t2}b)}
\end{scope}
\end{tikzpicture}};

\node[draw, ] (B1) at (\x, -1*\y){\begin{tikzpicture}
\oc{t1}{C}{_1}{-2}{5}{4}{\et}
\oc{x}{M}{}{-1}{4}{4}{\en}
\oc{y}{N}{}{2.5}{4}{4}{\en}
\oc{t2}{\rdual{C}}{_2}{5}{5}{2}{\etr}
\oc{t5}{C}{_5}{-1}{3}{2}{\et\dk}
\oc{e6}{E}{_6}{0}{4}{3}{\ep\dk}
\oc{ta}{\rdual{C}}{_a}{2}{3}{2}{\etr\lt}
\oc{eb}{\rdual{E}}{_b}{1}{4}{3}{\epr\lt}

\oc{tc}{C}{_c}{3}{3}{2}{\et\lt}
\oc{ed}{E}{_d}{4}{4}{3}{\ep\lt}

\oc{z}{P}{}{4}{2}{2}{\en}
\gc{g1}{-2}{4}{1}
\gc{g2}{3}{4}{1}
  \begin{scope}[on background layer]
\dlp{({t1}t)--({t2}t)}
\dlp{({t5}t)--({e6}b)}
\dlp{({ta}t)--({eb}b)}
\dlp{({tc}t)--({ed}b)}

\alp{ ({t1}b)--({x}b)--({e6}t)}
\alp{ ({t5}b)--({ta}b)}
\alp{ ({eb}t)--({y}b)--({ed}t)}
\alp{ ({tc}b)--({z}b)--({t2}b)}
\end{scope}
\end{tikzpicture}};

\node[draw, ] (B2) at (2*\x, -1*\y){\begin{tikzpicture}
\oc{t1}{C}{_1}{0}{5}{4}{\et}
\oc{x}{M}{}{1}{4}{4}{\en}
\oc{y}{N}{}{2.5}{9}{9}{\en}
\oc{t2}{\rdual{C}}{_2}{5}{5}{4}{\etr}
\oc{t5}{C}{_5}{6}{5}{4}{\et\dk}
\oc{e6}{E}{_6}{7}{6}{5}{\ep\dk}
\oc{ta}{\rdual{C}}{_a}{8}{7}{4}{\etr\lt}
\oc{eb}{\rdual{E}}{_b}{1}{9}{8}{\epr\lt}

\oc{tc}{C}{_c}{0}{7}{6}{\et\lt}
\oc{ed}{E}{_d}{4}{9}{8}{\ep\lt}

\oc{z}{P}{}{4}{4}{4}{\en}
\gc{g1}{2.5}{8}{7}
\gc{g2}{2.5}{6}{4}
  \begin{scope}[on background layer]
\dlp{ ({t1}t)--({t2}t)}
\dlp{({t5}t)--({e6}b)}
\dlp{({ta}t)--({g1}br)--({g1}tl)--({eb}b)}
\dlp{({tc}t)--({g1}bl)--({g1}tr)--({ed}b)}

\alp{ ({t1}b)--({x}b)--({g2}bl)--({g2}tr)--({e6}t)}
\alp{ ({t5}b)--({ta}b)}
\alp{ ({eb}t)--({y}b)--({ed}t)}
\alp{ ({tc}b)--({g2}tl)--({g2}br)--({z}b)--({t2}b)}
\end{scope}
\end{tikzpicture}};

\node[draw, ] (B3) at (2*\x, -2*\y){\begin{tikzpicture}
\oc{t1}{C}{_1}{-1}{5}{4}{\et}
\oc{x}{M}{}{1}{4}{4}{\en}
\oc{y}{N}{}{2.5}{9}{9}{\en}
\oc{t2}{\rdual{C}}{_2}{5}{5}{4}{\etr}
\oc{t5}{C}{_5}{6}{5}{4}{\et\dk}
\oc{e6}{E}{_6}{7}{6}{5}{\ep\dk}
\oc{ta}{\rdual{C}}{_a}{8}{7}{4}{\etr\lt}
\oc{eb}{\rdual{E}}{_b}{1}{9}{8}{\epr\lt}

\oc{tc}{C}{_c}{-1}{7}{6}{\et\lt}
\oc{ed}{E}{_d}{4}{9}{8}{\ep\lt}

\oc{ex}{E}{_x}{0}{6}{5}{\ep}
\oc{e8}{\rdual{E}}{_8}{1}{6}{5}{\epr\dk}

\oc{z}{P}{}{4}{4}{4}{\en}
\gc{g1}{2.5}{8}{7}
\gc{g2}{2.5}{6}{4}
  \begin{scope}[on background layer]
\dlp{({t1}t)--({ex}b)}
\dlp{({e8}b)--({t2}t)}
\dlp{({t5}t)--({e6}b)}
\dlp{({ta}t)--({g1}br)--({g1}tl)--({eb}b)}
\dlp{({tc}t)--({g1}bl)--({g1}tr)--({ed}b)}

\alp{ ({t1}b)--({x}b)--({g2}bl)--({g2}tr)--({e6}t)}
\alp{ ({t5}b)--({ta}b)}
\alp{ ({eb}t)--({y}b)--({ed}t)}
\alp{ ({tc}b)--({ex}t)}
\alp{ ({e8}t)--({g2}tl)--({g2}br)--({z}b)--({t2}b)}
\end{scope}
\end{tikzpicture}};

\node[draw, ] (B4) at (3*\x, -2*\y){\begin{tikzpicture}
\oc{t1}{C}{_1}{-1}{5}{4}{\et}
\oc{x}{M}{}{1}{4}{4}{\en}
\oc{y}{N}{}{2.5}{9}{9}{\en}
\oc{e6}{E}{_6}{4}{6}{5}{\ep\dk}
\oc{ta}{\rdual{C}}{_a}{5}{7}{4}{\etr\lt}
\oc{eb}{\rdual{E}}{_b}{1}{9}{8}{\epr\lt}

\oc{tc}{C}{_c}{-1}{7}{6}{\et\lt}
\oc{ed}{E}{_d}{4}{9}{8}{\ep\lt}

\oc{ex}{E}{_x}{0}{6}{5}{\ep}
\oc{e8}{\rdual{E}}{_8}{1}{6}{5}{\epr\dk}

\oc{z}{P}{}{4}{4}{4}{\en}
\gc{g1}{2.5}{8}{7}
\gc{g2}{2.5}{6}{4}
  \begin{scope}[on background layer]
\dlp{ ({t1}t)--({ex}b)}
\dlp{({e8}b)--({e6}b)}
\dlp{({ta}t)--({g1}br)--({g1}tl)--({eb}b)}
\dlp{({tc}t)--({g1}bl)--({g1}tr)--({ed}b)}

\alp{ ({t1}b)--({x}b)--({g2}bl)--({g2}tr)--({e6}t)}

\alp{ ({eb}t)--({y}b)--({ed}t)}
\alp{ ({tc}b)--({ex}t)}
\alp{({e8}t)--({g2}tl)--({g2}br)--({z}b)--({ta}b)}
\end{scope}
\end{tikzpicture}};

\node[draw, ] (C1) at (1.25*\x, 0){\begin{tikzpicture}
\oc{t1}{C}{_1}{-2}{5}{4}{\et}
\oc{x}{M}{}{-1}{4}{4}{\en}
\oc{y}{N}{}{2.5}{4}{4}{\en}
\oc{t2}{\rdual{C}}{_2}{7}{5}{4}{\etr}
\oc{t5}{C}{_5}{-1}{3}{2}{\et\dk}
\oc{e6}{E}{_6}{0}{4}{3}{\ep\dk}
\oc{ta}{\rdual{C}}{_a}{2}{3}{2}{\etr\lt}
\oc{eb}{\rdual{E}}{_b}{1}{4}{3}{\epr\lt}

\oc{tc}{C}{_c}{3}{3}{2}{\et\lt}
\oc{ed}{E}{_d}{4}{4}{3}{\ep\lt}
\oc{t7}{\rdual{C}}{_7}{6}{3}{2}{\etr\dk}
\oc{e8}{\rdual{E}}{_8}{5}{4}{3}{\epr\dk}

\oc{z}{P}{}{6}{4}{4}{\en}
\gc{g1}{-2}{4}{1}
\gc{g2}{3}{4}{1}
  \begin{scope}[on background layer]
\dlp{ ({t1}t)--({t2}t)}
\dlp{({t5}t)--({e6}b)}
\dlp{({ta}t)--({eb}b)}
\dlp{({tc}t)--({ed}b)}
\dlp{({t7}t)--({e8}b)}

\alp{ ({t1}b)--({x}b)--({e6}t)}
\alp{ ({t5}b)--({ta}b)}
\alp{ ({eb}t)--({y}b)--({ed}t)}
\alp{ ({tc}b)--({t7}b)}
\alp{ ({e8}t)--({z}b)--({t2}b)}
\end{scope}
\end{tikzpicture}};

\node[draw, ] (C2) at (3*\x, 0){\begin{tikzpicture}
\oc{t1}{C}{_1}{0}{5}{4}{\et}
\oc{x}{M}{}{1}{4}{4}{\en}
\oc{y}{N}{}{2.5}{9}{9}{\en}
\oc{t2}{\rdual{C}}{_2}{5}{5}{4}{\etr}
\oc{t5}{C}{_5}{6}{5}{4}{\et\dk}
\oc{e6}{E}{_6}{7}{6}{5}{\ep\dk}
\oc{ta}{\rdual{C}}{_a}{8}{7}{4}{\etr\lt}
\oc{eb}{\rdual{E}}{_b}{1}{9}{8}{\epr\lt}

\oc{tc}{C}{_c}{-3}{7}{4}{\et\lt}
\oc{ed}{E}{_d}{4}{9}{8}{\ep\lt}
\oc{t7}{\rdual{C}}{_7}{-1}{5}{4}{\etr\dk}
\oc{e8}{\rdual{E}}{_8}{-2}{6}{5}{\epr\dk}

\oc{z}{P}{}{4}{4}{4}{\en}
\gc{g1}{2.5}{8}{7}
\gc{g2}{2.5}{6}{4}
  \begin{scope}[on background layer]
\dlp{ ({t1}t)--({t2}t)}
\dlp{({t5}t)--({e6}b)}
\dlp{({ta}t)--({g1}br)--({g1}tl)--({eb}b)}
\dlp{({tc}t)--({g1}bl)--({g1}tr)--({ed}b)}
\dlp{({t7}t)--({e8}b)}

\alp{ ({t1}b)--({x}b)--({g2}bl)--({g2}tr)--({e6}t)}
\alp{ ({t5}b)--({ta}b)}
\alp{ ({eb}t)--({y}b)--({ed}t)}
\alp{ ({tc}b)--({t7}b)}
\alp{ ({e8}t)--({g2}tl)--({g2}br)--({z}b)--({t2}b)}
\end{scope}
\end{tikzpicture}};

\node[draw, ] (C3) at (3*\x, -1*\y){\begin{tikzpicture}
\oc{x}{M}{}{1}{4}{4}{\en}
\oc{y}{N}{}{2.5}{9}{9}{\en}
\oc{t2}{\rdual{C}}{_2}{5}{5}{4}{\etr}
\oc{t5}{C}{_5}{6}{5}{4}{\et\dk}
\oc{e6}{E}{_6}{7}{6}{5}{\ep\dk}
\oc{ta}{\rdual{C}}{_a}{8}{7}{4}{\etr\lt}
\oc{eb}{\rdual{E}}{_b}{1}{9}{8}{\epr\lt}

\oc{tc}{C}{_c}{0}{7}{4}{\et\lt}
\oc{ed}{E}{_d}{4}{9}{8}{\ep\lt}
\oc{e8}{\rdual{E}}{_8}{1}{6}{5}{\epr\dk}

\oc{z}{P}{}{4}{4}{4}{\en}
\gc{g1}{2.5}{8}{7}
\gc{g2}{2.5}{6}{4}
  \begin{scope}[on background layer]
\dlp{({t5}t)--({e6}b)}
\dlp{({ta}t)--({g1}br)--({g1}tl)--({eb}b)}
\dlp{({tc}t)--({g1}bl)--({g1}tr)--({ed}b)}
\dlp{({t2}t)--({e8}b)}

\alp{ ({tc}b)--({x}b)--({g2}bl)--({g2}tr)--({e6}t)}
\alp{ ({t5}b)--({ta}b)}
\alp{ ({eb}t)--({y}b)--({ed}t)}
\alp{ ({e8}t)--({g2}tl)--({g2}br)--({z}b)--({t2}b)}
\end{scope}
\end{tikzpicture}};

\node[draw, ] (C4) at (4*\x, -1*\y){\begin{tikzpicture}

\oc{x}{M}{}{1}{4}{4}{\en}
\oc{y}{N}{}{2.5}{9}{9}{\en}
\oc{e6}{E}{_6}{4}{6}{5}{\ep\dk}
\oc{ta}{\rdual{C}}{_a}{5}{7}{4}{\etr\lt}
\oc{eb}{\rdual{E}}{_b}{1}{9}{8}{\epr\lt}

\oc{tc}{C}{_c}{0}{7}{4}{\et\lt}
\oc{ed}{E}{_d}{4}{9}{8}{\ep\lt}
\oc{e8}{\rdual{E}}{_8}{1}{6}{5}{\epr\dk}

\oc{z}{P}{}{4}{4}{4}{\en}
\gc{g1}{2.5}{8}{7}
\gc{g2}{2.5}{6}{4}
  \begin{scope}[on background layer]
\dlp{({e8}b)--({e6}b)}
\dlp{({ta}t)--({g1}br)--({g1}tl)--({eb}b)}
\dlp{({tc}t)--({g1}bl)--({g1}tr)--({ed}b)}

\alp{ ({tc}b)--({x}b)--({g2}bl)--({g2}tr)--({e6}t)}
\alp{ ({e8}t)--({g2}tl)--({g2}br)--({z}b)--({ta}b)}
\alp{ ({eb}t)--({y}b)--({ed}t)}
\end{scope}
\end{tikzpicture}};

\node[draw, ] (C5) at (4*\x,-5*\y){\begin{tikzpicture}

\oc{x}{M}{}{1}{4}{4}{\en}
\oc{y}{N}{}{-1}{6}{6}{\en}
\oc{e6}{E}{_6}{2}{4}{3}{\ep\dk}
\oc{ta}{\rdual{C}}{_a}{-2}{5}{4}{\etr\lt}
\oc{eb}{\rdual{E}}{_b}{-3}{6}{5}{\epr\lt}

\oc{tc}{C}{_c}{0}{5}{4}{\et\lt}
\oc{ed}{E}{_d}{1}{6}{5}{\ep\lt}
\oc{e8}{\rdual{E}}{_8}{-4}{4}{3}{\epr\dk}

\oc{z}{P}{}{-3}{4}{4}{\en}
  \begin{scope}[on background layer]
\dlp{({e8}b)--({e6}b)}
\dlp{({ta}t)--({eb}b)}
\dlp{({tc}t)--({ed}b)}

\alp{ ({tc}b)--({x}b)--({e6}t)}
\alp{ ({e8}t)--({z}b)--({ta}b)}
\alp{ ({eb}t)--({y}b)--({ed}t)}
\end{scope}
\end{tikzpicture}};

\draw[->](B1)--(A3) node [midway , fill=white] {$C_E$};
\draw[<->](A3)--(B3)node [midway , fill=white] {$\gamma$};
\draw[->](B1)--(C1)node [midway , fill=white] {$\triangle^*$};
\draw[->](B2)--(C2)node [midway , fill=white] {$\triangle^*$};
\draw[->](B3)--(C3)node [midway , fill=white] {$\triangle^{-1}$};
\draw[->](B4)--(C4)node [midway , fill=white] {$\triangle^{-1}$};
\draw[<->](B1)--(B2)node [midway , fill=white] {$\gamma$};
\draw[->](B2)--(B3)node [midway , fill=white] {$C_E$};
\draw[->](B3)--(B4)node [midway , fill=white] {$P_C$};
\draw[<->](C1)--(C2)node [midway , fill=white] {$\gamma$};
\draw[->](C2)--(C3)node [midway , fill=white] {$P_C$};
\draw[->](C3)--(C4)node [midway , fill=white] {$P_C$};
\draw[<->](C4)--(C5)node [midway , fill=white] {$\gamma$};
\draw[<->](A3)--(A4)node [midway , fill=white] {$\gamma$};
\draw[->](A4)--(A5)node [midway , fill=white] {$P_C$};
\draw[<->](A5)--(B4)node [midway , fill=white] {$\gamma$};
\draw[->](B9)--(Z2)node [midway , fill=white] {$C_E$};

\draw[->](B9)--(B1)node [midway , fill=white] {$\triangle(\triangle^*)\triangle$};
\draw[->](B9)--(C1)node [midway , fill=white] {$\triangle(\triangle^*)\triangle(\triangle^*)$};
\draw[->](Z2)--(A3)node [midway , fill=white] {$\triangle(\triangle^*)\triangle$};
\draw[->](Z2)--(Z3)node [midway , fill=white] {$\triangle$};
\draw[->](Z3)--(A3)node [midway , fill=white] {$(\triangle^\ast)\triangle$};
\draw[<->](Z3)--(Z4)node [midway , fill=white] {$\gamma$};
\draw[->](Z4)--(A4)node [midway , fill=white] {$\triangle(\triangle^*)$};
\draw[->](Z4)--(Z5)node [midway , fill=white] {$P_C$};
\draw[->](Z5)--(A5)node [midway , fill=white] {$(\triangle^*)\triangle$};
\draw[->](C2)to [out =0, in =90, looseness =.5] node [midway , fill=white] {$(P_C)^2$}(C4);
\draw[<->](Z3)--(Y3)node [midway , fill=white] {$\gamma$};
\draw[<->](Y3)to node [midway , fill=white] {$\gamma$}(Y4);
\draw[<->](Z4)--(Y4)node [midway , fill=white] {$\gamma$};
\draw[->](Y4)--(Y5)node [midway , fill=white] {$P_C$};
\draw[<->](Z5)--(Y5)node [midway , fill=white] {$\gamma$};
\draw[->](Y5)--(A6)node [midway , fill=white] {$(\triangle^*)\triangle$};
\draw[<->](A5)--(A6)node [midway , fill=white] {$\gamma$};
\draw[->](A6)--(C5)node [midway , fill=white] {$\triangle^{-1}$};
\draw[<->](A6)to [out = 0, in = 0, looseness = .5]node [midway , fill=white] {$\gamma$}(B4);
\draw[->](Y5)--(Y6)node [midway , fill=white] {$\triangle^{-1}$};
\draw[->](Y6)--(C5)node [midway , fill=white] {$(\triangle^*)\triangle$};
\draw[->](Y4)--(X4)node [midway , fill=white] {$\triangle^{-1}$};
\draw[->](X4)--(Y6)node [midway , fill=white] {$P_C$};
\draw[->](Y3)--(X5)node [midway , fill=white] {$\triangle^{-1}$};
\draw[<->](X5)--(X4)node [midway , fill=white] {$\gamma$};
\end{tikzpicture}